\documentclass[11pt, reqno, oneside]{amsart}
\usepackage[lmargin=1in,rmargin=1in, bmargin=1in, tmargin=1in]{geometry}

\usepackage{amsmath,amsthm,amssymb}
\usepackage{tikz-cd}
\usepackage{tikz}
\usepackage{todonotes}
\usepackage{mathrsfs}
\usepackage{extarrows}
\usepackage{graphicx}
\usepackage{mathtools}
\usepackage{adjustbox}
\usepackage{colonequals}
\usepackage{stmaryrd}
\usepackage{macros}
\usepackage{mathdots}
\usepackage{xypic}
\newcommand{\sar}[2]{\ar@{}[#1]|-*[@]{#2}}

\usepackage[pagebackref, hidelinks]{hyperref}
\usepackage{IEEEtrantools}
\usepackage[utf8]{inputenc}
\usepackage[english]{babel}
\usepackage{comment}
\usepackage{csquotes}
\usepackage[shortlabels]{enumitem}
\usepackage{bbm}

\newcommand{\Unloz}{U(p^n)^\lozenge}
\newcommand{\Uncirc}{U(p^n)^\circ}
\newcommand{\Unone}{U_1(p^n)}
\newcommand{\Un}{U(p^n)}

\newcommand{\Pmu}{\mathscr{P}(\mu)} 

\newcommand{\QQ}{\mathbb{Q}}
\renewcommand{\tilde}{\widetilde}

\newcommand{\Iw}{\opn{Iw}}

\numberwithin{equation}{subsection}

\newtheorem{theorem}[equation]{Theorem}
\newtheorem{proposition}[equation]{Proposition}
\newtheorem{lemma}[equation]{Lemma}
\newtheorem{corollary}[equation]{Corollary}
\newtheorem*{corollary*}{Corollary}
\newtheorem{conjecture}[equation]{Conjecture}
\newtheorem{assumption}[equation]{Assumption}
\newtheorem*{assumption*}{Assumption}
\newtheorem{claim}[equation]{Claim}

\newcounter{alphalabels}

\newtheorem{theoremx}[alphalabels]{Theorem}

\theoremstyle{definition}

\newtheorem{definition}[equation]{Definition}
\newtheorem{notation}[equation]{Notation}
\newtheorem*{notation*}{Notation}

\newtheorem*{convention*}{Convention}

\newtheorem{test-data}[equation]{Test data}

\theoremstyle{remark}
\newtheorem{remark}[equation]{Remark}

\DeclareFontFamily{U}{wncy}{}
\DeclareFontShape{U}{wncy}{m}{n}{<->wncyr10}{}
\DeclareSymbolFont{mcy}{U}{wncy}{m}{n}
\DeclareMathSymbol{\sha}{\mathord}{mcy}{"58}

\newcommand{\ide}[1]{\mathfrak{#1}}
\newcommand{\mbb}[1]{\mathbb{#1}}

\newcommand{\opn}[1]{\operatorname{#1}}

\newcommand{\mbf}[1]{\mathbf{#1}}
\newcommand{\tbyt}[4]{\left( \begin{array}{cc} #1 & #2 \\ #3 & #4 \end{array} \right)}

\newcommand{\smat}[1]{\left(\begin{smallmatrix} #1 \end{smallmatrix}\right)}

\newcommand{\mat}[1]{\left(\begin{matrix} #1 \end{matrix}\right)}

\renewcommand{\cL}{\mathcal{L}}

\newcommand{\Addresses}{{
  \bigskip
  \footnotesize

  (Barrera Salazar) \textsc{Universidad de Santiago de Chile, Avenida Libertador Bernardo O'Higgins no. 3363, Estaci\'{o}n Central, Santiago, Chile}\par\nopagebreak
  \textit{E-mail address}: \texttt{daniel.barrera.s@usach.cl}

  \medskip

  (Graham) \textsc{Mathematical Institute, University of Oxford, Woodstock Road, Oxford OX2 6GG, United Kingdom}\par\nopagebreak
  \textit{E-mail address}: \texttt{andrew.graham@maths.ox.ac.uk}

  \medskip

  (Williams) \textsc{University of Nottingham, University Park, Nottingham, NG7 2RD, United Kingdom}\par\nopagebreak
  \textit{E-mail address}: \texttt{chris.williams1@nottingham.ac.uk}

}}

\newcommand{\Qpb}{\overline{\QQ}_p}

\title[ ]{Local-global compatibility and the exceptional zero conjecture for $\opn{GL}(3)$}
\author{Daniel Barrera Salazar, Andrew Graham, and Chris Williams}
\date{}

\begin{document}

\begin{abstract}
We prove exceptional zero conjectures for $p$-ordinary regular algebraic cuspidal automorphic representations of $\opn{GL}_3(\mbb{A})$ which are Steinberg at $p$. We make no self-duality assumptions. 

The paper has two parts. In Part 1, we use $p$-arithmetic cohomology to unconditionally prove an automorphic exceptional zero conjecture in this setting, using Gehrmann's automorphic $\mathcal{L}$-invariant. 

In Part 2 we prove, under mild assumptions that are expected to always hold, the equality of automorphic and Fontaine--Mazur $\mathcal{L}$-invariants, and thus deduce cases of the full Greenberg--Benois exceptional zero conjecture. As one of the key ingredients for this, we establish local-global compatibility at $\ell = p$ for Galois representations attached to $p$-ordinary torsion classes for $\opn{GL}_n$, confirming a conjecture of Hansen in this setting. We prove this for all $n$ following the strategy in the ``10-author paper'', and use the $n=3$ case to deduce the desired equality of $\mathcal{L}$-invariants.

\end{abstract}

\maketitle

\footnotesize
\tableofcontents
\normalsize

\section{Introduction}

There are many conjectures connecting arithmetic information to analytic properties of $L$-functions, such as the Birch--Swinnerton-Dyer (BSD) and Bloch--Kato conjectures, which concern the relationship between the size of Mordell--Weil/Selmer groups and complex $L$-values. A profitable approach to studying this relationship has been to use $p$-adic methods to prove $p$-parts of these conjectures. For example, if $E/\Q$ is an elliptic curve with (ordinary) semistable reduction at $p$, then there is a $p$-adic $L$-function $L_p(E,s)$ attached to $E$, related to the complex $L$-function via
\begin{equation}\label{eq:interpolation intro}
L_p(E,1) = e_p(E,\mathbbm{1}) \cdot L^{(p)}(E,1)/(\text{period}),
\end{equation}
where $e_p(-)$ is Coates--Perrin-Riou's modified Euler factor at $p$, and $L^{(p)}$ denotes the complex $L$-function with the Euler factor at $p$ removed. Via the ($p$-adic) Iwasawa Main Conjecture \cite{SU14,SkinnerSplitMult}, the analytic properties of $L_p(E,s)$ can be related to arithmetic invariants for $E$. As a consequence, if $e_p(E,\mathbbm{1}) \neq 0$, then \eqref{eq:interpolation intro} gives a non-trivial relation between the $p$-adic and complex values at $s=1$ which can be exploited to deduce the $p$-part of the BSD conjecture in analytic rank zero. If $E$ has split multiplicative reduction at $p$, however, then $e_p(E,\mathbbm{1}) = 0$; the $p$-adic $L$-function $L_p(E,s)$ has an \emph{exceptional zero} at $s=1$. In this case, \eqref{eq:interpolation intro} is vacuous and the relationship between $L_p(E,s)$ and $L(E,1)$ is not clear; more work is therefore needed to obtain similar results towards the BSD conjecture.

In the 1980s, Mazur, Tate, and Teitelbaum \cite{MTT86} conjectured an `exceptional zero formula'
 \begin{equation}\label{eq:EZF}
 \tfrac{d}{ds}L_p(E,s)\big|_{s=1} = \mathcal{L}_E \cdot L(E,1)/(\text{period})
 \end{equation}
whenever $E$ has split multiplicative reduction at $p$. In particular, differentiating removes the exceptional zero, establishing the desired link between $L_p(E,s)$ and $L(E,1)$, whilst introducing the \emph{$\mathcal{L}$-invariant} $\mathcal{L}_E$ of $E$ at $p$. This $\mathcal{L}$-invariant is known to have a variety of arithmetic descriptions. For example, $\mathcal{L}_E$ equals $\log_p(q)/\opn{ord}_p(q)$, where $q$ is the Tate period of $E/\Qp$; it measures the position of the Hodge filtration in the associated $(\varphi,N)$-module at $p$; it has connections to $p$-adic local Langlands for $E$; and it describes infinitesimal data in a Hida family through $E$. In particular, \eqref{eq:EZF} predicts a new, purely $p$-adic, connection between analysis and arithmetic. It also suggests that the first derivative of $L_p(E,s)$ at $s=1$ and the complex $L$-value $L(E,1)$ should carry the same arithmetic data.
 
Mazur--Tate--Teitelbaum's conjecture was proved by Greenberg and Stevens in a celebrated paper \cite{GS93}, and is one of the crucial ingredients in proving the aforementioned results towards the $p$-part of the BSD conjecture for elliptic curves with split multiplicative reduction at $p$ (see \cite{SkinnerSplitMult}).

In \cite{GreenbergTrivialZeros, Benois11}, Greenberg and Benois predicted an analogue of \eqref{eq:EZF} in huge generality, using a Galois $\mathcal{L}$-invariant. There have been decades of work proving versions of these conjectures for $\GL_2$, but proofs of higher rank cases are rare (e.g., \cite{RossoSym2,LR20}). Moreover, all known cases assume a self-duality condition.

In this article, we consider the analogue of this conjecture for automorphic representations of $\opn{GL}_3(\mbb{A})$ which are Steinberg at $p$ (see Assumption \ref{assumptions} below). In particular, across Theorems \ref{FirstMainThmInIntro}, \ref{SecondMainThmIntroAutLinv} and \ref{ThirdMainThmIntroEqualLinv}, we prove exceptional zero formulae for the $p$-adic $L$-functions attached to such automorphic representations, as constructed in the work of Loeffler and the third author \cite{LW21}. This provides the first examples of the Greenberg--Benois exceptional zero conjecture for regular algebraic cuspidal automorphic representations of $\opn{GL}_n(\mbb{A})$ without any self-duality assumption. 
One of the key ingredients is local-global compatibility at $\ell = p$ for the Galois representations associated with Hida families (and $p$-ordinary torsion classes) for $\opn{GL}_3$; we establish such a result for general $\GL_n$, following the strategy in \cite[\S 5]{10author} (see Theorem \ref{FourthMainThmintroLGcompat}). This result is of independent interest, and analogous results have had applications towards proving automorphy lifting theorems.

\subsection{Set-up and statement of the main results}

We now describe the main results of this article in more detail. Fix an isomorphism $\mbb{C} \cong \Qpb$ throughout. Let $n \geq 2$ and let $\pi$ be a regular algebraic cuspidal automorphic representation (RACAR) of $\opn{GL}_n(\mbb{A})$. Without loss of generality, we assume that $\pi$ is unitary, and we let $\omega_{\pi}$ denote its central character (which is a Dirichlet character). We assume that $\omega_{\pi}(-1) = (-1)^{(n+1)/2}$ if $n$ is odd (otherwise the complex $L$-function $L(\pi, s)$ has no critical values). 

Let $p$ be a prime. By the work of many people (culminating in \cite{ScholzeTorsion, HLTTrigid}), there exists a continuous semisimple Galois representation 
\[
\rho_{\pi} \colon G_{\mbb{Q}} \defeq \opn{Gal}(\overline{\mbb{Q}}/\mbb{Q}) \to \opn{GL}_n(L)
\]
associated with $\pi$, uniquely determined by the property that the trace of $\rho_{\pi}(\opn{Frob}_{\ell})$ matches with the eigenvalue of a certain Hecke operator acting on $\pi_{\ell}^{\opn{GL}_n(\mbb{Z}_\ell)}$ for all but finitely many $\ell \neq p$ (here, $\opn{Frob}_{\ell}$ denotes an arithmetic Frobenius at the prime $\ell$ and $L/\mbb{Q}_p$ is a sufficiently large finite extension). We impose the following ramification assumptions on $\pi$ and $\rho_{\pi}$:

\begin{assumption}\label{assumptions}
    Suppose that:
    \begin{itemize}
    \item[(1)] $\pi$ is $p$-ordinary and $\pi_p$ is isomorphic to the Steinberg representation of $\GL_n(\Qp)$; 
    \item[(2)] one has
    \begin{equation}\label{eq:steinberg at p}
    \rho_{\pi}|_{G_{\mbb{Q}_p}} \sim \mat{ 1 & *_1 & \dots & * & * \\ & \omega^{-1} & & * & * \\ & & \ddots & *_{n-2} & \vdots \\ & & & \omega^{2-n} & *_{n-1} \\ & & & & \omega^{1-n}}
    \end{equation}
    where $\omega$ is the $p$-adic cyclotomic character and $*_i$ are non-crystalline extensions (for all $i=1, \dots, n-1$).
    \end{itemize}
\end{assumption}

\begin{remark}
Under assumption (1), there conjecturally exist $p$-adic $L$-functions with exceptional zeros, without which the exceptional zero conjecture has no content. Note that (1) implies $\pi$ has trivial cohomological weight.

If (2) is satisfied and additionally $\overline{\rho}_{\pi}$ is absolutely irreducible and decomposed generic (see Remark \ref{Rem:DGandCMassumptionIntro} below), then by the work of Hevesi \cite{HevesiOrdinaryParts}, (1) is automatic. We note that the converse (1) $\Rightarrow$ (2) should also be true, however this is currently only known up to semisimplification (under the same irreducibility and decomposed generic assumptions; see Theorem \ref{Thm:Char0ExistenceOfGalRACAR}).
\end{remark}

We can quantify the extensions $*_i$ by their associated \emph{Fontaine--Mazur $\mathcal{L}$-invariant}. More precisely, by local class field theory we have an isomorphism $\opn{Hom}_{\opn{cts}}(\mbb{Q}_p^{\times}, L) \cong \opn{H}^1(\mbb{Q}_p, L)$, and we let $\mathcal{L}^{\opn{FM}}_{\pi, j} \in L$ denote the unique element such that 
\[
\langle \opn{log}_p - \mathcal{L}^{\opn{FM}}_{\pi, j} \opn{ord}_p, \; *_{(n-3)/2+j} \rangle = 0
\]
where: $j \in \tfrac{1}{2}\mbb{Z}$ is such that $(n-3)/2+j \in \mbb{Z} \cap [1, n-1]$, $\langle -, - \rangle \colon \opn{H}^1(\mbb{Q}_p, L) \times \opn{H}^1(\mbb{Q}_p, L(1)) \to L$ denotes the local Tate duality pairing and $\opn{log}_p$ (resp. $\opn{ord}_p$) denotes the $p$-adic logarithm (resp. $p$-adic valuation) normalised so that $\opn{log}_p(p) = 0$ (resp. $\opn{ord}_p(p) = 1$). Note that $\mathcal{L}^{\opn{FM}}_{\pi, j}$ exists by the assumption that $*_{(n-3)/2+j}$ is non-crystalline. 

Let $\chi$ be an even Dirichlet character, and let $L(\pi \times \chi, s) = L(\rho_{\pi} \otimes \chi, (n-1)/2 + s)$ denote the standard $L$-function associated with $\rho_{\pi} \otimes \chi$. If $\chi$ is trivial we will write $L(\pi, s)$, and our normalisation is such that $s=1/2$ is always the centre of the functional equation. The critical values of $L(\pi\times\chi, s)$ are at $s=1/2$ if $n$ is even, and at $s=0$ and $s=1$ if $n$ is odd (using our assumption on the central character $\omega_{\pi}$). Let $c = 1/2$ if $n$ is even and $c \in \{0, 1\}$ if $n$ is odd. Then the conjecture of Coates--Perrin-Riou \cite{coatesperrinriou89, coates89} predicts that there exists a $p$-adic measure $\mathscr{L}^{(c)}_p(\pi, -) \in \mathcal{O}_L[\![ \Zp^\times ] \!]$ such that for all finite-order even $\chi \colon \Zp^\times \to \Qpb^{\times}$, we have
\[
\mathscr{L}^{(c)}_p(\pi, \chi^{-1}) = e_{\infty}^{(c)}(\pi_\infty, c) \cdot e_p^{(c)}(\pi_p, \chi, c) \cdot \frac{L^{(p)}(\pi\times \chi, c)}{\Omega_{\pi, c}}
\]
where $e_{\infty}^{(c)}$ and $e_p^{(c)}$ denote the modified factors at $\infty$ and $p$ respectively, and $\Omega_{\pi, c} \in \mbb{C}^{\times}$ is a certain period such that $L(\pi \times \chi, c)/\Omega_{\pi, c}$ is algebraic for all $\chi$. Under the above assumptions, one can explicitly compute $e_p^{(c)}$; indeed, there exist constants $A \neq 0$ and $B$, depending only on $n$ and $c$, such that\footnote{For example when $n=3$, we have $E_p^{(0)}(s) = 1$ and $E_p^{(1)}(s) = -p^s$.}
\[
e_p^{(c)}(\pi_p, \mathbbm{1}, s) = E_p^{(c)}(s) \cdot (1-p^{s-c}) \cdot L(\pi_p,s), \qquad \text{where $E_p^{(c)}(s) = A p^{Bs}$}.
\]
 In particular $e_p^{(c)}(\pi_p,\mathbbm{1},s)$ always has a simple zero at $s = c$,
and hence $\mathscr{L}_p^{(c)}(\pi, \mathbbm{1}) = 0$; the $p$-adic $L$-function has an ``exceptional zero'' at the trivial character coming from $e_p(-)$. Assuming the existence of such a $p$-adic $L$-function, the Greenberg--Benois conjecture \cite[p.166]{GreenbergTrivialZeros} says:

\begin{conjecture}[The exceptional zero conjecture] \label{conj:EZC}
    With notation and assumptions as above, set $L^{(c)}_p(\pi, s) \defeq \mathscr{L}_p^{(c)}(\pi, \langle-\rangle^s)$ for $s \in \mbb{Z}_p$ and $\langle - \rangle \colon \Zp^\times \to 1+2p\Zp$ the natural map. Then
    \[
    \left. \frac{d}{ds} L_p^{(c)}(\pi, s) \right|_{s=0} = e_{\infty}^{(c)}(\pi_{\infty}, c) \cdot E_p^{(c)}(c) \cdot  \mathcal{L}^{\opn{FM}}_{\pi, c+1} \cdot \frac{L(\pi, c)}{\Omega_{\pi, c}} .
    \]
\end{conjecture}

In particular, the term $(1-p^{s-c})$ causing the exceptional zero has been replaced by the $\mathcal{L}$-invariant. As mentioned above, beyond the intrinsic beauty of tying together complex and $p$-adic analytic $L$-values with a $p$-adic arithmetic $\mathcal{L}$-invariant, the conjecture plays an important role in establishing $p$-parts of the Bloch--Kato conjecture for $\rho_{\pi}$ from the corresponding (conjectural) Iwasawa Main Conjecture. 

Let us briefly mention some previous work on this conjecture:
\begin{itemize}
    \item The case of $n=2$ was, as described above, settled in \cite{MTT86,GS93}.
    \item If $n \geq 3$ is odd and $\pi$ is essentially self-dual, then the existence of the $p$-adic $L$-function and the exceptional zero fomulae (for both $c=0,1$) are known by work of Rosso \cite{RossoSym2, RossoSym2II} and Liu--Rosso \cite{LR20}, under assumptions on the prime-to-$p$ ramification of $\pi$. 
    
    When $n=3$, essential self-duality implies $\pi \cong \opn{Sym}^2 f \otimes \xi$ is a twisted symmetric square lift, where $f$ is a weight $2$ cuspidal newform $f$ of level $\Gamma_1(N) \cap \Gamma_0(p)$ with split multiplicative reduction at $p$ (see \cite{RamakrishnanSelfDual}), and $\xi$ is an even Dirichlet character of prime-to-$p$ conductor with $\xi(p) = 1$. Then \cite{RossoSym2, RossoSym2II} proves the conjecture under the assumptions that $N$ is square-free and the nebentypus $\varepsilon_f$ of $f$ is trivial.
    \item If $n \geq 4$ and $\pi$ is of symplectic-type (which implies that $\pi$ is essentially self-dual), then the existence of the $p$-adic $L$-function is known \cite{LiuSunRelativeCC}, however to the best of the authors' knowledge, the exceptional zero fomula is not known in any cases.
\end{itemize}

In particular, nothing is known about the exceptional zero conjecture for $\opn{GL}_{n}/\mbb{Q}$ outside the essentially self-dual setting. One reason for this is that all of the known exceptional zero results described above exploit the fact that one can find $p$-adic families through $\pi$ with a Zariski dense set of classical specialisations -- something which is not expected for non-essentially-self-dual representations (see \cite{AshPollackStevens, CM09} and Remark \ref{Rem:DGandCMassumptionIntro} below). Furthermore, all of the aforementioned results crucially use algebro-geometric/complex analytic methods, exploiting the fact that $\pi$ is a functorial transfer from a group whose locally symmetric spaces are Shimura varieties. Such methods cannot be applied for general $\pi$.

Beyond this, even constructing the Coates--Perrin-Riou $p$-adic $L$-function is difficult. Recently, in work of Loeffler--Williams \cite{LW21}, the $p$-adic $L$-function for $\pi$ was constructed for automorphic representations of $\opn{GL}_3(\mbb{A})$ with no self-duality assumption. It is natural to ask whether one can prove an exceptional zero formula for this $p$-adic $L$-function. This is one of the main results of this article, which is the first case of an exceptional zero formula for non-essentially-self-dual regular algebraic cuspidal automorphic representations of $\opn{GL}_n(\mbb{A})$:

\begin{theoremx} \label{FirstMainThmInIntro}
    Let $n=3$ and let $p \geq 7$. Suppose that the residual representation $\overline{\rho}_{\pi} \colon G_{\mbb{Q}} \to \opn{GL}_3(\overline{\mbb{F}}_p)$ is irreducible and decomposed generic. Then the exceptional zero conjecture holds for either $L_p^{(0)}$ or $L_p^{(1)}$. Furthermore, if $\pi$ has ``non-parabolic infinitesimal deformations'' (see Remark \ref{Rem:DGandCMassumptionIntro} below), then the exceptional zero conjecture holds for both $L_p^{(0)}$ and $L_p^{(1)}$.
\end{theoremx}

Before describing the ideas that go into the proof of this theorem, let us make a few remarks.

\begin{remark}
Our first major result towards this is Theorem \ref{SecondMainThmIntroAutLinv}, where we prove an automorphic analogue of Conjecture \ref{conj:EZC}. In particular, we prove the exact analogue of the stated exceptional zero formula, for both $L_p^{(0)}$ and $L_p^{(1)}$, with the Fontaine--Mazur $\mathcal{L}$-invariant $\mathcal{L}_{\pi,c+1}^{\opn{FM}}$ replaced by Gehrmann's automorphic $\mathcal{L}$-invariant $\mathcal{L}_{\pi,c+1}^{\opn{Aut}}$ (from \cite{AutomorphicLinvariants}).  

Theorem \ref{SecondMainThmIntroAutLinv} is completely unconditional (beyond Assumption \ref{assumptions}(1), which is needed for the statement to make sense), and is even independent of the existence of $\rho_\pi$. Assumption \ref{assumptions}(2), and the other hypotheses in Theorem \ref{FirstMainThmInIntro}, are needed only to prove the equality of $\mathcal{L}$-invariants (Theorem \ref{ThirdMainThmIntroEqualLinv}).
\end{remark}

\begin{remark} \label{Rem:DGandCMassumptionIntro}
    By decomposed generic, we mean that there exists a prime $\ell$ where $\overline{\rho}_{\pi}$ is unramified, and such that none of the eigenvalues of $\opn{Ad}\overline{\rho}_{\pi}(\opn{Frob}_{\ell})$ are equal to $\ell^{\pm 1}$. Additionally, by ``non-parabolic infinitesimal deformations'', we mean that $\pi$ admits ``non-$s_i$-infinitesimal deformations'' for both $i=1, 2$ (as in \S \ref{Subsub:TheGaloisSideIntro} below). Informally, this means that $\pi$ varies in a $p$-adic family over a curve $\Sigma$ in the $2$-dimensional weight space $\mathcal{W}$ associated with $\opn{PGL}_3$, and for every simple root $s_i$ of $\opn{PGL}_3$, there exists a vector in the tangent space of $\Sigma$ at the origin which does not lie in the (one-dimensional) parabolic weight space associated with $s_i$. One can show that if $\pi$ is essentially self-dual, then $\pi$ \emph{always} has non-parabolic infinitesimal deformations, and if $\pi$ is not essentially self-dual, this is strongly suggested by the conjectures of Ash--Pollack--Stevens \cite{AshPollackStevens} and Calegari--Mazur \cite{CM09}. We expand on this in Remark \ref{eq:APS}.
\end{remark}

\begin{remark}\label{rem:intro self dual}
Even in the essentially self-dual case, where $\pi \cong \mathrm{Sym}^2f \otimes \xi$ for $f \in S_2(\Gamma_1(N)\cap\Gamma_0(p),\varepsilon_f)$, we obtain new results; we generalise \cite{RossoSym2, RossoSym2II} by removing the assumptions \emph{op.\ cit}.\ that $N$ is square-free and $\varepsilon_f$ is trivial (see Remark \ref{Rem:EZFforSym2generallevel}). Moreover, we do not need to impose the assumptions on $p$ and $\overline{\rho}_{\pi}$ in this case. 
\end{remark}

\begin{remark}
A prototype for this setting is that of Bianchi modular forms (RACARs for $\GL_2$ over an imaginary quadratic field). Here one encounters similar issues to those mentioned above: non-existence of $p$-adic families with a Zariski-dense set of classical specialisations, and underlying locally symmetric spaces that are not Shimura varieties. In this case, Conjecture \ref{conj:EZC} is not known. However, cases of the automorphic analogue of the conjecture were proved by the first and third authors in \cite{BW17}. The methods used there (modular symbols) were very explicit; finding a more conceptual framework for obtaining such results, that could apply to higher rank settings, was a key aim of the present paper, and our approach is summarised in \S\ref{sec:intro automorphic} below.

We hope to return to this case in future, using the methods of the present paper to prove cases of the full Greenberg--Benois exceptional zero conjecture for Bianchi modular forms.
\end{remark}

\subsection{Method of proof}

We now specialise to the case $n=3$. The proof of Theorem \ref{FirstMainThmInIntro} is naturally divided into two parts: an automorphic side (proving the automorphic exceptional zero formula, in Theorem \ref{SecondMainThmIntroAutLinv}) and a Galois side (proving an equality of automorphic and Fontaine--Mazur $\mathcal{L}$-invariants, in Theorem \ref{ThirdMainThmIntroEqualLinv}). This is reflected in the structure of this article. 

The key input on the Galois side is to show that the family of Galois representations associated with the Hida family through $\pi$ satisfies local-global compatibility at $\ell = p$ (Theorem \ref{FourthMainThmintroLGcompat}). In particular, we show this family is trianguline with the correct parameters (Theorem \ref{FourthMainThmintroLGcompat}), confirming cases of a conjecture of Hansen \cite[Conj.\ 1.2.2]{Han17}.

\subsubsection{The automorphic side}\label{sec:intro automorphic}

Let $G = \opn{GL}_3$ and let $B \subset G$ denote the standard upper-triangular Borel subgroup. We let $P_1, P_2 \subset G$ denote the parabolic subgroups 
\[
P_1 = \smat{* & * & * \\ & * & * \\ & * & *}, \quad \quad P_2 = \smat{* & * & * \\ * & * & * \\ & & *}, 
\]
and let $\overline{B}$, $\overline{P}_1$, and $\overline{P}_2$ denote the opposites of these parabolic subgroups. Just as the Galois $\mathcal{L}$-invariant was defined in terms of extensions of powers of the $p$-adic cyclotomic character, the automorphic $\mathcal{L}$-invariant is defined in terms of extensions of (generalised) Steinberg representations.

To describe these $\mathcal{L}$-invariants, we therefore first need to introduce some notation. Let $R$ be a $\mbb{Z}$-algebra, and consider the generalised Steinberg representation:
\[
\opn{St}_i^{\mathrm{sm}}(R) \defeq \{ f \colon \overline{P}_i(\mbb{Q}_p) \backslash G(\mbb{Q}_p) \to R \text{ smooth } \} /\{ \text{ constant functions }\}
\]
which is a smooth representation of $G(\mbb{Q}_p)$ via right-translation. We also set 
\[
\opn{St}^{\mathrm{sm}}(R) = \{ f \colon \overline{B}(\mbb{Q}_p) \backslash G(\mbb{Q}_p) \to R \text{ smooth }\}/\left(\opn{St}^{\mathrm{sm}}_1(R) + \opn{St}^{\mathrm{sm}}_2(R) \right)
\]
which is the usual Steinberg representation of $G(\mbb{Q}_p)$ (over $R$). If $R = L$ is a finite extension of $\mbb{Q}_p$, then we have continuous and locally analytic variants of these representations (denoted $\opn{St}^{\opn{cts}}(L)$, $\opn{St}^{\opn{la}}(L)$, etc.) by replacing the property of being smooth with being continuous or locally analytic (with respect to the $p$-adic topologies on both sides). We have the following result of Ding \cite{DingLinvariants} classifying extensions of locally analytic Steinberg representations:
\begin{align*}
    \opn{Hom}_{\opn{cts}}(\mbb{Q}_p^{\times}, L) &\xrightarrow{\sim} \opn{Ext}^1_{\opn{an}}(\opn{St}_i^{\mathrm{sm}}(L), \opn{St}^{\opn{la}}(L)) \\
    \lambda &\mapsto \mathcal{C}_{i,\lambda}
\end{align*}
where $\opn{Ext}^1_{\opn{an}}$ denotes the groups of extensions (up to isomorphism) in the abelian category of admissible locally analytic representations of $G(\mbb{Q}_p)$ (see \cite[\S 2.4]{AutomorphicLinvariants}).

Recall that $\pi$ is a (unitary) regular algebraic cuspidal automorphic representation of $G(\mbb{A}) = \opn{GL}_3(\mbb{A})$, which has trivial cohomological weight and $\pi_p \cong \opn{St}^{\mathrm{sm}}(\mbb{C})$. We let 
\[
\Pi_p \defeq \opn{Hom}_{G(\mbb{A}_f^p)}(\mathcal{W}_L(\pi^p), \tilde{\opn{H}}^2_c(X_{\opn{GL}_3}, L))
\]
where $\mathcal{W}_L(\pi^p)$ denotes the Whittaker model of $\pi^p$ over $L$ and
\[
\tilde{\opn{H}}^2_c(X_{\opn{GL}_3}, L) = \varinjlim_{K^p} \left( \Big(\varprojlim_s \varinjlim_{K_p} \tilde{\opn{H}}^2_c(X_{\opn{GL}_3, K^pK_p}, \mathcal{O}_{L}/p^s)\Big)[1/p] \right)
\]
denotes the completed cohomology (at infinite prime-to-$p$ level) associated with the locally symmetric space $X_{\opn{GL}_3}$ for $\opn{GL}_3$. The space $\Pi_p$ is naturally an admissible Banach representation of $G(\mbb{Q}_p)$, and should be closely related with $\rho_{\pi}|_{G_{\mbb{Q}_p}}$ via a hypothetical $p$-adic local Langlands correspondence for $\opn{GL}_3(\mbb{Q}_p)$. Furthermore, one has a natural $G(\mbb{Q}_p)$-equivariant embedding 
\begin{equation} \label{Eqn:CanEmbeddingIntro}
\opn{St}^{\mathrm{sm}}(L) \hookrightarrow \Pi_p
\end{equation}
(unique up to scaling by multiplicity one results for the cohomology of $X_{\opn{GL}_3}$). We will therefore define, following Gehrmann \cite{AutomorphicLinvariants}, the automorphic $\mathcal{L}$-invariants for $\pi$ in terms of extensions of Steinberg representations related with $\Pi_p$. 

More precisely, for $i=1, 2$, let $\bL^{\opn{Aut}}_i(\pi) \subset \opn{Hom}_{\opn{cts}}(\mbb{Q}_p^{\times}, L)$ denote the subspace of homomorphisms $\lambda$ such that the natural embedding (\ref{Eqn:CanEmbeddingIntro}) extends to a $G(\mbb{Q}_p)$-equivariant embedding $\mathcal{C}_{3-i, \lambda} \hookrightarrow \Pi_p$. It is shown\footnote{In \cite{AutomorphicLinvariants}, Gehrmann describes $\bL^{\opn{Aut}}_i(\pi)$ slightly differently, using extensions of locally analytic Steinberg representations and $p$-arithmetic cohomology. We recap this in \S\ref{sec:extensions of steinberg} and \S\ref{sec:def of L-invariants}. The equivalence between the $p$-arithmetic and completed cohomology perspectives is discussed in \S \ref{RelationWithCCSSec}.} in \cite{AutomorphicLinvariants} that $\bL^{\opn{Aut}}_i(\pi)$ is one-dimensional and $\opn{ord}_p \not\in \bL^{\opn{Aut}}_i(\pi)$. We therefore define the automorphic $\mathcal{L}$-invariant $\mathcal{L}_{\pi, i}^{\opn{Aut}}$ as the unique quantity such that $\opn{log}_p - \mathcal{L}_{\pi, i}^{\opn{Aut}} \opn{ord}_p \in \bL^{\opn{Aut}}_i(\pi)$. In Theorems \ref{Thm:LeftHalfEZF} and \ref{Them:RightHalfEZF} we prove the following: 

\begin{theoremx} \label{SecondMainThmIntroAutLinv}
    Let $n=3$ and $c \in \{0, 1\}$. Then 
    \[
    \left. \frac{d}{ds} L_p^{(c)}(\pi, s) \right|_{s=0} =   e_{\infty}(\pi_{\infty}, c) \cdot E_p^{(c)}(c) \cdot  \mathcal{L}^{\opn{Aut}}_{\pi, c+1} \cdot \frac{L(\pi, c)}{\Omega_{\pi, c}} .
    \]
\end{theoremx}

\begin{remark}
    In the main body of the article, $L_p^{(c)}$ is denoted $L_p^{-}$ for $c=0$ and $L_p^+$ for $c=1$. We also renormalise so that the exceptional zero of $L_p^+(\pi,s)$ is at $s=1$ (as it relates to $L(\pi,1)$). To obtain the statements in the main text, note that $E_p^{(0)}(0) = 1$ and $E_p^{(1)}(1) = -p$.
\end{remark}

 Let us sketch the proof when $c=0$; the case $c=1$ then follows from the functional equation for the ($p$-adic) $L$-functions, as $\mathcal{L}^{\opn{Aut}}_{\pi, 2} = \mathcal{L}^{\opn{Aut}}_{\pi^{\vee}, 1}$ (see Proposition \ref{prop:duality L-invariants}). For $c=0$, the idea is to extend the construction in \cite{LW21} to ``modular symbols on the Bruhat--Tits building of $\opn{GL}_3(\mbb{Q}_p)$''. For this, we aim to construct a ``big evaluation map''
\[
\alpha \colon \opn{Hom}_{G(\mbb{Q}_p)}(\opn{St}^{\mathrm{sm}}(L), \Pi_p) \times \opn{Hom}_{\opn{cts}}(\mbb{Q}_p^{\times}, L) \to L
\]
satisfying three key properties: for $\mu$ the class corresponding to the natural embedding \eqref{Eqn:CanEmbeddingIntro} above, we have
\begin{itemize}
    \item[(1)] $\alpha$ is bilinear, and $\alpha(\mu, -)$ vanishes on $\bL^{\opn{Aut}}_1(\pi)$,
    \item[(2)] $\alpha(\mu, \opn{ord}_p) = e_{\infty}(\pi_{\infty}, 0) \cdot \tfrac{L(\pi, 0)}{\Omega_{\pi, 0}}$,
    \item[(3)] $\alpha(\mu, \opn{log}_p) = \left. \frac{d}{ds} L_p^{(0)}(\pi, s) \right|_{s=0}$.
\end{itemize}
Combining these points, and using that $\log_p - \mathcal{L}_{\pi,1}^{\mathrm{Aut}}\opn{ord}_p \in \bL^{\opn{Aut}}_1(\pi)$, proves Theorem \ref{SecondMainThmIntroAutLinv} for $c=0$.

Making this strategy precise occupies all of \S\ref{sec:abstract exceptional zero formula}. In practice, we work with coefficients mod $p^s$, and then link this to Gehrmann's description of automorphic $\mathcal{L}$-invariants using a delicate interplay between smooth, continuous and locally analytic Steinberg representations (proved in \S\ref{sec:extensions of steinberg}). Morally, $\alpha$ is then given by the system
\begin{equation} \label{Eqn:InformalBigEvalMapIntro}
\alpha(\mu, \lambda) \approx p^{-N} \left( \left\langle \Big[p^N\mu\Big(u_0 \cdot \opn{pr}(c_{1,\lambda}[t] \otimes \phi_n)\Big) \ \opn{mod} \ p^s\Big]\Big|_H, \opn{Eis}_n \right\rangle \right)_{s \geq 1} \in L.
\end{equation}

Let us explain the notation in this expression. Firstly, $\opn{pr} \colon \opn{St}^{\mathrm{sm}}_1(\mathcal{O}_L) \otimes \opn{St}_2^{\mathrm{sm}}(\mathcal{O}_L) \to \opn{St}^{\mathrm{sm}}(\mathcal{O}_L)$ denotes the natural map given by multiplication of functions;
\[
\phi_n = \opn{ch}\left(\overline{P}_2(\mbb{Q}_p) \smat{1 & & p^n \mbb{Z}_p \\ & 1 & \mbb{Z}_p \\ & & 1} \right) \in \opn{St}_2^{\mathrm{sm}}(\mathcal{O}_L)
\]
is a characteristic function; and $c_{1,\lambda} \in Z^1(G(\mbb{Q}_p), \opn{St}_1^{\opn{cts}}(\mathcal{O}_L))$ is an explicit $1$-cocycle, such that (the locally analytic vectors in the extension corresponding to) the image of $c_{1,\lambda}$ under the natural map 
\[
\opn{H}^1(G(\mbb{Q}_p), \opn{St}_1^{\opn{cts}}(\mathcal{O}_L)) \xrightarrow{\opn{pr} \circ (- \otimes \opn{St}_2^{\mathrm{sm}}(\mathcal{O}_L))} \opn{Ext}^1_{G(\mbb{Q}_p)}(\opn{St}_2^{\mathrm{sm}}(L), \opn{St}^{\opn{cts}}(L))
\]
is $\mathcal{C}_{2,\lambda}$ (up to isomorphism). Secondly, $t = \opn{diag}(p, 1, 1) \in G(\mbb{Q}_p)$ and $u_0 = \smat{1 & & \\ & & -1 \\ & 1 &}$. Thirdly, $p^N$ is any power of $p$ such that $p^N \mu$ is integral and in the first term of the pairing $\langle -, - \rangle$, we are evaluating $p^N \mu$ at $u_0 \cdot \opn{pr}(c_{1,\lambda}[t] \otimes \phi_n)$, reducing modulo $p^s$, and pulling back to (a cover of) the locally symmetric space for $H = \opn{GL}_2 \times \opn{GL}_1$ (here, we take $n$ to be any integer which is sufficiently large compared to $s$). Finally, $\opn{Eis}_n$ is a certain cohomology class on (the cover of) the locally symmetric space for $H = \opn{GL}_2 \times \opn{GL}_1$ which is given by a Siegel unit on the first $\opn{GL}_2$-factor, and $\langle -, - \rangle$ denotes the Poincar\'e duality pairing. We then show the expression (\ref{Eqn:InformalBigEvalMapIntro}) is independent of the various choices.

The utility of the cocycles $c_{1,\lambda}$ was suggested by the work of Spie{\ss} \cite{Spi14} in the $\GL_2$-setting. That we can calculate and manipulate these cocycles explicitly is vital in establishing these three bullet points, and is at the heart of the proof. Such calculations take place in \S \ref{Sub:ExtensionsI}.

\subsubsection{The Galois side: local-global compatibility for $\GL_n$} \label{Subsub:TheGaloisSideIntro}

After establishing Theorem \ref{SecondMainThmIntroAutLinv} in the first part of this article, we then move on to the second part: showing that the automorphic and Galois $\mathcal{L}$-invariants coincide. To be able to describe this, we consider Hida families through $\pi$ in the $\opn{GL}_3$-eigenvariety. More precisely, let $w \colon \mathscr{E} \to \mathcal{W}$ denote the $\opn{GL}_3$-eigenvariety, where $\mathcal{W}$ denotes the $2$-dimensional weight space parameterising pairs of characters $(\kappa_1, \kappa_2)$ on $\mbb{Z}_p^{\times}$, which we view as a character of $T(\mbb{Z}_p) \cong (\mbb{Z}_p^{\times})^3$ by $(x_1, x_2, x_3) \mapsto \kappa_1(x_1x_3^{-1})\kappa_2(x_2x_3^{-1})$ (we ignore the trivial determinant twist for automorphic representations of $\opn{GL}_3$, since it plays no role in what follows). The unique $p$-stabilisation in $\pi$ is ordinary and gives rise to a point $x_{\pi} \in \mathscr{E}$ with $w(x_{\pi}) = (0, 0)$ (the pair of trivial characters). By \cite[Theorem 4.9]{HansenThorne}, we can choose an irreducible component $\Sigma \subset \mathscr{E}$ containing $x_{\pi}$ which is one-dimensional and (locally around $x_{\pi}$) maps isomorphically to its image under $w$. By abuse of notation, we will also denote the image by $\Sigma$. 

Let $s_1 = (1, -1, 0)$ and $s_2 = (0, 1, -1)$ denote the standard simple roots of $\opn{GL}_3$. We say that the family $\Sigma$ admits ``non-$s_i$-infinitesimal deformations'' if there exists a non-zero tangent vector $v \in T_{x_{\pi}}(\Sigma)$ such that $dw \circ v \colon T(\mbb{Z}_p) \to L$ does not factor through $s_i$. Note that this is automatic if $\Sigma$ is \emph{not} smooth at $x_{\pi}$. Furthermore, it is always the case that $\Sigma$ admits either ``non-$s_1$-infinitesimal deformations'' or ``non-$s_2$-infinitesimal deformations'', and  conjecturally $\Sigma$ should always admit both (see Remark \ref{Rem:DGandCMassumptionIntro}). Our main theorem on the equality of $\mathcal{L}$-invariants is then:

\begin{theoremx} \label{ThirdMainThmIntroEqualLinv}
    Let $n=3$ and $i \in \{1, 2\}$.
    \begin{enumerate}
        \item If $\pi$ is essentially self-dual (so $\pi \cong \opn{Sym}^2f\otimes \xi$ is the twisted symmetric square lift of a modular form $f$), then $\mathcal{L}^{\opn{Aut}}_{\pi, i} = \mathcal{L}^{\opn{FM}}_{\pi, i} = \mathcal{L}(f)$, where $\mathcal{L}(f)$ denotes the $\mathcal{L}$-invariant for $f$.
        \item Suppose that $p \geq 7$ and $\overline{\rho}_{\pi}$ is absolutely irreducible and decomposed generic. Then, if $\Sigma$ admits ``non-$s_{i}$-infinitesimal deformations'', one has
    \[
    \mathcal{L}^{\opn{Aut}}_{\pi, i} = \mathcal{L}^{\opn{FM}}_{\pi, i} .
    \]
    \end{enumerate}
\end{theoremx}

Combining Theorems \ref{SecondMainThmIntroAutLinv} and \ref{ThirdMainThmIntroEqualLinv} yields Theorem \ref{FirstMainThmInIntro}, noting that $\pi$ always has either non-$s_1$-infinitesimal or non-$s_2$-infinitesimal deformations.

The first step in proving Theorem \ref{ThirdMainThmIntroEqualLinv} is to reintepret the automorphic $\mathcal{L}$-invariant in terms of parameters associated with the family $\Sigma$. More precisely, let $t_1 = \opn{diag}(p, 1, 1) \in G(\mbb{Q}_p)$ and $t_2 = \opn{diag}(p, p, 1) \in G(\mbb{Q}_p)$. Associated with $\Sigma$, we obtain Hecke eigensystems $\underline{\alpha}_1, \underline{\alpha}_2 \in \mathcal{O}^+(\Sigma)$ such that for $y \in \Sigma$, $\underline{\alpha}_i(y)$ denotes the eigenvalue of the normalised $U_{p}$-operator associated with $t_i$ on the $p$-adic automorphic form associated with the specialisation of the Hida family at $y$. We want to stress that, in the non-essentially self-dual case, it is not expected that there exists a Zariski dense set of classical specialisations in $\Sigma$; so in general, the specialisation of the Hida family at $y$ is only a $p$-adic automorphic form. Let $v \in T_{x_{\pi}}(\Sigma)$ be a non-zero tangent vector which does not factor through $s_{i}$, and normalise $v$ such that $v \circ s_{i}^{\vee} = 1$. Following the work of Gehrmann--Rosso \cite{GehrmannRosso}, we show 
\[
\mathcal{L}^{\opn{Aut}}_{\pi, i} = \left\{ \begin{array}{cc} -(\partial_v\underline{\alpha}^2_1 \underline{\alpha}_2^{-1})(0, 0) & \text{ if } i=1 \\ -(\partial_v \underline{\alpha}^2_2)(0, 0) & \text{ if } i=2 \end{array} \right.
\]
and refer to this as the ``automorphic Benois--Colmez--Greenberg--Stevens formula''.

On the other hand, let $L[\epsilon]$ denote the ring of dual numbers over $L$, which we view as an $\mathcal{O}_{\Sigma}$-algebra via the map $\mathcal{O}_{\Sigma} \to L[\epsilon]$, $f \mapsto f(0, 0) + (\partial_v f)(0, 0) \cdot \epsilon$. Then, if we can show that there exists a Galois representation $\rho_{\underline{\pi}} \colon G_{\mbb{Q}} \to \opn{GL}_3(L[\epsilon])$ with the following properties:
\begin{itemize}
    \item $\rho_{\underline{\pi}}$ modulo $\epsilon$ is isomorphic to $\rho_{\pi}$;
    \item one has
    \[
    \rho_{\underline{\pi}}|_{G_{\mbb{Q}_p}} \sim \smat{ \delta_1 & * & * \\ & \delta_2 \omega^{-1} & * \\ & & \delta_3 \omega^{-2} }
    \]
    with $\delta_1(p) = 1 + (\partial_v \underline{\alpha}_1)(0, 0) \cdot \epsilon$, $\delta_1(p)\delta_2(p) = 1 + (\partial_v \underline{\alpha}_2)(0, 0) \cdot \epsilon$, $\delta_1(p)\delta_2(p)\delta_3(p) = 1$, and $\delta_{i}|_{\mbb{Z}_p^{\times}}$ is trivial;
\end{itemize}
then the ``Galois Benois--Colmez--Greenberg--Stevens formula'' tells us that
\[
\mathcal{L}^{\opn{FM}}_{\pi, i} = \left\{ \begin{array}{cc} -(\partial_v \delta_1(p) \delta_2^{-1}(p))(0, 0) & \text{ if } i=1 \\ -(\partial_v \delta_2(p) \delta_3^{-1}(p))(0, 0) & \text{ if } i=2 \end{array} \right.
\]
(see \cite{Benois11,DingLinvariants}). If we can find such a Galois representation, comparing both Benois--Colmez--Greenberg--Stevens formulae yields Theorem \ref{ThirdMainThmIntroEqualLinv}. The second bullet point is the titular local-global compatibility at $\ell = p$ for Hida families. The existence of the Galois representation $\rho_{\underline{\pi}}$ with the required properties follows from Theorem \ref{FourthMainThmintroLGcompat} below, and in fact, we establish this local-global compatibility for $\opn{GL}_n$ for general $n \geq 2$. 

To describe this result, we introduce some notation. Let $S$ denote a finite set of primes containing $p$, and let $\mbf{T}^{S, \opn{ord}}$ denote the abstract Hecke algebra (over $\mathcal{O}_L$) of spherical Hecke operators away from $S$ and the $U_p$-Hecke operators at $p$. We let $R\tilde{\Gamma}(X_{\opn{GL}_n, K^p}, \mathcal{O}_L)^{\opn{ord}}$ denote the ordinary part of completed cohomology for the locally symmetric space associated with $\opn{GL}_n$, for some appropriate prime-to-$p$ level $K^p$ which is hyperspecial outside $S$. This complex naturally lives in the derived category of $\Lambda = \mathcal{O}_L[\![T(\mbb{Z}_p)]\!]$-modules, and we let $\mathcal{T} \defeq \mbf{T}^{S, \opn{ord}}(R\tilde{\Gamma}(X_{\opn{GL}_n, K^p}, \mathcal{O}_L)^{\opn{ord}})$ denote the image of $\mbf{T}^{S,\opn{ord}}$ in the endomorphism ring (in the derived category) of $R\tilde{\Gamma}(X_{\opn{GL}_n, K^p}, \mathcal{O}_L)^{\opn{ord}}$. 

By the work of Scholze \cite{ScholzeTorsion}, attached to any maximal ideal $\ide{m} \subset \mathcal{T}$ with residue field $k = \mathcal{O}_L/\varpi$, one has a continuous semisimple Galois representation
\[
\overline{\rho}_{\ide{m}} \colon G_{\mbb{Q}} \to \opn{GL}_n(k)
\]
such that the characteristic polynomial of  $\overline{\rho}_{\ide{m}}(\opn{Frob}_{\ell})$ coincides with the Hecke polynomial associated with $\ide{m}$ (see Theorem \ref{Thm:GaloisRepForTorClass} for a precise statement). As an example, if $\pi$ is a regular algebraic cuspidal automorphic representation of $\opn{GL}_n(\mbb{A})$ such that $\pi_p$ is (nearly) ordinary, then one can associate a maximal ideal $\ide{m} \subset \mathcal{T}$ and $\overline{\rho}_{\ide{m}} = \overline{\rho}_{\pi}$ is the mod $\varpi$ reduction of the Galois representation associated with $\pi$.

We prove the following theorem in \S \ref{LGcompatAtl=pSec}:

\begin{theoremx} \label{FourthMainThmintroLGcompat}
    Let $n \geq 2$ be any integer and $p > 2n$. Suppose that $\overline{\rho}_{\ide{m}}$ is absolutely irreducible and decomposed generic. With notation as above, there exists a nilpotent ideal $J \subset \mathcal{T}_{\ide{m}}$ (with nilpotence degree only depending on $n$), and a continuous Galois representation
    \[
    \rho \colon G_{\mbb{Q}} \to \opn{GL}_n( \mathcal{T}_{\ide{m}}/J)
    \]
    such that:
    \begin{enumerate}
        \item For all $\ell \not\in S$, the characteristic polynomial $\opn{det}(1-\opn{Frob}^{-1}_{\ell} X | \rho)$ equals the image of the standard Hecke polynomial at $\ell$ for $\opn{GL}_n$ in $\mathcal{T}_{\ide{m}}/J$ (see Definition \ref{HXHeckeDef}).
        \item For every $g \in G_{\mbb{Q}_p}$, the characteristic polynomial of $\rho(g)$ equals $\prod_{i=1}^n (X - \chi_i(g))$, where $\chi_i \colon G_{\mbb{Q}_p} \to (\mathcal{T}_{\ide{m}}/J)^{\times}$ are the standard Galois characters in Definition \ref{DefOfUnivGaloisChars}.
        \item For each $g_1, \dots, g_n \in G_{\mbb{Q}_p}$, we have
        \[
        (\rho(g_1) - \chi_1(g_1)) \cdots (\rho(g_n) - \chi_n(g_n)) =0.
        \]
    \end{enumerate}
\end{theoremx}

\begin{remark}
We note that conditions (2) and (3) are a way to encode the local-global compatibility at $\ell = p$ purely in terms of the continuous determinant associated with $\rho$. Furthermore, the existence of such a Galois representation satisfying (1) essentially follows from \cite[Theorem 1.1.6]{CGHJMRS} (see also \cite{ScholzeTorsion, NewtonThorneTorsion}) -- our main contribution is parts (2) and (3). In fact, \cite[Theorem 1.1.6]{CGHJMRS} implies that there exists a Galois representation satisfying (1) with $J=0$, however it does not seem clear (to the authors) how to remove the nilpotent ideal when establishing properties (2) and (3).
\end{remark}

The proof of Theorem \ref{FourthMainThmintroLGcompat} occupies all of \S\ref{LGcompatAtl=pSec}, and follows the strategy of \cite[\S5]{10author}. More precisely, $\rho$ is constructed by realising the Hecke eigensystem for $\mathcal{T}_{\ide{m}}$ in the boundary cohomology of the Shimura variety associated with $\opn{GSp}_{2n}$, and using known local-global compatibility results at $\ell = p$ for standard Galois representations associated with automorphic representations of $\opn{GSp}_{2n}$ (i.e., known local-global compatibility results for essentially self-dual cuspidal automorphic representations of $\opn{GL}_{2n+1}$ in, e.g., \cite{BGGTII}). One of the key ingredients is the recent vanishing results for the mod $p$ cohomology of abelian-type Shimura varieties in \cite{YangZhu}.

\begin{remark}
The theorem proves cases of a conjecture of Hansen \cite[Conj.\ 1.1.2(iii)]{Han17}, which says that the Galois representation associated to (possibly non-classical) points in eigenvarieties are trianguline. In particular, it shows this for any non-Eisenstein, decomposed generic ordinary point in the $\GL_n$-eigenvariety. We remark that, over imaginary CM fields, cases of Hansen's conjecture have been obtained in \cite{10author} (in the ordinary setting) and \cite{McDonald-Trianguline} (in the finite-slope setting), however it is not possible to deduce the cases obtained from Theorem \ref{FourthMainThmintroLGcompat} from these works. 
\end{remark}

\begin{remark}
    The proof of Theorem \ref{FourthMainThmintroLGcompat} relies on the endoscopic classification for symplectic groups in \cite{ArthurBook}. Although \emph{op.cit.} is still conditional on forthcoming work, there has been a tremendous amount of progress towards making this unconditional in \cite{LocalIntertwiningRelations}. At present, the only thing that remains is the proof of the twisted weighted fundamental lemma.
\end{remark}

\subsection{Structure of the paper}
In Part 1 (\S\ref{sec:extensions of steinberg}--\S\ref{sec:automorphic EZF}), we prove Theorem \ref{SecondMainThmIntroAutLinv}. In \S\ref{sec:extensions of steinberg}, we recall the extensions of Steinberg representations, and describe the interplay between smooth, continuous, and locally analytic extensions. This section also contains the crucial explicit computation of the cocycle $c_{1,\lambda}$ (in \S\ref{sec:explicit computation}). In \S\ref{sec:p-arithmetic cohomology} we recap $p$-arithmetic cohomology and fix conventions for level subgroups and Eisenstein classes. The technical heart of Part 1 is \S\ref{sec:abstract exceptional zero formula}, where we construct our evaluation maps, and use them to attach a $p$-adic measure and period to an appropriate $p$-arithmetic class. The main result of this section is an abstract `exceptional zero formula' relating this measure and period (Theorem \ref{AbstractLinvTheorem}). In \S\ref{sec:recall L-invariants} we recall the $p$-adic $L$-functions and automorphic $\mathcal{L}$-invariants for $\pi$, and prove the duality $\mathcal{L}_{\pi,2}^{\opn{Aut}} = \mathcal{L}_{\pi^\vee,1}^{\opn{Aut}}$. We conclude Part 1 in \S\ref{sec:automorphic EZF}, where we draw all of this together in our specific context to deduce Theorem \ref{SecondMainThmIntroAutLinv}.

In Part 2, we prove Theorems \ref{ThirdMainThmIntroEqualLinv} and \ref{FourthMainThmintroLGcompat}. In \S\ref{sec:symm square} we first prove Theorem \ref{ThirdMainThmIntroEqualLinv} in essentially self-dual cases via a comparison to Rosso's work \cite{RossoSym2, RossoSym2II}. In \S\ref{Sec:NSDrepresentations}, we study the $\GL(3)$-eigenvariety and Koszul complexes of $p$-adic distributions, implementing the strategy of Gehrmann--Rosso \cite{GehrmannRosso} to prove Theorem \ref{ThirdMainThmIntroEqualLinv} in general, up to a crucial trianguline hypothesis (Theorem \ref{TriangulationInFamilies}). In \S\ref{LGcompatAtl=pSec}, the most important section in Part 2, we prove Theorem \ref{FourthMainThmintroLGcompat} and show that it implies Theorem \ref{TriangulationInFamilies}, thus completing the proof of Theorem \ref{ThirdMainThmIntroEqualLinv}  (and hence Theorem \ref{FirstMainThmInIntro}).

\subsection{Acknowledgements}
We would particularly like to thank James Newton for many discussions on the content of \S\ref{LGcompatAtl=pSec}. The influence of the work of Lennart Gehrmann on Part I of this paper will be evident to the reader, and we thank him for many conversations about his work. We would also like to thank Vaughan McDonald for comments on an earlier version of this article.

DBS was supported by ECOS230025, Anid Fondecyt Regular grant 1241702 and  a Lluís Santaló Visiting Position funded by the CRM in Barcelona. AG was funded by UK Research and Innovation grant MR/V021931/1. CW was supported by EPSRC Postdoctoral Fellowship EP/T001615/2. For the purpose of Open Access, the authors have applied a CC BY public copyright licence to any Author Accepted Manuscript (AAM) version arising from this submission.

\subsection{Notation}
 Let $G = \GL_3$, and let $H = \GL_2 \times \GL_1$, diagonally embedded in $G$. Let $L/\Q_p$ be a finite extension with ring of integers $\mathcal{O}_L$ and uniformiser $\varpi \in \cO_L$. We let $I_n$ denote the $(n \times n)$ identity matrix. For a reductive group $G/\mbb{Q}$, let $G(\mbb{Q})_+ \defeq G(\mbb{Q}) \cap G(\mbb{R})_+$, where $G(\mbb{R})_+ \subset G(\mbb{R})$ denotes the connected component of the identity.


\part{Automorphic \texorpdfstring{$\mathcal{L}$}{L}-invariants}
 
\section{Extensions of Steinberg representations}\label{sec:extensions of steinberg}

In this section, we study extensions of Steinberg representations, developing and extending \cite[\S 2]{AutomorphicLinvariants} in the specific setting of $\GL_3$. Crucially, in Proposition 
\ref{prop:Clambda1LuxembourgProp}, we explicitly compute values of a cocycle attached to a particular extension. We also prove compatibilities between locally analytic, continuous and smooth extension classes, with coefficients in $L, \cO_L$ and $\cO_L/\varpi^s$.

\subsection{Generalised Steinberg representations}

Let $R$ be a $\Z$-algebra. We denote the standard upper triangular Borel subgroup of $G$ by $B$. We let $P_1$ (resp. $P_2$) denote the standard parabolic of $G = \opn{GL}_3$ of type $(1, 2)$ (resp. $(2, 1)$). Let $M_1$ (resp.\ $M_2$) be its Levi subgroup.

For any parabolic $P \subset \GL_3$, we write $\overline{P}$ for the opposite parabolic.

\begin{definition}
    For $i=1, 2$, let $\opn{I}^{\opn{sm}}_{\overline{P}_i}(R)$ denote the $R[G(\Q_p)]$-module of smooth functions $f \colon \overline{P}_i(\Q_p) \backslash G(\Q_p) \to R$, with $G(\Qp)$ acting by right translation. Let $\opn{St}^{\opn{sm}}_i(R)$ denote the quotient of $\opn{I}^{\opn{sm}}_{\overline{P}_i}(R)$ by the constant functions. 

    Let $\opn{I}^{\opn{sm}}_{\overline{B}}(R)$ denote the $R[G(\Q_p)]$-module of smooth functions $f \colon \overline{B}(\Q_p) \backslash G(\Q_p) \to R$. We set
    \[
    \opn{St}^{\opn{sm}}(R) \defeq \opn{I}^{\opn{sm}}_{\overline{B}}(R) / \left( \opn{I}^{\opn{sm}}_{\overline{P}_1}(R) + \opn{I}^{\opn{sm}}_{\overline{P}_2}(R) \right) .
    \]
\end{definition}

We have the following:

\begin{lemma} \label{SteinbergProjLemma}
    For $i=1, 2$, the representations $\opn{I}^{\opn{sm}}_{\overline{P}_i}(R)$ and $\opn{St}_i^{\opn{sm}}(R)$ are projective $R$-modules which satisfy $\opn{I}^{\opn{sm}}_{\overline{P}_i}(R) = \opn{I}^{\opn{sm}}_{\overline{P}_i}(\Z) \otimes_{\Z} R$ and $\opn{St}_i^{\opn{sm}}(R) = \opn{St}_i^{\opn{sm}}(\Z) \otimes_{\Z} R$.
\end{lemma}
\begin{proof}
    Since $\overline{P}_i(\Q_p) \backslash G(\Q_p)$ is compact, one can easily verify that $\opn{I}^{\opn{sm}}_{\overline{P}_i}(R) = \opn{I}^{\opn{sm}}_{\overline{P}_i}(\Z) \otimes_{\Z} R$ and $\opn{St}_i^{\opn{sm}}(R) = \opn{St}_i^{\opn{sm}}(\Z) \otimes_{\Z} R$. Furthermore, one has $\opn{I}^{\opn{sm}}_{\overline{P}_i}(\Z) \cong \opn{St}_i^{\opn{sm}}(\Z) \oplus \Z$ as $\Z$-modules, so it is enough to show that $\opn{St}_i^{\opn{sm}}(\Z)$ is a projective $\Z$-module. But this follows from \cite[Thm.\  2.6]{AutomorphicLinvariants}.
\end{proof}

We also need continuous and locally analytic variants. For $P \in \{ P_1, P_2, B \}$, let $\opn{I}^{\opn{cts}}_{\overline{P}}(\mathcal{O}_L)$ (resp. $\opn{I}^{\opn{cts}}_{\overline{P}}(L)$, resp. $\opn{I}^{\opn{la}}_{\overline{P}}(L)$) denote the space of continuous functions $G(\Q_p) \to \mathcal{O}_L$ (resp. continuous functions $G(\Q_p) \to L$, resp. locally analytic functions $G(\Q_p) \to L$) which are left-invariant under the action of $\overline{P}(\Q_p)$. These representations have a continuous action of $G(\Q_p)$ by right-translation; this action is locally analytic on $\opn{I}_{\overline{P}}^{\opn{la}}(L)$. 

\begin{definition}\label{def:generalised steinberg}
    For $i=1, 2$ and $(R, \bullet) \in \{ (\mathcal{O}_L, \opn{cts}), (L, \opn{cts}), (L, \opn{la}) \}$, let $\opn{St}_i^{\bullet}(R)$ denote the quotient of $\opn{I}_{\overline{P}_i}^{\bullet}(R)$ by the submodule of constant functions. We also set
    \[
    \opn{St}^{\bullet}(R) \defeq \opn{I}_{\overline{B}}^{\bullet}(R)/\left( \opn{I}_{\overline{P}_1}^{\bullet}(R) + \opn{I}_{\overline{P}_2}^{\bullet}(R) \right) .
    \]

\end{definition}

\begin{remark}
	Note that $\opn{I}^{\opn{la}}_{\overline{P}}(L)$ can be seen as the locally analytic vectors in the admissible Banach representation $\opn{I}_{\overline{P}}^{\opn{cts}}(L)$ (in the sense of \cite[\S3]{ST02} and \cite{Eme17}). 
	
	Additionally, $\opn{St}_i^{\opn{la}}(L)$ (resp. $\opn{St}^{\opn{la}}(L)$) can also be seen as the locally analytic vectors in $\opn{St}_i^{\opn{cts}}(L)$ (resp. $\opn{St}^{\opn{cts}}(L)$), since passing to locally analytic vectors is an exact functor on the category of admissible Banach representations (see \cite[Thm.\ 7.1]{ST03}, or more precisely \cite[Thm.\ 2.2.3]{Pan22}).
\end{remark}

\begin{notation} \label{BulletRLambdaNotation}
    In what follows, when we write $(\bullet, R, \lambda)$, we mean one of the following cases:
    \begin{enumerate}
        \item $\bullet = \opn{sm}$, we take $R$ to be a $\Z$-algebra, and we let $\lambda \colon \Q_p^{\times} \to R$ be a smooth homomorphism (where the target is equipped with the additive group structure).
        \item $\bullet = \opn{cts}$, we take $R = \mathcal{O}_L$ or $R = L$, and we let $\lambda \colon \Q_p^{\times} \to R$ be a continuous homomorphism.
        \item $\bullet = \opn{la}$, we take $R = L$, and we let $\lambda \colon \Q_p^{\times} \to R$ be a continuous homomorphism (which is automatically locally analytic).
    \end{enumerate}
    When we write $\opn{St}^{\bullet}(R)^*$ and $\opn{St}_i^{\bullet}(R)^*$, we mean the $R$-linear dual in case (1), and the continuous $R$-linear dual in cases (2) and (3) (for the topologies on both the target and source).
\end{notation}

We have the following analogue of Lemma \ref{SteinbergProjLemma}:

\begin{lemma}\label{lem:steinberg base change}
Let $\ast \in \{1, 2, \varnothing\}$.
\begin{itemize}
\item[(i)] For $s\geq 1$ a positive integer, we have $\opn{St}_{\ast}^{\mathrm{cts}}(\cO_L) \otimes_{\cO_L} \cO_L/\varpi^s \cong \opn{St}_{\ast}^{\mathrm{sm}}(\cO_L/\varpi^s)$.
\item[(ii)] We have $\opn{St}_{\ast}^{\mathrm{cts}}(\cO_L) \otimes_{\cO_L} L \cong \opn{St}_{\ast}^{\mathrm{cts}}(L)$.
\end{itemize}
\end{lemma}
\begin{proof}
Let $P \in \{P_1,P_2,B\}$. If $f \colon \overline{P}(\Qp)\backslash G(\Qp) \to \cO_L$ is continuous, then its reduction modulo $\varpi^s$ is locally constant (i.e., smooth), so 
\begin{equation}\label{eq:smooth mod}
\opn{I}_{\overline{P}}^{\mathrm{cts}}(\cO_L) \otimes_{\cO_L}\cO_L/\varpi^s \cong \opn{I}^{\mathrm{sm}}_{\overline{P}}(\cO_L/\varpi^s).
\end{equation}
Also any continuous $g :  \overline{P}(\Qp)\backslash G(\Qp) \to L$ is bounded by compactness, so we have 
\begin{equation}
\label{eq:cts extend}\opn{I}_{\overline{P}}^{\mathrm{cts}}(\cO_L) \otimes_{\cO_L}L \cong \opn{I}^{\mathrm{cts}}_{\overline{P}}(L).
\end{equation}

If $\ast = i = 1,2$, then $\opn{St}^{\opn{cts}}_i(\cO_L)$ is defined by the short exact sequence 
\[
0 \to \cO_L \to \opn{I}^{\opn{cts}}_{\overline{P}_i}(\cO_L) \to \opn{St}^{\opn{cts}}_{i}(\cO_L) \to 0.
\]
Tensoring with $\cO_L/\varpi^s$ and using \eqref{eq:smooth mod} yields a short exact sequence 
\[
0 \to \cO_L/\varpi^s \to \opn{I}^{\opn{sm}}_{\overline{P}_i}(\cO_L/\varpi^s) \to \opn{St}^{\opn{cts}}_{i}(\cO_L)\otimes_{\cO_L}\cO_L/\varpi^s \to 0,
\]
noting this remains exact since 
\[
    \mathrm{Tor}_1^{\cO_L}(\opn{St}^{\mathrm{cts}}_{i}(\cO_L),\cO_L/\varpi^s) = \opn{St}^{\mathrm{cts}}_{i}(\cO_L)[\varpi^s] = 0.
\]
Similarly for $\ast = \varnothing$ we obtain an exact sequence 
\[
0 \to \opn{I}_{\overline{P}_1}^{\mathrm{sm}}(\cO_L/\varpi^s)+\opn{I}_{\overline{P}_2}^{\mathrm{sm}}(\cO_L/\varpi^s) \to \opn{I}^{\opn{sm}}_{\overline{B}}(\cO_L/\varpi^s) \to \opn{St}^{\opn{cts}}(\cO_L)\otimes_{\cO_L}\cO_L/\varpi^s \to 0.
\]
As these sequences define $\opn{St}^{\mathrm{sm}}_{\ast}(\cO_L/\varpi^s)$, we deduce (i). Part (ii) follows similarly after tensoring the sequences defining $\mathrm{St}_{\ast}^{\opn{cts}}(\cO_L)$ with $L$, noting we retain exactness as $L$ is $\cO_L$-flat.
\end{proof}

\subsection{Extensions I} \label{Sub:ExtensionsI}

Let $(\bullet, R, \lambda)$ be as in Notation \ref{BulletRLambdaNotation}, so $\lambda$ is an additive character of $\Qp^\times$. We now construct extensions of Steinberg representations (and, ultimately, 1-cocycles) attached to such characters. Our chief interest will later be in the characters $\lambda = \log_p$ and $\mathrm{ord}_p$.   

Let $v_1$ be the character of $M_1(\Qp) = \GL_1(\Qp)\times\GL_2(\Qp)$ defined by
\[
v_1\left( \begin{smallmatrix} a & & \\ & b & c \\ & d & e \end{smallmatrix} \right) = a^{-1} \cdot \opn{det}\left( \begin{smallmatrix} b & c \\ d & e \end{smallmatrix} \right).
\]
This naturally extends to a character of $\overline{P}_1(\Qp)$ via projection to the Levi $M_1(\Qp)$. Denote by $\tau_{\lambda}$ the two dimensional representation of $\overline{P}_1(\Q_p)$ given by
\[
\tau_\lambda: p \longmapsto \tbyt{1}{(\lambda \circ v_1)(p)}{}{1}. 
\]
 We let $\opn{I}^{\bullet}_{\overline{P}_1}(\tau_{\lambda})$ denote the $\bullet$-induction of this representation, that is
\begin{align*}
\opn{I}^{\bullet}_{\overline{P}_1}(\tau_\lambda) = \Big\{f : G(\Qp) \to R^{\oplus 2} : f&\text{ satisfies property $\bullet$, and}
\\
&f(p \cdot -) = \tau_{\lambda}(p) \cdot f(-) \ \forall p \in \overline{P}_1(\Qp)\Big\}.
\end{align*}

Note that:
\begin{itemize}
	\item the map $R \to R^{\oplus 2}$, $r \mapsto (r,0)$ induces an $R[G(\Qp)]$-module map $\jmath_1 : \opn{I}_{\overline{P}_1}^\bullet(R) \to \opn{I}_{\overline{P}_1}^\bullet(\tau_\lambda)$, and
	\item the map $R^{\oplus 2} \to R$, $(r_1,r_2) \mapsto r_2$ induces an $R[G(\Qp)]$-module map $\jmath_2 : \opn{I}_{\overline{P}_1}^\bullet(\tau_\lambda) \to \opn{I}_{\overline{P}_1}^\bullet(R)$. 
\end{itemize}

\begin{lemma}\label{lem:short exact sequence extension}
	We have a short exact sequence 
\begin{equation} \label{IndExactSeqEqn}
	0 \to \opn{I}^{\bullet}_{\overline{P}_1}(R) \xrightarrow{\ \ \jmath_1 \ \ } \opn{I}^{\bullet}_{\overline{P}_1}(\tau_{\lambda}) \xrightarrow{\ \ \jmath_2 \ \ } \opn{I}^{\bullet}_{\overline{P}_1}(R) \to 0 .
\end{equation}
\end{lemma}

\begin{proof}
The map $\jmath_1$ is visibly injective, and it is clear from the definitions that the sequence is exact at $\opn{I}^\bullet_{\overline{P}_1}(\tau_\lambda)$. It remains to show that $\jmath_2$ is surjective.

Let $s \colon G(\Q_p) \to \overline{P}_1(\Q_p)$ be a continuous, left $\overline{P}_1(\Q_p)$-equivariant map; we will fix a specific choice, hence showing existence of such a map, in Lemma \ref{lem:g = pwk}. Now, given any $F \in \opn{I}^{\bullet}_{\overline{P}_1}(R)$, define $\tilde{F} : G(\Qp) \to R^{\oplus 2}$ by
\[
\tilde{F}(g) =  \begin{pmatrix}(\lambda \circ v_1 \circ s)(g) F(g)\\F(g)\end{pmatrix}.
\]
For any $p \in \overline{P}_1(\Qp)$, we have $F(pg) = F(g)$ and $s(pg) = ps(g)$, so
\begin{align*}
	\tilde{F}(pg) = \begin{pmatrix}(\lambda \circ v_1 \circ s)(pg) F(pg)\\F(pg)\end{pmatrix} &= \begin{pmatrix}(\lambda  \circ v_1 \circ s)(g) F(g) + (\lambda\circ  v_1)(p)F(g)\\F(g)\end{pmatrix}\\
	 &= \tau_\lambda(p)\begin{pmatrix}(\lambda \circ v_1  \circ s)(g) F(g)\\F(g)\end{pmatrix}.
\end{align*}
Hence $\tilde{F} \in \opn{I}^{\bullet}_{\overline{P}_1}(\tau_{\lambda})$. As visibly $\jmath_2(\tilde{F}) = F$, we deduce surjectivity.
\end{proof}

We can pullback and pushforward the exact sequence (\ref{IndExactSeqEqn}) along the maps $R \hookrightarrow \opn{I}^{\bullet}_{\overline{P}_1}(R)$ and $\opn{I}^{\bullet}_{\overline{P}_1}(R) \twoheadrightarrow \opn{St}^{\bullet}_1(R)$ to obtain an exact sequence\footnote{Explicitly, $\cE_{1,\lambda}^\bullet(R)$ is the quotient
$\cE_{1,\lambda}^\bullet(R) = \{f \in \opn{I}_{\overline{P}_1}^\bullet(\tau_\lambda) : \jmath_2(f) \in R\}/\jmath_1(R) \subset \opn{I}_{\overline{P}_1}^\bullet(\tau_\lambda)/\jmath_1(R).
$
}
\begin{equation}\label{eq:St extension}
0 \to \opn{St}^{\bullet}_1(R) \to \mathcal{E}_{1,\lambda}^{\bullet}(R) \to R \to 0 .
\end{equation}

We can describe this extension by a cocycle as follows. Let $s$ be a choice of splitting as in the proof of Lemma \ref{lem:short exact sequence extension}. Then we obtain a cocycle $c_{1,\lambda}^\bullet \in Z^1(G(\Q_p), \opn{St}^{\bullet}_1(R))$ by the formula
\begin{equation}\label{eq:c cocycle}
c_{1,\lambda}^\bullet[x] = \Big[ g \longmapsto (\lambda \circ v_1 \circ s)(gx) - (\lambda \circ v_1 \circ s)(g) \Big].
\end{equation}
Here note that $(g \mapsto \lambda v_1 s(gx) - \lambda v_1s(g)) \in \opn{I}_{\overline{P}_1}^\bullet(R)$, and we take its image in $\mathrm{St}^{\bullet}_1(R)$. 

\begin{remark}
Of course, this cocycle depends on the choice of splitting $s \colon G(\Q_p) \to \overline{P}_1(\Q_p)$. However, its image in $\opn{H}^1(G(\Q_p), \opn{St}^{\bullet}_1(R))$ does not; it is the image of $1 \in R = \h^0(G(\Qp),R)$ under the boundary map 
\[
\h^0(G(\Qp),R) \to \h^1(G(\Qp),\opn{St}_1^\bullet(R))
\]
in the long exact sequence attached to \eqref{eq:St extension}.
\end{remark}

\subsection{Explicit computation of $c_{1,\lambda}^\bullet$}\label{sec:explicit computation}

For our later results, we must explicitly describe a value of the cocycle $c_{1,\lambda}^\bullet$ defined above. We do this in Proposition \ref{prop:Clambda1LuxembourgProp}. For this, we now introduce an explicit section $s$ (which we fix for the rest of the article, unless specified otherwise). 

\begin{notation}\label{not:weyl reps}
	Let ${^{M_1}W}$ denote the set of minimal length representatives for $W_{M_1} \backslash W_G$, where $M_1$ is the Levi of $P_1$. Explicitly, we take:
	\[
	{^{M_1}W} = \left\{ \opn{id}, \smat{ & 1 & \\ 1 & & \\ & & 1 }, \smat{ &  & 1\\ 1 & &  \\ & 1 & } \right\}.
	\]
\end{notation}

\begin{lemma}\label{lem:g = pwk}
	For any $g \in G(\Qp)$, there exist unique 
	\[
	\overline{p} \in \overline{P}_1(\Qp), \qquad w \in {^{M_1}}W, \qquad \text{and} \qquad k\in\opn{Iw}\cap w^{-1}N_1(\Zp)w
	\]
	such that 
	\[
	g = \overline{p} \cdot w \cdot k.
	\]
	In particular, there is a well-defined left-$\overline{P}_1(\Qp)$-equivariant continuous map $s\colon  G(\Qp) \to \overline{P}_1(\Qp)$ defined by $s(g) \defeq \overline{p}$.
\end{lemma}

Although this is a standard result, we provide an explicit proof; the notation introduced will also be useful in the proof of Proposition \ref{prop:Clambda1LuxembourgProp}. Note that $\overline{P}_1(\Qp)\backslash G(\Qp) \cong \mathbb{P}^2(\Qp)$ via the map
\[
\smallthreemat{a}{b}{c}{\star}{\star}{\star}{\star}{\star}{\star} \longmapsto [a:b:c].
\]

\begin{proof}
	Suppose $g = \overline{p}wk = \overline{p}'w'k'$ are two such expressions. Then $w = w'$ by disjointness in the Bruhat decomposition; and rearranging gives 
	\[
	(\overline{p})^{-1}\overline{p}' = wk(k')^{-1}w^{-1}.
	\]
	The right-hand side is in $N_1(\Zp)$, whilst the left is in $\overline{P}_1(\Qp)$. As the intersection of these is trivial, we deduce $\overline{p} = \overline{p}'$ and $k = k'$. In particular, any such expression as in the lemma must be unique, and (provided we have existence) the association $g \mapsto s(g) \defeq \overline{p}$ is well-defined.

    \medskip
	
	It remains to prove existence. It is equivalent to show $g = \overline{p}nw$, for $n \in N_1(\Zp) \cap w\opn{Iw}w^{-1}$. Let 
	\[
	g  =\smallthreemat{a}{b}{c}{\star}{\star}{\star}{\star}{\star}{\star} \in G(\Qp).
	\]
	We break into a series of cases, corresponding to different subsets of $\mathbb{P}^2(\Qp)$.
	\begin{itemize}
		\item[(a)] If $a \neq 0$, then without loss of generality scale by $a^{-1}I_3 \in \overline{P}_1(\Qp)$ to assume $a = 1$. Then:
        \begin{itemize}
        \item[(a-1)] If $b,c \in \Zp$, then let $w = \mathrm{id}$, and
		\[
		k = n = \smallthreemat{1}{b}{c}{}{1}{}{}{}{1} \in \opn{Iw}\cap N_1(\Zp).
		\]
Then
\[
s(g) = \overline{p} \defeq gn^{-1} \in \overline{P}_1(\Qp),
\]
		and $g= \overline{p}n = \overline{p}k$ is the desired expression.

\item[(a-2)] If $v_p(b) < 0$ and $v_p(b) \leq v_p(c)$, then let $w = \smat{ & 1 & \\ 1 & & \\ & & 1 }$,
\[n = \smallthreemat{1}{b^{-1}}{b^{-1}c}{}{1}{}{}{}{1} \in N_1(\Zp), \qquad 
k = w^{-1}nw = \smallthreemat{1}{}{}{b^{-1}}{1}{b^{-1}c}{}{}{1} \in \opn{Iw}\cap w^{-1}N_1(\Zp) w.
\]
Then $s(g) = \overline{p}\defeq gw^{-1}n^{-1} \in \overline{P}_1(\Qp),$ and $g = \overline{p}nw = \overline{p}wk.$

\item[(a-3)] If $v_p(c) < 0$ and $v_p(c) < v_p(b)$, then let $w = \smat{ &  & 1\\ 1 & & \\ & 1 &  }$, 
\[n = \smallthreemat{1}{c^{-1}}{bc^{-1}}{}{1}{}{}{}{1} \in N_1(\Zp),\qquad k= w^{-1}nw = \smallthreemat{1}{}{}{}{1}{}{c^{-1}}{bc^{-1}}{1} \in \opn{Iw}\cap w^{-1}N_1(\Zp) w.
\]
Then $s(g) = \overline{p}\defeq gw^{-1}n^{-1} \in \overline{P}_1(\Qp),$ and $g = \overline{p}nw = \overline{p}wk.$
\end{itemize}

\medskip

Note (a-1), (a-2) and (a-3) cover all possible cases for $(b,c) \in \Qp^2$.

\bigskip
        
		\item[(b)] If $a = 0$ and $b\neq 0$, then without loss of generality we may assume $b=1$. Then:
        \begin{itemize}
        \item[(b-1)] If $c \in \Zp$, then let $w = \smat{ & 1 & \\ 1 & & \\ & & 1 }$, and
		\[
		n = \smallthreemat{1}{}{c}{}{1}{}{}{}{1} \in N_1(\Zp), \qquad k = w^{-1}nw = \smallthreemat{1}{}{}{}{1}{c}{}{}{1} \in \opn{Iw}\cap w^{-1}N_1(\Zp) w.
        \]
        Then $s(g) = \overline{p} \defeq gw^{-1}n^{-1} \in \overline{P}_1(\Qp),$ and $g = \overline{p}nw = \overline{p}wk$.

        \item[(b-2)] If $v_p(c) < 0$, then let $w = \smat{ &  & 1\\ 1 & & \\ & 1 &  }$, and
        		\[
		n = \smallthreemat{1}{}{c^{-1}}{}{1}{}{}{}{1} \in N_1(\Zp), \qquad k = w^{-1}nw = \smallthreemat{1}{}{}{}{1}{}{}{c^{-1}}{1} \in \opn{Iw}\cap w^{-1}N_1(\Zp) w.
        \]
        Then $s(g) = \overline{p} \defeq gw^{-1}n^{-1} \in \overline{P}_1(\Qp),$ and $g = \overline{p}nw = \overline{p}wk$.
\end{itemize}

        \bigskip
        
		\item[(c)] Finally if $a=b=0$, then without loss of generality we may assume $c=1$. Then
		\[
		s(g) = \overline{p} \defeq g \smat{ &  & 1\\ 1 & &  \\ & 1 & }^{-1} \in \overline{P}_1(\Qp),
		\]
		so $w =  \smat{ &  & 1\\ 1 & &  \\ & 1 & }$ and $g = \overline{p}w$ is the desired expression (so $k = I_3$).
	\end{itemize}
	As this covers all possible $a,b,c$, it covers all $g$, and every $g \in G(\Qp)$ can be written uniquely as $g = \overline{p}wk$. The function $g \mapsto s(g) = \overline{p}$ is visibly continuous and left-$\overline{P}_1(\Qp)$-equivariant.
\end{proof}

With this choice of $s$, we can explicitly calculate the cocycle $c_{1,\lambda}^\bullet$ evaluated at the element $t \defeq \opn{diag}(p, 1, 1) \in G(\Q_p)$.

\begin{proposition} \label{prop:Clambda1LuxembourgProp}
The element $c_{1,\lambda}^\bullet[t] \in \opn{St}_1^\bullet(R)$ can be represented by a function 
\[
F \colon \overline{P}_1(\Q_p) \backslash G(\Q_p) \to R \qquad \in \opn{I}_{\overline{P}_1}^\bullet(R)
\]
which has support in $\overline{P}_1(\Q_p)N_1(\Z_p)$ and satisfies
        \[
        F(g) = \left\{ \begin{array}{cc} -2\lambda(p) & \text{ if } x, y \in p\Zp \\ - 2\lambda(x) & \text{ if } x \in \Zp^\times, y \in \Zp \\ -2\lambda(y) & \text{ if } x \in p\Zp, y \in \Zp^\times, \end{array} \right.
        \]
        where we write $g = \overline{p}\smat{1 & x & y \\ & 1 & \\ & & 1} \in \overline{P}_1(\Qp)N_1(\Z_p)$.
\end{proposition}

\begin{proof}
    By above $c_{1,\lambda}^\bullet[t] \in \opn{St}_1^\bullet(R)$ is represented by the function $F' \in \opn{I}_{\overline{P}_1}^\bullet(R)$ defined by 
    \[
    F'(g) = (\lambda \circ v_1 \circ s)(g t) - (\lambda \circ v_1 \circ s)(g), \qquad g\in G(\Qp).
    \]
As $\lambda(p) \in R \subset \opn{I}_{\overline{P}_1}^\bullet(R)$ is constant, $c_{1,\lambda}^\bullet[t]$ is also represented by the function
    \[
        F(g) \defeq F'(g) - \lambda(p).
    \]

     We now calculate this explicitly. By Lemma \ref{lem:g = pwk}, any $g \in G(\Qp)$ has the form $g = s(g)nw$, for unique $s(g)\in \overline{P}_1(\Qp)$, $n = wkw^{-1} \in N_1(\Zp) \cap w\opn{Iw}w^{-1}$ and $w \in {^{M_1}}W$. By left-$\overline{P}_1(\Qp)$-invariance, we may without loss of generality assume $s(g) = 1$, so 
     \[
        g = \smallthreemat{1}{x}{y}{}{1}{}{}{}{1}w \in (N_1(\Zp) \cap w \opn{Iw} w^{-1}) \cdot w.
     \]
     Thus $(\lambda\circ v_1\circ s)(g) = \lambda(1) = 0$,   and
     \begin{equation}\label{eq:simpler F'}
     F'(g) = (\lambda\circ v_1\circ s)(gt) = (\lambda\circ v_1\circ s)(nwt).
     \end{equation}
     We break the computation into cases based on $n$ and $w$.

\medskip
    
    \textbf{Case 1:} First suppose $w \neq \mathrm{id}$. Then $w = \smat{ & 1 & \\ 1 & & \\ & & 1 }$ or $\smat{ &  & 1\\ 1 & &  \\ & 1 & }$, and in either case, we have
    \[
    wtw^{-1} = \smallthreemat{1}{}{}{}{p}{}{}{}{1} =: t'.
    \]
   Let also 
   \[
    n' = (t')^{-1}n t' = \smallthreemat{1}{px}{y}{}{1}{}{}{}{1} \in N_1(\Zp) \cap w\opn{Iw}w^{-1}.
\]
Then we have 
    \[
            gt = nwt = nt'w = t'n'w.
    \]
 In particular,
    \[
        s(gt) = t'.
    \]
    As $v_1(t') = p$, we have
    \begin{align}\label{eq:compute F'}
F'(g) = \lambda\circ v_1\circ s(gt) = \lambda\circ v_1\left(t'\right) = \lambda(p).
    \end{align}
    It follows that $F(g) = F'(g) - \lambda(p) = 0$. 
    In particular, $F$ vanishes identically on both cells 
    \[
        \overline{P}_1(\Qp)\cdot (N_1(\Zp) \cap w \opn{Iw} w^{-1}) \cdot w = \overline{P}_1(\Qp) \cdot w \cdot \opn{Iw}, \quad \quad  w=\smat{ & 1 & \\ 1 & & \\ & & 1 }, \smat{ &  & 1\\ 1 & &  \\ & 1 & };
    \]
    hence $F$ is supported on the remaining cell $\overline{P}_1(\Qp)N_1(\Zp)$.

    \medskip

    \textbf{Case 2:} Now suppose $w = \mathrm{id}$, so $g = n = \smallthreemat{1}{x}{y}{}{1}{}{}{}{1} \in N_1(\Zp)$. We wish to compute $F'(g) = (\lambda\circ v_1\circ s)(gt).$ We have
    \[
        gt = nt = t\smallthreemat{1}{p^{-1}x}{p^{-1}y}{}{1}{}{}{}{1},
    \]
    so as $t \in \overline{P}_1(\Qp)$ and $s$ is left-$\overline{P}_1(\Qp)$-equivariant, we have 
    \begin{equation}\label{eq:sgt}
    s(gt) = t\cdot s\left(\smallthreemat{1}{p^{-1}x}{p^{-1}y}{}{1}{}{}{}{1}\right).
    \end{equation}
    In the proof of Lemma \ref{lem:g = pwk}, we explicitly computed $s\left(\smallthreemat{1}{b}{c}{}{1}{}{}{}{1}\right)$ for arbitrary $b,c \in \Qp$. We now break into subcases, depending on the valuations of $x$ and $y$.
    
    \begin{itemize}\setlength{\itemsep}{5pt}
        \item[(a-1)] If $x, y \in p \Z_p$, then we have
        \[
        s\left(\smallthreemat{1}{p^{-1}x}{p^{-1}y}{}{1}{}{}{}{1}\right) = 1.
        \]
        Hence $s(g t) = t$. As $v_1(t) = p^{-1}$, by \eqref{eq:simpler F'} we see 
        \[
        F'(g) = 
        \lambda(p^{-1}) = -\lambda(p),
    \] 
    so $F(g) = -2\lambda(p)$.

        \item[(a-2)] If $x \in \Z_p^{\times}$ and $y \in \Z_p$, then as in Lemma \ref{lem:g = pwk}(a-2) we have 
        \begin{align*}
         s\left(\smallthreemat{1}{p^{-1}x}{p^{-1}y}{}{1}{}{}{}{1}\right) &= \smallthreemat{1}{p^{-1}x}{p^{-1}y}{}{1}{}{}{}{1}\smallthreemat{}{1}{}{1}{}{}{}{}{1}\smallthreemat{1}{-px^{-1}}{-x^{-1}y}{}{1}{}{}{}{1}\\
         &= \smallthreemat{p^{-1}x}{}{}{1}{-px^{-1}}{-x^{-1}y}{}{}{1},
        \end{align*}
        so by \eqref{eq:sgt},  $s(gt) = \smallthreemat{x}{}{}{1}{-px^{-1}}{-x^{-1}y}{}{}{1}$. Thus 
        \[
            F'(g) = \lambda \circ v_1\circ s(gt) = \lambda(px^{-2}) = \lambda(p)-2\lambda(x),\]
            so $F(g) = -2\lambda(x)$.

        \item[(a-3)] If $x \in p\Z_p$ and $y \in \Z_p^{\times}$, then as in Lemma \ref{lem:g = pwk}(a-3) we have
         \begin{align*}
         s\left(\smallthreemat{1}{p^{-1}x}{p^{-1}y}{}{1}{}{}{}{1}\right) &= \smallthreemat{1}{p^{-1}x}{p^{-1}y}{}{1}{}{}{}{1}\smallthreemat{}{1}{}{}{}{1}{1}{}{}\smallthreemat{1}{-py^{-1}}{-xy^{-1}}{}{1}{}{}{}{1}\\
         &= \smallthreemat{p^{-1}y}{}{}{}{}{1}{1}{-py^{-1}}{-xy^{-1}},
        \end{align*}
         so by \eqref{eq:sgt},  $s(gt) = \smallthreemat{y}{}{}{}{}{1}{1}{-py^{-1}}{-xy^{-1}}$. Thus 
        \[
            F'(g) = \lambda \circ v_1\circ s(gt) = \lambda(py^{-2}) = \lambda(p)-2\lambda(y),\]
            so $F(g) = -2\lambda(y)$.

    \end{itemize}
     This covers all cases, thus completes the proof.
\end{proof}

We introduce the following notation:

\begin{notation}
    Let 
    \[
    \delta^{\bullet} \colon \opn{Hom}_{\bullet}(\Q_p^{\times}, R) \to \opn{Ext}_{R[G(\Q_p)]}^1\left(R, \opn{St}^{\bullet}_1(R) \right)
    \]
    denote the map $\delta^{\bullet}(\lambda) = \mathcal{E}^{\bullet}_{1, \lambda}$, where $\opn{Ext}_{R[G(\Q_p)]}^1$ denotes the $R$-module of extensions in the category of abstract $R[G(\Q_p)]$-modules (i.e., we are forgetting any topology that these representations may carry). One can verify, using the description of $\mathcal{E}^{\bullet}_{1, \lambda}$ via the cocycle $c_{1,\lambda}^\bullet$, that $\delta^{\bullet}$ is $R$-linear.
\end{notation}

\begin{lemma}
Let $s \geq 1$ be a positive integer. We have the following commutative diagram:

\begin{equation} \label{FirstExtDiagramEqn}
\begin{tikzcd}
{\opn{Hom}_{\opn{sm}}(\Q_p^{\times}, \mathcal{O}_L/\varpi^s)}  \arrow[r, "\delta^{\opn{sm}}"]             & {\opn{Ext}^1_{\mathcal{O}_L/\varpi^s[G(\Q_p)]}\left( \mathcal{O}_L/\varpi^s, \opn{St}^{\opn{sm}}_1(\mathcal{O}_L/\varpi^s) \right)}                                                            \\
{\opn{Hom}_{\opn{sm}}(\Q_p^{\times}, \mathcal{O}_L)} \arrow[u, "\pmod{\varpi^s}"] \arrow[d, hook]           &                                                           \\
{\opn{Hom}_{\opn{cts}}(\Q_p^{\times}, \mathcal{O}_L)} \arrow[d, hook] \arrow[r, "\delta^{\opn{cts}}"] \arrow[bend right=73,uu, swap, "\!\!\!\!\pmod{\varpi^s}"] & {\opn{Ext}^1_{\mathcal{O}_L[G(\Q_p)]}\left( \mathcal{O}_L, \opn{St}_1^{\opn{cts}}(\mathcal{O}_L) \right)} \arrow[uu, "- \otimes \mathcal{O}_L/\varpi^s"'] \arrow[d, "- \otimes L"] \\
{\opn{Hom}_{\opn{cts}}(\Q_p^{\times}, L)} \arrow[r, "\delta^{\opn{cts}}"] \arrow[rd, "\delta^{\opn{la}}", ']                             & {\opn{Ext}^1_{L[G(\Q_p)]}\left(L, \opn{St}_1^{\opn{cts}}(L) \right)}                                                                                  \\
                                                                                                & {\opn{Ext}^1_{L[G(\Q_p)]}\left(L, \opn{St}_1^{\opn{la}}(L) \right)} \arrow[u]                                                                                                               
\end{tikzcd}
\end{equation}
\end{lemma}

\begin{proof}
Note any continuous homomorphism $\Qp^\times \to \cO_L$ is smooth modulo $\varpi^s$, so the curved map is well-defined. The right-hand maps can be described using the identification $\opn{Ext}^1_{R[G(\mbb{Q}_p)]}(R,M) \cong \h^1(G(\mbb{Q}_p),M)$ and, via Lemma \ref{lem:steinberg base change}, natural functorialities in $R$ and $M$. We can check commutativity at the level of cocycles, whence it follows from the description of $c_{1,\lambda}^\bullet$.
\end{proof}

\subsection{Extensions II}

We now construct certain extensions of $\opn{St}^{\opn{sm}}_2(R)$. Firstly, we note that there is a natural $G(\Q_p)$-equivariant map
\[
\opn{I}^{\bullet}_{\overline{P}_1}(R) \otimes_R \opn{I}^{\opn{sm}}_{\overline{P}_2}(R) \to \opn{I}^{\bullet}_{\overline{B}}(R)
\]
given by $(f_1, f_2) \mapsto f_1 \cdot f_2$. This induces a well-defined $G(\Q_p)$-equivariant morphism $\opn{pr} \colon \opn{St}_1^{\bullet}(R) \otimes_R \opn{St}_2^{\opn{sm}}(R) \to \opn{St}^{\bullet}(R)$. We consider the following composition (denoted $\Delta^{\bullet}$): 
\begin{align*}
\opn{Hom}_{\bullet}(\Q_p^{\times}, R) &\xrightarrow{\delta^{\bullet}} \opn{Ext}^1_{R[G(\Q_p)]}(R, \opn{St}_1^{\bullet}(R)) \\ &\xrightarrow{-\otimes \opn{St}_2^{\opn{sm}}(R)} \opn{Ext}^1_{R[G(\Q_p)]}(\opn{St}_2^{\opn{sm}}(R), \opn{St}_1^{\bullet}(R) \otimes_R \opn{St}_2^{\opn{sm}}(R)) \\ &\xrightarrow{\opn{pr}} \opn{Ext}^1_{R[G(\Q_p)]}(\opn{St}_2^{\opn{sm}}(R), \opn{St}^{\bullet}(R))
\end{align*}
where the second map is well-defined because $\opn{St}_2^{\opn{sm}}(R)$ is a projective (and hence flat) $R$-module (see Lemma \ref{SteinbergProjLemma}). By the diagram (\ref{FirstExtDiagramEqn}), we see that we have a commutative diagram:

\begin{equation} \label{SecondExtDiagramEqn}
\begin{tikzcd}
{\opn{Hom}_{\opn{sm}}(\Q_p^{\times}, \mathcal{O}_L/\varpi^s)}  \arrow[r, "\Delta^{\opn{sm}}"]             & {\opn{Ext}^1_{\mathcal{O}_L/\varpi^s[G(\Q_p)]}\left( \opn{St}_2^{\opn{sm}}(\mathcal{O}_L/\varpi^s), \opn{St}^{\opn{sm}}(\mathcal{O}_L/\varpi^s) \right)}                                                            \\
{\opn{Hom}_{\opn{sm}}(\Q_p^{\times}, \mathcal{O}_L)} \arrow[u, "\pmod{\varpi^s}"] \arrow[d, hook]           &                                                           \\
{\opn{Hom}_{\opn{cts}}(\Q_p^{\times}, \mathcal{O}_L)} \arrow[d, hook] \arrow[r, "\Delta^{\opn{cts}}"] \arrow[bend right=73,uu, swap, "\!\!\!\!\pmod{\varpi^s}"] & {\opn{Ext}^1_{\mathcal{O}_L[G(\Q_p)]}\left( \opn{St}_2^{\opn{sm}}(\mathcal{O}_L), \opn{St}^{\opn{cts}}(\mathcal{O}_L) \right)} \arrow[uu, "- \otimes \mathcal{O}_L/\varpi^s"'] \arrow[d, "- \otimes L"] \\
{\opn{Hom}_{\opn{cts}}(\Q_p^{\times}, L)} \arrow[r, "\Delta^{\opn{cts}}"] \arrow[rd, "\Delta^{\opn{la}}"']       & {\opn{Ext}^1_{L[G(\Q_p)]}\left(\opn{St}_2^{\opn{sm}}(L), \opn{St}^{\opn{cts}}(L) \right)}                                                                                                              \\
                                                                                                                      & {\opn{Ext}^1_{L[G(\Q_p)]}\left(\opn{St}_2^{\opn{sm}}(L), \opn{St}^{\opn{la}}(L) \right)} \arrow[u]                                                                                                    
\end{tikzcd}
\end{equation}

Note that the vertical maps given by tensor products are well-defined by Lemma \ref{lem:steinberg base change}. 

\begin{remark} \label{ExplicitDeltaCocycleRemark}
    Since $\opn{St}_2^{\opn{sm}}(R)$ is a projective $R$-module, we have an identification
    \begin{equation}\label{eq:ext1 = h1}
    \opn{Ext}^1_{R[G(\Q_p)]}(\opn{St}_2^{\opn{sm}}(R), \opn{St}^{\bullet}(R)) = \opn{H}^1\left( G(\Q_p), \opn{Hom}_R(\opn{St}_2^{\opn{sm}}(R), \opn{St}^{\bullet}(R)) \right) .
    \end{equation}
    Under this identification, for $\lambda \in \mathrm{Hom}_{\bullet}(\Qp^\times,R)$ the element $\Delta^\bullet_{1,\lambda} \defeq \Delta^\bullet(\lambda)$ can be explicitly described by the cocycle
    \begin{equation}\label{eq:Delta_1}
    \Delta_{1,\lambda}^\bullet \colon x \mapsto \left[ f \mapsto \opn{pr}(c_{1,\lambda}^\bullet[x] \otimes f) \right] .
    \end{equation}
\end{remark}

The following proposition describes the relation with the extensions considered in \cite[\S 2]{AutomorphicLinvariants}.

\begin{proposition} \label{AbstractExtFactorsThroughLAExtProp}
    The $L$-linear morphism $\Delta^{\opn{la}}$ factors as
    \[
    \opn{Hom}_{\opn{cts}}(\Q_p^{\times}, L) \xrightarrow{\sim} \opn{Ext}^1_{\opn{an}}(\opn{St}_2^{\opn{sm}}(L), \opn{St}^{\opn{la}}(L)) \to \opn{Ext}^1_{L[G(\Q_p)]}(\opn{St}_2^{\opn{sm}}(L), \opn{St}^{\opn{la}}(L))
    \]
    where the first map is the isomorphism in \cite[Theorem 2.15]{AutomorphicLinvariants} (well-defined up to scalar) and the second map is the natural one given by ``forgetting topologies''. Here $\opn{Ext}^1_{\opn{an}}$ denotes the group of extensions of locally analytic $G(\Q_p)$-representations as in \cite[\S 2.4]{AutomorphicLinvariants}. 
\end{proposition}
\begin{proof}
    Let us recall how the isomorphism in \cite[Theorem 2.15]{AutomorphicLinvariants} is constructed. Let $\Lambda \colon \overline{B}(\Q_p) \to L$ be a continuous homomorphism, and let $\sigma_{\Lambda}$ denote the two-dimensional representation of $\overline{B}(\Q_p)$ over $L$ given by
    \[
    b \mapsto \tbyt{1}{\Lambda(b)}{}{1} .
    \]
    Let $\opn{I}^{\opn{la}}_{\overline{B}}(\sigma_{\Lambda})$ denote the locally analytic induction of this representation. One has an exact sequence
    \[
    0 \to \opn{I}^{\opn{la}}_{\overline{B}}(L) \to \opn{I}^{\opn{la}}_{\overline{B}}(\sigma_{\Lambda}) \to \opn{I}^{\opn{la}}_{\overline{B}}(L) \to 0
    \]
    and by pulling back along $\opn{I}^{\opn{sm}}_{\overline{P}_2}(L) \hookrightarrow \opn{I}_{\overline{B}}^{\opn{la}}(L)$ and pushing out along $\opn{I}_{\overline{B}}^{\opn{la}}(L) \to \opn{St}^{\opn{la}}(L)$, we obtain a $L$-linear morphism\footnote{Note the strong analogy with Lemma \ref{lem:short exact sequence extension}.}
    \[
    \cC_{-} \colon \opn{Hom}_{\opn{cts}}(\overline{B}(\Q_p), L) \to \opn{Ext}^1_{\opn{an}}\left(\opn{I}^{\opn{sm}}_{\overline{P}_2}(L), \opn{St}^{\opn{la}}(L)\right), \qquad \Lambda \longmapsto \cC_\Lambda.
    \]
    Restriction from $\overline{P}_2$ to $\overline{B}$ realises $\opn{Hom}_{\opn{cts}}(\overline{P}_2(\Q_p), L) \subset \opn{Hom}_{\opn{cts}}(\overline{B}(\Q_p), L)$. As explained in \cite[Lemma 2.13]{AutomorphicLinvariants}, if $\Lambda \in \opn{Hom}_{\opn{cts}}(\overline{P}_2(\Q_p), L)$ lies in this subspace, then $\mathcal{C}_{\Lambda} = 0$; and as in \cite[Corollary 2.14, Theorem 2.15]{AutomorphicLinvariants}, one has isomorphisms
    \begin{align}\label{eq:lennart isos}
    \opn{Hom}_{\opn{cts}}(\Q_p^{\times}, L) \cong \opn{Hom}_{\opn{cts}}&(\overline{B}(\Q_p), L)/\opn{Hom}_{\opn{cts}}(\overline{P}_2(\Q_p), L)\\
    &\labelisorightarrow{\ \ \cC_{-} \ \ }\opn{Ext}^1_{\opn{an}}(\opn{I}^{\opn{sm}}_{\overline{P}_2}(L), \opn{St}^{\opn{la}}(L)) \labelisoleftarrow{\ \ \ } \opn{Ext}^1_{\opn{an}}(\opn{St}_2^{\opn{sm}}(L), \opn{St}^{\opn{la}}(L)),\notag
    \end{align}
    where the right-hand map is the natural one arising from the map $\opn{I}_{\overline{P}_2}^{\opn{sm}}(L) \twoheadrightarrow \opn{St}_2^{\opn{sm}}(L)$. 
    
    In \cite{AutomorphicLinvariants}, the first isomorphism is taken to be $\lambda \mapsto \Lambda'_{\lambda} \defeq [b \mapsto \lambda(t_2t_1^{-1})]$, where $\opn{diag}(t_1, t_2, t_3)$ denotes the projection of $b$ to the torus. We set $\Lambda_{\lambda} \defeq [b \mapsto (\lambda\circ v_1)(t) = \lambda(t_2t_3t_1^{-1})]$ and note 
    \[
    \Lambda_{\lambda}(b) = \Lambda'_{\lambda}(b) + \Xi_\lambda(b)
    \]
    where $\Xi_\lambda(b) = \lambda(t_3)$. Since $\Xi_\lambda \in \opn{Hom}_{\opn{cts}}(\overline{P}_2(\Q_p), L)$, we see that $\Lambda'_{\lambda}$ and $\Lambda_{\lambda}$ represent the same homomorphism in the quotient; therefore, we can (and do) take the first isomorphism to be
    \begin{align*}
    \opn{Hom}_{\opn{cts}}(\Q_p^{\times}, L) &\isorightarrow \opn{Hom}_{\opn{cts}}(\overline{B}(\Q_p), L)/\opn{Hom}_{\opn{cts}}(\overline{P}_2(\Q_p), L)\\
    \lambda &\longmapsto \Lambda_\lambda,
    \end{align*}
 and this does not change the isomorphism in \cite[Theorem 2.15]{AutomorphicLinvariants}.

    Since $\opn{H}^0(G(\Q_p), \opn{St}^{\opn{la}}(L)) = 0$, one sees that the natural map 
    \begin{equation} \label{InjStToIndEqn}
    \gamma \colon \opn{H}^1\left( G(\Qp), \opn{Hom}_L(\opn{St}_2^{\opn{sm}}(L), \opn{St}^{\opn{la}}(L)) \right) \to \opn{H}^1\left( G(\Qp), \opn{Hom}_L(\opn{I}_{\overline{P}_2}^{\opn{sm}}(L), \opn{St}^{\opn{la}}(L)) \right)
    \end{equation}
    is injective. Combining with \eqref{eq:lennart isos}, we get a commutative diagram
    \[
        \begin{tikzcd}
 \opn{Hom}_{\opn{cts}}(\Q_p^{\times}, L)\arrow[r,"\sim"]\arrow[rd,swap,"\lambda \mapsto \cC_{\Lambda_\lambda}"] & \opn{Hom}_{\opn{cts}}(\overline{B}(\Q_p), L)/\opn{Hom}_{\opn{cts}}(\overline{P}_2(\Q_p), L) \ar[d,"\sim"]&\\
    &\opn{Ext}^1_{\opn{an}}(\opn{I}^{\opn{sm}}_{\overline{P}_2}, \opn{St}^{\opn{la}})\arrow[d,swap, "\beta"]\arrow[r,"\alpha", "\sim"'] & \opn{Ext}^1_{\opn{an}}(\opn{St}_2^{\opn{sm}}, \opn{St}^{\opn{la}}) \arrow[d,"\beta'"]\\
    & \h^1(G(\Qp), \opn{Hom}_L(\opn{I}_{\overline{P}_2}^{\opn{sm}}, \opn{St}^{\opn{la}})) & \h^1(G(\Qp), \opn{Hom}_L(\opn{St}^{\opn{sm}}_2, \opn{St}^{\opn{la}}))\arrow[l,swap,hook,"\gamma"],
        \end{tikzcd}
    \]
    with all coefficients in $L$. Here $\alpha$ is the inverse of the final isomorphism in \eqref{eq:lennart isos}, and $\beta,\beta'$ are the compositions of the forgetful maps $\opn{Ext}^1_{\opn{an}} \to \opn{Ext}^1_{L[G(\Qp)]}$ with the identification \eqref{eq:ext1 = h1} (and its analogue for $\opn{I}_{\overline{P}_2}^{\opn{sm}}$). The proposition then follows if we can show
    \[
        \Delta^{\opn{la}}(\lambda) = \beta'\circ\alpha(\cC_{\Lambda_\lambda}).
    \]
    By injectivity of $\gamma$, it suffices to show that 
    \[
        \gamma(\Delta^{\opn{la}}(\lambda)) = \beta(\cC_{\Lambda_\lambda}).
    \]
    As in Remark \ref{ExplicitDeltaCocycleRemark}, we see $\gamma(\Delta^{\opn{la}}(\lambda))$ is represented by the cocycle
    \begin{equation}\label{eq:delta cocycle}
x \mapsto \Big(F \longmapsto \Big[g \mapsto \Big((\lambda\circ v_1\circ s)(gx) - (\lambda\circ v_1 \circ s)(g)\Big)F(g)\Big]\Big),
    \end{equation}
    for $F \in \opn{I}_{\overline{P}_2}^{\opn{sm}}(L)$. To compute $\beta(\cC_{\Lambda_\lambda})$, note that exactly as in Lemma \ref{lem:short exact sequence extension} -- and because of our choice of $\Lambda_\lambda$ -- we have an $L$-linear splitting of the extension $\cC_{\Lambda_\lambda}$ induced by
    \[
    r\colon F \mapsto \left( g \mapsto \begin{pmatrix}(\lambda \circ v_1 \circ s)(g)F(g)\\ F(g)\end{pmatrix} \right) \; \in \; \opn{I}_{\overline{B}}^{\opn{la}}(\sigma_{\Lambda_{\lambda}}), \qquad F \in \opn{I}_{\overline{P}_2}^{\opn{sm}}(L).
    \]
 The associated cocycle representing $\beta(\cC_{\Lambda_\lambda})$ is given by $x \mapsto (F \mapsto [x\cdot r(x^{-1} \cdot F) - r(F)])$, i.e.,
      \[
x \mapsto \left(F \longmapsto \left[g \mapsto \begin{pmatrix}\big((\lambda\circ v_1\circ s)(gx) - (\lambda\circ v_1 \circ s)(g)\big)F(g)\\0\end{pmatrix}\right]\right),
    \]
  which agrees exactly with \eqref{eq:delta cocycle}. Thus $\gamma(\Delta^{\opn{la}}(\lambda)) = \beta(\cC_{\Lambda_\lambda})$, as required. 
\end{proof}

\section{\texorpdfstring{$p$}{p}-arithmetic cohomology}\label{sec:p-arithmetic cohomology}

In this section, we set-up notation and recap $p$-arithmetic cohomology from \cite{AutomorphicLinvariants}. We also describe Eisenstein classes in this setting.

\subsection{Notation}\label{sec:level notation}
Fix henceforth a neat compact open subgroup $K^p \subset G(\mbb{A}_f^p)$. Recall $H = \GL_2\times\GL_1$,  and let $K^p_H = K^p \cap H(\mbb{A}_f^p)$. Possibly shrinking $K^p$ we assume $K^p_H\subset \prod_{\ell \neq p}H(\Z_{\ell})$. 

For $G$, we will always work with Iwahori-level $\opn{Iw}\subset G(\Zp)$ at $p$; but it will be necessary to vary the level at $p$ for $H$. For $U \subset H(\Qp)$ any open compact subgroup, consider the finite set:
\[
A_U \defeq H(\Q)_+ \backslash H(\mbb{A}_f) /K^p_H U,
\]
and let $I_U \subset H(\bA_f)$ be a set of representatives for $A_U$. For any $x \in I_U$, we denote by 
\[
    \Gamma_{U}(x) \defeq H(\Q)_+ \cap x K^p_H U x^{-1}
    \]
the corresponding arithmetic subgroup. As we assumed $K^p$ neat, the group $\Gamma_{U}(x)$ is torsion-free.

\subsection{Representations and $p$-arithmetic cohomology}

We recall notions from \cite[\S3]{AutomorphicLinvariants}.

\begin{definition}
Let $\pi_0(H(\mbb{R}))$ denote the set of connected components of $H(\mbb{R})$.
\begin{itemize}
\item For an $R[H(\Q)]$-module $M$, set
\[
\mathcal{A}_H(M) = \left\{ f \colon \pi_0(H(\mbb{R})) \times H(\mbb{A}_f^p)/K^p_H \to M \right\}
\]
This is a left-$H(\Q)$-module via 
\[
(\gamma\cdot f)([h_\infty], h^p) = \gamma \cdot f([\gamma^{-1}h_\infty], \gamma^{-1}h^p), \qquad [h_\infty] \in \pi_0(H(\R)), \ h^p \in H(\A_f^p)/K_H^p.
\]
\item Let $D_H$ be the Steinberg module of $H$, as described in \cite[Def.\ 3.21]{AutomorphicLinvariants}, and set $\mathcal{A}_{H, c}(M) = \opn{Hom}_{\Z}(D_H, \mathcal{A}_H(M))$. If $f \in \cA_{H,c}(M)$ we write $f(d,[h_\infty],h^p) \defeq f(d)([h_\infty],h^p)$, where $d \in D_H$. We obtain a left $H(\Q)$-action by
\[
(\gamma\cdot f)(d,[h_\infty], h^p) = \gamma \cdot f(\gamma^{-1}\cdot d, [\gamma^{-1}h_\infty], \gamma^{-1}h^p).
\]
\item For $M$ an $R[G(\Q)]$-module, define $G(\Q)$-modules $\cA_G(M)$ and $\cA_{G,c}(M)$ by direct analogy.
\end{itemize}
\end{definition}

A particularly important case is where $M$ is a space of functions on a set $X$; for example, $M = \opn{St}^{\opn{sm}}(R)^*$ or $M = \opn{Ind}_U^{H(\Qp)}R$. We can then naturally represent elements of $\cA_{H,c}(M)$ as functions
\[
D_H \times \pi_0(H(\R)) \times H(\A_f^p)/K^p_H \times X \to R.
\]

\begin{remark}
We note some slight differences to \cite{AutomorphicLinvariants}. Because $H$ and $G$ are not semisimple, we include $\pi_0(H(\R))$ and $\pi_0(G(\R))$ to account for connected components. 

We also dualise: the notation $\cA(K^p_H, M; R)$ \emph{op.cit.} corresponds to what we call $\cA_H(M^{*})$, where $M^*$ is the $R$-linear dual. Note in particular that via the duality between compact induction and regular induction, Gehrmann's space $\cA(K^p_H \times U; R) = \cA(K^p_H, \opn{c-Ind}_U^{H(\Qp)} R; R)$ is, in our notation, denoted $\cA_{H}(\opn{Ind}_U^{H(\Qp)} R)$ (and similarly for $\cA_{H,c}$, $\cA_G$ and $\cA_{G,c}$). 
\end{remark}

\begin{remark}
The cohomology of the above representations (``$p$-arithmetic cohomology'') provides a reinterpretation of classical arithmetic cohomology. To see this, recall $U \subset H(\Q_p)$ is an open compact subgroup, with attached objects $I_U$ and $\Gamma_U(x)$ as in \S\ref{sec:level subgroups}. If $R$ is a ring, then Shapiro's lemma defines a natural isomorphism
\begin{equation}\label{eq:shapiro}
\opn{H}^i\left( H(\Q), \mathcal{A}_{H}(\opn{Ind}_{U}^{H(\Q_p)}R) \right) \isorightarrow \bigoplus_{x \in I_U} \opn{H}^i\left( \Gamma_U(x), R \right).
\end{equation}
We also have a compactly-supported version
\begin{equation}\label{eq:shapiro compact}
\opn{H}^i\left( H(\Q), \mathcal{A}_{H, c}(\opn{Ind}_{U}^{H(\Q_p)}R) \right) \isorightarrow \bigoplus_{x \in I_U} \opn{H}^i\Big( \Gamma_U(x), \opn{Hom}(D_H, R) \Big),
\end{equation}
where $\Gamma_U(x)$ acts on $\opn{Hom}(D_H, R)$ by $(y \cdot f)(d) = f(y^{-1}d)$. Explicitly, \eqref{eq:shapiro compact} takes a cocycle $c$ to the collection of cocycles $(c_x)_{x \in I_U}$ satisfying
\[
c_x[y_1, \dots, y_i](d) = c[y_1, \dots, y_i](d, [H(\mbb{R})_+], x^p, x_p)
\]
where $d \in D_H$ and $y_j \in \Gamma_U(x)$. 
\end{remark}

Gehrmann uses the above, together with the familiar identity between arithmetic and Betti cohomology, to show the following.  Let $Z_G$ be the centre of $G$, let $K_\infty = Z_G(\R)O_3(\R)$, let $K_\infty^H = K_\infty\cap H(\R)$, and let $K_\infty^\circ$ and $(K_\infty^H)^\circ$ be their neutral components. Then let
\[
X_{G,K^p\opn{Iw}} \defeq G(\Q)\backslash G(\A)/K_\infty^\circ K^p\opn{Iw}, \qquad X_{H,K^p_HU} \defeq H(\Q)\backslash H(\A)/(K_\infty^H)^\circ K^p_HU.
\]
\begin{proposition}\label{prop:arithmetic to betti}
\begin{itemize}
\item[(i)] For all $i \geq 0$, we have natural isomorphisms
\begin{align*}
\opn{H}^i\left( G(\Q), \mathcal{A}_{G}(\opn{Ind}_{\opn{Iw}}^{G(\Q_p)}R) \right) &\cong \h^i\left(X_{G, K^p\opn{Iw}}, R\right),\\
\opn{H}^i\left( H(\Q), \mathcal{A}_{H}(\opn{Ind}_{U}^{H(\Q_p)}R) \right) &\cong \h^i\left(X_{H, K^p_HU}, R\right).
\end{align*}

\item[(ii)] For all $i \geq 0$, we have natural isomorphisms
\begin{align*}
\opn{H}^i\left( G(\Q), \mathcal{A}_{G,c}(\opn{Ind}_{\opn{Iw}}^{G(\Q_p)}R) \right) &\cong \hc{i+2}\left(X_{G, K^p\opn{Iw}}, R\right),\\
\opn{H}^i\left( H(\Q), \mathcal{A}_{H,c}(\opn{Ind}_{U}^{H(\Q_p)}R) \right) &\cong \hc{i+1}\left(X_{H, K^p_HU}, R\right).
\end{align*}
\end{itemize}
\end{proposition}

\begin{proof}
Part (i) is proved after \cite[Cor.\ 3.4]{AutomorphicLinvariants}. Part (ii) is explained in \S3.6 \emph{op.\ cit}. We recall that there $l$ is used to denote the number of simple roots, so $l = 2$ for $G$ and $l = 1$ for $H$, explaining the degree shift in (ii). The only difference from \emph{op.\ cit}.\ is that we must include $\pi_0(G(\R))$ to handle connected components, but given this all of the relevant isomorphisms and arguments pass over \emph{mutatis mutandis} to the reductive setting.
\end{proof}

Now, there is a natural pairing
\[
\langle -, - \rangle_{U} \colon \opn{H}^{i-1}\left( H(\Q), \mathcal{A}_{H, c}(\opn{Ind}_{U}^{H(\Q_p)}R) \right) \times \opn{H}^{2-i}\left( H(\Q), \mathcal{A}_{H}(\opn{Ind}_{U}^{H(\Q_p)}R) \right) \to R
\]
corresponding to Poincar\'e duality on the Betti cohomology. It is equivariant for the $H(\mbb{A}_f)$-action (up to necessary modifications to $K^p_H$), and restriction and corestriction are adjoint under the pairing. Note that the pairing is not perfect in general (but we will not need it to be).

\subsection{$p$-arithmetic cohomology for the Steinberg representation}\label{sec:steinberg tree}

We consider another important $p$-arithmetic cohomology group. The following is a variant on \cite[Prop.\ 3.6(b)]{AutomorphicLinvariants}.

Let $\varphi_{\opn{Iw}} : \overline{B}(\Qp)\backslash G(\Qp) \to R$ denote the characteristic function $\opn{ch}(\overline{B}(\Qp)\cdot\opn{Iw})$. When $R = L$, a finite extension of $\Qp$, or $R = \cO_L$, its ring of integers, the image of $\varphi_{\Iw}$ in $\opn{St}^{\opn{sm}}(R)$ generates the Iwahori-invariants $\opn{St}^{\opn{sm}}(R)^{\opn{Iw}}$. We then have a natural $G(\Qp)$-equivariant map
\begin{align*}
\opn{St}^{\opn{sm}}(R)^* &\longrightarrow \opn{Ind}_{\opn{Iw}}^{G(\Qp)}R,\\
\nu &\longmapsto \big(g \mapsto \nu(g\cdot \varphi_{\opn{Iw}})\big).
\end{align*}
This induces a map
\[
\rho_{\mathrm{Iw}}^i \colon \h^i\Big(G(\Q), \cA_{G,c}(\opn{St}^{\opn{sm}}(R)^*)\Big) \to \h^i\Big(G(\Q), \cA_{G,c}(\opn{Ind}^{G(\Qp)}_{\opn{Iw}} R)\Big) \cong \hc{i+2}\Big(X_{G, K^p\opn{Iw}}, R\Big).
\]

Consider the Hecke operator $U_{p,1} = [\opn{Iw} t \opn{Iw}]$ acting on $\opn{St}^{\opn{sm}}(R)^{\opn{Iw}}$, recalling $t = \smallthreemat{p}{}{}{}{1}{}{}{}{1}$.

\begin{lemma}\label{lem:image fixed by Up1}
The operator $U_{p,1}$ acts trivially on the image of $\rho_{\mathrm{Iw}}^i$. 
\end{lemma}
\begin{proof}
It is standard that $\varphi_{\opn{Iw}}$ is fixed by $U_{p,1}$. Write $U_{p,1} = \sum_j \gamma_j$; then $U_{p,1} \cdot \rho_{\mathrm{Iw}}^i$ is the map induced by
\begin{align*}
\mu \mapsto \Big(g \mapsto \sum_j \mu(g\gamma_j\cdot \varphi_{\opn{Iw}})\Big) &= \Big(g \mapsto \mu\Big(g \cdot \Big[\sum_j\gamma_j\cdot \varphi_{\opn{Iw}}\Big]\Big)\Big)\\
&= \Big(g \mapsto \mu(g\cdot U_{p,1}\varphi_{\opn{Iw}})\Big) = \Big(g \mapsto \mu(g\cdot\varphi_{\Iw})\Big), 
\end{align*}
which also induces $\rho_{\Iw}^i$.
\end{proof}

\begin{remark}
In \cite[Prop.\ 3.6]{AutomorphicLinvariants}, Gehrmann proves that $\rho_{\mathrm{Iw}}^i$ is an isomorphism after passing to $\pi$-isotypic components, where $\pi$ is any automorphic representation that is Steinberg at $p$; see Proposition \ref{prop:lift to tree} of the present paper.

The group $\h^0(G(\Q), \cA_{G,c}(\opn{Ind}_{\opn{Iw}}^{G(\Qp)}R))$ can be considered as a space of `modular symbols' for $\GL_3$. The space $\h^0(G(\Q), \cA_{G,c}(\opn{St}(R)^*))$ is then an analogue of modular symbols over the Bruhat--Tits tree, and Gehrmann's result says any Steinberg-at-$p$ eigensymbol can be lifted to the tree. For explicit analogues in the case of Bianchi modular forms, see \cite[\S4]{BW17} or \cite[\S3.3]{VW19}.
\end{remark}

\subsection{Level subgroups at $p$}\label{sec:level subgroups}
We now introduce some level subgroups at $p$. Let $n \geq 1$, and define 
\[
\xymatrix@C=0.1mm@R=3mm{
\Un \sar{d}{\cap}&\defeq& \left\{h_p \equiv \smallmatrd{1}{}{}{1} \times 1 \ \text{modulo }p^n\right\} & &\\
\Uncirc \sar{d}{\cap}&\defeq& \left\{h_p \equiv \smallmatrd{x}{}{}{1} \times x \ \text{modulo }p^n\right\}&&\\
\Unloz \sar{d}{\cap}&\defeq& \left\{h_p \equiv \smallmatrd{*}{}{}{1} \times * \ \text{modulo }p^n\right\}&&\\
\Unone &\defeq & \left\{h_p \equiv \smallmatrd{*}{*}{}{1} \times * \ \text{modulo }p^n\right\} &\subset & H(\Zp).
}
\]

Recalling $A_U, I_U$ and $\Gamma_U(x)$ from \S\ref{sec:level notation}, we adopt the notation 
\[
    A_1(p^n) \defeq A_{\Unone}, \qquad I_1(p^n) \defeq I_{\Unone}, \qquad \Gamma_1(p^n)_x = \Gamma_{\Unone}(x),
\]
and similarly for $\Unloz, \Uncirc$ and $\Un$. Note that $A_1(p^n)\cong \cA_1 \times\cA_2$, where $\cA_i$ are finite groups which are quotients of groups of the form $\left( \Z/D_i \Z \right)^{\times}$ with $p \nmid D_i$. 

We can and do fix a choice $I_1(p^n) = I(p^n)^\lozenge \subset H(\widehat{\Z}^p) = \prod_{\ell \neq p} H(\Z_{\ell})$, considered in $H(\A_f)$ by placing 1 at $p$.

\begin{lemma}\label{lem:gamma in Ghat} For $n\geq 1$  we have $H(\Q)_+\cap H(\widehat{\Z}^p)\Unloz \subset \Un$. In particular, if $x\in I_1(p^n)$, we have $\Gamma(p^n)^{\lozenge}_x \subset \Un$. 
\end{lemma}
\begin{proof}
If $y \in H(\Q)_+ \cap H(\widehat{\Z}^p)\Unloz$, write $y = (y_1, y_2)$ with $y_1 \in \GL_2(\Q)_+$ and $y_2 \in \Q_{>0}$. Then $\opn{det}y_1$ and $y_2$ are positive; and they are $\ell$-adic units for all primes $\ell$. Hence $\det y_1 = y_2 = 1$. The $p$-component of $y_1$ lies in $\Unloz$ by definition, hence $y_1 \in \opn{SL}_2(\Z_{(p)})$ has the form $y_1 \equiv \opn{diag}(z, 1)$ $\pmod{p^n}$. Since $\det(y_1) = 1$, one must have $z=1$ (modulo $p^n$), hence $y_1$ is congruent to the identity modulo $p^n$, and $y \in \Un$.

The final statement follows as $\Gamma(p^n)^\lozenge_x = H(\Q)_+\cap xK^p_H\Unloz x^{-1}$, and $xK^p_Hx^{-1} \subset H(\widehat{\Z}^p)$.
\end{proof}

\subsection{Eisenstein classes} \label{EisensteinClassesSSec}

Let $c$ be a positive integer prime to $6pD_2$. Let $\eta_2$ be an even Dirichlet character of conductor $D | D_2$, and we assume that $\eta_2$ factors through $\left( \Z/D_2 \Z \right) \twoheadrightarrow A_2$. We assume that the local component of $K^p_H$ at any prime $\ell | c$ is $H(\Z_{\ell})$. Let ${_c\mathcal{S}_0(\mbb{A}_f^2, \Z)}$ be the Schwartz functions which vanish at $(0, 0)$ and equal $\opn{ch}(\Z_\ell^2)$ at any prime $\ell | c$ (see \cite[\S 5.3.3]{LW21}). This carries a natural $\GL_2(\A_f)$-action, and thus an $H(\A_f)$-action via the natural projection $H \twoheadrightarrow \GL_2$.

For $\Phi \in {}_c\cS_0(\bA_f^2,\Z)$ fixed by $K^p_HU$ with $U \subset H(\Qp)$ an open compact subgroup, let ${}_c\opn{Eis}_{\Phi}^0$ be the associated Betti--Eisenstein class from \cite[Thm.\ 5.7]{LW21}; it is the Betti realisation of Beilinson's motivic Eisenstein class from \cite{Bei86}. Since we work in trivial weight, it admits a simpler description in terms of Siegel units (see, e.g., \cite[\S3.1]{LW18}). Via the $\GL_2$-analogue of Proposition \ref{prop:arithmetic to betti}, we view ${}_c\opn{Eis}_\Phi^0 \in \h^1(\GL_2(\Q),\cA_{\GL_2}(\opn{Ind}_{U_{\opn{GL}_2}}^{\GL_2(\Qp)}\Z))$, where $U_{\opn{GL}_2}$ denotes the image of $U$ under the natural map $H(\mbb{Q}_p) \twoheadrightarrow \opn{GL}_2(\mbb{Q}_p)$.

Let $\mbb{Z}_{\eta_2}$ denote the ring of integers of a number field in which $\eta_2$ is valued. Then, we have a $\opn{GL}_2(\Q_p)$-equivariant association 
\[
{_c\mathcal{S}_0(\mbb{A}_f^2, \Z)}^{K^p_HU} \to \opn{H}^1\left( H(\Q), \mathcal{A}_H(\opn{Ind}_{U}^{H(\Q_p)}\Z_{\eta_2}) \right)
\]
given by $\Phi \mapsto \mathcal{E}_{\Phi, U} \defeq {_c\opn{Eis}_{\Phi}^0} \boxtimes [\widehat{\eta}_2]$, where $\boxtimes$ is the exterior cup product (we are viewing $[\widehat{\eta}_2]$ as a cohomology class for the $\opn{GL}_1$ Shimura set). This map is $\GL_2(\Qp)$-equivariant in that for any $g \in \GL_2(\Qp$), the following diagram commutes:
\begin{equation}\label{eq:Eis equiv}
\xymatrix{
{_c\mathcal{S}_0(\mbb{A}_f^2, \Z)}^{K^p_HU} \ar[r]\ar[d]^{g \cdot -} & \opn{H}^1\left( H(\Q), \mathcal{A}_H(\opn{Ind}_{U}^{H(\Q_p)}\Z_{\eta_2}) \right)\ar[d]^{g\cdot -}\\
{_c\mathcal{S}_0(\mbb{A}_f^2, \Z)}^{K^p_H\cdot gUg^{-1}} \ar[r] & \opn{H}^1\left( H(\Q), \mathcal{A}_H(\opn{Ind}_{gUg^{-1}}^{H(\Q_p)}\Z_{\eta_2}) \right).
}
\end{equation}
Here $g$ acts on the right-hand spaces through its inclusion into $H(\Q_p)$. In particular, if $U' \subset U$ and we can choose a set of representatives for $U/U'$ in $\opn{GL}_2(\mbb{Q}_p) \times \{ 1 \} \subset H(\mbb{Q}_p)$, then we have a trace map $\opn{tr} = \opn{tr}^{U}_{U'}$ and a compatibility 
\begin{equation}\label{eq:Eis trace}
    \cE_{\mathrm{tr}(\Phi),U} = \mathrm{tr}\left(\cE_{\Phi,U'}\right) \in \h^1\left(H(\Q), \mathcal{A}_H(\opn{Ind}_{U}^{H(\Q_p)}\Z_{\eta_2}) \right) \qquad \text{for all } \Phi \in {_c\cS_0(\bA_f^2,\Z)}^{K^p_HU'}.
\end{equation}

\begin{notation}\label{not:schwartz}
\begin{itemize}
\item[(i)] At $p$, we will have a particular interest in the Schwartz function $\Phi_p^n \defeq \opn{ch}(p^n \Z_p, 1+p^n \Z_p) \in \cS(\Qp^2,\Z)$, which we can naturally view with coefficients in $R$. 
\item[(ii)] We shall later need to make particular choices of prime-to-$p$ Schwartz functions. When the choice of $\Phi^{(p)}$ is fixed and implicit (for example, throughout \S\ref{sec:abstract exceptional zero formula}), as far as possible we will drop it from notation. To this end, we will often write $\Phi_n$ for the function $\Phi^{(p)}\Phi_p^n$. 
\end{itemize}
\end{notation}


\section{An abstract exceptional zero formula}\label{sec:abstract exceptional zero formula}

Let $\mu \in \h^0(G(\Q),\cA_{G,c}(\opn{St}^{\opn{sm}}(\cO_L)^*))$. In this section, we attach a $p$-adic measure $\sL(\mu) \in \cO_L[\![\Zp^\times]\!]$ to $\mu$, and then prove an abstract exceptional zero formula (with an abstract $\mathcal{L}$-invariant) for $\sL(\mu)$.

Throughout, $p$ is any prime, $L/\Q_p$ is finite with ring of integers $\cO_L$ and uniformiser $\varpi$, such that $\eta_2$ is valued in $\mathcal{O}_L^{\times}$. We continue to impose the assumptions in \S \ref{EisensteinClassesSSec}. To ease notation, we will drop the superscript ``$\opn{sm}$'' when discussing smooth representations. Unless otherwise specified, $R$ will be a general $\Z$-algebra.

\begin{convention*}
We will fix tame data $\Phi^{(p)}$ and $K^p$, as defined in the previous section, and in this section drop them from all notation, adopting Notation \ref{not:schwartz}. We emphasise that all of our constructions will, however, depend on the choice of $\Phi^{(p)}$.
\end{convention*}

\subsection{Evaluation maps for $p$-arithmetic cohomology}

We now define evaluation maps on the $p$-arithmetic cohomology that will allow us to construct the $p$-adic measure $\sL_\mu$. First, we construct a $p$-arithmetic restriction map from $G$ to $H$, for which we need the following data:
\begin{notation}\label{not:evaluation data}
\begin{itemize}
\item[(i)] \cite[Prop.\ 1.1]{Reeder} gives a $H(\mbb{Q})$-equivariant injection $\iota \colon D_H \hookrightarrow D_G$. 
\item[(ii)] Recall we fixed, in \S\ref{sec:level subgroups}, a set of representatives $I \defeq I_1(p^n) = I(p^n)^{\lozenge} \subset \prod_{\ell \neq p}H(\Z_{\ell})$. 

\item[(iii)] We have $\det(\Unloz) = \Zp^\times \times \Zp^\times$; thus by strong approximation, for any $h_\infty \in H(\R)$, $h^p \in H(\A_f^p)$ and $h_p \in H(\Qp)$ there exist $h \in H(\Q)$, $x \in I$, $k^p\in K^p_H$ and $ k_p \in \Unloz$ such that
\[
[h_\infty] = h[H(\R)_+], \ h^p = hxk^p, \ h_p = hk_p.
\]
Here $x$ is unique given $I$; but then $h, k^p$ and $k_p$ are only well-defined up to $\Gamma(p^n)^\lozenge_x$, as for any $\gamma \in \Gamma(p^n)^\lozenge_x$, we can write $[h_\infty] = h[H(\R)_+] = h \gamma [H(\R)_+]$, $h^p = hxk^p = h\gamma\cdot x \cdot x^{-1}\gamma^{-1}xk^p$ and $h_p = hk_p = h\gamma\cdot \gamma^{-1}k_p$.
\end{itemize}
\end{notation}

\begin{definition}\label{def:Psi_n}
    Let $n \geq 1$. Define a map
    \[
        \Psi_{n} : \opn{H}^0\left( G(\Q), \mathcal{A}_{G, c}(\opn{St}(R)^*) \right)  \times \opn{St}(R)^{\Un} \longrightarrow \opn{H}^0\left( H(\Q), \mathcal{A}_{H, c}\big(\opn{Ind}^{H(\Q_p)}_{\Unloz}R\big) \right)
    \]
    by
      \[
    \Psi_{n}(\mu,\varphi)(d, [h_\infty], h^p, h_p) \defeq \mu(h^{-1} \iota(d), [G(\mbb{R})_+], x, \varphi),
    \]
    for $h$ and $x$ as in (iii) above.
\end{definition}

We must show this is well-defined. The only ambiguity in the definition comes through $h$, which, by (iii) above, we can replace with $h\gamma$ for $\gamma \in \Gamma(p^n)^\lozenge_x$.

\begin{lemma}\label{lem:ind of reps}
The function $\Psi_n(\mu,\varphi) \in \cA_{H,c}\left(\opn{Ind}_{\Unloz}^{H(\Qp)}R\right)$ is well-defined and independent of the choice of representatives $I$.
\end{lemma}
\begin{proof}
First we consider well-definedness. We must check the definition is unchanged after replacing $h$ with $h\gamma$ for $\gamma \in \Gamma(p^n)_x^\lozenge \subset H(\Q)_+$. By $G(\Q)$-invariance of $\mu$ we have
\begin{equation}\label{eq:Xi well defined}
\mu((h\gamma)^{-1}\iota(d), [G(\R)_+], x, \varphi) = \mu(h^{-1} \iota(d), \gamma[G(\R)_+],\gamma x,\gamma\cdot\varphi).
\end{equation}
We have $\gamma \in U(p^n)$ by Lemma \ref{lem:gamma in Ghat}, so $\gamma\varphi= \varphi$; and as $\gamma \in H(\R)_+$ we have $\gamma[G(\R)_+] = [G(\R)_+]$. Finally as $\gamma \in xK^p_Hx^{-1}$ we see $\gamma x = x\cdot x^{-1}\gamma x \in xK^p_H$. As $F$ is right-invariant under $K^p_H$, we conclude the right-hand side of \eqref{eq:Xi well defined} is equal to $\mu(h^{-1}\iota(d),[G(\R)_+],x,\varphi)$. Thus $\Psi_n(\mu,\varphi)$ is a well-defined function $D_H \times \pi_0(H(\R)) \times H(\A_f^p) \times H(\Qp) \to R$. It is right-invariant under $K^p_H$ in the $h^p$-variable, since $\mu$ is invariant under $K^p$; and it is right-invariant under $\Unloz$ by construction. Thus $\Psi_n(\mu,\varphi) \in \cA_{H,c}(\opn{Ind}_{\Unloz}^{H(\Qp)} R)$ is well-defined.

We now show it is independent of the choice of $I$. Let $I' = \{x_i'\} \subset \prod_{\ell \neq p} H(\mbb{Z}_{\ell})$ be another set of representatives, defining a function $\Psi_n'(\mu,\varphi) \in \cA_{H,c}(\opn{Ind}_{\Unloz}^{H(\Qp)}R)$. There exist $h_i \in H(\Q)_+$, $k_i \in K^p_H$ and $u_i\in \Unloz$ such that 
\[
x_i' = h_ix_ik_iu_i \qquad \forall i.
\]
By assumption $\mu$ is invariant under $h_i$. As both $x_i'$ and $x_i$ are trivial at $p$, we see $h_{i,p} = u_i^{-1} \in \Unloz$. Moreover $h_i^p \in x_i'K_H^px_i^{-1} \subset \prod_{\ell \neq p}H(\Z_\ell)$, so exactly as in Lemma \ref{lem:gamma in Ghat},  we see $h_{i} \in \Un$. In particular $h_i^{-1}$ fixes $\varphi$. Additionally $h_i^{-1}[G(\R)_+] = [G(\R)_+]$. We thus compute that 
\begin{align*}
\Psi_n'(\mu,\varphi)(d,[h_\infty],h^p,h_p) &= \Psi_n'(\mu,\varphi)(d,h[H(\R)_+],hx_i'k^p, hk_p) = \mu(h^{-1}\iota(d),[G(\R)_+], x_i', \varphi)\\
&= \mu(h^{-1}\iota(d), [G(\R)_+], h_ix_ik_i, \varphi)\\
&= \mu(h_i^{-1}h^{-1}\iota(d), h_i^{-1}[G(\R)_+], x_i, h_i^{-1}\varphi)\\
&= \mu((hh_i)^{-1}\iota(d), [G(\R)_+], x_i, \varphi) = \Psi_n(\mu,\varphi)(d,[h_\infty], h^p,h_p).
\end{align*}
The fourth equality is $G(\Q)$-invariance of $F$, and the last uses that $h^p = hx_i'k^p = (hh_i)x_i(k_ik^p)$.
\end{proof}

\begin{lemma}\label{lem:Psi invariant}
We have $\Psi_n(\mu,\varphi) \in \h^0\Big(H(\Q), \cA_{H,c}(\opn{Ind}_{\Unloz}^{H(\Qp)}R)\Big)$.
\end{lemma}
\begin{proof}
We must show $\Psi_n(\mu,\varphi)$ is fixed by $H(\Q)$. For $\gamma \in H(\Q)$ we have
\begin{align*}
(\gamma^{-1} \cdot \Psi_n(\mu,\varphi))(d, [h_\infty], h^p,h_p) &= \Psi_n(\mu,\varphi)(\gamma d, \gamma[h_\infty],\gamma h^p,\gamma h_p)\\
&= \mu((\gamma h)^{-1}\gamma d, [G(\R)_+], x, \varphi) = \Psi_n(\mu,\varphi)(d,[h_\infty],h^p,h_p),
\end{align*}
where we write $h^p = hxk^p$ as above, and note $(\gamma h)^{-1}\gamma = h^{-1}$.
\end{proof}

\begin{remark}
If $\varphi$ is fixed by the larger group $\Unloz$, we have the simpler description
\[
\Psi_{n}(\mu,\varphi)(d,[h_\infty],h^p,h_p) = \mu(\iota(d), [h_\infty], h^p, h_p\cdot \varphi).
\]
\end{remark}

The following is our $p$-arithmetic evaluation map. Recall $\Phi_n$ and $\cE_{\Phi,U}$ from \S\ref{EisensteinClassesSSec}.

\begin{definition}\label{def:Psi tilde}
Define an $R$-bilinear map
    \begin{align*}
    \widetilde{\Psi}_n \colon \opn{H}^0\left( G(\Q), \mathcal{A}_{G, c}(\opn{St}(R)^*) \right)\times \opn{St}(R)^{U(p^n)} &\longrightarrow R\\
    (\mu,\varphi)& \longmapsto \left\langle \Psi_{n}(\mu,\varphi), \mathcal{E}_{\Phi_n, \Unloz} \right\rangle_{\Unloz}.
    \end{align*}
\end{definition}

\subsection{$p$-adic measures attached to $p$-arithmetic cohomology classes}\label{sec:measure}

The following definition and lemma, describing certain elements of and manipulations in $\opn{St}(R)^*$, will repeatedly be useful.

\begin{definition}\label{def:psi and varphi}
Let $\cA,\cB \subset \Zp$ be compact open subsets with $\cA$ stable under multiplication by $-1$. Let $(\bullet,R)$ be as in Notation \ref{BulletRLambdaNotation}, and let $f : \cB \to R$ be a $\bullet$-function. 

Let $\psi_{\cA,\cB,f} : \overline{B}(\Qp) \backslash G(\Qp) \to R$ be the function
    \[
        \psi_{\cA,\cB,f}(g) = \left\{ \begin{array}{cc} f(y) & \text{ if } g \in \overline{B}(\Q_p) \left( \begin{smallmatrix} 1 & x & y \\ & 1 & pz \\ & & 1 \end{smallmatrix} \right) \text{ with } x \in \cA,\  y \in \cB, \ z \in \Z_p \\ 0 & \text{ otherwise, }  \end{array} \right.
    \]
    and let $\varphi_{\cA,\cB,f} : \overline{B}(\Qp)\backslash G(\Qp) \to R$ be the function 
    \[
    \varphi_{\cA,\cB,f}(g) = \left\{ \begin{array}{cc} f(y) & \text{ if } g \in \overline{B}(\Q_p) \left( \begin{smallmatrix} 1 & y & x \\ & 1 & z \\ & & 1 \end{smallmatrix} \right) \text{ with } x \in \cA, \ y \in \cB, \ z \in \Z_p \\ 0 & \text{ otherwise. }  \end{array} \right.
    \]
\end{definition}

Define matrices
\begin{equation}\label{eq:u}
u_0 = \smallthreemat{1}{}{}{}{}{-1}{}{1}{}, \qquad u = \smallthreemat{1}{}{1}{}{1}{}{}{}{1}u_0.
\end{equation}

\begin{lemma}\label{lem:psi and varphi}
In the set-up of Definition \ref{def:psi and varphi}, we have $u_0\cdot[\varphi_{\cA,\cB,f}] = -[\psi_{\cA,\cB,f}]$ in $\opn{St}^\bullet(R)$.
\end{lemma}

\begin{proof}
Let $\xi_{\cA,\cB,f} : \overline{P}_1(\Qp)\backslash G(\Qp) \to R$ be defined by    
\[
    \xi_{\cA,\cB,f}(g) = \left\{ \begin{array}{cc} f(y) & \text{ if } g \in \overline{P}_1(\Q_p) \left( \begin{smallmatrix} 1 & x & y \\ & 1 & 0 \\ & & 1 \end{smallmatrix} \right) \text{ with } x \in \cA,\  y \in \cB \\ 0 & \text{ otherwise. }  \end{array} \right.
    \]
    As this is left-invariant under $\overline{P}_1(\Qp)$, its image in $\opn{St}^\bullet(R)$ is 0. The result thus follows if we show
    \begin{equation}\label{eq:comparison in St}
        u_0 \cdot \varphi_{\cA,\cB,f} = \xi_{\cA,\cB,f} - \psi_{\cA,\cB,f}.
    \end{equation}
    First note that:

 \begin{claim} We have
        $\overline{P}_1(\Qp) \left(\begin{smallmatrix} 1 & \cA & \cB\\ & 1 & 0 \\ & & 1 \end{smallmatrix}\right) = \overline{B}(\Qp)u_0^{-1} \left(\begin{smallmatrix} 1 & \cA & \cB \\ & 1 &  \\ & \Zp & 1 \end{smallmatrix}\right) \sqcup \overline{B}(\Qp) \left(\begin{smallmatrix} 1 & \cA & \cB \\ & 1 & p\Zp \\ &  & 1 \end{smallmatrix}\right).$
    \end{claim}
 \emph{Proof of claim:} The Bruhat decomposition tells us $\GL_2(\Qp) = \overline{B}_2(\Qp)\smallmatrd{}{1}{-1}{}\opn{Iw}_2^- \sqcup \overline{B}_2(\Qp)\opn{Iw}_2^-$, where $\opn{Iw}_2^- = \{g \in \GL_2(\Zp): g \newmod{p} \in \overline{B}_2(\mbb{F}_p)\}$. Via Iwahori decomposition (of the form $NT\overline{N}$ in the first factor, and $\overline{N}TN$ in the second) we see that
    \[
        \GL_2(\Qp) = \overline{B}_2(\Qp)\smallmatrd{}{1}{-1}{}\smallmatrd{1}{}{\Zp}{1} \sqcup \overline{B}_2(\Qp)\smallmatrd{1}{p\Zp}{}{1}.
    \]
    The claim follows by considering this computation in the bottom-right $2\times 2$ square in $\GL_3$.

    \medskip

    Now, since $u_0 \smallthreemat{1}{y}{-x}{}{1}{-z}{}{}{1}u_0^{-1} = \smallthreemat{1}{x}{y}{}{1}{}{}{z}{1}$, and $\cA$ is stable under multiplication by $-1$, we have
    \[
    \Big(u_0\cdot\varphi_{\cA,\cB,f}\Big)(g) = \left\{ \begin{array}{cc} f(y) & \text{ if } g \in \overline{B}(\Q_p) u_0^{-1}\left( \begin{smallmatrix} 1 & x & y \\ & 1 &  \\ & z& 1 \end{smallmatrix} \right) \text{ with } x \in \cA, \ y \in \cB, \ z \in \Z_p \\ 0 & \text{ otherwise, }  \end{array} \right.
    \]
    From the definitions and claim, we see  $\xi_{\cA,\cB,f} = u_0 \cdot \varphi_{\cA,\cB,f} + \psi_{\cA,\cB,f}$, yielding \eqref{eq:comparison in St} and the lemma.
\end{proof}

Until the end of \S\ref{sec:measure}, we specialise to the setting where $R = \cO_L$. We introduce the following notation: for any map $f \colon \left(\mbb{Z}/p^n\mbb{Z}\right)^{\times} \to \mathcal{O}_L$, let $f^{\iota}$ denote the function given by $f^{\iota}(x) = f(-x)$ for any $x \in \left(\mbb{Z}/p^n\mbb{Z}\right)^{\times}$.

\begin{definition}
If $f \colon (\Z/p^n\Z)^\times \to \mathcal{O}_L$ is any function, let
\[
\psi_{f} \defeq \left[\psi_{p^n\Zp,\Zp^\times,f}\right]  \in \opn{St}(\cO_L).
\]
be as in Definition \ref{def:psi and varphi}, where we lift $f$ to a smooth function on $\Zp^\times$ under the projection $\Zp^\times \twoheadrightarrow (\Z/p^n\Z)^\times$. Then for $\mu \in \opn{H}^0\left( G(\Q), \mathcal{A}_{G, c}(\opn{St}(\cO_L)^*) \right)$, let\footnote{The involution $(-)^{\iota}$ here may seem unnatural, but it is an artifact of our conventions; ultimately, putting it here multiplies the measure $\sL(\mu)$ in Proposition \ref{prop:measure} by the Dirac measure $[-1]$. This gives cleaner comparison to \cite{LW21}.}
\begin{align*}
\sL_{n}(\mu, -) : \opn{Maps}((\Z/p^n\Z)^\times,\cO_L) &\longrightarrow \cO_L,\\
f & \longmapsto \widetilde{\Psi}_n(\mu,-\psi_{f^{\iota}}).
\end{align*}
\end{definition}

It is standard (see e.g.\ \cite[Prop.\ 3.15]{ENT-notes}) that $\sL_{n}(\mu)$ can be interpreted as an element of $\cO_L[(\Z/p^n\Z)^\times]$. We have natural projection maps
\[
\opn{Norm}^{n+1}_n : \cO_L[(\Z/p^{n+1}\Z)^\times] \longrightarrow \cO_L[(\Z/p^n\Z)^\times].
\]
The following reinterpretation of \cite{LW21} relates these constructions to $p$-adic measures.

\begin{proposition}\label{prop:measure}
We have $\opn{Norm}^{n+1}_n(\sL_{n+1}(\mu)) = \sL_{n}(\mu)$. Thus there is an $\cO_L$-linear map
\begin{align*}
\opn{H}^0\left( G(\Q), \mathcal{A}_{G, c}(\opn{St}(\cO_L)^*) \right) & \longrightarrow \cO_L[\![\Zp^\times]\!] \cong \varprojlim_n \cO_L[(\Z/p^n\Z)^\times] \\
\mu &\longmapsto \sL(\mu) \defeq \big(\sL_{n}(\mu)\big)_{n \geq 1}.
\end{align*}
\end{proposition}

\begin{proof}
   Recall the map $\rho_{\mathrm{Iw}}^0$ and operator $U_{p,1}$ from \S\ref{sec:steinberg tree}. For ease of notation, let $\phi \defeq \rho_{\opn{Iw}}^0(\mu) \in \opn{H}^0(G(\mbb{Q}), \mathcal{A}_{G, c}(\opn{Ind}^{G(\mbb{Q}_p)}_{\opn{Iw}}\mathcal{O}_L))$. By Lemma \ref{lem:image fixed by Up1}, we know $\phi$ is fixed by $U_{p,1}$. This can be considered as a class in $\hc{2}(X_{G,K^p\Iw}, \cO_L)$, yielding an input to \cite[\S6.2]{LW21}; we will directly compare $\sL_{n}(\mu)$ with the elements ${}_c\Xi_n^{[0,0]}(\phi, \Phi^{(p)})$ constructed there, and use Corollary 6.12 \emph{op.cit.}\ to conclude.

   \medskip

    Let $f : (\Z/p^n\Z)^\times \to \cO_L$ be any function. Recall 
    \[
        \varphi_{\opn{Iw}} = \opn{ch}\big(\overline{B}(\Qp)\cdot \opn{Iw}) = \varphi_{\Zp,\Zp,\mathbbm{1}},
    \]
    the latter in the notation of Definition \ref{def:psi and varphi}. We first claim that as elements of $\opn{St}(\cO_L)$, we have
    \begin{equation}\label{eq:psi f sum}
    -\psi_{f^{\iota}} = \sum_{b \in \left(\mbb{Z}/p^n\mbb{Z}\right)^{\times}} f(-b)\delta_b u t^n \cdot [\varphi_{\opn{Iw}}],
    \end{equation}
    where $\delta_b = \opn{diag}(-b, 1, 1) \in H(\mbb{Z}_p)$ for $b \in \mbb{Z}_p^{\times}$. Indeed,
    \[
        t^n\cdot \varphi_{\opn{Iw}} = t^n\cdot\varphi_{\Zp,\Zp,\mathbbm{1}} = \varphi_{p^n\Zp,p^n\Zp,\mathbbm{1}},
    \]
    and then Lemma \ref{lem:psi and varphi} implies 
    \[
        u_0 t^n\cdot[\varphi_{\opn{Iw}}] = -[\psi_{p^n\Zp,p^n\Zp,\mathbbm{1}}] = -\left[\opn{ch}\left(\overline{B}(\Qp)\smallthreemat{1}{p^n\Zp}{p^n\Zp}{}{1}{p\Zp}{}{}{1}\right)\right] \in \opn{St}(\cO_L).
    \]
    One then checks that
    \begin{equation} \label{Eqn:TwistedElementequalsPsifiota}
    f(-b) \delta_b u t^n \cdot [\varphi_{\opn{Iw}}] = - f(-b) \left[\opn{ch}\left( \overline{B}(\Q_p) \left( \begin{smallmatrix} 1 & p^n \mbb{Z}_p & b + p^n \mbb{Z}_p \\ & 1 & p\mbb{Z}_p \\ & & 1 \end{smallmatrix} \right) \right)\right] = -[\psi_{p^n\Zp, b+p^n\Zp,f^{\iota}}]
    \end{equation}
    as elements of $\opn{St}(\mathcal{O}_L)$. Equation \eqref{eq:psi f sum} follows after taking the sum. 
    
    Recall $\Uncirc$ from \S\ref{sec:level subgroups}, and $\phi = \rho^0_{\Iw}(\mu)$; and consider the class
    \[
    \iota^*_{n}(\phi) \in \opn{H}^0\Big(H(\mbb{Q}), \mathcal{A}_{H, c}(\opn{Ind}^{H(\mbb{Q}_p)}_{\Uncirc} \mathcal{O}_L)\Big)
    \]
    defined by 
    \[
    (d, [h_{\infty}], h^p, h_p) \mapsto \mu(\iota(d), [h_{\infty}], h^p, h_p u t^n \cdot \varphi_{\opn{Iw}}),
    \]
    where as usual we view $H \subset G$ block diagonally. In classical language, $\iota^*_n(\phi)$ is the Iwahori level class for $G$, twisted by the action of $ut^n$ and pulled back to $H$. Finally, let 
    \[
    \nu_1 : H \to \GL_1, \qquad (h_1,h_2) \mapsto \tfrac{\det h_1}{h_2}
    \]
    be the map from \cite[Def.\ 4.2]{LW21}, and let $[f] \in \opn{H}^0(H(\mbb{Q}), \mathcal{A}_H(\opn{Ind}_{U(p^n)^{\circ}}^{H(\mbb{Q}_p)}\mathcal{O}_L))$ be the class given by
    \[
    ([h_{\infty}], h^p, h_p) \mapsto f(\nu_1(h_p) |\!| \nu_1(h^p h_p) |\!|)
    \]
    where $|\!|-|\!|$ denotes the finite part of the adelic norm. 

    An explicit computation (using the claim above) shows that
    \begin{equation} \label{Eqn:TracePullbackPsin}
    \opn{tr}_{U(p^n)^{\circ}}^{U(p^n)^{\lozenge}}( \iota_n^*(\phi) \cup [f]) = \Psi_{n}(\mu, -\psi_{f^{\iota}})
    \end{equation}
    since $\{ \delta_b \}$ are a set of representatives of $U(p^n)^{\lozenge}/U(p^n)^{\circ}$. Indeed, for $x \in I(p^n)^{\lozenge}$, $b \in \left( \mbb{Z}/p^n\mbb{Z} \right)^{\times}$, $d \in D_H$, and $h_{\infty} \in H(\mbb{R})$, we have $\nu_1(\delta_b)|\!|\nu_1(x\delta_b)|\!| = -b$, and hence
    \begin{align*}
    \left( \iota_n^*(\phi) \cup [f] \right)(d, [h_{\infty}], x, \delta_b) &= f(-b)\mu(\iota(d), [h_{\infty}], x, \delta_b ut^n \cdot \varphi_{\opn{Iw}}) \\
    &= f(-b)\mu(\iota(d), h \cdot [G(\mbb{R})_+], h x, h \cdot \delta_b ut^n \cdot \varphi_{\opn{Iw}}) \\
    &= \mu(h^{-1} \cdot \iota(d), [G(\mbb{R})_+], x, f(-b)\delta_b ut^n \cdot \varphi_{\opn{Iw}}) \\ &= \Psi_n(\mu, -[\psi_{p^n\Zp, b+p^n\Zp,f^{\iota}}])(d, [h_{\infty}], x, 1) .
    \end{align*}
    Here $h \in H(\mbb{Q})$ is any element such that $h[H(\mbb{R})_+] = [h_{\infty}]$ and $x^{-1} h x \in K^p_H U(p^n)^{\circ}$. The second equality follows the fact that $\mu$ is invariant under right-translation by $K^p_H$ and that $U(p^n)^{\circ}$ fixes $\delta_b ut^n \cdot \varphi_{\opn{Iw}}$, and the third equality follows from the $G(\mbb{Q})$-equivariance of $\mu$. The final equality follows from (\ref{Eqn:TwistedElementequalsPsifiota}). We see that 
    \begin{align*}
    \opn{tr}_{U(p^n)^{\circ}}^{U(p^n)^{\lozenge}}(\iota_n^*(\phi) \cup [f])(d, [h_{\infty}], x, 1) &= \sum_{b \in \left( \mbb{Z}/p^n\mbb{Z} \right)^{\times}} (\iota_n^*(\phi) \cup [f])(d, [h_{\infty}], x, \delta_b) \\ &= \sum_{b \in \left( \mbb{Z}/p^n\mbb{Z} \right)^{\times}} \Psi_n(\mu, -[\psi_{p^n\Zp, b+p^n\Zp,f^{\iota}}]))(d, [h_{\infty}], x, 1) \\ &= \Psi_n(\mu, -\psi_{f^{\iota}})(d, [h_{\infty}], x, 1)
    \end{align*}
    where the last equality follows from the bilinearity of $\Psi_n$. The equality (\ref{Eqn:TracePullbackPsin}) now follows because both sides are determined by their values on $(d, [h_{\infty}], x, 1)$ (using Notation \ref{not:evaluation data}(iii)).

    We now see that 
    \begin{align}
        \mathscr{L}_{n}(\mu,f) &= \langle \Psi_{n}(\mu, -\psi_{f^{\iota}}), \mathcal{E}_{\Phi_n, \Unloz} \rangle_{\Unloz} \notag\\
        &= \langle \opn{tr}_{\Uncirc}^{\Unloz}( \iota_n^*(\phi) \cup [f] ), \mathcal{E}_{\Phi_n, \Unloz} \rangle_{\Unloz} \notag\\
        &= \langle \iota_n^*(\phi) \cup [f] , \mathcal{E}_{\Phi_n, \Uncirc} \rangle_{\Uncirc} \notag\\
        &= \langle \iota_n^*(\phi) , \mathcal{E}_{\Phi_n, \Uncirc} \cup [f]\rangle_{\Uncirc} . \label{eq:LW pullback}
    \end{align}
    This now bears direct comparison to analogous objects constructed in \cite[\S6.2]{LW21};    
    in the notation \emph{op.\ cit}., this is precisely the element ${}_{c}\Xi^{[0,0]}_n \in \cO_L[(\Z/p^n\Z)^\times]$ evaluated at $f$. Precisely, the pairing defining ${}_c\Xi_n^{[0,0]}$ there is a pushforward under $\iota_n$; using adjointness of pushforward and pullback, we see that pairing equals \eqref{eq:LW pullback}.

   Now, the result follows from \cite[Corollary 6.12]{LW21}, which shows the ${_c\Xi^{[0,0]}_n}$ fit into a $p$-adic measure as $n$ varies. Here we use crucially that $U_{p,1}$ fixes $\phi$. 
\end{proof}

\subsection{$p$-adic periods attached to $p$-arithmetic cohomology classes}\label{sec:period}

We will also consider other values of this evaluation map. We begin with the following analogue of $\Psi_n$, where we replace $U(p^n)^\lozenge$ with $U_1(p^n)$ and specialise the Steinberg variable.

\begin{definition} \label{OrdPeriodDefinition}
    Let $\varphi_{\opn{Iw}} \in \opn{St}(\mathcal{O}_L)$ denote the Iwahori-invariant vector represented by the characteristic function of $\overline{B}(\Q_p)\opn{Iw}$. For any integer $m \geq 1$, we let
    \[
    \Psi_{\opn{Iw}, m}(\mu) \in \opn{H}^0\left( H(\Q), \mathcal{A}_{H, c}(\opn{Ind}^{H(\Q_p)}_{U_1(p^m)}\mathcal{O}_L) \right)
    \]
    denote the element satisfying
    \[
    \Psi_{\opn{Iw},m}(\mu)(d, [h_{\infty}], h^p, h_p) = \mu(\iota(d), [h_{\infty}], h^p, h_p u_0 \cdot \varphi_{\opn{Iw}})
    \]
    for $d \in D_H$, $[h_{\infty}] \in \pi_0(H(\mbb{R}))$, $h^p \in H(\mbb{A}_f^p)$, $h_p \in H(\Q_p)$. This is well-defined, since $u_0^{-1}U_1(p^m)u_0 \subset \opn{Iw}$ for all $m$. Note also that $\Psi_{\opn{Iw},m}(\mu)$ is nothing but $\Psi_{\opn{Iw},1}(\mu)$ considered at level $U_1(p^m)$, so this class has no significant dependence on $m$. Indeed, we have
    \begin{align*} 
\left\langle \Psi_{\opn{Iw}, m}(\mu), \mathcal{E}_{\Phi_m, U_1(p^m)} \right\rangle_{U_1(p^m)} &= \left\langle \opn{res}(\Psi_{\opn{Iw}, 1}(\mu)), \mathcal{E}_{\Phi_m, U_1(p^m)} \right\rangle_{U_1(p^m)}\\
&= \left\langle \Psi_{\opn{Iw}, 1}(\mu), \opn{tr}\left(\mathcal{E}_{\Phi_m, U_1(p^m)} \right)\right\rangle_{U_1(p)} = \left\langle \Psi_{\opn{Iw}, 1}(\mu), \mathcal{E}_{\Phi_1, U_1(p)} \right\rangle_{U_1(p)}.
\end{align*}
In particular, the quantity
    \begin{equation}\label{eq:Pmu}
    \sP(\mu) \defeq \left\langle \Psi_{\opn{Iw}, m}(\mu), \mathcal{E}_{\Phi_m, U_1(p^m)} \right\rangle_{U_1(p^m)} \in \cO_L
    \end{equation}
    is independent of $m$. 
    \end{definition}

We have the following relation.  Let $\xi_n = \opn{ch}\left(\overline{B}(\Q_p) \smat{1 & p \Z_p & p^n \Z_p \\ & 1 & \Z_p \\ & & 1}\right) = \varphi_{p^n\Zp,p\Zp,\mathbbm{1}} \in \opn{St}(\mathcal{O}_L)$.

\begin{lemma}\label{lem:period vs evaluation}
    We have
    \[
    \widetilde{\Psi}_n(\mu, u_0\cdot \xi_n) = \Pmu
    \]
    for any $n \geq 1$.
\end{lemma}
\begin{proof}
    Let $\{ z \}$ be a set of representatives for $p\mbb{Z}_p/p^n\mbb{Z}_p$. Then, we calculate that
    \begin{equation} \label{eq:trace 1 to n} 
    \sum_z \smat{1 & z & \\ & 1 & \\ & & 1} \cdot u_0 \cdot \xi_n = u_0 t \cdot \varphi_{\opn{Iw}} = t u_0 \cdot \varphi_{\opn{Iw}} = u_0 \cdot \xi_1.
    \end{equation}
    We consider two Hecke operators:
    \[
    V_p = [U(p)^{\lozenge} \cdot t \cdot U_1(p)], \quad \quad V_p^t = [U_1(p) \cdot t^{-1} \cdot U(p)^{\lozenge}] .
    \]
    Note that $V_p^t \cdot \Phi_1 = \Phi_1$.

    We first calculate that
    \begin{align*}
        \widetilde{\Psi}_1(\mu, u_0 \cdot \xi_1) &= \left\langle \Psi_{1}(\mu, u_0\cdot \xi_1), \mathcal{E}_{\Phi_1, U(p)^{\lozenge}} \right\rangle_{U(p)^{\lozenge}} \\
        &= \left\langle \opn{res}_{U(p)^{\lozenge}}^{U(p)^{\lozenge} \cap U_1(p^n)} \Psi_{1}(\mu, u_0 \cdot \xi_1), \mathcal{E}_{\Phi_n, U(p)^{\lozenge} \cap U_1(p^n)} \right\rangle_{U(p)^{\lozenge} \cap U_1(p^n)} \\
        &= \left\langle \opn{tr}_{U(p^n)^{\lozenge}}^{U(p)^{\lozenge} \cap U_1(p^n)} \Psi_{n}(\mu, u_0 \cdot \xi_n), \mathcal{E}_{\Phi_n, U(p)^{\lozenge} \cap U_1(p^n)} \right\rangle_{U(p)^{\lozenge} \cap U_1(p^n)} \\
        &= \left\langle \Psi_{n}(\mu, u_0 \cdot \xi_n,), \mathcal{E}_{\Phi_n, U(p^n)^{\lozenge}} \right\rangle_{U(p^n)^{\lozenge}} \\
        &= \widetilde{\Psi}_n(\mu, u_0 \cdot \xi_n).
    \end{align*}
    For the second equality we have used that $\opn{tr}_{U(p)^{\lozenge} \cap U_1(p^n)}^{U(p)^{\lozenge}} \mathcal{E}_{\Phi_n, U(p)^{\lozenge} \cap U_1(p^n)} = \mathcal{E}_{\Phi_1, U(p)^{\lozenge}}$, and the third uses \eqref{eq:trace 1 to n}. This shows independence of $n$ and allows us to assume $n=1$.

    Another calculation shows:
    \begin{align*}
        \widetilde{\Psi}_1(\mu, u_0 \cdot \xi_1) &= \langle \Psi_{1}(\mu, u_0 \cdot \xi_1), \mathcal{E}_{\Phi_1, U(p)^{\lozenge}} \rangle_{U(p)^{\lozenge}} \\
        &= \langle V_p \cdot \Psi_{\opn{Iw}, 1}(\mu), \mathcal{E}_{\Phi_1, U(p)^{\lozenge}} \rangle_{U(p)^{\lozenge}} \\
        &= \langle \Psi_{\opn{Iw}, 1}(\mu), V_p^t \cdot \mathcal{E}_{\Phi_1, U(p)^{\lozenge}} \rangle_{U_1(p)} \\
        &= \langle \Psi_{\opn{Iw}, 1}(\mu), \mathcal{E}_{\Phi_1, U_1(p)} \rangle_{U_1(p)} \\
        &= \Pmu
    \end{align*}
    where for the second equality we have used $t u_0 \cdot \varphi_{\opn{Iw}} = u_0 \cdot \xi_1$, and the penultimate equality uses $V_p^t \cdot \Phi_1 = \Phi_1$ and equivariance of the Eisenstein classes.
\end{proof}

\subsection{$p$-smooth representations}
The rest of this section is concerned with an abstract exceptional zero formula. One of the key ingredients will be an analogue of the above evaluation maps on $\h^1$ rather than $\h^0$. To define this, we need some abstract notions. We return to the case of allowing $R$ to be any $\Z$-algebra. For any integer $n \geq 1$, let
\[
\widehat{G}_n \defeq \opn{ker}(G(\Z_{(p)}) \to G(\Z/p^n\Z)) \subset G(\Q) .
\]

\begin{definition}
    Let $M$ be a left $R[G(\Q)]$-module. We define the ``$p$-smooth part'' to be
    \[
    M^{p\opn{-sm}} \defeq \bigcup_{n \geq 1} M^{\widehat{G}_n} \subset M .
    \]
\end{definition}

\begin{lemma}
The $G(\Q)$-action on $M$ preserves $M^{p\opn{-sm}}$ (which is thus a $G(\Q)$-representation). 
\end{lemma}
\begin{proof}
If $m \in M^{p\opn{-sm}}$, then $m$ is fixed by some $\widehat{G}_{N_m}$, hence by $\widehat{G}_N$ for all $N \geq N_m$. 

Let $g \in G(\Q)$. We have an analogue of the Cartan decomposition 
\[
    G(\Q) = \bigsqcup_{a\geq b \geq c}G(\Z_{(p)}) t_{abc}G(\Z_{(p)}), \qquad t_{abc} = \smallthreemat{p^a}{}{}{}{p^b}{}{}{}{p^c};
\]
write $g = g_1t_{abc} g_2$ in this decomposition. For $n > a$, we have $t_{abc}^{-1}\widehat{G}_nt_{abc} \subset \widehat{G}_{n-a}$. 

Let $N > N_m+a$, and $h \in \widehat{G}_N$. As each $\widehat{G}_n \triangleleft G(\Z_{(p)})$ is normal, and by the above, there exists $h' \in \widehat{G}_{N-a} \subset \widehat{G}_{N_m}$ such that $hg = gh'$. Then $h\cdot(g\cdot m) = gh'\cdot m = g\cdot m$, since $\widehat{G}_{N_m}$ fixes $m$. In particular
\[
g\cdot m \in M^{\widehat{G}_N} \subset M^{p\opn{-sm}}. \qedhere
\]
\end{proof}

\subsection{A preliminary map} \label{APrelimMapSSEc}

Recall $\opn{St}_2(R) = \opn{St}_2^{\opn{sm}}(R)$ from Definition \ref{def:generalised steinberg}. The goal of this section is to construct maps 
\[
\Xi_n \colon \mathcal{A}_{G, c}(\opn{St}_2(R)^*)^{\widehat{G}_n} \to \opn{H}^0\left( H(\Q), \mathcal{A}_{H, c}(\opn{Ind}_{U(p^n)^{\lozenge}}^{H(\Q_p)}R) \right)
\]
satisfying Proposition \ref{prop:preliminary map}. 

\subsubsection{Definition of the maps}

We first define the maps and show they are well-defined. Let 
\[
    \phi_n = \left[\opn{ch}\left( \overline{P}_2(\Q_p)\smat{1 & & p^n\Z_p \\ & 1 & \Z_p \\ & & 1}\right)\right] \in \opn{St}_2(R).
    \]

\begin{definition}\label{def:Xi_n}
Define
    \[
    \Xi_n(F)(d,[h_\infty],h^p,h_p) \defeq F(h^{-1}\iota(d), [G(\mbb{R})_+], x, u_0 \cdot \phi_n).
    \]
\end{definition}

  Here we use the same conventions as when defining $\Psi_n$ in Definition \ref{def:Psi_n}; in particular we write $[h_\infty] = h[H(\R)_+]$, $h^p = hxk^p$ and $h_p = hk_p$. The map bears strong comparison to $\Psi_n$, but we now impose the weaker assumption of $\widehat{G}_n$-invariance, are working with $\opn{St}_2(R)^*$ rather than $\opn{St}(R)^*$, and fix the Steinberg element to be $u_0 \cdot \phi_n$. As with $\Psi_n$, we must show this is well-defined, as $h$ is only well-defined up to $\Gamma(p^n)^\lozenge_x$. First we note:

\begin{lemma}\label{l:calculation} If $k_p\in U(p^n)^{\lozenge}$, then $k_p u_0 \cdot \phi_n = u_0 \cdot \phi_n$.
\end{lemma}
\begin{proof} It suffices to prove that $\overline{P}_2(\Q_p)\smat{1 & & p^n\Z_p \\ & 1 & \Z_p \\ & & 1}= \overline{P}_2(\Q_p)\smat{1 & & p^n\Z_p \\ & 1 & \Z_p \\ & & 1}u_0^{-1}k_pu_0$. We write $k_p = \smat{a & b & 0 \\ c & d & 0 \\ 0 & 0 & e}$ with $a, e \in \mbb{Z}_p^{\times}$, $d \equiv 1 \pmod{p^n}$ and $b, c \equiv 0 \pmod{p^n}$. Let $\delta = ad - bc \in \mbb{Z}_p^{\times}$. For the inclusion $\supset$, note that for each $x\in p^n\Z_p$ and $y\in \Z_p$ we have
\[
\smat{1 & & x \\ & 1 & y \\ & & 1} u_0^{-1}k_pu_0= \smat{a-cx & 0 & 0 \\ -cy & e & 0 \\ -c & 0 & \tfrac{\delta}{a-cx}}\smat{1 & 0 & \tfrac{-b+dx}{a-cx} \\ 0 & 1 & \tfrac{\delta y}{e(a-cx)} \\ 0 & 0 & 1} \in \overline{P}_1(\Qp)\smat{1 & & p^n\Z_p \\ & 1 & \Z_p \\ & & 1}.
\]
For the other inclusion, it is enough to consider the above with $k_p^{-1}$ instead of $k_p$.
\end{proof}

\begin{lemma} \label{Lem:XInFiswell-defined}
The function $\Xi_n(F) \in \cA_{H,c}(\opn{Ind}_{U(p^{n})^\lozenge}^{H(\Qp)}R)$ is well-defined, independent of the choice of representatives $I(p^n)^\lozenge$, and $H(\Q)$-invariant.
\end{lemma}
\begin{proof}
This is directly analogous to Lemmas \ref{lem:ind of reps} and \ref{lem:Psi invariant}, replacing the $U(p^n)$-invariance of $\varphi$ with Lemma \ref{l:calculation}. We note that by Lemma \ref{lem:gamma in Ghat} the elements $\gamma$ and $h_i$ considered in Lemma \ref{lem:ind of reps} actually lie in $\widehat{G}_n$, so the arguments with $G(\Q)$-invariance in Lemma \ref{lem:ind of reps} transfer \emph{mutatis mutandis} to the present setting. 
\end{proof}

\subsubsection{Compatibility in $n$: statement}

Now we consider variation in $n$. The set of representatives $I(p^n)^{\lozenge}$ can be chosen independently of $n$, so we may assume $I(p^n)^{\lozenge} = I(p^{n+1})^{\lozenge} =: I$. 

\begin{proposition}\label{prop:preliminary map}
For $n \geq 1$, the maps $\Xi_n$ satisfy the following two properties:
\begin{itemize}
    \item[(a)] One has a commutative diagram
    \[
\begin{tikzcd}
{\mathcal{A}_{G, c}(\opn{St}_2(R)^*)^{\widehat{G}_{n+1}}} \arrow[rr, "\Xi_{n+1}"]    &  & {\opn{H}^0\left( H(\Q), \mathcal{A}_{H, c}(\opn{Ind}_{U(p^{n+1})^{\lozenge}}^{H(\Q_p)}R) \right)} \arrow[d, "\opn{tr}_{U(p^{n+1})^{\lozenge}}^{U(p^n)^{\lozenge} \cap U_1(p^{n+1})}"] \\
                                                                                     &  & {\opn{H}^0\left( H(\Q), \mathcal{A}_{H, c}(\opn{Ind}_{U(p^n)^{\lozenge}\cap U_1(p^{n+1})}^{H(\Q_p)}R) \right)}                                                                        \\
{\mathcal{A}_{G, c}(\opn{St}_2(R)^*)^{\widehat{G}_n}} \arrow[uu] \arrow[rr, "\Xi_n"] &  & {\opn{H}^0\left( H(\Q), \mathcal{A}_{H, c}(\opn{Ind}_{U(p^n)^{\lozenge}}^{H(\Q_p)}R) \right)} \arrow[u, "\opn{res}_{U(p^n)^{\lozenge}}^{U(p^n)^{\lozenge} \cap U_1(p^{n+1})}"']      
\end{tikzcd}
    \]
    where $\opn{res}$ and $\opn{tr}$ denote the natural restriction and trace maps.
    \item[(b)] Let $t = \opn{diag}(p, 1, 1)$. One has a commutative diagram:
    \[
\begin{tikzcd}
{\mathcal{A}_{G, c}(\opn{St}_2(R)^*)^{\widehat{G}_{n+1}}} \arrow[rr, "\Xi_{n+1}"]              &  & {\opn{H}^0\left( H(\Q), \mathcal{A}_{H, c}(\opn{Ind}_{U(p^{n+1})^{\lozenge}}^{H(\Q_p)}R) \right)}                                                                        \\
                                                                                               &  & {\opn{H}^0\left( H(\Q), \mathcal{A}_{H, c}(\opn{Ind}_{t^{-1}U(p^{n+1})^{\lozenge}t}^{H(\Q_p)}R) \right)} \arrow[u, "t \cdot"']                                               \\
{\mathcal{A}_{G, c}(\opn{St}_2(R)^*)^{\widehat{G}_n}} \arrow[uu, "t\cdot"] \arrow[rr, "\Xi_n"] &  & {\opn{H}^0\left( H(\Q), \mathcal{A}_{H, c}(\opn{Ind}_{U(p^n)^{\lozenge}}^{H(\Q_p)}R) \right)} \arrow[u, "\opn{res}_{U(p^n)^{\lozenge}}^{t^{-1}U(p^{n+1})^{\lozenge}t}"']
\end{tikzcd}
    \]
\end{itemize}
\end{proposition}

\subsubsection{Compatibility in $n$: proofs}

We will prove each part of Proposition \ref{prop:preliminary map} separately.

\begin{lemma} \label{Lem:Prop(a)ofCompatinn}
    The maps $\Xi_n$ satisfy property (a) of Proposition \ref{prop:preliminary map}.
\end{lemma}
\begin{proof}
     Let $F \in \mathcal{A}_{G, c}(\opn{St}_2(R)^*)^{\widehat{G}_n}$. Let $d \in D_H$, $h_{\infty} \in H(\mbb{R})$, $h^p \in H(\mbb{A}_f^p)$ and $h_p \in H(\mbb{Q}_p)$, and let $h \in H(\mbb{Q})$, $x \in I$, $k^p \in K^p_H$ and $k_p \in U(p^{n+1})^{\lozenge} \subset U(p^n)^{\lozenge}$ be such that:
     \[
     [h_{\infty}] = h[H(\mbb{R})_+], \; h^p = h x k^p, \; h_p = hk_p .
     \]
     Then we have:
    \[
    \left(\opn{res}_{U(p^n)^{\lozenge}}^{U(p^n)^{\lozenge} \cap U_1(p^{n+1})} \circ\  \Xi_{n}\right)(F)(d, [h_\infty], h^p, h_p) = F(h^{-1} \cdot \iota(d), [G(\mbb{R})_+], x, u_0 \cdot \phi_n) .
    \]
    Let
    \[
            J_n \defeq \left\{\smat{1 & -p^n b & \\ & 1 & \\ & & 1} : b \in \Z/p\Z\right\} \subset H(\Q)_+\cap \widehat{G}_n \subset H(\A).
    \]
    Considering $J_n\subset U(p^n)^\lozenge$ gives a set of representatives for $[U(p^n)^{\lozenge} \cap U_1(p^{n+1})] / U(p^{n+1})^{\lozenge}$. In particular,
    \[
    \left(\opn{tr}_{U(p^{n+1})^{\lozenge}}^{U(p^n)^{\lozenge} \cap U_1(p^{n+1})} \circ\  \Xi_{n+1}\right)(F)(d, [h_\infty], h^p, h_p) = \sum_{v \in J_n} \Xi_{n+1}(F)(d, [h_\infty], h^p, h_pv).
    \]
    We compute
    \begin{align*}
    \Xi_{n+1}(F)(d, [h_\infty], h^p, h_pv) &= \Xi_{n+1}(F)(d, h[H(\R)_+], hxk^p, hk_pv)\\
    &= \Xi_{n+1}(F)(d, hv[H(\R)_+], hv \cdot (v^p)^{-1}xk^p, hvk_p'),
    \end{align*}
    where $k_p' = v^{-1}k_p v \in v^{-1}U(p^{n+1})^\lozenge v= U(p^{n+1})^\lozenge$. Note that $(v^{p})^{-1}x = v^{-1}xv_p$ represents the same double coset as $x$. Via Lemma \ref{Lem:XInFiswell-defined}, we may compute $\Xi_{n+1}(F)$ with respect to the set of representatives $I' = \{(v^p)^{-1}x : x \in I\}$, giving
    \begin{align*}
    \Xi_{n+1}(F)(d, [h_\infty], h^p, h_pv) &= F(v^{-1}h^{-1}\iota(d), [G(\R)_+], (v^p)^{-1}x, u_0\cdot\phi_{n+1})\\ 
    &= F(h^{-1}\iota(d), [G(\R)_+], x, vu_0\cdot \phi_{n+1}),
    \end{align*}
    where in the second equality we use $\widehat{G}_n$-invariance of $F$.

A direct computation shows
    \[
    \sum_v (u_0^{-1}v u_0) \cdot \phi_{n+1} = \sum_b \smat{1 & & p^n b  \\ & 1 & \\ & & 1} \cdot \phi_{n+1} = \phi_n.
    \]
In particular
    \[
    \sum_v v u_0 \cdot \phi_{n+1} = u_0 \cdot \phi_n.
    \]
    Combining everything, we find
    \begin{align*}
    \left(\opn{tr}_{U(p^{n+1})^{\lozenge}}^{U(p^n)^{\lozenge} \cap U_1(p^{n+1})} \circ \Xi_{n+1}\right)(F)(d, [h_\infty], h^p, h_p) &= \sum_{v} F(h^{-1}\iota(d), [G(\R)_+], x, vu_0\cdot \phi_{n+1})\\
    &= F(h^{-1}\iota(d), [G(\R)_+], x, u_0\cdot \phi_{n})\\
    &= \left(\opn{res}_{U(p^n)^{\lozenge}}^{U(p^n)^{\lozenge} \cap U_1(p^{n+1})} \circ \Xi_{n}\right)(F)(d, [h_\infty], h^p, h_p),
    \end{align*}
    as required. 
\end{proof}

\begin{lemma}
    The maps $\Xi_n$ satisfy property (b) of Proposition \ref{prop:preliminary map}.
\end{lemma}
\begin{proof}
    Let $I=I(p^{n+1})^{\lozenge} = \{ x_i \}$ be a set of representatives. Then $I' = \{ x_i' = t^{-1}x_i \}$ is a set of representatives (in $\prod_{\ell \neq p}H(\Z_{\ell})$) for the level $U(p^{n})^{\lozenge}$. Let $F \in \mathcal{A}_{G, c}(\opn{St}_2(R)^*)^{\widehat{G}_n}$. Then, with the same notation at the start of the proof of Lemma \ref{Lem:Prop(a)ofCompatinn},
    \begin{align*}
    \Xi_{n+1}(t \cdot F)(d, [h_\infty], h^p,h_p) &= \Xi_{n+1}(t\cdot F)(d, h[H(\R)_+], hxk^p, hk_p)\\
    &= (t\cdot F)(h^{-1}\iota(d),[G(\R)_+], x, u_0 \cdot \phi_{n+1})\\
    &= F((ht)^{-1}d, t^{-1}[G(\R)_+], t^{-1}x, t^{-1}u_0 \cdot \phi_{n+1}).
    \end{align*}
    Note $t$ commutes with $u_0$ and $t^{-1} \cdot \phi_{n+1} = \phi_n$, whilst $t^{-1}$ fixes $[G(\R)_+]$. Also, $t^{-1}k_p t \in U(p^n)^{\lozenge}$. This thus equals
    \begin{align}
    F((ht)^{-1}d, [G(\R)_+], t^{-1}x, u_0 \cdot \phi_n) &= \Xi_{n}(F)(d, [h_\infty], ht(t^{-1}x)k^p, ht \cdot t^{-1}k_pt ) \label{eq:I'} \\ 
    &= \Xi_{n}(F)(d, [h_\infty], h^p,h_p t) \nonumber
    \end{align}
    as required. Note that in (\ref{eq:I'}) we compute $\Xi_n$ with respect to the representatives $I' = \{t^{-1}x : x \in I\}$, which is valid by Lemma \ref{Lem:XInFiswell-defined}.
\end{proof}

\subsection{Big evaluation maps}

We now define the big evaluation maps. Firstly, let 
\[
\widetilde{\Xi}_n \colon \mathcal{A}_{G, c}(\opn{St}_2(R)^*)^{\widehat{G}_n} \to R
\]
denote the $R$-linear map given by
\[
\widetilde{\Xi}_n(F) \defeq \langle \Xi_n(u_0 \cdot F), \mathcal{E}_{\Phi_n, U(p^n)^{\lozenge}} \rangle_{U(p^n)^{\lozenge}} 
\]
noting  $u_0$ normalises $\widehat{G}_n$ and recalling $\Phi_n$ from Notation \ref{not:schwartz}. We emphasise here that $u_0$ is acting on $F$ as an element of $G(\Q)$, so that the following lemma holds. Recall $\widetilde{\Psi}_n$ from Definition \ref{def:Psi tilde}, and the $G(\Q_p)$-equivariant morphism $\opn{pr} \colon \opn{St}_1(R) \otimes \opn{St}_2(R) \to \opn{St}(R)$. 

\begin{lemma}\label{lem:classical vs big}
For any $\mu \in \h^0(G(\Q), \cA_{G,c}(\opn{St}(R)^*))$ and $\xi \in \opn{St}_1(R)^{U(p^n)}$, we have
    \[
    \widetilde{\Psi}_n\big(\mu, \ u_0\cdot\opn{pr}(\xi\otimes\phi_n)\big) = \widetilde{\Xi}_n(\mu_\xi),
    \]
    where $\mu_\xi \in \mathcal{A}_{G, c}(\opn{St}_2(R)^*)^{\widehat{G}_n}$ is the element 
    \[
        \mu_\xi(d, [g_{\infty}], g^p, \phi) = \mu(d, [g_{\infty}], g^p, \opn{pr}(\xi \otimes \phi))
    \]
    for $d \in D_G$, $[g_{\infty}] \in \pi_0(G(\mbb{R}))$, $g^p \in G(\mbb{A}_f^p)$, and $\phi \in \opn{St}_2(R)$.
\end{lemma}
\begin{proof}
With notation from Definition \ref{def:Xi_n} with $F = u_0 \cdot \mu_{\xi}$, we compute that 
\begin{align*}
(u_0 \cdot \mu_{\xi})(h^{-1}\iota(d), [G(\mbb{R})_+], x, u_0 \cdot \phi_n) &= \mu_{\xi}(u_0^{-1}h^{-1}\iota(d), u_0^{-1}[G(\mbb{R})_+], u_0^{-1}x, \phi_n) \\
&= \mu(u_0^{-1}h^{-1}\iota(d), u_0^{-1}[G(\mbb{R})_+], u_0^{-1}x, \opn{pr}(\xi \otimes \phi_n)) \\
&= \mu(h^{-1}\iota(d), [G(\mbb{R})_+], x, u_0 \cdot \opn{pr}(\xi \otimes \phi_n))
\end{align*}
where the last equality uses the $G(\mbb{Q})$-invariance of $\mu$. From the definition of $\Psi_n$ (Definition \ref{def:Psi_n}), we therefore see that $\Psi_n(\mu, u_0 \cdot \opn{pr}(\xi \otimes \phi_n)) = \Xi_n(u_0 \cdot \mu_{\xi})$, which implies the claim. 
\end{proof}

\begin{lemma}
We have $\widetilde{\Xi}_{n+1}(F) = \widetilde{\Xi}_n(F)$. In particular, there is a map
\[
\widetilde{\Xi}_{\infty} \colon \mathcal{A}_{G, c}(\opn{St}_2(R)^*)^{p\opn{-sm}} \to R
\]
defined by $\widetilde{\Xi}_\infty \defeq \varinjlim_n \widetilde{\Xi}_n$.
\end{lemma}
\begin{proof}
We compute 
\begin{align*}
\widetilde{\Xi}_{n+1}(F) &= \left\langle \Xi_{n+1}(u_0 \cdot F), \mathcal{E}_{\Phi_{n+1}, U(p^{n+1})^{\lozenge}} \right\rangle_{U(p^{n+1})^{\lozenge}} \\
&= \left\langle \opn{tr} \big(\Xi_{n+1}(u_0 \cdot F)\big), \mathcal{E}_{\Phi_{n+1}, U(p^n)^{\lozenge} \cap U_1(p^{n+1})} \right\rangle_{U(p^n)^{\lozenge} \cap U_1(p^{n+1})} \\
&= \left\langle \opn{res}\big(\Xi_n(u_0 \cdot F)\big), \mathcal{E}_{\Phi_{n+1}, U(p^n)^{\lozenge} \cap U_1(p^{n+1})} \right\rangle_{U(p^n)^{\lozenge} \cap U_1(p^{n+1})}\\
&= \left\langle \Xi_n(u_0 \cdot F), \opn{tr} \left(\mathcal{E}_{\Phi_{n+1}, U(p^n)^{\lozenge} \cap U_1(p^{n+1})} \right)\right\rangle_{U(p^n)^{\lozenge}}\\
&= \left\langle \Xi_n(u_0 \cdot F), \mathcal{E}_{\mathrm{tr}(\Phi_{n+1}), U(p^n)^{\lozenge}})\right\rangle_{U(p^n)^{\lozenge}}.
\end{align*}
Here the second and fourth equalities use that restriction and trace are adjoint under the Poincar\'e pairing, the third equality is by Proposition \ref{prop:preliminary map}(a), and the last equality is trace compatibility \eqref{eq:Eis trace} of Eisenstein classes. Now, for $b,c \in \Zp$ let
\begin{equation}\label{eq:ubc}
v_{b,c} \defeq \mat{1 & 0 \\ p^nb & 1+p^nc}.
\end{equation}
Then $\{v_{b,c} \times 1 : b,c \in \Z/p\Z\} \subset U(p^n)^\lozenge$ is a set of representatives for $U(p^n)^{\lozenge} / U(p^n)^{\lozenge} \cap U_1(p^{n+1})$. We compute 
\begin{align*}
\opn{tr}(\Phi_{p,n+1}) &= \sum_{b,c}v_{-b,-c}\cdot \opn{ch}(p^{n+1}\Zp, 1+p^{n+1}\Zp)\\
&= \sum_{b,c}\opn{ch}(p^n(b+p\Zp), 1+p^n(c+p\Zp)) =  \Phi_{p,n}.
\end{align*}
We conclude that
\[
   \widetilde{\Xi}_{n+1}(F) = \langle \Xi_n(u_0 \cdot F), \mathcal{E}_{\Phi_{n}, U(p^n)^{\lozenge}} \rangle_{U(p^n)^{\lozenge}} = \widetilde{\Xi}_n(F),
\]
as required. The second statement is immediate.
\end{proof}

\begin{lemma}\label{lem:t equiv}
    The map $\widetilde{\Xi}_{\infty}$ is $t^{\Z}$-equivariant, where the target has the trivial action.
\end{lemma}
\begin{proof}
    It suffices to prove $\widetilde{\Xi}_{\infty}(t \cdot F) = \widetilde{\Xi}_{\infty}(F)$. Recall that $t$ commutes with $u_0$. Let $F$ be fixed by $\widehat{G}_n$. Then
    \begin{align*}
        \widetilde{\Xi}_{\infty}(t \cdot F) &= \widetilde{\Xi}_{n+1}(t \cdot F) \\
        &= \left\langle \Xi_{n+1}(t \cdot u_0 \cdot F), \mathcal{E}_{\Phi_{n+1}, U(p^{n+1})^{\lozenge}} \right\rangle_{U(p^{n+1})^{\lozenge}} \\
        &= \left\langle t \cdot \opn{res}\big(\Xi_n(u_0\cdot F)\big), \mathcal{E}_{\Phi_{n+1}, U(p^{n+1})^{\lozenge}} \right\rangle_{U(p^{n+1})^{\lozenge}} \\
        &= \left\langle \opn{res}\big(\Xi_n(u_0\cdot F)\big), t^{-1} \cdot \mathcal{E}_{\Phi_{n+1}, U(p^{n+1})^{\lozenge}} \right\rangle_{t^{-1} U(p^{n+1})^{\lozenge} t} \\
        &= \left\langle \opn{res}\big(\Xi_n(u_0\cdot F)), \mathcal{E}_{t^{-1}\cdot\Phi_{n+1}, t^{-1}U(p^{n+1})^{\lozenge}t} \right\rangle_{t^{-1} U(p^{n+1})^{\lozenge} t} \\
        &= \left\langle \Xi_n(u_0\cdot F), \opn{tr}\left( \mathcal{E}_{t^{-1}\cdot\Phi_{n+1}, t^{-1}U(p^{n+1})^{\lozenge}t} \right)\right\rangle_{U(p^n)^{\lozenge}} \\
        &= \left\langle \Xi_n(u_0\cdot F),  \mathcal{E}_{\opn{tr}(t^{-1}\cdot\Phi_{n+1}), U(p^{n})^{\lozenge}} \right\rangle_{U(p^n)^{\lozenge}}, 
    \end{align*}
    where the third equality follows from Proposition \ref{prop:preliminary map}(b), the fifth equality being \eqref{eq:Eis equiv} and the last again being \eqref{eq:Eis trace}. Now, a set of representatives of $U(p^n)^{\lozenge}/t^{-1}U(p^{n+1})^{\lozenge}t$ is $\{v_{b,c} \times 1 :  b \in \Z/p^2 \Z, c \in \Z/p\Z\}$, for $v_{b,c}$ as in \eqref{eq:ubc}. We calculate that
    \begin{align*}
    \opn{tr}\left(t^{-1}\cdot \Phi_{p,n+1}\right) = \sum_{b, c} v_{-b,-c}\smat{p^{-1} & \\ & 1} \cdot \Phi_{p,n+1} &= \sum_{b, c} v_{-b,-c} \opn{ch}(p^{n+2}\Z_p, 1+p^{n+1}\Z_p)\\
    &= \sum_{b, c} \opn{ch}( p^n(b+p^2\Z_p), 1+p^n(c+p\Z_p)) = \Phi_{p, n}.
    \end{align*}
    Plugging this in, we find
    \[
\widetilde{\Xi}_\infty(t\cdot F) = \left\langle \Xi_n(u_0\cdot F), \mathcal{E}_{\Phi_n, U(p^n)^{\lozenge}} \right\rangle_{U(p^n)^{\lozenge}} = \widetilde{\Xi}_n(F) = \widetilde{\Xi}_{\infty}(F),
    \]
    as required.
\end{proof}

The following is the main definition of this subsection.

\begin{definition}
    Define the \emph{big evaluation map}
    \[
    \opn{Ev}^{\opn{big}}_R \colon \opn{H}^1(G(\Q), \mathcal{A}_{G, c}(\opn{St}_2(R)^*)^{p\opn{-sm}}) \to R
    \]
    as $\opn{Ev}^{\opn{big}}_R([c]) = \widetilde{\Xi}_{\infty}( c[t])$, where $c$ is any choice of cocycle representing $[c]$. 
\end{definition}

\begin{lemma}
The map $\opn{Ev}^{\opn{big}}_R([c])$ is well-defined (i.e., independent of the choice of cocycle $c$).
\end{lemma}
\begin{proof}
Any other choice of cocycle $c'$ differs by a coboundary $b[g] = g\cdot F - F$ for some $F \in \cA_{G,c}(\opn{St}_2(R)^*)^{p\opn{-sm}}$. As $\widetilde{\Xi}_\infty(b[t]) = 0$ by Lemma \ref{lem:t equiv}, we see $\widetilde{\Xi}_\infty(c[t]) = \widetilde{\Xi}_\infty(c'[t])$, as required.
\end{proof}

\subsection{Abstract $\mathcal{L}$-invariants and the exceptional zero formula}\label{sec:abstract L-invariants}

We now place ourselves in one of the situations described in Notation \ref{BulletRLambdaNotation}. We will often drop ``$\opn{sm}$'' from the notation. Recall the cocycle $c_{1,\lambda}^\bullet \in Z^1(G(\Qp), \opn{St}_1^\bullet(R))$ from \eqref{eq:c cocycle}, where $\lambda \in \opn{Hom}_{\bullet}(\Q_p^{\times}, R)$. Define a map
\[
\Upsilon^{\bullet}_{\lambda} \colon \opn{H}^0(G(\Q), \mathcal{A}_{G, c}(\opn{St}^{\bullet}(R)^*)) \longrightarrow Z^1(G(\Q), \mathcal{A}_{G, c}(\opn{St}_2(R)^*))
\]
by
\[
\Upsilon_{\lambda}^\bullet(\mu)[z] \defeq \left( (d, g_{\infty}, g^p, \phi) \mapsto \mu\left(d, g_{\infty}, g^p, \opn{pr}\Big(c_{1,\lambda}^\bullet[z] \otimes \phi\Big)\right)  \right)
\]
where $z \in G(\Q)$, $d \in D_G$, $g_{\infty} \in \pi_0(G(\mbb{R}))$, $g^p \in G(\mbb{A}_f^p)$, $\phi \in \opn{St}_2(R)$.

\begin{remark}\label{rem:Delta cup product}
   We can alternatively describe $\Upsilon^\bullet_{\lambda}$ as follows. Recall the cocycle
\[
\Delta^{\bullet}_{1,\lambda} \in Z^1(G(\Q_p), \opn{Hom}_R(\opn{St}_2(R), \opn{St}^{\bullet}(R)))
\]
from Remark \ref{ExplicitDeltaCocycleRemark}. From the definitions, we see $\Upsilon^{\bullet}_{\lambda}(\mu)$ is the cup product $\mu \cup \Delta^{\bullet}_{1,\lambda}|_{G(\Q)}$ composed with the natural map $\mathcal{A}_{G, c}(\opn{St}^{\bullet}(R)^*) \otimes \opn{Hom}_R(\opn{St}_2(R), \opn{St}^{\bullet}(R)) \to \mathcal{A}_{G, c}(\opn{St}_2(R)^*)$. 

Later, in \S\ref{sec:define L-invariants}, because of this description and in line with our wider notation, we will denote this map $\Upsilon_{1,\lambda}^\bullet$. To ease notation, in this section we will continue to omit the subscript 1.
\end{remark}

\begin{lemma}\label{lem:smooth lands in smooth}
    For $\bullet = \opn{sm}$, the map $\Upsilon^{\opn{sm}}_{\lambda}$ lands in $Z^1(G(\Q), \mathcal{A}_{G, c}(\opn{St}_2(R)^*)^{p\opn{-sm}})$.
\end{lemma}
\begin{proof}
As $\opn{St}^{\opn{sm}}_1(R)$ is a smooth $G(\Qp)$-module, $c_{1,\lambda}^\bullet[z]$ is fixed by $\widehat{G}_{n(z)}$ for some $n(z) \gg 0$ (depending on $z$). For $g \in \widehat{G}_{n(z)}$ we thus have
\begin{align*}
\Upsilon_\lambda^{\opn{sm}}(\mu)[z](gd,gg_\infty,gg^p,g\cdot \phi) &= \mu\left(gd, gg_\infty, gg^p, \opn{pr}\left(c_{1,\lambda}^\bullet[z] \otimes g\cdot \phi\right)\right)\\
&= \mu\left(gd, gg_\infty, gg^p, \opn{pr}\left(g\cdot c_{1,\lambda}^\bullet[z] \otimes g\cdot \phi\right)\right)\\
&=\mu\left(d, g_\infty, g^p, \opn{pr}\left(c_{1,\lambda}^\bullet[z] \otimes \phi\right)\right)= \Upsilon_\lambda^{\opn{sm}}(\mu)[z](d,g_\infty,g^p,\phi),
\end{align*}
where the penultimate equality is $G(\Q)$-invariance of $\mu$. In particular for every $\mu$ and $z$, $\Upsilon_\lambda^{\opn{sm}}(\mu)[z]$ is fixed by $\widehat{G}_{n(z)}$ for some $n(z) \gg 0$; so $\Upsilon_\lambda^{\opn{sm}}(\mu)[z] \in \cA_{G,c}(\opn{St}_2(R)^*)^{p\opn{-sm}}$.
\end{proof}

\begin{remark}
In particular, we may apply $\opn{Ev}^{\opn{big}}_R$ to the image of $\Upsilon_\lambda^{\opn{sm}}$. This lemma captures the fundamental necessity of developing the theory of $p$-smooth coefficients: the space $\cA_{G,c}(-)^{p\opn{-sm}}$ is an intermediary between $\cA_{G,c}(-)$ and $\cA_{G,c}(-)^{G(\Q)}$, that is large enough to capture the image of $\Upsilon_\lambda^{\opn{sm}}$, but small enough that we can construct evaluation maps on $\cA_{G,c}(-)^{p\opn{-sm}}$.
\end{remark}

We consider the following situation. Let $\mu^{\opn{cts}} \in \opn{H}^0\left( G(\Q), \mathcal{A}_{G, c}(\opn{St}^{\opn{cts}}(\mathcal{O}_L)^*) \right)$, and let 
\[
\mu_s \in \opn{H}^0\left( G(\Q), \mathcal{A}_{G, c}(\opn{St}^{\opn{sm}}(\mathcal{O}_L/\varpi^s)^*) \right), \quad \quad \mu^{\opn{la}} \in \opn{H}^0\left( G(\Q), \mathcal{A}_{G, c}(\opn{St}^{\opn{la}}(L)^*) \right),
\]
denote the images of $\mu^{\opn{cts}}$ under the natural maps to these cohomology groups (here $s \geq 1$ is any integer). Suppose there exists an element $\mathcal{L}_{\mu} \in L$ such that the image $[\Upsilon^{\opn{la}}_{\lambda}(\mu^{\opn{la}})]$ of $\Upsilon^{\opn{la}}_{\lambda}(\mu^{\opn{la}})$ in $\opn{H}^1\left( G(\Q), \mathcal{A}_{G, c}(\opn{St}_2(L)^*) \right)$ is zero, where $\lambda = \opn{log}_p - \mathcal{L}_{\mu} \opn{ord}_p$.

\begin{lemma} \label{UniformkwhichkillsLemma}
    There exists an integer $k \geq 1$ such that: $p^k \mathcal{L}_{\mu} \in \mathcal{O}_L$, and for all $s \geq 1$ the image of $\Upsilon^{\opn{sm}}_{p^k\lambda}(\mu_s)$ in $\opn{H}^1\left(G(\Q), \mathcal{A}_{G, c}(\opn{St}_2(\mathcal{O}_L/\varpi^s)^*)^{p\opn{-sm}}\right)$ is zero.
\end{lemma}
\begin{proof}
    Let $k_0 \geq 1$ be such that $p^{k_0}\mathcal{L}_{\mu} \in \mathcal{O}_L$, so that $p^{k_0}\lambda \in \opn{Hom}_{\opn{cts}}(\Qp^\times,\cO_L)$. By the commutative diagram  \eqref{SecondExtDiagramEqn}, combined with Remarks \ref{ExplicitDeltaCocycleRemark} and \ref{rem:Delta cup product}, for any $r \geq k_0$ the diagram
\[
\begin{tikzcd}
\mu_s & \h^0(G(\Q),\cA_{G,c}(\opn{St}^{\opn{sm}}(\cO_L/\varpi^s)^*)) \arrow[rr, "\Upsilon_{p^{r}\lambda}^{\opn{sm}}"]  && \h^1(G(\Q), \mathcal{A}_{G, c}(\opn{St}_2(\cO_L/\varpi^s)^*)) \\
\mu^{\opn{cts}}\ar[u, mapsto]\arrow[dd,mapsto] & \h^0(G(\Q),\cA_{G,c}(\opn{St}^{\opn{cts}}(\cO_L)^*)) \arrow[u, "-\otimes \cO_L/\varpi^s"']\arrow[d, "-\otimes L"]\arrow[rr, "\Upsilon_{p^{r}\lambda}^{\opn{cts}}"]  && \h^1(G(\Q), \mathcal{A}_{G, c}(\opn{St}_2(\cO_L)^*)) \arrow[d, "- \otimes L"]\arrow[u, "-\otimes \cO_L/\varpi^s"'] \\
&\h^0(G(\Q),\cA_{G,c}(\opn{St}^{\opn{cts}}(L)^*))  \arrow[d]\arrow[rr, "\Upsilon_{p^{r}\lambda}^{\opn{cts}}"]&& \h^1(G(\Q), \mathcal{A}_{G, c}(\opn{St}_2(L)^*)) \\
\mu^{\opn{la}} & \h^0(G(\Q),\cA_{G,c}(\opn{St}^{\opn{la}}(L)^*)) \arrow[rru, "\Upsilon_{p^{r}\lambda}^{\opn{la}}", '] &&                                                                       
\end{tikzcd}
\]
commutes. (Here we compose $\Upsilon_{p^r\lambda}^\bullet$ with the quotient map $[-] : Z^1 \to \h^1$). 

 First take $r = k_0$. By assumption, the image of $\mu^{\opn{la}}$ along the bottom diagonal arrow is zero, and it follows that $[\Upsilon^{\opn{cts}}_{p^{k_0}\lambda}(\mu^{\opn{cts}})]$ is zero after tensoring with $L$; in particular, it is torsion in the cohomology group $\opn{H}^1\left( G(\Q), \mathcal{A}_{G, c}(\opn{St}_2(\mathcal{O}_L)^*) \right)$. Since the module $\opn{St}_2(\mathcal{O}_L)$ is flawless (see \cite[Definition 2.1, Theorem 2.6]{AutomorphicLinvariants}), this cohomology group is a finitely-generated $\mathcal{O}_L$-module (\cite[Proposition 3.2(a)]{AutomorphicLinvariants}). This means that there exists an integer $k' \geq 1$ such that $p^{k'}$ kills $[\Upsilon^{\opn{cts}}_{p^{k_0}\lambda}(\mu^{\opn{cts}})]$. 
 
 Now set $k = k_0+k'$. By the above $[\Upsilon_{p^k\lambda}^{\opn{cts}}(\mu^{\opn{cts}})] = 0$. Using the diagram above with $r = k$, we see that for any $s \geq 1$, we have
 \begin{equation}\label{eq:0 not smooth}
    0 = [\Upsilon^{\opn{sm}}_{p^k\lambda}(\mu_s)] \in \opn{H}^1\left( G(\Q), \mathcal{A}_{G, c}(\opn{St}_2(\mathcal{O}_L/\varpi^s)^*) \right).
\end{equation}
It remains to show that it is zero in the cohomology with $p$-smooth coefficients.

By \eqref{eq:0 not smooth} there exists $m\in\cA_{G,c}(\opn{St}_2(\cO_L/\varpi^s)^*))$ such that
\[
\Upsilon^{\opn{sm}}_{p^k\lambda}(\mu_s)[z] = z\cdot m - m.
\]
We must show $m$ is $p$-smooth, i.e., there exists $n \gg 0$ such that $\widehat{G}_n$ fixes $m$, as then $\Upsilon^{\opn{sm}}_{p^k\lambda}(\mu_s)$ is a coboundary in $Z^1(G(\Q),\cA_{G,c}(\opn{St}_2(\cO_L/\varpi^s)^*)^{p\opn{-sm}})$. For this, note that $\mu_s$ is fixed by all of $G(\Q)$ by definition. To define $\Upsilon^{\opn{sm}}_{p^k\lambda}(\mu_s)$, we cup with $\Delta_{1,p^k\lambda}^{\opn{sm}} \in Z^1(G(\Qp), \opn{Hom}_R(\opn{St}_2(\cO_L/\varpi^s),\opn{St}(\cO_L/\varpi^s)))$ (after restriction to $G(\Q)$). As $\Delta_{1,p^k\lambda}^{\opn{sm}}$ is defined on $G(\Qp)$, and we are working with coefficients in finite characteristic, by \cite[\S 3]{SmoothParts} it is continuous. In particular it vanishes on some open neighbourhood of the identity; so its restriction to $G(\Q)$ must vanish on $\widehat{G}_n$ for $n\gg 0$. It follows that $\mu_s\cup \Delta_{1,p^k\lambda}^{\opn{sm}}|_{G(\Q)}$ vanishes on $\widehat{G}_n$, so $\Upsilon_{p^k\lambda}^{\opn{sm}}(\mu_s)$ does too. Thus $m$ is $p$-smooth and we are done.
\end{proof}

The main result of this section is the following:

\begin{theorem} \label{AbstractLinvTheorem}
     Let $\mu^{\opn{cts}}$ be any element of $\h^0(G(\Q), \cA_{G,c}(\opn{St}^{\opn{cts}}(\cO_L)^*))$, and let $\mu$ (resp.\ $\mu^{\opn{la}}$) be its image in $\h^0(G(\Q),\cA_{G,c}(\opn{St}^{\opn{sm}}(\cO_L)^*))$ (resp.\ $\h^0(G(\Q), \cA_{G,c}(\opn{St}^{\opn{la}}(\cO_L)^*))$. Let $\sL(\mu)$ (resp. $\Pmu$) denote the measure (resp. period) associated with $\mu$ as in Proposition \ref{prop:measure} (resp.\ Definition \ref{OrdPeriodDefinition}).
     
     Let $\mathcal{L}_\mu \in L$ be any element such that $[\Upsilon_{\log_p - \mathcal{L}_\mu\opn{ord}_p}^{\opn{la}}(\mu^{\opn{la}})] = 0$. Then in $L$, we have an equality
    \[
    \mathscr{L}(\mu,\opn{log}_p) = \mathcal{L}_{\mu} \cdot \Pmu.
    \]
\end{theorem}
\begin{proof}
    Note $\log_p$ is not a smooth function on $\Zp^\times$; however, its reduction modulo $\varpi^s$ is smooth for all $s \geq 1$. In particular for fixed $s$, $\log_p \newmod{\varpi^s}$ factors through $(\Z/p^n\Z)^\times$ for sufficiently large $n$, and then from the definitions
    \begin{equation}\label{eq:measure mod s}
        \sL(\mu,\log_p) \newmod{\varpi^s} = \widetilde{\Psi}_n\left(\mu_s, -\psi_{\log_{p}^{\iota}}\right) = \widetilde{\Psi}_n\left(\mu_s, -\psi_{\log_{p}}\right).
    \end{equation}
    where we have used the fact that $\opn{log}_p^{\iota} = \opn{log}_p$. As such, we will first prove the formula modulo $\varpi^s$, using the machinery developed above. 

    Let $k \geq 1$ be as in Lemma \ref{UniformkwhichkillsLemma}, and set $\lambda = \opn{log}_p - \mathcal{L}_{\mu} \opn{ord}_p$. Fix $s\geq 1$ and take $R = \mathcal{O}_L/\varpi^s$. By that lemma, we have $\opn{Ev}_{\mathcal{O}_L/\varpi^s}^{\opn{big}}\left( \Upsilon^{\opn{sm}}_{p^k\lambda}(\mu_s) \right) = 0$, hence by linearity, we see that
    \begin{equation}\label{eq:evaluation equality}
    p^k\opn{Ev}_{\mathcal{O}_L/\varpi^s}^{\opn{big}}\left( \Upsilon^{\opn{sm}}_{\opn{log}_p}(\mu_s) \right) = p^k \mathcal{L}_{\mu} \cdot \opn{Ev}_{\mathcal{O}_L/\varpi^s}^{\opn{big}}\left( \Upsilon^{\opn{sm}}_{\opn{ord}_p}(\mu_s) \right) .
    \end{equation}
    Then for $n \gg 1$, we see that (by Lemma \ref{lem:classical vs big})
    \begin{align*}
    &\opn{Ev}_{\mathcal{O}_L/\varpi^s}^{\opn{big}}\left( \Upsilon^{\opn{sm}}_{\opn{log}_p}(\mu_s) \right) = \widetilde{\Xi}_n( \Upsilon^{\opn{sm}}_{\opn{log}_p}(\mu_s)[t] ) = \widetilde{\Psi}_n\Big(\mu_s,\ u_0 \cdot \opn{pr}(c_{1,\opn{log}_p}^{\opn{sm}}[t] \otimes \phi_n  )\Big) \\
    &\opn{Ev}_{\mathcal{O}_L/\varpi^s}^{\opn{big}}\left( \Upsilon^{\opn{sm}}_{\opn{ord}_p}(\mu_s) \right) = \widetilde{\Xi}_n( \Upsilon^{\opn{sm}}_{\opn{ord}_p}(\mu_s)[t] ) =  \widetilde{\Psi}_n\Big(\mu_s, \ u_0 \cdot \opn{pr}(c_{1,\opn{ord}_p}^{\opn{sm}}[t] \otimes \phi_n ) \Big) .
    \end{align*}
    Using Proposition \ref{prop:Clambda1LuxembourgProp}, we see that for $n \gg 0$, we have
    \begin{align*}
        \opn{pr}(c_{1,\opn{log}_p}^{\opn{sm}}[t] \otimes \phi_n ) &= -2 \varphi_{p^n\Zp, \Zp^\times. \opn{log}_p}, \\
        \opn{pr}(c_{1,\opn{ord}_p}^{\opn{sm}}[t] \otimes \phi_n ) &= - 2 \varphi_{p^n\Zp,p\Zp,\mathbbm{1}} = -2\xi_n.
    \end{align*}
    In particular, by Lemma \ref{lem:psi and varphi}, we see that in $\opn{St}(\cO_L/\varpi^s)$, we have
    \begin{align*}
        u_0 \cdot \opn{pr}(c_{1,\opn{log}_p}^{\opn{sm}}[t] \otimes \phi_n ) &= 2 \psi_{p^n\Zp,\Zp^\times,\opn{log}_{p}} = 2\psi_{\log_p}.
    \end{align*}
    For sufficiently large $n$, by bilinearity we thus have
    \begin{align*}
        \opn{Ev}_{\cO_L/\varpi^s}^{\opn{big}}\Big(\Upsilon_{\log_p}^{\opn{sm}}(\mu_s)\Big) &= -2\widetilde{\Psi}_n(\mu_s, -\psi_{\log_p})
        \equiv -2\mathscr{L}(\mu,\log_p) \newmod{\varpi^s},\\
         \opn{Ev}_{\cO_L/\varpi^s}^{\opn{big}}\Big(\Upsilon_{\opn{ord}_p}^{\opn{sm}}(\mu_s)\Big)&= -2\widetilde{\Psi}_n(\mu_s, u_0\cdot \xi_n)
        \equiv -2\Pmu \newmod{\varpi^s},
    \end{align*}
    where the top line is \eqref{eq:measure mod s}, and the bottom follows from reducing Lemma \ref{lem:period vs evaluation} modulo $\varpi^s$. Combining with \eqref{eq:evaluation equality} thus gives
    \[
    -2 p^k \sL(\mu,\opn{log}_p) \equiv -2 p^k \mathcal{L}_{\mu} \cdot \Pmu \pmod{\varpi^s}.
    \]
   As this holds for all $s \geq 1$, we deduce $-2p^k\sL(\mu,\log_p) = -2p^k\mathcal{L}_\mu\cdot \Pmu$ in $\cO_L$, hence conclude.
\end{proof}

\begin{remark}
Attached to $\sL(\mu)$ is an analytic function $L(\mu,-):  \Zp \to \cO_L$ defined by $L(\mu,s) \defeq \int_{\Zp^\times} \langle x \rangle^s \ d\sL(\mu,x)$, where $\langle - \rangle \colon \mbb{Z}_p^{\times} \to 1+ 2p\mbb{Z}_p$ is the natural map. Then $\sL(\mu,\log_p) = \tfrac{d}{ds}L(\mu,s)|_{s=0}$.
\end{remark}

\section{\texorpdfstring{$p$}{p}-adic \texorpdfstring{$L$}{L}-functions and \texorpdfstring{$\mathcal{L}$}{L}-invariants for GL(3)}\label{sec:recall L-invariants}

The theory developed so far has been for abstract $p$-arithmetic cohomology. We now introduce our specific automorphic representations $\pi$ of interest, recall the $p$-adic $L$-functions of $\pi$ and describe their exceptional zeros, recap the construction of the two $\mathcal{L}$-invariants for $\pi$, and finally prove a duality formula for these $\mathcal{L}$-invariants.

\subsection{Automorphic representations} \label{AutoRepSubSec}

Let $\pi$ be a unitary regular algebraic cuspidal automorphic representation of $G(\mbb{A})$, which has central character $\widehat{\omega}_{\pi}$ for some Dirichlet character $\omega_{\pi}$. Let $\pi = \bigotimes'_v \pi_v$ denote its restricted tensor product decomposition over all places of $\Q$.

\begin{assumption} \label{SteinbergAssumptionOnpi}
    We assume that the weight of $\pi$ is trivial and that $\pi_p$ is Steinberg, i.e., there exists a $G(\Q_p)$-equivariant isomorphism $\pi_p \cong \opn{St}^{\opn{sm}}(\mbb{C})$ (which we implicitly fix for the rest of the article). In particular, this implies that $\omega_{\pi}$ is finite-order of prime-to-$p$ conductor. We moreover assume $\omega_\pi(-1) = 1$ (else $L(\pi,s)$ has no critical values).
\end{assumption}

Let $L(\pi, s)$ denote the standard $L$-function associated with $\pi$; by our assumptions, the critical values of this $L$-function are at $s=0$ and $s=1$.

Finally, let $\psi \colon \mbb{A} \to \mbb{C}^{\times}$ denote the standard additive character, and let $\mathcal{W}_{\psi}(\pi)$ denote the Whittaker model associated with $\psi$ and $\pi$. This decomposes as $\mathcal{W}_{\psi}(\pi) = \bigotimes_v' \mathcal{W}_{\psi_v}(\pi_v)$.

\subsection{The $p$-adic $L$-functions of $\pi$} \label{sec:p-adic l-functions}

We recap the results of \cite{LW21}, which construct $p$-adic $L$-functions for $\pi$. Note that $\opn{St}^{\opn{sm}}(\C)$ is $B$-ordinary, hence both $P_1$-ordinary and $P_2$-ordinary; see \cite[Example 2.13]{LW21}. Let
\begin{align*}
\opn{Crit}^-_p(\pi) &= \{(\eta,0) : \eta\text{ even Dirichlet character of $p$-power conductor}\},\\
\opn{Crit}^+_p(\pi) &= \{(\eta,1) : \eta\text{ even Dirichlet character of $p$-power conductor}\}.
\end{align*}
Let $e_{\infty}^-(\pi_\infty, s) = e_{\infty}^-(\pi_\infty\times\eta_\infty, s)$ and $e_p^-(\pi_p \times \eta_p, s)$ denote the (non-zero) modified $L$-factors at $\infty$ and $p$ for $(\eta,0) \in \opn{Crit}^-_p(\pi)$, as in \cite[Definitions 2.6, 2.16]{LW21}, where $\eta$ is an even Dirichlet character of $p$-power conductor. For $\eta = 1$, similar to Example 2.21 \emph{op.\ cit}.\ (noting $\alpha_p$ there is $p^{-1}$) we have 
\[
e_p^-(\pi_p, s) = (1-p^{s}) \cdot (1-p^{-1-s})^{-1} = (1-p^{s}) \cdot L(\pi_p, s) .
\]
Also let $e_\infty^+(\pi_\infty,s) = e_\infty^+(\pi_\infty\times\eta_\infty,s)$ and $e_p^+(\pi_p\times\eta_p,s)$ be the analogues for the right-half; for $\eta = 1$ one has
\[
e_p^+(\pi_p,s) = \varepsilon(\sigma_p,s)^{-1}\cdot L(\sigma_p, 1-s)^{-1} \cdot L(\sigma_p,s) = - p^s (1-p^{s-1})\cdot L(\pi_p,s),
\]
where $\sigma_p$ is the twist of the Steinberg representation for $\opn{GL}_2(\mbb{Q}_p)$ by $|\cdot|^{1/2}$. Here we are using the fact that, if $\opn{St}_{\opn{GL}_d}(\mbb{C})$ denotes the Steinberg representation of $\opn{GL}_d(\mbb{Q}_p)$, then 
\begin{equation}\label{eq:steinberg epsilon}
\varepsilon(\opn{St}_{\opn{GL}_d}(\mbb{C}), s) = (-1)^{d-1}p^{(d-1)(1/2-s)}
\end{equation}
with respect to $\psi$ and the standard choice of Haar measure (see \cite[(3.2.3), (4.3.4)]{WedLLC}). 

As $\pi_p$ is self-dual, we have a `functional equation' $e^+(\pi_p,1-s) = \gamma(\pi_p,s) \cdot e^-(\pi_p,s).$ Here $\varepsilon(-)$ and $\gamma(-)$ are the usual $\varepsilon$- and $\gamma$-factors for $\pi_p$, defined with respect to the additive character $\psi$.

\begin{theorem}[Loeffler--Williams] \label{thm:LW21}
There exist measures $\sL_p^\pm(\pi) \in \cO_L[\![\Zp^\times]\!]$ such that:
\begin{itemize}
\item[(i)] For all $(\eta,0) \in \opn{Crit}^-_p(\pi)$, we have
\[
\int_{\Zp^\times} \eta^{-1}(x) \cdot d\sL_p^-(\pi) = e_\infty^-(\pi_\infty,0) \cdot e_p^-(\pi_p\times\eta_p, 0) \cdot \frac{L^{(p)}(\pi\times\eta, 0)}{\Omega_\pi^-}.
\]
\item[(i)] For all $(\eta,1) \in \opn{Crit}^+_p(\pi)$, we have
\[
\int_{\Zp^\times} \eta^{-1}(x) \cdot d\sL_p^+(\pi) = e_\infty^+(\pi_\infty,1)\cdot e_p^+(\pi_p\times\eta_p, 1) \cdot \frac{L^{(p)}(\pi\times\eta, 1)}{\Omega_\pi^+}.
\]
\end{itemize}
\end{theorem}

Here the period $\Omega_\pi^-$ is as defined in \cite[(9.1)]{LW21}, and $\Omega_\pi^+ \defeq \Omega_{\pi^\vee}^-$. As an artifact of the construction, $\Omega_\pi^\pm$ involves a fixed auxiliary choice $\eta_2$ of Dirichlet character of prime-to-$p$ conductor as in \S \ref{EisensteinClassesSSec}; we assume that $\eta_2 \omega_{\pi}$ is not congruent modulo $p$ to any character of $p$-power conductor.

Since $\pi_p$ is self-dual, $\pi^\vee$ also satisfies Assumption \ref{SteinbergAssumptionOnpi}; and by \cite[\S10]{LW21}, we have a $p$-adic functional equation. To explain this, let $\iota$ be the involution on measures induced by precomposition with the map $x \mapsto x^{-1}$ on $\Zp^\times$, let $N_\pi^{(p)}$ be the prime-to-$p$ part of the conductor of $\pi$, let $[N]$ be the Dirac measure given by evaluation at $N$, and let $\varepsilon(\pi_f^{\vee,p},s)$ be the finite prime-to-$p$ part of the epsilon factor of $\pi^\vee$. Then 
\begin{equation}\label{eq:p-adic functional equation}
\sL_p^+(\pi) = \varepsilon\big(\pi_f^{\vee,p},0\big)^{-1} \cdot \iota\Big([N_\pi^{(p)}] \cdot \sL_p^-(\pi^\vee)\Big).
\end{equation}

\begin{remark}
We have normalised $\sL_p^+$ slightly differently to \cite{LW21}, omitting the $\opn{tw}_1$ so as to give a uniform treatment of exceptional zeros. In particular our $\sL_p^+$ is not the one from  \cite[Conj.\ B]{LW21}, which has an exceptional zero at the character $x \mapsto x$ (rather than the trivial character).
\end{remark}

\begin{remark}
The measure $\sL_p^-(\pi)$ exists for any $P_1$-nearly-ordinary $\pi$, whilst $\sL_p^+(\pi)$ exists for any $P_2$-ordinary $\pi$. In particular these two $p$-adic $L$-functions correspond exactly to the two maximal standard parabolics in $\GL_3$. This observation provides a conceptual connection between a $p$-adic $L$-function and its corresponding $\mathcal{L}$-invariant.
\end{remark}

For $s \in \Z_p$, let $\langle - \rangle^s = \exp_p(s\cdot \log_p(\langle-\rangle)) \colon \Z_p^{\times} \to \Q_p^{\times}$, where $\langle-\rangle$ is projection to $1+2p\Z_p$ (via the splitting given by the Teichm\"{u}ller character). Define functions $L_p^\pm(\pi,-) : \Zp\to L$ by
\[
L_p^-(\pi,s) \defeq \int_{\Zp^\times} \langle x\rangle^s \cdot d\sL_p^\pm(\pi), \qquad L_p^+(\pi,s) \defeq \int_{\Zp^\times} \langle x\rangle^{s-1} \cdot d\sL_p^\pm(\pi).
\]
From the descriptions of $e_p^\pm$ above, the following is immediate. 

\begin{lemma}
We have $e_p^-(\pi_p,0) = e_p^+(\pi_p,1) = 0$. In particular, $L_p^-(\pi,0) = L_p^+(\pi,1) = 0$.
\end{lemma}

\subsection{Automorphic \texorpdfstring{$\mathcal{L}$}{L}-invariants for $G$}\label{sec:define L-invariants}

 We now recall Gehrmann's automorphic $\mathcal{L}$-invariants for $\pi$, one for each simple root/maximal standard parabolic of $\GL_3$. These are subspaces of $\opn{Hom}_{\opn{cts}}(\Qp^\times,L)$. We then prove a duality between these two $\mathcal{L}$-invariants.

\subsubsection{Definition of \texorpdfstring{$\mathcal{L}$}{L}-invariants}\label{sec:def of L-invariants}

Recall in \S\ref{sec:extensions of steinberg} that, for any $\lambda \in \opn{Hom}_{\opn{cts}}(\mbb{Q}_p^{\times}, R)$, we constructed an extension
\[
\mathcal{E}^{\bullet}_{1, \lambda} \in \opn{H}^1(G(\mbb{Q}_p), \opn{St}_1^{\bullet}(R))
\]
of the Steinberg representation for $P_1$, represented by a cocycle $c_{1,\lambda}^\bullet \in \Z^1(G(\Qp),\opn{St}_1^{\bullet}(R))$. Reversing the roles of $P_1$ and $P_2$, to $\lambda$ we can similarly define $\mathcal{E}^{\bullet}_{2, \lambda} \in \opn{H}^1(G(\mbb{Q}_p), \opn{St}_2^{\bullet}(R))$ with associated cocycle\footnote{As before, this cocycle depends on a choice that we will not yet specify. In Lemma \ref{lem:explicit c2}, we make an explicit choice of $c_{2,\lambda}^\bullet$. Regardless, its class in cohomology is independent of this choice, so the $\mathcal{L}$-invariants here are well-defined.} $c_{2,\lambda}^\bullet \in Z^1(G(\Qp),\opn{St}_2^{\bullet}(R))$.

\begin{definition}
\begin{itemize}
\item[(i)] Let $\Delta_{i, \lambda}^{\bullet} \in \opn{H}^1(G(\mbb{Q}_p), \opn{Hom}_R(\opn{St}_{3-i}^{\opn{sm}}(R), \opn{St}^{\bullet}(R)))$ denote the cohomology class represented by the cocycle
\[
x \mapsto \left\{ \begin{array}{cc} f \mapsto \opn{pr}(c_{1,\lambda}^\bullet[x] \otimes f) & \text{ if } i=1 \\ f \mapsto \opn{pr}(f \otimes c_{2,\lambda}^\bullet[x]) & \text{ if } i=2. \end{array} \right. 
\]

\item[(ii)]    Let 
    \[
    \Upsilon^{\opn{la}}_{i,\lambda} \colon \opn{H}^0(G(\mbb{Q}), \mathcal{A}_{G, c}(\opn{St}^{\opn{la}}(L)^*)) \to \opn{H}^1(G(\mbb{Q}), \mathcal{A}_{G, c}(\opn{St}_{3-i}^{\opn{sm}}(L)^*))
    \]
    denote the morphism $\mu \mapsto \mu \smile \Delta^{\opn{la}}_{i, \lambda}$ using the natural map 
    \[
    \mathcal{A}_{G, c}(\opn{St}^{\opn{la}}(L)^*) \otimes_L \opn{Hom}_L(\opn{St}^{\opn{sm}}_{3-i}(L), \opn{St}^{\opn{la}}(L)) \to \mathcal{A}_{G, c}(\opn{St}_{3-i}^{\opn{sm}}(L)^*) .
    \]
\end{itemize}
\end{definition}

\begin{remark}This recovers $\Delta_{1,\lambda}$ from \eqref{eq:Delta_1}. We used $\Upsilon_{1,\lambda}^{\opn{la}}$ in \S\ref{sec:abstract L-invariants}, where it was denoted $\Upsilon_{\lambda}^{\opn{la}}$.
\end{remark}

Now let $\pi$ be any automorphic representation satisfying Assumption \ref{SteinbergAssumptionOnpi}; we define $\mathcal{L}$-invariants for $\pi$ by considering how $\Upsilon_{i,\lambda}^{\opn{la}}$ acts on the $\pi^p$-part of cohomology, as we now define.

\begin{definition}
Let $K^p \subset G(\mbb{A}^p_f)$ be a sufficiently small compact open subgroup, and $\mbb{T} = C_c(K^p \backslash G(\mbb{A}_f^p) / K^p, L)$ the prime-to-$p$ Hecke algebra of level $K^p$. Let $L$ be any characteristic zero field for which the Whittaker model $\mathcal{W}_{\psi}(\pi)$ has a model $\mathcal{W}_{\psi}(\pi, L)$ over $L$. For any $L[G(\Q_p)]$-representation $M$, let
\[
\opn{H}^i\left(G(\Q), \mathcal{A}_{G, c}(M) \right)[\pi^p] \defeq \opn{Hom}_{\mbb{T}}\left( \mathcal{W}_{\psi}(\pi_f^p, L)^{K^p}, \opn{H}^i\left(G(\Q), \mathcal{A}_{G, c}(M) \right) \right).
\]
\end{definition}

    Note that $\Upsilon^{\opn{la}}_{i,\lambda}$ is prime-to-$p$ Hecke equivariant, so in particular induces a map
    \[
            \Upsilon_{i,\lambda}^{\opn{la}}[\pi] : \opn{H}^0(G(\mbb{Q}), \mathcal{A}_{G, c}(\opn{St}^{\opn{la}}(L)^*))[\pi^p] \to \opn{H}^1(G(\mbb{Q}), \mathcal{A}_{G, c}(\opn{St}_{3-i}^{\opn{sm}}(L)^*))[\pi^p].
    \]
    By Proposition \ref{Prop:SmLaCtsIsoTree}, we can view $\opn{H}^0(G(\mbb{Q}), \mathcal{A}_{G, c}(\opn{St}^{\opn{cts}}(L)^*))[\pi^p]$ as a subspace of the module $\opn{H}^0(G(\mbb{Q}), \mathcal{A}_{G, c}(\opn{St}^{\opn{la}}(L)^*))[\pi^p]$. We let $\Upsilon^{\opn{cts}}_{i, \lambda}[\pi]$ denote the restriction of $\Upsilon^{\opn{la}}_{i, \lambda}[\pi]$ to this submodule.

    \begin{definition}\label{def:L-invariant}
    We define
    \[
    \bL^{\opn{Aut}}_i(\pi) \subset \opn{Hom}_{\opn{cts}}(\mbb{Q}_p^{\times}, L)
    \]
    to be the subspace of $\lambda \in \opn{Hom}_{\opn{cts}}(\mbb{Q}_p^{\times}, L)$ such that $\Upsilon_{i, \lambda}^{\opn{cts}}[\pi] = 0$.
\end{definition}

\begin{remark}
    The appearance of $3-i$ in $\Upsilon^{\opn{la}}_{i, \lambda}$ is an artifact of our conventions. If we identify $i$ with the standard simple root $e_i - e_{i+1}$ for $\opn{GL}_3$ (i.e., $i=1$ corresponds to $(1, -1, 0)$ and $i=2$ corresponds to $(0, 1, -1)$), then the parabolics $P_1$ and $P_2$ would be denoted $P_{\{2\}}$ and $P_{\{1\}}$ in \cite{AutomorphicLinvariants}. The shift $3-i$ is included so that indexing in $\mbb{L}_i^{\opn{Aut}}(\pi)$ matches with that \emph{op.cit.}.
\end{remark}

We can recover $\mathcal{L}$-invariants as scalars via:

\begin{proposition}[{c.f., \cite[Proposition 3.14]{AutomorphicLinvariants}}] \label{Prop:AutLinvSubspaceIsOneDim}
    The subspace $\bL^{\opn{Aut}}_i(\pi) \subset \opn{Hom}_{\opn{cts}}(\Q_p^{\times}, L)$ has codimension one (and hence is one-dimensional), and $\opn{ord}_p \not\in \bL^{\opn{Aut}}_i(\pi)$.
\end{proposition}
\begin{proof}
    Note that $\opn{H}^0\left(G(\Q), \mathcal{A}_{G, c}(\opn{St}^{\opn{cts}}(L)^*) \right)[\pi^p]$ is one-dimensional by Lemma \ref{Lem:Pip-isotypic-piece1dim}, Proposition \ref{prop:lift to tree}, and Proposition \ref{Prop:SmLaCtsIsoTree} below. Then \cite[Proposition 3.14]{AutomorphicLinvariants} (in combination with Proposition \ref{AbstractExtFactorsThroughLAExtProp}) implies that $\opn{ord}_p \not\in \bL^{\opn{Aut}}_i(\pi)$, and the morphism $\Upsilon^{\opn{cts}}_{i, \opn{ord}_p}[\pi]$ is an isomorphism. This proves the claim. Note that, although we are working with the ``compactly-supported'' version of $p$-arithmetic cohomology, the cited results still apply, as justified in \cite[\S 3.6]{AutomorphicLinvariants}.
\end{proof}

\begin{definition}\label{def:L-invariant scalar}
The \emph{automorphic $\mathcal{L}$-invariant} $\mathcal{L}^{\opn{Aut}}_{\pi,i}\in L$ is the unique quantity such that 
\[
\opn{log}_p - \mathcal{L}^{\opn{Aut}}_{\pi,i} \cdot \opn{ord}_p \in \bL^{\opn{Aut}}_i(\pi).
\]
\end{definition}

\subsubsection{Duality between \texorpdfstring{$\mathcal{L}$}{L}-invariants}\label{sec:duality L-invariants}

We have the following relation between automorphic $\mathcal{L}$-invariants for $\pi$ and $\pi^{\vee}$. Its proof will occupy all of \S\ref{sec:duality L-invariants}.

\begin{proposition}\label{prop:duality L-invariants}
    Let $\pi$ be an automorphic representation satisfying Assumption \ref{SteinbergAssumptionOnpi}. Then $\pi^{\vee}$ satisfies Assumption \ref{SteinbergAssumptionOnpi} and 
    \[
    \bL^{\opn{Aut}}_2(\pi) = \bL_1^{\opn{Aut}}(\pi^{\vee}).
    \]
\end{proposition}

Consider the involution $\vartheta \colon G \to G$ given by $g^{\vartheta} = {^t g^{-1}}$ and assume, without loss of generality, that $K^p$ is stable under $\vartheta$. This involution induces a natural involution $\vartheta \colon \mbb{T} \to \mbb{T}$, where $\mbb{T}$ denotes the prime-to-$p$ Hecke algebra of level $K^p$.

Away from $p$, we have the following (see \cite[\S2]{JPSS-RankinSelberg}).
\begin{lemma}
We have an isomorphism 
\[
\theta \colon \mathcal{W}_{\psi}(\pi_f^p, L) \xrightarrow{\sim} \mathcal{W}_{\psi^{-1}}((\pi_f^{\vee})^p, L)
\]
given by $f \mapsto [g \mapsto f(w_0 g^{\vartheta})]$, where $w_0$ denotes the longest Weyl element of $G$. This isomorphism is semilinear for the action of $G(\mbb{A}_f^p)$ with respect to the involution $\vartheta$.
\end{lemma}

At $p$, one can verify analogously that we have an isomorphism
\[
\theta \colon \opn{St}_1^{\bullet}(R) \xrightarrow{\sim} \opn{St}_2^{\bullet}(R)
\]
induced from $f \mapsto [g \mapsto f(w_0 g^{\vartheta})]$ (where $g \in G(\mbb{Q}_p)$, and $f \colon \overline{P}_1(\mbb{Q}_p) \backslash G(\mbb{Q}_p) \to R$). This $R$-linear isomorphism is semilinear for the action of $G(\mbb{Q}_p)$ via the involution $\vartheta$.

\begin{lemma}\label{lem:explicit c2}
    Let $c_{1,\lambda}^\bullet \in Z^1(G(\mbb{Q}_p), \opn{St}^{\bullet}_1(R))$ be the cocycle representing $\mathcal{E}^{\bullet}_{1, \lambda}$, and let $c_{2,\lambda}^\bullet \in Z^1(G(\mbb{Q}_p), \opn{St}_2^{\bullet}(R))$ be the cocycle
    \[
    c_{2,\lambda}^\bullet[x] \defeq \theta( c_{1,\lambda}^\bullet[x^{\vartheta}] ), \quad \quad x \in G(\mbb{Q}_p) .
    \]
    Then $c_{2,\lambda}^\bullet$ represents the extension $\mathcal{E}^{\bullet}_{2, \lambda}$.
\end{lemma}
\begin{proof}
    This is a direct computation involving the explicit description of the cocycles as 
    \[
    c^{\bullet}_{i,\lambda}[x](g) = (\lambda \circ v_i \circ s_i)(gx) - (\lambda \circ v_i \circ s_i)(g)
    \]
    where $v_2\smat{a & b & \\ c & d & \\ & & e} = \opn{det}\smat{a & b \\ c & d}^{-1} e$, and $s_i \colon G(\mbb{Q}_p) \to \overline{P}_i(\mbb{Q}_p)$ is a continuous, left $\overline{P}_i(\mbb{Q}_p)$-equivariant map. More precisely, once we have fixed a choice of $s_1$, we can take $s_2(g) \defeq w_0 s_1(w_0 g^{\vartheta})^{\vartheta} w_0^{-1}$. Note that $v_2(w_0 p^{\vartheta} w_0^{-1}) = v_1(p)$ for $p \in \overline{P}_1(\mbb{Q}_p)$, hence we have $(v_2 \circ s_2)(g) = (v_1 \circ s_1)(w_0 g^{\vartheta})$. Then we compute:
    \[
    c_{1,\lambda}^\bullet[x^{\vartheta}](w_0 g^{\vartheta}) = (\lambda \circ v_1 \circ s_1)(w_0 (gx)^{\vartheta}) - (\lambda \circ v_1 \circ s_1)(w_0 g^{\vartheta}) = (\lambda \circ v_2 \circ s_2)(gx) - (\lambda \circ v_2 \circ s_2)(g)
    \]
    and the lemma follows.
\end{proof}

We now give the analogous relationship between the classes $\Delta_{i,\lambda}^\bullet$ used in the definition of $\mathcal{L}$-invariants. Similarly to above, we have a natural $R$-linear isomorphism
\[
\theta \colon \opn{St}^{\bullet}(R) \xrightarrow{\sim} \opn{St}^{\bullet}(R)
\]
induced from $f \mapsto [g \mapsto f(w_0 g^{\vartheta})]$ which is semilinear for the action of $G(\mbb{Q}_p)$ via $\vartheta$. Similarly, we have a $\vartheta$-semilinear isomorphism
\[
\Theta \colon \opn{Hom}_R(\opn{St}_2^{\opn{sm}}(R), \opn{St}^{\bullet}(R)) \xrightarrow{\sim} \opn{Hom}_R(\opn{St}_1^{\opn{sm}}(R), \opn{St}^{\bullet}(R))
\]
given by $\Theta(-) = \theta \circ (-) \circ \theta$. It is simple to check that $\Delta^{\bullet}_{1, \lambda}$ is mapped to $\Delta^{\bullet}_{2, \lambda}$ under the natural map 
\[
\opn{H}^1(G(\mbb{Q}_p), \opn{Hom}_R(\opn{St}_{2}^{\opn{sm}}(R), \opn{St}^{\bullet}(R))) \to \opn{H}^1(G(\mbb{Q}_p), \opn{Hom}_R(\opn{St}_{1}^{\opn{sm}}(R), \opn{St}^{\bullet}(R)))
\]
given by $c \mapsto [x \mapsto \Theta(c[x^{\vartheta}])]$.

Recall an element $\xi \in \mathcal{A}_{G, c}(M)$ is described as 
\[
\xi[d] \colon \pi_0(G(\mbb{R})) \times G(\mbb{A}_f^p)/K^p \to M
\]
where $d \in D_G$. For $M = \opn{St}^{\opn{la}}(L)^*$ or $M_i = \opn{St}_{i}(L)^*$, we consider the following isomorphism:
\[
\Theta_{\mathcal{A}} \colon \mathcal{A}_{G, c}(M_{(2)}) \xrightarrow{\sim} \mathcal{A}_{G, c}(M_{(1)})
\]
given by 
\[
\Theta_{\mathcal{A}}(\xi)[d]([g_{\infty}], g^p) = \theta^*(\xi[d^{\vartheta}]([g^{\vartheta}_{\infty}], (g^p)^{\vartheta} ))
\]
which is semilinear for the action of $G(\mbb{Q})$ via $\vartheta$. Here the notation $M_{(i)}$ means we either omit or include the subscript in both the source and target. Note that $\vartheta$ indeed induces a natural semilinear involution on $D_G$ because $\vartheta$ sends proper maximal parabolics to proper maximal parabolics (recall the definition of $D_G$ in \cite[\S 3.6]{AutomorphicLinvariants} for example). 

We have a commutative diagram:

\[
\begin{tikzcd}
{\opn{H}^0(G(\mbb{Q}), \mathcal{A}_{G, c}(\opn{St}^{\opn{la}}(L)^*))} \arrow[d] \arrow[r, "{\Upsilon^{\opn{la}}_{1, \lambda}}"] & {\opn{H}^1(G(\mbb{Q}), \mathcal{A}_{G, c}(\opn{St}_{2}^{\opn{sm}}(L)^*))} \arrow[d] \\
{\opn{H}^0(G(\mbb{Q}), \mathcal{A}_{G, c}(\opn{St}^{\opn{la}}(L)^*))} \arrow[r, "{\Upsilon^{\opn{la}}_{2, \lambda}}"]           & {\opn{H}^1(G(\mbb{Q}), \mathcal{A}_{G, c}(\opn{St}_{1}^{\opn{sm}}(L)^*))}          
\end{tikzcd}
\]
where the vertical maps are induced from $\Theta_{\mathcal{A}}$ (the right-hand vertical map sends a cocycle $c$ to $x \mapsto \Theta_{\mathcal{A}}(c[x^{\vartheta}])$). The vertical maps are isomorphisms and $\vartheta$-semilinear for the action of $\mbb{T}$.

We are now in a position to prove the main result of this subsection.

\bigskip

\noindent \emph{Proof of Proposition \ref{prop:duality L-invariants}.}
    From the previous section, it suffices to show that 
    \[
    \Theta_{\mathcal{A}} \colon \opn{H}^0(G(\mbb{Q}), \mathcal{A}_{G, c}(\opn{St}^{\opn{la}}(L)^*)) \xrightarrow{\sim} \opn{H}^0(G(\mbb{Q}), \mathcal{A}_{G, c}(\opn{St}^{\opn{la}}(L)^*))
    \]
    induces an isomorphism
    \[
    \opn{H}^0(G(\mbb{Q}), \mathcal{A}_{G, c}(\opn{St}^{\opn{la}}(L)^*))[(\pi^{\vee})^p] \xrightarrow{\sim} \opn{H}^0(G(\mbb{Q}), \mathcal{A}_{G, c}(\opn{St}^{\opn{la}}(L)^*))[\pi^p] ,
    \]
    and similarly for the continuous version and the right-hand vertical map in the above commutative diagram. But the isomorphism 
    \[
    \opn{Hom}_{\mbb{T}}(\mathcal{W}_{\psi^{-1}}((\pi^{\vee}_f)^p, L), \opn{H}^0(G(\mbb{Q}), \mathcal{A}_{G, c}(\opn{St}^{\opn{la}}(L)^*)) \xrightarrow{\sim} \opn{Hom}_{\mbb{T}}(\mathcal{W}_{\psi}(\pi_f^p, L), \opn{H}^0(G(\mbb{Q}), \mathcal{A}_{G, c}(\opn{St}^{\opn{la}}(L)^*))
    \]
    is given by $(-) \mapsto \Theta_{\mathcal{A}} \circ (-) \circ \theta$. \qed

\section{The automorphic exceptional zero formula}\label{sec:automorphic EZF}

Let $\pi$ be an automorphic representation satisfying Assumption \ref{SteinbergAssumptionOnpi}. We now prove our first main result: the automorphic exceptional zero formula for $\pi$.

\subsection{Cohomology classes for $\pi$}\label{ss: L invariants}

We now attach $p$-arithmetic cohomology classes to cusp forms in the Whittaker model for $\pi$. Let $K^p$ be any sufficiently small prime-to-$p$ level.

\begin{lemma} \label{Lem:Pip-isotypic-piece1dim}
    Let $K_{\infty}^{\circ} \subset G(\mbb{R})$ denote the connected component of the identity of the maximal compact-mod-centre subgroup of $G(\mbb{R})$. We have an isomorphism
    \[
    \opn{H}^2(\mathfrak{g}, K_{\infty}^{\circ}; \pi_{\infty})\otimes \pi_p^{\Iw} \cong \opn{H}^0\left(G(\Q), \mathcal{A}_{G, c}(\opn{Ind}_{\opn{Iw}}^{G(\Q_p)}\mbb{C}) \right)[\pi^p],
    \]
    and both sides are one-dimensional over $\mbb{C}$.
\end{lemma}

\begin{proof}
Proposition \ref{prop:arithmetic to betti} yields a Hecke-equivariant isomorphism between the $\h^0(-)$ in the right-hand side and $\hc{2}(X_{G,K^p\Iw},\C)$. By e.g.\ \cite{Clo90}, we have
\[
\hc{2}(X_{G,K^p\Iw},\C)[\pi^p] \cong \h_{\opn{cusp}}^2(X_{G,K^p\Iw},\C)[\pi^p].
\]
Moreover considering the standard description of cuspidal cohomology, and by strong multiplicity one, we have 
\[
    \h_{\opn{cusp}}^2(X_{G,K^p\Iw},\C)[\pi^p] \cong \h^2(\fg,K_\infty^\circ; \pi_\infty) \otimes \pi_p^{\Iw} \otimes (\pi_f^p)^{K^p}[\pi^p].
\]
Now $(\pi_f^p)^{K^p}[\pi^p]$ is a line, generated by the inverse of any intertwining into the Whittaker model. Choosing one, and chaining together the isomorphisms above, yields the claimed isomorphism. 

The final statement follows since both constituents of the tensor product are 1-dimensional.
\end{proof}

Now:
\begin{itemize}
\item let $\zeta_\infty \in \opn{H}^2(\mathfrak{g}, K_{\infty}^{\circ}; \pi_{\infty})$ denote a fixed generator; and

\item let $\varphi_p \in \pi_p$ denote the vector which is a scalar multiple of the indicator function $\opn{ch}(\overline{B}(\Q_p) \opn{Iw}) \in \opn{St}^{\opn{sm}}(\mbb{C})$ under the isomorphism in Assumption \ref{SteinbergAssumptionOnpi}, and satisfies $W_p(1) = W_{\varphi_p}(1) = 1$.
\end{itemize}
Under the isomorphism in Lemma \ref{Lem:Pip-isotypic-piece1dim}, this yields an element $\phi \in \h^0(G(\Q),\cA_{G,c}(\opn{Ind}_{\Iw}^{G(\Qp)}\C))[\pi^p]$, that is, a Hecke equivariant map 
\[
\phi \colon \mathcal{W}_{\psi}(\pi^p_f)^{K^p} \to \opn{H}^0\left(G(\Q), \mathcal{A}_{G, c}(\opn{Ind}_{\opn{Iw}}^{G(\Q_p)}\mbb{C}) \right).
\]
Let $E$ be the Hecke field of $\pi$, which is a number field. We then have an $E$-rational structure $\cW_\psi(\pi_f,E) \subset \cW_\psi(\pi_f)$, and we let $\Theta_{\pi} \in \C^\times$ be the Whittaker period as in \cite[\S 3.3.1]{LW21}; this is well-defined up to $E^\times$, and 
\[
\Theta_\pi^{-1} \cdot \phi\Big(\cW_{\psi}(\pi_f^p,E)^{K^p}\Big) \subset \h^0(G(\Q),\cA_{G,c}(\opn{Ind}_{\Iw}^{G(\Qp)}E)).
\]
If $L$ is a finite extension of $\Qp$ containing $E$, then we can naturally consider the image of $\Theta_\pi^{-1} \cdot \phi$ (on the $E$-rational Whittaker functions) with $L$-coefficients. 

Recall the map $\rho_{\Iw}^0$ from \S\ref{sec:steinberg tree}.  As $\pi_p$ is Steinberg, we have the following key result:

\begin{proposition}\label{prop:lift to tree}
    Post-composition with $\rho_{\Iw}^0$ induces an isomorphism
    \begin{equation} \label{LiftToTreeArithCohEqn}
    \rho_{\Iw}^0 : \opn{H}^0\left(G(\Q), \mathcal{A}_{G, c}(\opn{St}^{\opn{cts}}(L)^*) \right)[\pi^p] \to \opn{H}^0\left(G(\Q), \mathcal{A}_{G, c}(\opn{Ind}_{\opn{Iw}}^{G(\Q_p)}L) \right)[\pi^p].
    \end{equation}
\end{proposition}
\begin{proof}
    This is \cite[Proposition 3.6(b)]{AutomorphicLinvariants}, using that $\pi_p$ is Steinberg.
\end{proof}

\begin{proposition} \label{Prop:SmLaCtsIsoTree}
The natural maps $\opn{St}^{\opn{cts}}(L)^* \to \opn{St}^{\opn{la}}(L)^* \to \opn{St}^{\opn{sm}}(L)^*$ induce morphisms
\begin{equation} \label{CtsToSmoothIsoEqn}
\opn{H}^0\left(G(\Q), \mathcal{A}_{G, c}(\opn{St}^{\opn{cts}}(L)^*) \right) \hookrightarrow \opn{H}^0(G(\Q), \mathcal{A}_{G, c}(\opn{St}^{\opn{la}}(L)^*)) \twoheadrightarrow  \opn{H}^0\left(G(\Q), \mathcal{A}_{G, c}(\opn{St}^{\opn{sm}}(L)^*) \right),
\end{equation}
and the composition is an isomorphism.
\end{proposition}
\begin{proof}
    This follows from (the compactly-supported version of) \cite[Proposition 3.11]{AutomorphicLinvariants}.
\end{proof}

Combining these propositions with the discussion above, for $\bullet \in \{\opn{sm}, \opn{cts}\}$ we have 
\[
    \dim_L \opn{H}^0\left(G(\Q), \mathcal{A}_{G, c}(\opn{St}^{\bullet}(L)^*) \right)[\pi^p] = 1,
\]
and we may lift $\Theta_\pi^{-1} \cdot \phi$ uniquely to a Hecke-equivariant function
\[
\phi^{\bullet} : \cW_\psi(\pi_f^p, E)^{K^p} \longrightarrow \opn{H}^0\left(G(\Q), \mathcal{A}_{G, c}(\opn{St}^{\bullet}(L)^*) \right).
\]
(This depends on $\Theta_\pi$, but we suppress this). We let $\phi^{\opn{la}}$ denote the composition of $\phi^{\opn{cts}}$ with the first map in (\ref{CtsToSmoothIsoEqn}).

\subsection{Local zeta integrals and test data}\label{sec:test data}
As input to the map $\phi^{\opn{cts}}$, we take test data for the following local zeta integrals.

\begin{definition}
For any Schwartz function $\Phi_v \in \mathcal{S}(\Q_v^{2})$, smooth character $\chi \colon \Q_v^{\times} \to \mbb{C}^{\times}$, and $s \in \mbb{C}$, we let $W_{\Phi_v}(-; \chi, s)$ denote the ($\GL_2$) Whittaker function as in \cite[\S 8.1]{LW21}. For $s_1, s_2 \in \mbb{C}$, $W \in \mathcal{W}_{\psi_v}(\pi_v)$, $\Phi_v \in \mathcal{S}(\Q_v^2)$ and smooth characters $\chi_1, \chi_2 \colon \Q_v^{\times} \to \mbb{C}^{\times}$, we define
\[
Z(W, \Phi_v; \chi_1, \chi_2, s_1, s_2) = \int_{\left(N_2 \backslash \opn{GL}_2\right)(\Q_v)} W(\iota(g, 1)) W_{\Phi_v}(g; \chi_2, s_2) \chi_1(\opn{det}g)|\opn{det}g|^{s_1 - \tfrac{1}{2}} dg .
\]
\end{definition}

This factorises as 
\[
Z(W, \Phi_v; \chi_1, \chi_2, s_1, s_2) = \tilde{Z}(W, \Phi_v; \chi_1, \chi_2, s_1, s_2) \cdot L(\pi \times \chi_1, s_1+s_2-1/2) L(\pi \times \chi_1\chi_2^{-1}, s_1-s_2+1/2)
\]
where $\tilde{Z}(W, \Phi_v; \chi_1, \chi_2, s_1, s_2)$ is entire in both complex variables (see \cite[Theorem 8.2]{LW21}).

\begin{test-data} \label{AssumptionsOnTestVectors}
    We fix the following local test data:
    \begin{itemize}
        \item If $v \nmid p \infty$, by \cite{JPSS-RankinSelberg} there exists a finite set $I_v$ and sets $(W_{v,i})_{i\in I_v} \subset \cW_{\psi}(\pi_v,E)$ and $(\Phi_{v,i})_{i\in I_v} \subset {}_c\cS(\Q_v^2,\Z)$  such that 
        \[
        \sum_i \tilde{Z}(W_{v,i}, \Phi_v; 1, \omega_{\pi, v} \eta_{2, v}, s_1, s_2) = 1.
        \]

        \item If $v \nmid p\infty$ and $\pi_v$ and $\eta_{2,v}$ are unramified, we may take $I_v = \{1\}$ a singleton, and take $W_{v,1}$ and $\Phi_{v,1}$ to be the normalised spherical data.  

        \item If $v = \infty$, we let $\Phi_{\infty}(x, y) = 2^{-1}(x+iy)^2\opn{exp}(-\pi(x^2+y^2)) \in \cS_0(\R^2)$.
        \item If $v = p$, we let $\varphi_p \in \pi_p$ denote the vector which is a scalar multiple of the indicator function $\opn{ch}(\overline{B}(\Q_p) \opn{Iw}) \in \opn{St}^{\opn{sm}}(\mbb{C})$ under the isomorphism in Assumption \ref{SteinbergAssumptionOnpi}, and satisfies $W_p(1) = W_{\varphi_p}(1) = 1$. We let $\Phi_p = \opn{ch}(p^m \Z_p, 1+p^m \Z_p) \in \cS(\Qp^2)$ for an integer $m \geq 1$ (to be specified later).
    \end{itemize}
\end{test-data}

We suppose $K^p$ fixes all of this test data. Let $I = \prod I_v$, a finite set (as all but finitely many $I_v$ are singletons). For each $i = (i_v) \in I$ we then obtain an element $W^p_i \defeq \otimes_{v\nmid p\infty} W_{v,i_v} \in \cW_{\psi}(\pi_f^p,E)^{K^p}$ and a prime-to-$p$ Schwartz function $\Phi^{(p)}_i \defeq \otimes_v\Phi_{v,i_v}$. 

For each $i \in I$, let $\mu^{\opn{cts}}_i \defeq \phi^{\opn{cts}}(W^p_i) \in \h^0(G(\Q),\cA_{G,c}(\opn{St}^{\opn{cts}}(L)^*))$. Let  $\mu_i$ be its image under the natural map $\opn{St}^{\opn{cts}}(L)^* \to \opn{St}^{\opn{sm}}(L)^*$, and likewise $\mu_i^{\opn{la}}$ under $\opn{St}^{\opn{cts}}(L)^* \to \opn{St}^{\opn{la}}(L)^*$. Note in all cases $\mu_i^\bullet = \phi^{\bullet}(W_i^p)$. 

Similarly to \cite[\S3.3.3]{LW21}, we may rescale $\Theta_\pi$ to ensure that
\[
\mu^{\opn{cts}}_i \in \h^0(G(\Q),\cA_{G,c}(\opn{St}^{\opn{cts}}(\cO_L)^*)),
\]
hence also $\mu_i$, is integral. Then for each $i \in I$, we have a measure $\sL(\mu_i, -, \Phi^{(p)}_i)$ and a period $\sP(\mu_i,\Phi^{(p)}_i)$ as constructed in \S\ref{sec:measure} and \S\ref{sec:period} respectively. Note we now emphasise the dependence on the prime-to-$p$ Schwartz function, which we dropped from previous notation. 

We combine this data into a single measure and and single period attached to $\pi$:
\begin{definition}
\begin{itemize}
\item Let $\sL(\pi,-) \defeq \sum_{i\in I}\sL(\mu_i,-,\Phi_i^{(p)}) \in \cO_L[\![\Zp^\times]\!]$.
\item Let $\sP(\pi) \defeq \sum_{i\in I}\sP(\mu_i,\Phi_i^{(p)}) \in \cO_L$.
\end{itemize}
\end{definition}

\subsection{Abstract vs.\ automorphic $\mathcal{L}$-invariants}

The following relates the automorphic $\mathcal{L}$-invariants for $\pi$ to the abstract $\mathcal{L}$-invariants used in \S\ref{sec:abstract L-invariants}. We maintain the notation $\mu_i^\bullet$ for $i\in I$ from \S\ref{sec:test data}.

\begin{lemma}
For each $i \in I$, we have $[\Upsilon^{\opn{la}}_{\log_p - \mathcal{L}^{\opn{Aut}}_{\pi,1}\opn{ord}_p}(\mu_i^{\opn{la}})] = 0$. 
\end{lemma}

\begin{proof}
For the first, note by Definitions \ref{def:L-invariant} and \ref{def:L-invariant scalar} that $\Upsilon_{\log_p - \mathcal{L}^{\opn{Aut}}_{\pi,1}\opn{ord}_p}^{\opn{la}}[\pi]$ is zero on the generator $\phi^{\opn{cts}} \in \h^0(G(\Q),\cA_{G,c}(\opn{St}^{\opn{cts}}(L)^*))[\pi^p]$, hence zero on $\phi^{\opn{la}}$. As this map is induced by post-composition with $\Upsilon_{\log_p -\mathcal{L}^{\opn{Aut}}_{\pi,1}\opn{ord}_p}^{\opn{la}}$, it follows that 
\[
\Big[\Upsilon_{\log_p -\mathcal{L}^{\opn{Aut}}_{\pi,1}\opn{ord}_p}^{\opn{la}}(\phi^{\opn{la}}(W^p))\Big] = 0 \qquad \forall W^p \in \cW_\psi(\pi_f^p,E)^{K^p}.
\]
The result follows since $\mu_i^{\opn{la}}= \phi^{\opn{la}}(W_i^p)$.
\end{proof}

We immediately deduce the following exceptional zero formula:

\begin{proposition}\label{prop:EZF with automorphic}
We have
\[
\sL(\pi,\log_p) = \mathcal{L}^{\opn{Aut}}_{\pi,1} \cdot \sP(\pi).
\]
\end{proposition}

\begin{proof}
The lemma allows us to apply Theorem \ref{AbstractLinvTheorem} to $\mu_i^{\opn{cts}}$ with $\mathcal{L}_{\mu_i} = \mathcal{L}^{\opn{Aut}}_{\pi,1}$, which yields
\[
\sL(\mu_i, \log_p, \Phi^{(p)}_i) = \mathcal{L}^{\opn{Aut}}_{\pi,1} \cdot \sP(\mu_i, \Phi^{(p)}_i).
\]
Taking the sum over $i$ yields the claimed expression.
\end{proof}

\subsection{Comparison of measures}

We now directly compare the measure $\sL(\pi)$ to the $p$-adic $L$-function $\sL_p^-(\pi)$ from \cite{LW21} (as described in Theorem \ref{thm:LW21}). Let $dh$ denote the Tamagawa measure for $H$. We have the following comparison.

\begin{proposition} \label{LogpCalculationProp} There exists $C \in \Z$ such that
    \[
    \sL(\pi,-) = p^C \cdot \Big(c^2 - (\eta_2 \omega_{\pi})(c)^{-1}[c^{-1}]\Big) \cdot \opn{vol}(K^p_H; dh)^{-1} \cdot \sL_p^-(\pi, -) .
    \]
\end{proposition}
\begin{proof} 
Let $i \in I$. Combining Propositions \ref{prop:lift to tree} and \ref{prop:arithmetic to betti}(ii), attached to $\mu_i$ is a class $\phi_i \in \hc{2}(X_{G,K^p\Iw},L)$. There exists\footnote{The factor of $p^C$ accounts for any difference between the integral structures chosen in the present paper and in \cite{LW21}, which we do not attempt (or need) to carefully compare.} $C \in \Z$ such that each $p^C\phi_i$ has coefficients in $\cO_L$. Then $\sL_p^-(\pi)$ is constructed using the tame test data $\{p^C\phi_i, \Phi^{(p)}_i\}_i$, as explained in \S9 \emph{op.\ cit}. The construction there also uses the same choices at infinity and $p$, so to compare the measures $\sL(\pi)$ and $\sL_p^-(\pi)$, it suffices to compare the methods of construction. The compatibility of the two constructions was described in the proof of Proposition \ref{prop:measure}.

To explain the additional terms, note when constructing $\sL(\pi)$ we have not included the volume term from \cite[Prop.\ 6.11]{LW21}, hence the appearance of $\opn{vol}(K^p_H;dh)$ here. Finally, the construction of $\sL(\pi,-)$ uses $c$-smoothed Eisenstein series, differing from the non-$c$-smoothed version by the stated factor in $c$, as explained in \cite[\S6.3]{LW21}.
\end{proof}

For $s \in \Zp$, let $\sL(\pi,s) \defeq \int_{\Zp^\times} \langle x\rangle^s \cdot d\sL(\pi)$. The decomposition in Proposition \ref{LogpCalculationProp} can be evaluated multiplicatively on any character, hence on $\langle x\rangle^s$, so
\[
\sL(\pi,s) = p^C \cdot \left(c^2 - (\eta_2 \omega_{\pi})(c)^{-1}c^{-s}\right) \cdot \opn{vol}(K^p_H; dh)^{-1} \cdot L_p^-(\pi, s).
\]

\begin{corollary}\label{cor:derivative}
We have
\begin{align*}
\sL(\pi,\log_p) &= \left.\tfrac{d}{ds}\sL(\pi,s)\right|_{s=0}\\
&= p^C \cdot \Big(c^2 - (\eta_2 \omega_{\pi})(c)^{-1})\Big) \cdot \opn{vol}(K^p_H; dh)^{-1} \cdot \left.\tfrac{d}{ds}L_p^-(\pi,s)\right|_{s=0}.
\end{align*}
\end{corollary}
\begin{proof}
As $\langle x\rangle^s = \opn{exp}_p(s \cdot \log_p(\langle x\rangle))$, we see the first derivative at $s=0$ is computed by integrating $\log_p(x) = \log_p(\langle x\rangle)$. The second equality is then the chain rule, recalling that $L_p(\pi, 0) = 0$.
\end{proof}

\subsection{Computation of the period}

The goal of this section is to prove the following:

\begin{proposition} \label{OrdCalculationProp}
   We have
    \[
    \sP(\pi) = p^C (c^2 - (\eta_2\omega_{\pi})(c)^{-1}) \opn{vol}(K^p_H; dh)^{-1} \cdot e_{\infty}(\pi, 0) \cdot \frac{L(\pi, 0)}{\Omega_{\pi}^-} .
    \]
\end{proposition}

We will prove this by first relating the summands $\sP(\mu_i,\Phi^{(p)}_i)$ to a Rankin--Selberg automorphic period, and then computing this period via local zeta integrals. From now on, we let $m \geq 1$ be an arbitrary integer, and take $\Phi_p = \Phi_p^m$.

Let $\{ \delta_1, \delta_2 \}$ and $\{ \delta_1', \dots, \delta_5' \}$ denote bases of $\left( \ide{gl}_2/\ide{K}_{2, \infty}^{\circ} \right)^{\vee}$ and $\left( \ide{gl}_3/\ide{K}_{3, \infty}^{\circ} \right)^{\vee}$ respectively, satisfying the compatibility as in \cite[\S 7.2.3]{LW21}. By the presentation of Lie algebra cohomology using the Chevalley--Eilenberg complex, we can view 
\[
\opn{H}^2\left( \ide{gl}_3, K_{3, \infty}^{\circ}; \pi_{\infty} \right) \subset \bigwedge {\!}^2 \left( \ide{gl}_3/\ide{K}_{3, \infty}^{\circ} \right)^{\vee} \otimes \pi_{\infty} .
\]
Write 
\[
\zeta_{\infty} = \sum_{1 \leq r < s \leq 5} [\delta'_r \wedge \delta'_s] \otimes \varphi_{\infty, r, s} 
\]
where $\zeta_{\infty}$ is the fixed generator of this one-dimensional cohomology group, and set $\varphi_{\infty} = \varphi_{\infty, 1, 3}- \varphi_{\infty, 2, 3} \in \pi_{\infty}$. Recall $\phi_p \in \pi_p^{\Iw}$ from \S\ref{ss: L invariants}, and let $\varphi^p_i$ be the pullback of $W^p_i$ under our fixed Whittaker intertwining. We have the following:

\begin{lemma} \label{PequalsAutPerLemma}
    One has
    \[
    \sP(\mu_i,\Phi^{(p)}_i) = p^C \opn{vol}(K^p_H U_1(p^m); dh)^{-1} \Theta_{\pi}^{-1} \int_{H(\Q)\backslash H(\mbb{A})/\mbb{R}_{>0}} \varphi_i(\iota(h)u_0) \cdot {_c\mathcal{E}}^2_{\Phi_i^{(p)}\Phi_p}(\opn{pr}_1(h)) \cdot \widehat{\eta}_2(\nu_2(h)) dh
    \]
    where $\varphi_i= \varphi_{\infty} \otimes \varphi_p \otimes \varphi_p^i$ and $\nu_2 : H \to \GL_1$ is projection to the second factor.
\end{lemma}
\begin{proof}
    Again we compare to \cite{LW21}. As in the proof of Proposition \ref{LogpCalculationProp}, attached to $\mu_i$ is a class $p^C\phi_i \in \hc{2}(X_{G,K^p\Iw},\cO_L)$. In the definition of $\sP(\mu_{i},\Phi^{(p)}_i)$, note that $\Psi_{\Iw,m}(\mu_i)$ is simply the restriction of $u_0 \cdot p^C\phi_i$ from $G$ to $H$. Then $\sP(\mu_{\varphi_f})$ is the pairing of this restriction with an Eisenstein class, namely, the pullback version of the pairing in \cite[\S7.2]{LW21}. The result now follows from exactly the same reasoning as \S7.2-7.3 \emph{op.\ cit}.
\end{proof}

We now relate this to a Rankin--Selberg period, which in turn can be related to a product of local zeta integrals. For ease of notation set $C' = c^2 - (\eta_2\omega_{\pi})(c^{-1})$. 

\begin{lemma} \label{AutPerEqualsRSproductLemma}
    One has
    \begin{align*} 
    \int_{H(\Q)\backslash H(\mbb{A})/\mbb{R}_{>0}} &\varphi_i(\iota(h)u_0)  \cdot {_c\mathcal{E}}^2_{\Phi_i^{(p)}\Phi_p}(\opn{pr}_1(h)) \cdot  \widehat{\eta}_2(\nu_2(h)) dh \\ &= C' \cdot Z(u_0 \cdot W_p, \Phi_p; 1, \omega_{\pi, p}\eta_{2, p}, 1/2, 0) \cdot \prod_{v \neq p} Z(W_{v,i_v}, \Phi_{v,i_v}; 1, \omega_{\pi, v}\eta_{2, v}, 1/2, 0).
    \end{align*}  
\end{lemma}
\begin{proof}
    This follows from \cite[\S 7.3]{LW21} \emph{mutatis mutandis}.  The factor $C'$ arises from comparing the $c$-smoothed Eisenstein series with the non-$c$-smoothed version.
\end{proof}

\begin{corollary}\label{cor:period = L-values}
After possibly rescaling $\zeta_\infty$, one has
\begin{align*}
\sP(\pi) = C' p^C \opn{vol}(K^p_H U_1(p^m); dh)^{-1} &\cdot Z(u_0 \cdot W_p, \Phi_p; 1, \omega_{\pi, p}\eta_{2, p}, 1/2, 0)\\
&\times e_{\infty}^-(\pi, 0) \cdot \frac{L^{(p)}(\pi, 0) \cdot L^{(p)}(\pi \times (\omega_{\pi}\eta_{2})^{-1}, 1)}{\Theta_\pi}.
\end{align*}
\end{corollary}
\begin{proof}
Summing over $i\in I$, we find from Assumption \ref{AssumptionsOnTestVectors} that all of the finite prime-to-$p$ zeta integrals reduce to the claimed $L$-values. The factor at infinity is computed in \cite[\S9.5]{LW21}. 
\end{proof}

From now on, we assume that $\zeta_{\infty}$ is normalised so that the conclusion of Corollary \ref{cor:period = L-values} holds. We now compute the local zeta integral at $p$.

\begin{lemma} \label{ZetaAtpCalcLemma}
    Set $\chi_2 = \omega_{\pi, p}\eta_{2, p}$. With notation as in Lemma \ref{AutPerEqualsRSproductLemma}, one has
    \[
    Z(u_0 \cdot W_p, \Phi_p; 1, \omega_{\pi, p}\eta_{2, p}, 1/2, 0) = \frac{1}{p^{2m}(1-p^{-2})} L(\pi_p, 0) \frac{\mathcal{E}_0 \cdot L(\pi_p \times \chi_2^{-1}, 1)}{\varepsilon(\pi_p \times \chi_2^{-1}, 1)} 
    \]
    for $m \geq 1$ sufficiently large, where $\mathcal{E}_0 = \frac{L(\pi_p \times \eta_{2, p}, 0)}{L(\pi_p^{\vee} \times \chi_2, 0)}$. Furthermore, $\mathcal{E}_0$ is a $p$-adic unit.
\end{lemma}
\begin{proof}
    We follow (and freely use the notation from) \cite[Theorem 8.7]{LW21}. We first note that
    \[
    Z(u_0 \cdot W_p, \Phi_p; 1, \chi_2, 1/2, 0) = \gamma(\pi_p \times \chi_2^{-1}, 1)^{-1} \cdot Y(u_0 \cdot W_p, \Phi_p; 1, \chi_2, 1/2, 0)
    \]
    by \cite[Theorem 8.4]{LW21}. For $m \geq 1$ sufficiently large, we have
    \[
    Y(u_0 \cdot W_p, \Phi_p; 1, \chi_2, 1/2, 0) = \frac{1}{p^{2m}(1-p^{-2})} y(u_0 \cdot W_p; 1, \chi_2; 1/2, 0)(1) 
    \]
    so it suffices to compute
    \[
    y(u_0 \cdot W_p; 1, \chi_2; s_1, s_2)(1) = \int_{(\Q_p^{\times})^2} W_p\smat{x & & \\ & 1 & \\ & & y^{-1}} |x|^{s_1+s_2-3/2}|y|^{s_2-s_1-1/2}\chi_2(y) d^{\times} x d^{\times}y
    \]
    where $W_p$ is the Iwahori fixed vector satisfying $W_p(1) = 1$. By \cite[Proposition 2.25]{LW21}, we have
    \begin{align*}
     y(u_0 \cdot W_p; 1, \chi_2; s_1, s_2)(1) &= \sum_{m,n \geq 0} \omega_{\pi,p}(p^{-n}) W_p\smat{p^{m+n} & & \\ & p^n & \\ & & 1} p^{-m(s_1+s_2-3/2)} p^{-n(s_2-s_1-1/2)} \chi_2(p^n) \\
     &= \sum_{m,n \geq 0} \eta_{2,p}(p^{n}) W_p\smat{p^{m+n} & & \\ & p^n & \\ & & 1} p^{-m(s_1+s_2-3/2)} p^{-n(s_2-s_1-1/2)} \\
     &= \sum_{m,n \geq 0} \eta_{2,p}(p^{n}) W_{\sigma}^{\opn{new}}\smat{p^n & \\ & 1} p^{-m(s_1+s_2+1/2)} p^{-n(s_2-s_1+1)} \\
     &= \left( \sum_{m \geq 0} p^{-m(s_1+s_2+1/2)} \right) \cdot \left( \sum_{n \geq 0} \eta_{2, p}(p^n) W_{\sigma}^{\opn{new}}\smat{p^n & \\ & 1} p^{-n(s_2-s_1+1)} \right)
    \end{align*}
    Here $W_{\sigma}^{\opn{new}}$ is the normalised new vector for $\sigma = \opn{St}_{\opn{GL}_2}^{\infty}(\mbb{C}) \otimes | \cdot |^{-1/2}$, where $\opn{St}_{\opn{GL}_2}^{\infty}(\mbb{C})$ denotes the Steinberg representation for $\opn{GL}_2(\Q_p)$. Using the standard formula for local $L$-factors in terms of Whittaker new vectors \cite[Prop.\ 6.17]{Gel75}, we have
    \begin{align*}
    y(u_0 \cdot W_p; 1, \chi_2; s_1, s_2)(1) &= \zeta_p(s_1+s_2+1/2) \cdot L(\opn{St}^{\infty}_{\opn{GL}_2}(\mbb{C}) \otimes \eta_{2, p}, s_2-s_1+1) \\
    &= L(|\cdot|_p, s_1+s_2-1/2) \cdot L(|\cdot |_p \otimes \eta_{2, p}, s_2-s_1+1/2) \\
    &= L(\pi_p, s_1+s_2-1/2) \cdot L(\pi_p \times \eta_{2, p}, s_2-s_1+1/2).
    \end{align*}
    We now use the fact that 
    \[
    \gamma(\pi_p \times \chi_2^{-1}, 1)^{-1} = \frac{L(\pi_p \times \chi_2^{-1}, 1)}{\varepsilon(\pi_p \times \chi_2^{-1}, 1)L(\pi_p^{\vee} \times \chi_2, 0)} .
    \]
    The last part follows from \cite[Lemma 8.8]{LW21}.
\end{proof}

\begin{proof}[{Proof of Proposition \ref{OrdCalculationProp}}]
    Combining Lemmas \ref{PequalsAutPerLemma}, \ref{AutPerEqualsRSproductLemma}, and \ref{ZetaAtpCalcLemma}, and using the definition of $\Omega_\pi^-$ from \cite[(9.1)]{LW21}, we see that
    \begin{align*}
    \sP(\pi) = p^C C' \opn{vol}(K^p_H U_1(p^m); dh)^{-1} \frac{1}{p^{2m}(1-p^{-2})} \cdot e_{\infty}^-(\pi, 0) \cdot \frac{L(\pi, 0)}{\Omega_{\pi}^-} .
    \end{align*} 
    Now we note that the volume of $U_1(p^m)$ satisfies
    \[
    \opn{vol}(U_1(p^m); dh)^{-1} = [\opn{GL}_2(\Z_p) : \smat{* & * \\ 0 & 1} \pmod{p^m}] = p^{m}(1-p^{-1}) \cdot p^m(1+p^{-1}) = p^{2m}(1-p^{-2}) .
    \]
\end{proof}

\subsection{The exceptional zero formulas}

The following is the first main theorem of this article. It is the direct analogue of Conjecture \ref{conj:EZC} using the automorphic $\mathcal{L}$-invariant.

\begin{theorem}[The left-half exceptional zero formula] \label{Thm:LeftHalfEZF}
    One has
    \[
    \left. \frac{d}{ds}L_p^-(\pi, s) \right|_{s=0} = e_{\infty}^-(\pi_{\infty}, 0) \cdot \mathcal{L}^{\opn{Aut}}_{\pi,1} \cdot \frac{L(\pi, 0)}{\Omega_{\pi}^-} .
    \]
\end{theorem}
\begin{proof}
    Choose $c$ such that $c^2 - (\omega_{\pi}\eta_2)(c^{-1}) \neq 0$ (which is always possible because we are assuming that $\omega_{\pi}\eta_2$ is not congruent modulo $p$ to any character of $p$-power conductor). The result now follows by combining Proposition \ref{prop:EZF with automorphic} with Corollary \ref{cor:derivative} and Proposition \ref{OrdCalculationProp}, which yields the claimed formula multiplied on both sides by the (non-zero) factor $p^C \cdot (c^2-(\omega_{\pi}\eta_2)(c^{-1}))\cdot \opn{vol}(K^p_H; \opn{dh})^{-1}$. Dividing through gives the result.
\end{proof}

We also have an analogue for the other $p$-adic $L$-function for $\pi$.

\begin{theorem}[The right-half exceptional zero formula] \label{Them:RightHalfEZF}
    One has
    \[
    \left. \frac{d}{ds}L_p^+(\pi, s) \right|_{s=1} = -p\cdot e_{\infty}^+(\pi_{\infty}, 1) \cdot \mathcal{L}^{\opn{Aut}}_{\pi,2} \cdot \frac{L(\pi, 1)}{\Omega_{\pi}^+} .
    \]
\end{theorem}
\begin{proof}
From the $p$-adic functional equation \eqref{eq:p-adic functional equation}, we have
\[
L_p^+(\pi,s) = \varepsilon\big(\pi_f^{\vee,p},0\big)^{-1} \cdot \langle N_\pi^{(p)}\rangle^{1-s}\int_{\Zp^\times} \langle x \rangle^{1-s}\cdot d\sL_p^-(\pi^\vee).
\]
Differentiating and evaluating at $s=1$, the chain rule gives two terms; but one of them contains $\int_{\Zp^\times} \mathbbm{1}\cdot d\sL_p^-(\pi)$, which vanishes (via the exceptional zero). We are left with
\[
\left. \frac{d}{ds}L_p^+(\pi, s) \right|_{s=1} = \varepsilon\big(\pi_f^{\vee,p},0\big)^{-1} \cdot \int_{\Zp^\times} -\log_p(x)\cdot d\sL_p^-(\pi^\vee).
\]
Bringing out the $-1$ and applying the left-half exceptional zero formula to $\pi^\vee$, we obtain
\[
\left. \frac{d}{ds}L_p^+(\pi, s) \right|_{s=1} = - \varepsilon(\pi_f^{\vee,p},0)^{-1} \cdot  e_\infty^-(\pi_{\infty}^\vee,0) \cdot \mathcal{L}^{\opn{Aut}}_{\pi^\vee,1} \cdot \frac{L(\pi^\vee, 0)}{\Omega_{\pi^\vee}^-}.
\]
We have functional equations $e_\infty^-(\pi_{\infty}^\vee,0)L(\pi^\vee,0) = \varepsilon(\pi_f^\vee,0)\cdot e_\infty^+(\pi_{\infty},1)L(\pi,1)$ (from \cite[(6)]{coates89}), $\Omega_{\pi^\vee}^- = \Omega_{\pi}^+$ (by definition), and $\mathcal{L}^{\opn{Aut}}_{1,\pi^\vee} = \mathcal{L}^{\opn{Aut}}_{\pi,2}$ (Proposition \ref{prop:duality L-invariants}). Combining all of this yields
\[
\left. \frac{d}{ds}L_p^+(\pi, s) \right|_{s=1} = - \varepsilon(\pi_p^\vee,0) \cdot  e_\infty^+(\pi_{\infty},1) \cdot \mathcal{L}^{\opn{Aut}}_{\pi,2} \cdot \frac{L(\pi,1)}{\Omega_{\pi}^+}.
\]
We conclude as $\varepsilon(\pi_p^\vee,0) = \varepsilon(\opn{St}_{\GL_3}(\C),0) = p$ by \eqref{eq:steinberg epsilon}.
\end{proof}

\subsection{Relation with completed cohomology} \label{RelationWithCCSSec}

We note that we have the following relation between $\mathbb{L}_i^{\opn{Aut}}(\pi)$ and completed cohomology of the locally symmetric space for $G$. Let $X_{G, K^pK_p}$ denote the locally symmetric space for $G = \opn{GL}_3$ of level $K^p K_p$. We define the compactly supported completed cohomology to be:
\[
\widetilde{\opn{H}}^*_c\left( X_{G, K^p}, L \right) \defeq \left( \varprojlim_n \varinjlim_{K_p} \opn{H}^*_c\left( X_{G, K^pK_p}, \mathcal{O}_L/\varpi^n \right) \right)[1/\varpi]
\]
which carries commuting actions of $\mbb{T}$ and $G(\Q_p)$. Let 
\[
\widetilde{\opn{H}}^2_c\left( X_{G, K^p}, L \right)[\pi^p] \defeq \opn{Hom}_{\mbb{T}}\left( \mathcal{W}_{\psi}(\pi_f^p, L)^{K^p}, \widetilde{\opn{H}}^2_c\left( X_{G, K^p}, L \right) \right)
\]
denote the prime-to-$p$ isotypic piece. Recall from Proposition \ref{AbstractExtFactorsThroughLAExtProp} that any continuous homomorphism $\lambda \colon \Q_p^{\times} \to L$ gives rise to an extension (in the category of admissible locally analytic representations)
\[
0 \to \opn{St}^{\opn{la}}(L) \to \mathcal{C}_{\Lambda_{\lambda}} \to \opn{St}_2^{\opn{sm}}(L) \to 0 .
\]
Let us denote this by $\mathcal{C}_{1, \lambda} = \mathcal{C}_{\Lambda_{\lambda}}$ to indicate the dependence on $i=1$. By reversing the roles of $P_1$ and $P_2$, we also obtain an extension $\mathcal{C}_{2,\lambda}$ of $\opn{St}_1^{\opn{sm}}(L)$ by $\opn{St}^{\opn{la}}(L)$ in exactly the same way.

\begin{proposition}
    One has a unique up to scalar $G(\Q_p)$-equivariant embedding
    \begin{equation} \label{StsmoothEmbedCompletedCoh}
    \opn{St}^{\opn{sm}}(L) \hookrightarrow \widetilde{\opn{H}}^2_c\left( X_{G, K^p}, L \right)[\pi^p] .
    \end{equation}
    Moreover, the embedding (\ref{StsmoothEmbedCompletedCoh}) extends to a $G(\Q_p)$-equivariant embedding
    \[
    \mathcal{C}_{i, \lambda} \hookrightarrow \widetilde{\opn{H}}^2_c\left( X_{G, K^p}, L \right)[\pi^p]
    \]
    if and only if $\lambda \in \bL^{\opn{Aut}}_i(\pi)$.
\end{proposition}
\begin{proof}
    This follows from \cite[\S 5]{AutomorphicLinvariants} since the group $\opn{GL}_3$ is ``relatively definite''.
\end{proof}

\part{Local-global compatibility and Fontaine--Mazur \texorpdfstring{$\mathcal{L}$}{L}-invariants} \label{Part2Label}

In this second part of the article, we now discuss the exceptional zero formulae involving the Fontaine--Mazur $\mathcal{L}$-invariants. More precisely, we prove that the automorphic $\mathcal{L}$-invariants coincide with the Fontaine--Mazur $\mathcal{L}$-invariants (under some additional assumptions if $\pi$ is not essentially self-dual). We have two approaches to this. The first only applies under the assumption that $\pi$ is essentially self-dual -- in this case, we can compare with the work of Rosso \cite{RossoSym2, RossoSym2II}. The second is to prove this comparison between $\mathcal{L}$-invariants if $\pi$ is not necessarily essentially self-dual (which occupies the majority of Part \ref{Part2Label}). The key input is to prove that the family of Galois representations associated with $\pi$ over the $\opn{GL}_3$-eigenvariety admits a triangulation in families, which amounts to establishing local-global compatibility results extending \cite{10author}. We can then apply the method of Gehrmann--Rosso \cite{GehrmannRosso}.

\section{The symmetric square of a weight two modular form}
\numberwithin{equation}{section}

\label{sec:symm square}

Let $f$ be a cuspidal newform of weight $2$ and level $\Gamma_1(N) \cap \Gamma_0(p)$ and nebentypus $\varepsilon_f$, where $N$ is prime to $p$. We assume that $\varepsilon_f(p) = 1$ (otherwise the symmetric square $p$-adic $L$-function of $f$ will have no exceptional zeros). Let $\pi$ denote the unitary normalisation of the automorphic representation of $G(\mbb{A})$ associated with the symmetric square lift $\opn{Sym}^2f$, as constructed in \cite{GelbartJacquet}.  It satisfies the following properties:
\begin{itemize}
    \item $\pi$ is a regular algebraic cuspidal automorphic representation of $\opn{GL}_3(\mbb{A})$, which satisfies $\pi^{\vee} \cong \pi \otimes \varepsilon_f^{-1}$, and is cohomological for the trivial representation.
    \item The local component at $p$ satisfies $\pi_p \cong \opn{St}^{\opn{sm}}(\mbb{C})$ (because $\varepsilon_f(p) = 1$).
    \item Let $\rho_f \colon G_{\mbb{Q}} \to \opn{GL}_2(\Qpb)$ denote the $p$-adic Galois representation associated with $f$, and set $\rho_{\pi} = \opn{Sym}^2 \rho_{f}$. Then
    \[
    L(\pi, s) = L(\rho_{\pi}, s+1)
    \]
    for $s \in \mbb{C}$.
    \item The restriction $\rho_{\pi}|_{G_{\mbb{Q}_p}}$ is semi-stable, and there exists a basis $\{v_0, v_1, v_2\}$ of the associated filtered $(\varphi, N)$-module $D \defeq \mbf{D}_{\opn{st}}(\rho_{\pi}|_{G_{\mbb{Q}_p}})$ such that 
    \[
    \varphi v_i = p^i v_i, \quad N v_i = v_{i-1}, \quad \quad i=0,1,2,
    \]
    with the convention that $v_{-1} = 0$.
\end{itemize}
Indeed, for the first bullet point see, for example, \cite[Theorem A]{NewtonThorneSymII} noting that $f$ cannot have CM because $p$ divides its conductor exactly once. The remaining bullet points follow from known local-global compatibility results for essentially self-dual regular algebraic cuspidal automorphic representations of $\opn{GL}_2(\mbb{A})$ and $\opn{GL}_3(\mbb{A})$ (see, e.g., \cite{loeffler-weinstein, Caraiani12, BGGTII} and the references therein, noting that the trivial weight for $\opn{GL}_3$ is Shin-regular).

The jumps in the Hodge filtration $\opn{Fil}^i D$ on $D$ are at $i=0,1,2$, and as $D$ is weakly admissible, there exist unique quantities $\mathcal{L}_{1, \pi}^{\opn{FM}}, \mathcal{L}_{2, \pi}^{\opn{FM}} \in \Qpb$ such that $v_1 - \mathcal{L}_{1, \pi}^{\opn{FM}} v_0$ (resp. $v_2 - \mathcal{L}_{2, \pi}^{\opn{FM}} v_1$) is a basis of $\opn{Fil}^1D/\opn{Fil}^2D$ (resp. $(\opn{Fil}^2D + \langle v_0 \rangle)/\langle v_0 \rangle$). Since $\rho_{\pi}^{\vee} \cong \rho_{\pi}(2) \otimes \varepsilon_f^{-1}$, one can easily check that $\mathcal{L}_{\opn{Sym}^2f}^{\opn{FM}} \defeq \mathcal{L}_{1, \pi}^{\opn{FM}} = \mathcal{L}_{2, \pi}^{\opn{FM}}$. We call $\mathcal{L}_{\opn{Sym}^2f}^{\opn{FM}}$ the \emph{Fontaine--Mazur $\mathcal{L}$-invariant} associated with $\opn{Sym}^2 f$.

\begin{remark}
    One can easily verify directly from the definitions that $\mathcal{L}_{\opn{Sym}^2f}^{\opn{FM}} = \mathcal{L}_{f}^{\opn{FM}}$, where $\mathcal{L}_{f}^{\opn{FM}}$ denotes the Fontaine--Mazur $\mathcal{L}$-invariant associated with $\rho_f$ (c.f., \cite[Remark 4.2]{MokLinvariant}).
\end{remark}

On the other hand, since $\pi^{\vee} \cong \pi \otimes \varepsilon_f^{-1}$, we have $\mathcal{L}^{\opn{Aut}}_{\opn{Sym}^2f} \defeq \mathcal{L}^{\opn{Aut}}_{\pi,1} = \mathcal{L}^{\opn{Aut}}_{\pi,2}$ (see \S \ref{sec:duality L-invariants} and \cite[\S 3.4]{AutomorphicLinvariants}). We refer to this quantity as the automorphic $\mathcal{L}$-invariant associated with $\opn{Sym}^2 f$. We obtain the following theorem relating the Fontaine--Mazur and automorphic $\mathcal{L}$-invariants:

\begin{theorem} \label{Thm:RossoComparison}
    With notation as above, one has $\mathcal{L}^{\opn{FM}}_{\opn{Sym}^2f} = \mathcal{L}^{\opn{Aut}}_{\opn{Sym}^2f}$.
\end{theorem}
\begin{proof}
We will give two proofs of this theorem, with the first under more restrictive hypotheses. 

Suppose first that $N$ is square-free and $\varepsilon_f$ is trivial. Under these assumptions, Rosso \cite{RossoSym2, RossoSym2II} constructs a $p$-adic measure $\mathscr{L}_p^{\opn{imp}}(\pi, -)$ on $1+2p\mbb{Z}_p$ which satisfies: for any Dirichlet character $\eta$ of $p$-power conductor factoring through $\Z_p^{\times} \to 1+2p\Z_p$, one has
\[
\mathscr{L}^{\opn{imp}}_p(\pi, \eta) = C_{2, \eta^{-1}} \cdot E_1(2, \eta^{-1}) \cdot \frac{L(\pi \times \eta, 0)}{\pi \cdot W'(f) \cdot \langle f, f \rangle_{\opn{Pet}}}
\]
where the $\pi$ in the denominator is the real number $3.14\cdots$, and
\begin{itemize}
    \item $C_{2, \eta} = G(\eta) N_f^{-1} 2^{-5/2}$,
    \item $E_1(2, \eta) = (1-\eta_0(p))$ where $\eta_0$ is the primitive character associated with $\eta$,
    \item $W'(f)$ is the prime-to-$p$ part of the root number associated with $f$,
    \item $\langle f, f \rangle_{\opn{Pet}}$ is the Petersson inner product of $f$ with itself.
\end{itemize}
We first relate this with the $p$-adic $L$-function in \cite{LW21}. More precisely, noting that $e^-_p(\pi_p \times \eta_p, 0) = G(\eta^{-1})E_1(2, \eta^{-1})L(\pi_p \times \eta_p, 0)$, we have 
\[
\mathscr{L}_p^{\opn{imp}}(\pi, \eta) = (N_f 2^{5/2} W'(f))^{-1} e^-_p(\pi_p \times \eta_p, 0) \cdot \frac{L(\pi \times \eta, 0)}{L(\pi_p \times \eta_p, 0) \cdot \pi \langle f, f \rangle_{\opn{Pet}}} .
\]
Set $Q \defeq (N_f 2^{5/2} W'(f)) e^-_{\infty}(\pi_{\infty}, 0) \pi \langle f, f \rangle_{\opn{Pet}} (\Omega_{\pi}^-)^{-1} \neq 0$. Then, since this quantity is independent of $\eta$ (factoring through $\mbb{Z}_p^{\times} \twoheadrightarrow 1 + 2p \mbb{Z}_p$), we have 
\[
\mathscr{L}_{\pi}^-|_{1+2p\mbb{Z}_p} = Q \cdot \mathscr{L}^{\opn{imp}}(\pi, -)
\]
by comparing interpolation formulae and using Weierstrass preparation. The theorem now follows from combining \cite[Theorem 1.3]{RossoSym2II} and Theorem \ref{Thm:LeftHalfEZF} with the fact that $L(\pi, 0) \neq 0$.

For the alternative proof, we now drop the assumption that $N$ is square-free and $\varepsilon_f$ is trivial. In this case, the theorem now follows from \S \ref{SubSubProofOfRossoComparison} below.
\end{proof}

\begin{remark}
    We note that, in \cite{RossoSym2, RossoSym2II}, an exceptional zero formula is obtained for general $N$ at the expense of replacing the $p$-adic $L$-function with one interpolating values of the (complex) imprimitive $L$-function. In this generality, it does not seem possible to employ a similar strategy as in first proof of Theorem \ref{Thm:RossoComparison} to show that the Fontaine--Mazur and automorphic $\mathcal{L}$-invariants coincide, since the imprimitive $L$-function can vanish at $s=0$ for non-square-free $N$. This is why we use a different strategy (which also works for non-self-dual automorphic representations of $\opn{GL}_3(\mbb{A})$) to obtain the equality of $\mathcal{L}$-invariants for general $N$.
\end{remark}

\begin{remark} \label{Rem:EZFforSym2generallevel}
    Set $L_p(\opn{Sym}^2 f, s) \defeq L_p^-(\pi, s)$, $L(\opn{Sym}^2f, s) \defeq L(\pi, s)$ and $\Omega(\opn{Sym}^2 f) \defeq e_{\infty}^-(\pi_{\infty}, 0)^{-1} \cdot \Omega_{\pi}^-$. As a consequence of the \emph{second} proof of Theorem \ref{Thm:RossoComparison} and Theorem \ref{Thm:LeftHalfEZF}, we obtain the exceptional zero formula 
    \[
    \left. \frac{d}{ds}L_p(\opn{Sym}^2f, s) \right|_{s=0} = \mathcal{L}_{\opn{Sym}^2 f}^{\opn{FM}} \cdot \frac{L(\opn{Sym}^2 f, 0)}{\Omega(\opn{Sym}^2 f)}
    \]
    as predicted by Greenberg \cite{GreenbergTrivialZeros} (see also \cite{Benois11}). When $N$ is square-free and $\varepsilon_f$ trivial, this has of course already been proven in \cite{RossoSym2II}; however for general $N$ and $\varepsilon_f$, this result appears to be new (to the best of the authors' knowledge).
\end{remark}

\begin{remark}
If $\xi$ is a Dirichlet character of prime-to-$p$ conductor with $\xi(p) = 1$, then $\cL_{\pi,i}^{\opn{Aut}} = \cL_{\pi\otimes\xi,i}^{\opn{Aut}}$ and $\cL_{\pi,i}^{\opn{FM}} = \cL_{\pi\otimes\xi,i}^{\opn{FM}}$. Thus Remark \ref{rem:intro self dual} in the introduction follows from Theorem \ref{Thm:RossoComparison}.
\end{remark}

\section{The Benois--Colmez--Greenberg--Stevens formulae} \label{Sec:NSDrepresentations}

\numberwithin{equation}{subsection}

In this section, we extend Theorem \ref{Thm:RossoComparison} to automorphic representations of $\opn{GL}_3(\mbb{A})$ which are not necessary essentially self-dual, under certain assumptions on the existence of infinitesimal deformations of triangulations (Assumptions \ref{CalegariMazurAssumption} and \ref{TriangulationInFamilies}). In particular, these assumptions can be verified in the self-dual setting, leading to the second proof of Theorem \ref{Thm:RossoComparison}.

\subsection{Families} \label{SubSec:FamiliesDefect1}

Let $\mathcal{W}$ denote the adic space over $\opn{Spa}(\Q_p, \Z_p)$ whose $(R, R^+)$-points parameterise pairs of continuous characters $\kappa = (\kappa_1, \kappa_2)$ with $\kappa_i \colon \Z_p^{\times} \to R^{\times}$. The space $\mathcal{W}$ is isomorphic to the disjoint union of $2$-dimensional open discs; in particular, let $\mathcal{W}_0$ denote the connected component containing the trivial character. Then $\mathcal{W}_0$ is the adic generic fibre of $\opn{Spf}\Z_p[\![T_1, T_2]\!]$ with the universal character $\kappa^{\opn{univ}} = (\kappa_1^{\opn{univ}}, \kappa_2^{\opn{univ}})$ satisfying $\kappa_i^{\opn{univ}} = 1$ on the roots of unity $\mu_{p-1}$ (if $p$ odd) or $\mu_2$ (if $p=2$), and
\[
\kappa_i^{\opn{univ}}(v) = 1+T_i, \; \; \; \quad (i=1, 2).
\]
Here $v \in 1+2p\Z_p$ is a fixed topological generator with $\log_p(v) = 1$; thus $\kappa_i^{\opn{univ}}(x) = (1+T_i)^{\log_p(x)}$ for all $x \in \Zp^\times$. We can view any point $\kappa \in \mathcal{W}(R, R^+)$ as a character of $T(\Z_p)$ by the formula $(z_1, z_2, z_3) \mapsto \kappa_1(z_1z_3^{-1})\kappa_2(z_2z_3^{-1})$. We will use the same notation if there is no confusion.

\begin{remark}
In this normalisation, the one-dimensional ``pure weight space'', supporting cuspidal cohomology, corresponds to the closed subspace $\mathcal{W}^{\opn{pure}} \defeq \{ \kappa^{\opn{univ}}_2 = 1 \} \subset \mathcal{W}$.
\end{remark}

Let $\pi$ be as in \S \ref{AutoRepSubSec}. We now explain how to deform $\pi$ in a $p$-adic family. We note that, since the defect of the group $G = \opn{GL}_3$ is $1$, we can only deform $\pi$ in a family over a one-dimensional locus in $\mathcal{W}$ (not necessarily the pure weight locus), and in general this family is not expected to have a Zariski-dense set of classical specialisations.

Let $\Omega \subset \mathcal{W}_0$ denote the (affinoid) disc of radius $p^{-\rho}$ centered at the trivial character, for some large $\rho \in \Q_{>0}$ (which we will frequently increase as needed). Let $r = r_{\Omega}$ denote the radius of analyticity of the universal character of $\Omega$. We define $D^{r\opn{-an}}_{\Omega}$ to be the space of $r$-analytic distributions (over $\Omega$) on the Iwahori subgroup of $G(\Q_p)$. More precisely, let
\[
A^{r\opn{-an}}_{\Omega} = \{ f \colon \opn{Iw} \to \mathcal{O}(\Omega) : f \text{ is } r\text{-analytic}, f(b \cdot -) = \kappa_{\Omega}(b) \cdot f(-) \; \forall b \in \overline{B}(\Q_p) \cap \opn{Iw} \},
\]
where a function is $r$-analytic if it is analytic on each open disc of radius $p^{-r}$ in $\opn{Iw}$ (see e.g.\ \cite[\S3.2.2]{BW20}), and let $D^{r\opn{-an}}_{\Omega}$ denote the continuous $\cO(\Omega)$-dual of $A^{r\opn{-an}}_{\Omega}$. Let $T^+$ (resp.\ $T^-$) be the submonoid of $T(\Qp)$ of elements which have non-negative valuation when paired with positive (resp.\ negative) roots. The space $A^{r\opn{-an}}_\Omega$ carries left-actions of $\opn{Iw}$ (by right translation) and $T^-$; the latter is induced by the map $N(\Zp) \to N(\Zp)$, $n \mapsto [t]t^{-1}nt$, where $[t]$ is the image of $t$ under the natural projection $T(\Qp) \twoheadrightarrow T(\Zp)$ (cf.\ \cite[\S3.4]{BW20}). Dualising, $D^{r\opn{-an}}_{\Omega}$ carries left actions of $\opn{Iw}$ and $T^+$. Note the two actions agree for any $t \in T^\pm \cap \opn{Iw}$.

Let $\mathcal{K}^p \subset G(\mbb{A}_f^p)$ be a compact open subgroup such that $(\pi^p_f)^{\mathcal{K}^p}$ is one-dimensional. We let $K^p \subset \mathcal{K}^p$ denote a sufficiently small normal compact open subgroup, and we fix a finite set $S$ of primes containing $p$ and such that $K^p = K^S K_{S \backslash \{p\}}$ with $K_{S \backslash \{p\}} \subset G(\mbb{Q}_{S\backslash\{p\}})$ and $K^S = \prod_{\ell \not\in S} G(\mbb{Z}_{\ell})$. Let $K = K^p\opn{Iw}$ to ease notation, and let $X_{G, K}$ denote the locally symmetric space associated with $G$ of level $K$. Let $\mbf{T}^S = C^{\infty}_c(G(\mbb{A}_f^S) /\!/G(\widehat{\mbb{Z}}^S), \mbb{Z})$ denote the abstract spherical Hecke algebra at primes not in $S$ (with respect to the standard choices of Haar measures).

As $(\pi_f^S)^{K^S}$ is a line, $\mbf{T}^S$ acts on $(\pi_f^p)^{K^p}$ via an algebra homomorphism $\alpha^S \colon \mbf{T}^S \to \mbb{C}$. Let $\alpha_p \colon \Z[T^+] \to \mbb{C}$ denote the algebra homomorphism satisfying 
\[
U_t \cdot \varphi_{\opn{Iw}} = \alpha_p(t) \varphi_{\opn{Iw}}
\]
for $\varphi_{\opn{Iw}} \in \pi_p \cong \opn{St}^{\opn{sm}}(\C)$ any non-zero vector fixed by $\opn{Iw}$, where $U_t$ is the Hecke operator associated with $t$. Both $\alpha^S$ and $\alpha_p$ are definable over a number field, and hence a finite extension $L/\Q_p$. We note that $\alpha_p$ is in fact the trivial character. 

Set $\mbf{T}^{S, \opn{ord}} = \mbf{T}^S \otimes_{\mbb{Z}} \mbb{Z}[T^+]$, and let $\ide{m} \subset \mbf{T}^{S, \opn{ord}}_L$ denote the kernel of $\alpha^S \otimes \alpha_p$. We base-change everything to $L$, but omit this from the notation. Finally if $\ide{p} \subset \mathcal{O}(\Omega)$ is a prime ideal, we let $\Omega_{\ide{p}} \subset \Omega$ denote the Zariski closed subspace cut out by $\ide{p}$.

\begin{proposition} \label{Prop:FamiliesThrupi}
    With notation as above, we have:
    \begin{enumerate}
        \item $\opn{H}^i_c(X_{G,K}, D^{r\opn{-an}}_{\{1\}})^{\opn{ord}, \mathcal{K}^p/K^p}_{\ide{m}}$ is zero if $i \neq 2, 3$ and one-dimensional over $L$ for $i=2, 3$. 
        \item For $\rho$ sufficiently large, there exists a height one prime ideal $\ide{p} \subset \mathcal{O}(\Omega)$ contained in $(T_1, T_2)$, and an algebra homomorphism
        \[
        \underline{\alpha}^S \otimes \underline{\alpha}_p \colon \mbf{T}^{S, \opn{ord}}_{\mathcal{O}(\Omega)/\ide{p}} \to \mathcal{O}(\Omega)/\ide{p} = \mathcal{O}(\Omega_{\ide{p}})
        \]
        which specialises to $\alpha^S \otimes \alpha_p$ at the trivial character, with the following property: the localised cohomology group $\opn{H}^i_c(X_{G,K}, D^{r\opn{-an}}_{\Omega_{\ide{p}}})^{\opn{ord}, \mathcal{K}^p/K^p}_{I}$ is zero if $i \neq 2, 3$, and free of rank one over $\mathcal{O}(\Omega_{\ide{p}})$ if $i=2,3$, where $I$ denotes the kernel of $\underline{\alpha}^S \otimes \underline{\alpha}_p$. Moreover
        \[
        \opn{H}^i_c(X_{G,K}, D^{r\opn{-an}}_{\Omega_{\ide{p}}})^{\opn{ord}, \mathcal{K}^p/K^p}_{I} \otimes_{\mathcal{O}(\Omega_{\ide{p}})} L \cong \opn{H}^i_c(X_{G,K}, D^{r\opn{-an}}_{\{1\}})^{\opn{ord}, \mathcal{K}^p/K^p}_{\ide{m}}
        \]
        for all $i$, where the tensor product is over the trivial character.
        
        \item One has $\underline{\alpha}_p(T^+) \subset \mathcal{O}^+(\Omega_{\ide{p}})^{\times}$ and (after possibly increasing $\rho$) $\underline{\alpha}_{p, 1}\underline{\alpha}_{p, 2}\underline{\alpha}_{p, 3} = 1$, where $\underline{\alpha}_{p, j}$ denotes the character on the $(j, j)$-th element of the diagonal matrix (note $\underline{\alpha}_p$ extends naturally to a character of $T(\Q_p)$).
    \end{enumerate}
\end{proposition}
\begin{proof}
Cuspidal cohomology for $\GL(3)$ is supported in degrees 2 and 3 by \cite[Proposition 4.1]{FSdecomposition} and Matsushima's formula; part (1) then follows from Stevens' control theorem in this setting (see, e.g., \cite[Thm.\ 4.4]{BW20}). 

For part (2), one can apply \cite[Thm.\ 4.9]{HansenThorne} (noting that the compactly-supported version follows from the same proof). More precisely, since the non-abelian Leopoldt conjecture is known when the defect of the group is $1$ \cite[Thm.\ 1.1.5]{Han17}, one has $\opn{H}^i_c(X_{G,K}, D^{r\opn{-an}}_{\Omega})^{\opn{ord}, \mathcal{K}^p/K^p}_{\m} = 0$ if $i \neq 3$, and $\opn{H}^3_c(X_{G,K}, D^{r\opn{-an}}_{\Omega})^{\opn{ord}, \mathcal{K}^p/K^p}_{\m} \cong \mathcal{O}(\Omega)_{\{1\}}/J$ for a principal ideal $J$ generated by a non-zero-divisor. One can now choose a height one prime ideal $\ide{p} \subset \mathcal{O}(\Omega)$ contained in $(T_1, T_2)$, which contains $J$ after localising, and the Tor spectral sequence implies that:
\begin{align*}
\opn{H}^3_c(X_{G,K}, D^{r\opn{-an}}_{\Omega_{\ide{p}}})^{\opn{ord}, \mathcal{K}^p/K^p}_{\m} &\cong \mathcal{O}(\Omega)_{\{1\}}/\ide{p} = \mathcal{O}(\Omega_{\ide{p}})_{\{1\}} \\
\opn{H}^2_c(X_{G,K}, D^{r\opn{-an}}_{\Omega_{\ide{p}}})^{\opn{ord}, \mathcal{K}^p/K^p}_{\m} &\cong \opn{Tor}^{\mathcal{O}(\Omega)_{\{1\}}}_1(\opn{H}^3_c(X_{G,K}, D^{r\opn{-an}}_{\Omega})^{\opn{ord}, \mathcal{K}^p/K^p}_{\m}, \mathcal{O}(\Omega)_{\{1\}}/\ide{p}) \cong \mathcal{O}(\Omega_{\ide{p}})_{\{1\}} \\
\opn{H}^j_c(X_{G,K}, D^{r\opn{-an}}_{\Omega_{\ide{p}}})^{\opn{ord}, \mathcal{K}^p/K^p}_{\m} &\cong \opn{Tor}^{\mathcal{O}(\Omega)_{\{1\}}}_{3-j}(\opn{H}^3_c(X_{G,K}, D^{r\opn{-an}}_{\Omega})^{\opn{ord}, \mathcal{K}^p/K^p}_{\m}, \mathcal{O}(\Omega)_{\{1\}}/\ide{p}) = 0 \quad (j \neq 2, 3) 
\end{align*}
where the last isomorphism in the second line follows from the fact that $J$ is generated by a non-zero-divisor (and $J \subset \ide{p}$). Via a standard rigid delocalisation process (as in \cite[Lem.\ 2.10]{BDJ17}), we deduce both the existence of the character $\underline{\alpha}^S \otimes \underline{\alpha}_p$, and that $\opn{H}^i_c(X_{G,K}, D^{r\opn{-an}}_{\Omega_{\ide{p}}})^{\opn{ord}, \mathcal{K}^p/K^p}_{I}$ is free of rank one over $\cO(\Omega_{\pri})$. The final statement of (2) follows from the Tor spectral sequence relating $\opn{H}^i_c(X_{G,K}, D^{r\opn{-an}}_{\Omega_{\ide{p}}})^{\opn{ord}, \mathcal{K}^p/K^p}_{\m}$ and $\opn{H}^i_c(X_{G,K}, D^{r\opn{-an}}_{\{1\}})^{\opn{ord}, \mathcal{K}^p/K^p}_{\m}$.

Finally, to see (3), note that $\Z[T^+]$ acts by integral units by ordinarity, giving the first statement. For the second, note $\underline{\alpha}_{p,1}\underline{\alpha}_{p,2}\underline{\alpha}_{p,3}$ is the character of $\Qp^\times$ describing the action of the centre $Z(\GL_3(\Qp))$ on $\opn{H}^i_c(X_{G,K}, D^{r\opn{-an}}_{\Omega_{\ide{p}}})^{\opn{ord}, \mathcal{K}^p/K^p}_{I}$. As we have restricted to the 2-dimensional `null' weight space, the centre acts by a finite-order character at each point, and trivially at $\m$ (as $\pi_p$ is Steinberg). Thus, by shrinking if necessary, we can assume it acts trivially at $I$.
\end{proof}

\subsubsection{The essentially self-dual case} \label{Subsub:TheESDCaseFamilies}

We can say slightly more if $\pi$ is essentially self-dual. In this case, $\pi$ is a twisted symmetric square lift of a weight $2$ cuspidal newform $f$ of level $\Gamma_1(N) \cap \Gamma_0(p)$ and nebentypus $\varepsilon_f$ satisfying $\varepsilon_f(p) = 1$ (as in \S\ref{sec:symm square}). Then $f$ is necessarily ordinary at $p$ -- let $\mbf{f} = \sum_{n=1}^{\infty} a_n(\mbf{f}) q^n$ denote the (unique) Hida family passing through $f$ defined over a small open neighbourhood of the trivial weight\footnote{Our convention is that the specialisation of $\mbf{f}$ at an integer weight $k \geq 0$ in the one-dimensional weight space for $\GL_2$ is a weight $k+2$ modular form.} in the (one-dimensional) $\GL_2$ weight space.

By the rigidity of eigenvarieties and $p$-adic Langlands functoriality \cite[\S 3.2]{JoNew} and \cite[\S5.4]{Han17}, and since $\varepsilon_{\mbf{f}}(p) = 1$, we see that one may take $\Omega_{\ide{p}}$ contained in the pure weight locus $\mathcal{W}^{\opn{pure}}$ and 
\[
\underline{\alpha}_p\left( \smat{p & & \\ & 1 & \\ & & 1} \right) = \underline{\alpha}_p\left( \smat{p & & \\ & p & \\ & & 1} \right) = a_p(\mbf{f})^2 .
\]

\subsection{The automorphic Benois--Colmez--Greenberg--Stevens formula}

With notation as in \S \ref{SubSec:FamiliesDefect1}, set $\Sigma = \Omega_{\ide{p}}$. The tangent space $T_1(\Sigma)$ of $\Sigma$ at the origin is at least one-dimensional over $L$. We assume the following weaker version of the Calegari--Mazur conjecture (see \cite{CM09}):

\begin{assumption} \label{CalegariMazurAssumption}
    Let $T_1(\mathcal{W})$ denote the tangent space of $\mathcal{W}$ at the identity, and let $PT_1(\mathcal{W})$ denote the space of lines in $T_1(\mathcal{W})$. One has a natural map 
    \begin{equation} \label{CMweakeqn}
    T_1(\Sigma) \backslash \{ 0 \} \to T_1(\mathcal{W}) \backslash \{ 0 \} \to PT_1(\mathcal{W}) .
    \end{equation}
    We assume that:
    \begin{itemize}
        \item[(A1)] There exists an element $0 \neq v \in T_1(\Sigma)$ which maps to the line $[(v_1, v_2)] \in PT_1(\mathcal{W})$ such that $[(v_1, v_2)] \neq [(1, 1)]$.
        \item[(A2)] There exists an element $0 \neq v \in T_1(\Sigma)$ which maps to the line $[(v_1, v_2)] \in PT_1(\mathcal{W})$ such that $[(v_1, v_2)] \neq [(2, -1)]$. 
    \end{itemize}
    Note that, automatically, either (A1) or (A2) holds.
\end{assumption}

\begin{remark} 
For $i= 1, 2$, assumption (Ai) is precisely the property of admitting ``non-$s_i$-infinitesimal deformations'' as in \S \ref{Subsub:TheGaloisSideIntro}.
\end{remark}

\begin{remark}\label{eq:APS}
  The papers of Ash--Pollack--Stevens \cite{AshPollackStevens} and Calegari--Mazur \cite{CM09} study classical $p$-adic families for $\GL(3)$ and Bianchi modular forms respectively. They both conjecture that non-essentially-self-dual representations do not vary in classical $p$-adic families. Calegari--Mazur go further, predicting that such representations never vary infinitesimally along a rational line. The analogue of this stronger conjecture should also hold for $\GL(3)$; and this would say that if $\pi$ is non-essentially-self-dual, then the image of (\ref{CMweakeqn}) does not contain a line of the form $[(\lambda_1, \lambda_2)]$ with $\lambda_1, \lambda_2 \in \Q$. Assumptions (A1) and (A2) would both be implied by this analogue.
\end{remark}

\begin{remark} \label{Rem:ESDispureweightline}
    If $\pi$ is essentially self-dual, then Assumption \ref{CalegariMazurAssumption} automatically holds because $\Sigma$ is contained in the pure weight locus $\mathcal{W}^{\opn{pure}}$ (so $(1, 0)$ is a basis of $T_1(\Sigma)$).
\end{remark}

Let $v = (v_1, v_2) \in T_1(\Sigma)$ be a non-zero tangent vector. We will consider the first-order infinitesimal behaviour of the family in the direction of this tangent vector. Let $L[\epsilon]$ denote the ring of dual numbers, and note we have a natural map $\mathcal{O}(\Sigma) \to L[\epsilon]$ given by $T_1 \mapsto v_1\epsilon$ and $T_2 \mapsto v_2\epsilon$; in particular, $T_i^2 \mapsto 0$. As $\kappa_i^{\opn{univ}}(x) = (1+T_i)^{\log_p(x)}$, the universal character over $L[\epsilon]$ is then given by
\begin{equation}\label{eq:kappa epsilon}
\kappa_{i, \epsilon}(x) = 1 + \opn{log}_p(x) v_i \epsilon
\end{equation}
for $x \in \Z_p^{\times}$. Set $\kappa_{3, \epsilon} = \kappa_{1, \epsilon}^{-1} \kappa_{2, \epsilon}^{-1}$.

\begin{definition} \label{TriangParamsDef}
    For $i=1, 2,3$, let $\delta_{i,v} \colon \Q_p^{\times} \to L[\epsilon]^{\times}$ denote the unique character such that $\delta_{i, v}|_{\Z_p^{\times}} = \kappa_{i, \varepsilon}$ and $\delta_{i, v}(p) = \underline{\alpha}_{p, i}(p)$ (specialised in this first-order neighbourhood). Note $\delta_{i, v}$ is trivial modulo $\epsilon$.
\end{definition}

\subsubsection{Koszul resolutions} \label{Subsub:KoszulResolutions}

One of the key ingredients in proving the relation between the automorphic and Galois $\mathcal{L}$-invariants is the locally analytic Koszul resolution of Kohlhasse--Schraen, which we now recall. For this we introduce some notation. Let $R$ be a $\mbb{Q}_p$-affinoid algebra, and let $\chi \colon \overline{B}(\mbb{Q}_p) \to T(\mbb{Q}_p) \to R^{\times}$ be a locally analytic character. Let $\chi_0 \colon \opn{Iw} \cap \overline{B}(\mbb{Q}_p) \to R^{\times}$ denote its restriction, which we suppose is $r_{\chi_0}$-analytic.

Recall that, for $r \geq r_{\chi_0}$, we let
\[
A^{r\mathrm{-an}}_{\chi_0} = \{ f \colon \opn{Iw} \to R : f \text{ is } r\text{-analytic}, f(b \cdot -) = \chi_0(b) \cdot f(-) \; \forall b \in \overline{B}(\Q_p) \cap \opn{Iw} \}
\]
denote the $r$-analytic induction, which comes equipped with an action of $\opn{Iw}$ and $T^{-} \subset T(\mbb{Q}_p)$. For any $t \in T^-$, let $U_t \colon G(\mbb{Q}_p) \to \opn{End}_R(A^{r\mathrm{-an}}_{\chi_0})$ denote the unique bi-$\opn{Iw}$-equivariant morphism satisfying: $\opn{supp}(U_t) = \opn{Iw} \cdot t \cdot \opn{Iw}$, and $U_t(t) = t \star -$ (where $\star$ denotes the action of $T^-$ on $A^{r\mathrm{-an}}_{\chi_0}$). This gives rise to an endomorphism of $\opn{c-Ind}_{\opn{Iw}}^{G(\mbb{Q}_p)}(A^{r\mathrm{-an}}_{\chi_0})$ which we also denote by $U_t$.

We let 
\[
\opn{I}_{\overline{B}}^{\opn{la}}(\chi) = \left\{ f \colon G(\mbb{Q}_p) \to R \text{ locally analytic } : f(b \cdot -) = \chi(b) f(-) \text{ for all } b \in \overline{B}(\mbb{Q}_p) \right\}
\]
which is a representation of $G(\mbb{Q}_p)$. One has a natural $\opn{Iw}$-equivariant map $A^{r\mathrm{-an}}_{\chi_0} \to \opn{I}_{\overline{B}}^{\opn{la}}(\chi)$ given by extending a function on the open cell $\overline{B}(\mbb{Q}_p) \cdot \opn{Iw}$ by zero. This induces a $G(\mbb{Q}_p)$-equivariant map 
\begin{equation}\label{eq:extend by zero}
\opn{c-Ind}_{\opn{Iw}}^{G(\mbb{Q}_p)}(A^{r\mathrm{-an}}_{\chi_0}) \to \opn{I}_{\overline{B}}^{\opn{la}}(\chi).
\end{equation}

Let $t_1 = t= \opn{diag}(p, 1, 1)$ and $t_2 = \opn{diag}(p, p, 1)$, so $t_1^{-1}, t_2^{-1} \in T^{-}$. For $i=1, 2$, we let $y_i = U_{t_i^{-1}} - \chi(t_i)$ denote the corresponding endomorphism of $\opn{c-Ind}_{\opn{Iw}}^{G(\mbb{Q}_p)}(A^{r\mathrm{-an}}_{\chi_0})$. The Koszul complex is
\[
    \Lambda^{\bullet}_{R}(R^{\oplus 2}) \otimes_{R} \opn{c-Ind}_{\opn{Iw}}^{G(\Q_p)}(A^{r\mathrm{-an}}_{\chi_0}),
\]
concentrated in degrees $[-2,0]$,  with differentials $d_k : \Lambda^k \otimes (-) \to \Lambda^{k-1} \otimes (-)$  given by
\[
d_2(e_1\wedge e_2 \otimes f) = e_2 \otimes y_1(f) - e_1 \otimes y_2(f), \qquad d_1(e_i \otimes f) = y_i(f).
\]
Here $e_1, e_2$ denotes the standard basis of $R^{\oplus 2}$.

\begin{lemma}[{c.f., \cite[Thm.\ 2.5]{KS12} and \cite[Thm.\ 3.4]{GehrmannRosso}}] \label{Lem:KoszulResolution}
With notation as above, suppose that $\chi$ is trivial on the centre of $G(\mbb{Q}_p)$. Then the augmented Koszul complex
\[
\Lambda^{\bullet}_{R}(R^{\oplus 2}) \otimes_{R} \opn{c-Ind}_{\opn{Iw}}^{G(\Q_p)}(A^{r\mathrm{-an}}_{\chi_0} ) \rightarrow \mathrm{I}_{\overline{B}}^{\opn{la}}(\chi)\rightarrow 0
\]
is exact.
\end{lemma}
\begin{proof} 
Note that, by the hypothesis on $\chi$, we can identify the induction $\opn{I}_{\overline{B}}^{\opn{la}}(\chi)$ as an induction from the lower-triangular Borel in $\mathrm{PGL}_3$. Similarly, we can view $A_{\chi_0}^{r\mathrm{-an}}$ as an induction to the image of $\opn{Iw}$ in $\opn{PGL}_3$. Thus, the lemma is a consequence of the same result for $\mathrm{PGL}_3$ proved by Kohlhasse and Schraen \cite[Thm.\ 2.5]{KS12}. We note that in \emph{loc.cit.}, we are free to arrange $\overline{B}$ to be the choice of Borel and $\opn{Iw}$ to be the choice of Iwahori subgroup.
\end{proof}

Let $D^{r\mathrm{-an}}_{\chi_0}$ denote the continuous $R$-linear dual of $A^{r\mathrm{-an}}_{\chi_0}$. Let $M^{\bullet}$ denote the following complex of $G(\mbb{Q}_p)$-representations
\[
\Lambda^{\bullet}_{R}(R^{\oplus 2}) \otimes_{R} \opn{Ind}_{\opn{Iw}}^{G(\Q_p)}(D^{r\mathrm{-an}}_{\chi_0})
\]
concentrated in degrees $[0, 2]$, with differentials given by the dual of $d_k$ above (identifying $M^k$ with the continuous dual of $\Lambda^k_R(R^{\oplus 2}) \otimes_R \opn{c-Ind}_{\opn{Iw}}^{G(\mbb{Q}_p)}(A^{r\mathrm{-an}}_{\chi_0})$). Note that the dual of $d_1$ is given by
\[
d^*_1(\mu) = e_1 \otimes z_1(\mu) + e_2 \otimes z_2(\mu),
\]
where $z_i = U_{t_i} - \chi(t_i)$ (resp. $U_{t_i}$) denotes the dual of $y_i$ (resp. $U_{t_i^{-1}}$). From Lemma \ref{Lem:KoszulResolution} we obtain:

\begin{corollary} \label{c: dual spectral sequence}
Let $L/\mbb{Q}_p$ be a finite extension and set $R=L$ or $R = L[\epsilon]$ with $\epsilon^2 = 0$. With notation and assumptions as in Lemma \ref{Lem:KoszulResolution}, we obtain  a resolution 
\[
0\rightarrow \mathcal{A}_{G, c}(\mathrm{I}_{\overline{B}}^{\opn{la}}(\chi)^{\ast})\rightarrow \mathcal{A}_{G, c}(M^{\bullet})
\]
as $G(\Q)$-representations. In particular $\opn{H}^0(G(\mbb{Q}), \mathcal{A}_{G, c}(\opn{I}_{\overline{B}}^{\opn{la}}(\chi)^*))$ is identified with the submodule
\[
\opn{ker}(U_{t_1} - \chi(t_1)) \cap \opn{ker}(U_{t_2} - \chi(t_2)) \subset \opn{H}^0(G(\mbb{Q}), \mathcal{A}_{G, c}(\opn{Ind}_{\opn{Iw}}^{G(\mbb{Q}_p)}D^{r\opn{-an}}_{\chi_0})).
\]
\end{corollary}
\begin{proof}
    Note that we have a resolution as $G(\mbb{Q}_p)$-representations
    \[
    0 \to \opn{I}_{\overline{B}}^{\opn{la}}(\chi)^* \to M^{\bullet}.
    \]
    Indeed this follows from the same argument using the Hahn--Banach theorem as in the proof of \cite[Prop.\ 3.9]{GehrmannRosso}; there exactness is only discussed for coefficients in a finite extension of $\mbb{Q}_p$, but the extension to $R=L[\epsilon]$ is immediate. Applying $\mathcal{A}_{G, c}$ to this resolution preserves exactness; $\opn{Hom}_{\mbb{Z}}(D_G, -)$ is exact, since $D_G$ is a projective $\mbb{Z}$-module. 

    Truncating this resolution to the short exact sequence $0 \to \cA_{G,c}(\opn{I}_{\overline{B}}^{\opn{la}}(\chi)^*) \to \cA_{G,c}(M^0) \to \opn{Im}(d_1^*) \to 0$, and taking the long exact sequence in $G(\Q)$-cohomology, we see $\h^0(G(\Q),\cA_{G,c}(\opn{I}_{\overline{B}}^{\opn{la}}(\chi)^*))$ is identified with the kernel of 
    \[
    \opn{H}^0(G(\mbb{Q}), \mathcal{A}_{G, c}(M^0)) \xrightarrow{d_1^*} \opn{H}^0(G(\mbb{Q}), \mathcal{A}_{G, c}(M^1)),
    \]
    so the last part follows from the explicit description of $d_1^*$ above.
\end{proof}

\begin{remark}
    It is likely that a version of Corollary \ref{c: dual spectral sequence} holds for general affinoid algebras $R$. Since we only need the result when $R = L$ or $R = L[\epsilon]$, we restrict to this case in order to avoid any complications with the failure of the Hahn--Banach theorem over general affinoid algebras.
\end{remark}

\subsubsection{Automorphic $\mathcal{L}$-invariants via infinitesimal deformations}

We have the following result:

\begin{proposition} \label{Prop:AutBCGSformula}
    Let $\partial \delta_{i, v} \colon \Q_p^{\times} \to L$ denote the unique continuous homomorphism satisfying $\delta_{i, v} = 1 + \partial \delta_{i, v} \epsilon$. Then:
    \begin{enumerate}
        \item $\partial \delta_{1, v} \neq \partial \delta_{2, v}$ if Assumption \ref{CalegariMazurAssumption}(A1) holds, and $\partial \delta_{2, v} \neq \partial \delta_{3, v}$ if Assumption \ref{CalegariMazurAssumption}(A2) holds.
        \item One has 
        \[
            \partial \delta_{1, v} - \partial \delta_{2, v} \in \bL^{\opn{Aut}}_1(\pi) \qquad \text{and}\qquad\partial \delta_{2, v} - \partial \delta_{3, v} \in \bL^{\opn{Aut}}_2(\pi).
        \]
        In particular, if Assumption \ref{CalegariMazurAssumption}(A1) (resp. Assumption \ref{CalegariMazurAssumption}(A2)) holds, then $\partial \delta_{1, v} - \partial \delta_{2, v}$ is a basis of $\bL^{\opn{Aut}}_1(\pi)$ (resp. $\partial \delta_{2, v} - \partial \delta_{3, v}$ is a basis of $\bL^{\opn{Aut}}_2(\pi)$).
    \end{enumerate}
\end{proposition}
\begin{proof} 
Note $\left(\partial \delta_{1, v} - \partial \delta_{2, v}\right)|_{\Z_p^{\times}} = (v_1 - v_2)\opn{log}_p$ and $\left(\partial \delta_{2, v} - \partial \delta_{3, v}\right)|_{\Z_p^{\times}} = (v_1 + 2v_2)\opn{log}_p$. Part (1) now follows from Assumption \ref{CalegariMazurAssumption}. We now prove part (2) following \cite{GehrmannRosso}.

Let $\chi \defeq \delta_{1, v}\delta_{2, v}^{-1} \colon \Q_p^{\times} \to L[\varepsilon]^{\times}$ and set $\lambda = \partial \delta_{1, v} - \partial \delta_{2, v}$. Therefore, $\chi = 1 + \lambda \cdot \epsilon$. We extend $\chi$ to $\overline{B}(\Q_p)$ by $\chi(b)= \chi(t_1^{-1}t_2)$ where $b= \smallthreemat{t_1}{}{}{\ast}{t_2}{}{\ast}{\ast}{t_3}$, and we consider the locally analytic induction $\opn{I}_{\overline{B}}^{\opn{la}}(\chi)$ introduced in \S \ref{Subsub:KoszulResolutions} (with $R = L[\epsilon]$). Let $r \geq r_{\chi_0}$, and take $\mathcal{K}^p$ to be the prime-to-$p$ level in the definition of $\mathcal{A}_{G, c}(-)$. Then we have a commutative diagram:
\begin{equation} \label{GRliftingtoTreeinfamilies}
\begin{tikzcd}
{R\Gamma(G(\Q), \mathcal{A}_{G, c}(\opn{I}_{\overline{B}}^{\opn{la}}(\chi)^*) )} \arrow[d] \arrow[r] & {R\Gamma( G(\Q), \mathcal{A}_{G, c}(\opn{Ind}_{\opn{Iw}}^{G(\mbb{Q}_p)} D^{r\mathrm{-an}}_{\chi_0}))} \arrow[d] \\
{R\Gamma(G(\Q), \mathcal{A}_{G, c}(\opn{I}_{\overline{B}}^{\opn{la}}(L)^*) )} \arrow[r]              & {R\Gamma( G(\Q), \mathcal{A}_{G, c}(\opn{Ind}_{\opn{Iw}}^{G(\mbb{Q}_p)} D^{r\mathrm{-an}}_{1}))}               
\end{tikzcd}
\end{equation}
where both vertical maps are specialisation mod $\epsilon$, and the horizontal maps are induced from \eqref{eq:extend by zero}. Let $I'$ denote the kernel $\underline{\alpha}^S$ in $\mbf{T}_{L[\epsilon]}^S$ and let $\ide{m}$ denote the kernel of $\alpha^S \otimes \alpha_p$. From Corollary \ref{c: dual spectral sequence} and Proposition \ref{Prop:FamiliesThrupi} we deduce that for the cohomology in degree $0$, the composition of the top horizontal map and the right-hand vertical map in (\ref{GRliftingtoTreeinfamilies}) localised at $I'$ has image which contains $\opn{H}^0( G(\Q), \mathcal{A}_{G, c}(\opn{Ind}_{\opn{Iw}}^{G(\mbb{Q}_p)} D^{r\mathrm{-an}}_{1}))_{\ide{m}}^{\opn{ord}}$. Here we use the analogue of Proposition \ref{prop:arithmetic to betti} with non-trivial coefficients to compare $p$-arithmetic and compactly-supported cohomology (as in \cite[\S3.6]{AutomorphicLinvariants}).

Additionally, by Corollary \ref{c: dual spectral sequence}, the bottom horizontal map is injective on cohomology in degree $0$ with image the space where $U_{t_1}$ and $U_{t_2}$ act trivially. Combining both of these properties, we see that the hypotheses in \cite[Proposition 2.4]{GehrmannRosso} hold, and we conclude by applying \emph{loc.cit.} (noting that the compactly-supported version follows from exactly the same proof). 

One uses a similar argument for $\partial\delta_{2, v} - \partial\delta_{3, v}$, and we leave the details to the reader.
\end{proof}

\subsubsection{Proof of Theorem \ref{Thm:RossoComparison}} \label{SubSubProofOfRossoComparison}

Assume that we are in the essentially self-dual setting of \S \ref{Subsub:TheESDCaseFamilies}. Then we may take $\Sigma = \Omega_{\pri}$ to be a small disc around the origin in the pure weight line $\mathcal{W}^{\opn{pure}}$, with universal character $(\kappa, 1)$ for $\kappa(x) = \opn{exp}(k \opn{log}(x))$ for all $x \in 1+2p\mbb{Z}_p$, where $k \in \mathcal{O}(\Sigma)$ with $|k| \leq p^{-\rho}$. View $a_p = a_p(\mbf{f})$ as a rigid analytic function in the variable $k$ and set $v = (1, 0)$. Then we calculate that
\begin{itemize}
\item $(\partial \delta_{1, v} - \partial \delta_{2, v})|_{\mbb{Z}_p^{\times}} = \opn{log}_p$, and
\item recalling $\underline{\alpha}_{p,1}(p) = a_p(\mathbf{f})^2$ and noting that $\underline{\alpha}_{p,2}(p) = 1$, we have
\[
\partial \delta_{1, v}(p) - \partial \delta_{2, v}(p) = \partial \underline{\alpha}_{p,1}(p) = \partial a_p(\mbf{f})^2 = 2 \left. \tfrac{d a_p}{d k} \right|_{k=0}.
\]
\end{itemize}
Combining, we see
\[
\partial \delta_{1, v} - \partial \delta_{2, v} = \opn{log}_p + 2 \left. \tfrac{d a_p}{d k} \right|_{k=0}  \cdot \opn{ord}_p .
\]
By Remark \ref{Rem:ESDispureweightline} and Proposition \ref{Prop:AutBCGSformula}, we see that $\partial \delta_{1, v} - \partial \delta_{2, v}$ is a basis of $\bL^{\opn{Aut}}_1(\pi)$; thus 
\[
    \mathcal{L}_{\opn{Sym}^2f}^{\opn{Aut}} = \mathcal{L}^{\opn{Aut}}_{\pi,1} = -2  \left.\tfrac{d a_p}{d k} \right|_{k=0} = \mathcal{L}_{\opn{Sym}^2f}^{\opn{FM}},
\]
yielding Theorem \ref{Thm:RossoComparison}. The last equality here is \cite{HidaLinvariant} (see also \cite{MokLinvariant}). \qed

\subsection{The Galois Benois--Colmez--Greenberg--Stevens formula}

We now consider the Galois side. Suppose that $p \geq 7$ and let $\omega \colon \Q_p^{\times} \to \Q_p^{\times}$ denote the $p$-adic cyclotomic character (i.e., $\omega(p) = 1$ and $\omega(x) = x$ for $x \in \Z_p^{\times}$). Let $\rho_{\pi} \colon G_{\mbb{Q}} \to \opn{GL}_3(L)$ denote the $p$-adic Galois representation associated with $\pi$. 

\begin{assumption} \label{AbsIrredDecompGenAssumption}
    We assume that the mod $p$ residual representation $\overline{\rho}_{\pi}$ is absolutely irreducible and decomposed generic. By the latter condition, we mean that there exists an odd prime $\ell \neq p$ such that: $\overline{\rho}_{\pi}$ is unramified at $\ell$, and neither $\ell$ nor $\ell^{-1}$ is an eigenvalue of $\opn{Ad}\overline{\rho}_{\pi}(\opn{Frob}_{\ell}) \in \opn{GL}_9(\overline{\mbb{F}}_p)$.
\end{assumption}

By \cite[Corollary 5.5.2]{10author} (see Theorem \ref{Thm:Char0ExistenceOfGalRACAR}), we have
\[
\rho_{\pi}|_{G_{\mbb{Q}_p}} \sim \left( \begin{array}{ccc} 1 & *_1 & * \\ & \omega^{-1} & *_2 \\ & & \omega^{-2} \end{array} \right)
\]
for certain extension classes $*_i \in \opn{H}^1(\mbb{Q}_p, L(1))$. We also make the following assumption:

\begin{assumption} \label{NonSplitExtensionsAssumption}
    Suppose that $*_1$ and $*_2$ are non-split extensions.
\end{assumption}

\begin{remark}
    Whether $*_i$ is split or non-split should be closely related to whether $\pi$ has CM or not. Since $\pi_p$ is assumed to be the Steinberg representation, $\pi$ cannot have CM and we therefore expect at least one of $*_1$ or $*_2$ to be non-split. Furthermore, $\pi_p$ being Steinberg should correspond to maximal rank monodromy at $p$ on $\mbf{D}_{\opn{st}}(\rho_{\pi}|_{G_{\mbb{Q}_p}})$, so we expect the stronger property that $*_1$ and $*_2$ are non-crystalline. Unfortunately, even understanding when the local Galois representation at $p$ associated with a modular form is split or non-split is still an open problem.
\end{remark}

\begin{definition} 
For $i= 1, 2$, we denote by $\bL^{\opn{FM}}_i(\pi) \subset \opn{Hom}_{\opn{cts}}(\Q_p^{\times}, L) \cong \opn{H}^1(\Q_p, L)$ the one-dimensional subspace given by the orthogonal complement of $\langle *_i \rangle \subset \opn{H}^1(\Q_p, L(1))$ under Tate local duality. If $*_i$ is non-crystalline, we let $\mathcal{L}_{\pi,i}^{\opn{FM}} \in L$ denote the unique element such that $\opn{log}_p - \mathcal{L}_{\pi,i}^{\opn{FM}} \opn{ord}_p$ is a basis of $\bL^{\opn{FM}}_i(\pi)$.
\end{definition}

\subsubsection{Families of trianguline representations}

We have the following theorem, which follows from the results in \S \ref{LGcompatAtl=pSec}. Its proof will be postponed until \S \ref{SubSub:ProofOfTriangThm}.

\begin{theorem} \label{TriangulationInFamilies}
    There exists a rank $3$ trianguline $(\varphi, \Gamma)$-module $D_{L[\epsilon]}$ over $L[\epsilon]$ with parameters 
        \[
        \{ \delta_{1, v}, \delta_{2, v}\omega^{-1}, \delta_{3, v}\omega^{-2} \}
        \]
        (with notation as in Definition \ref{TriangParamsDef}). Moreover, the specialisation of $D_{L[\epsilon]}$ mod $\epsilon$ corresponds, via $p$-adic Hodge theory, to $\rho_{\pi}|_{G_{\mbb{Q}_p}}$.
\end{theorem}

\subsubsection{Equality of $\mathcal{L}$-invariants}

We now prove an equality of automorphic and Fontaine--Mazur $\mathcal{L}$-invariants without assuming any self-duality.

\begin{theorem} 
Suppose that Assumptions \ref{AbsIrredDecompGenAssumption} and \ref{NonSplitExtensionsAssumption} hold. Then at least one of $\bL^{\opn{FM}}_1(\pi) = \bL^{\opn{Aut}}_1(\pi)$ or $\bL^{\opn{FM}}_2(\pi) = \bL^{\opn{Aut}}_2(\pi)$ holds. More precisely, for $\opn{i} \in \{1,2\}$, if Assumption \ref{CalegariMazurAssumption}(A$\opn{i}$) holds then $\bL^{\opn{FM}}_{\opn{i}}(\pi) = \bL^{\opn{Aut}}_{\opn{i}}(\pi)$.
\end{theorem}
\begin{proof}
    By Proposition \ref{Prop:AutBCGSformula}, it suffices to show $\partial \delta_{1, v} - \partial \delta_{2, v} \in \bL^{\opn{FM}}_1(\pi)$ or $\partial \delta_{2, v} - \partial \delta_{3, v} \in \bL^{\opn{FM}}_2(\pi)$. This follows from the Benois--Colmez--Greenberg--Stevens formula \cite{Benois11}, \cite[Thm.\ 3.4]{DingLinvariants}.
\end{proof}

As a consequence, we obtain the following by combining this with Theorems \ref{Thm:LeftHalfEZF} and \ref{Them:RightHalfEZF}.

\begin{corollary}
    Suppose that Assumptions \ref{AbsIrredDecompGenAssumption} and \ref{NonSplitExtensionsAssumption} hold. If Assumption \ref{CalegariMazurAssumption}(A1) holds, then we obtain the exceptional zero formula:
    \[
    \left. \frac{d}{ds}L_p^-(\pi, s) \right|_{s=0} = e_{\infty}^-(\pi_\infty, 0) \cdot \mathcal{L}_{\pi, 1}^{\opn{FM}} \cdot \frac{L(\pi, 0)}{\Omega_{\pi}^-} .
    \]
    If Assumption \ref{CalegariMazurAssumption}(A2) holds, then we obtain the exceptional zero formula:
    \[
    \left. \frac{d}{ds}L_p^+(\pi, s) \right|_{s=1} = -p\cdot e_{\infty}^+(\pi_\infty, 1) \cdot \mathcal{L}_{\pi, 2}^{\opn{FM}} \cdot \frac{L(\pi, 1)}{\Omega_{\pi}^+} .
    \]
    In particular, the exceptional zero conjecture of Greenberg--Benois holds for at least one of $L_p^{-}(\pi, s)$ or $L_p^+(\pi, s)$ under Assumptions \ref{AbsIrredDecompGenAssumption} and \ref{NonSplitExtensionsAssumption}.
\end{corollary}

\begin{remark}
    Note that the statement in the above corollary is indeed well-defined because $\opn{ord}_p \not\in \bL^{\opn{Aut}}_i(\pi)$ (Proposition \ref{Prop:AutLinvSubspaceIsOneDim}), hence $*_i$ is non-crystalline if $\bL_i^{\opn{FM}}(\pi) = \bL^{\opn{Aut}}_i(\pi)$.
\end{remark}


\section{Local-global compatibility at \texorpdfstring{$\ell = p$}{l = p}} \label{LGcompatAtl=pSec}

The goal of this section is to prove Theorem \ref{TriangulationInFamilies}. In order to do this, we will extend the local-global compatibility results in the ordinary case in \cite[\S 5]{10author} to cuspidal automorphic representations of $\opn{GL}_3(\mbb{A})$, with the main result being Theorem \ref{Thm:LevelbcmLGcompat}. In fact, there is nothing special about $\opn{GL}_3$ in this section, and we obtain local-global compatibility results for cuspidal automorphic representations of $\GL_n(\mbb{A})$, for any $n \geq 2$. Since we are working in a more general setting, the notation in this section therefore differs slightly from the rest of the article.

\subsection{Notation and conventions}

Let $n \geq 2$ be an integer.\footnote{Almost everything in this section also applies with $n=1$; however we exclude this case since the main result (Theorem \ref{Thm:LevelbcmLGcompat}) becomes trivial.} Throughout this section, we fix the following notation and conventions.

\subsubsection{The groups}

Let $\tilde{G} = \opn{GSp}_{2n}$ denote the general symplectic group in $2n$-variables where, for any $\mbb{Z}$-algebra $R$, one has
\[
\tilde{G}(R) = \left\{ x \in \opn{GL}_{2n}(R) : \begin{array}{c} x^t J x = s(x)J \\ \text{ for some } s(x) \in R^{\times} \end{array} \right\}, \quad \quad J = \smat{ & & & & & 1 \\ & & & & \iddots & \\ & & & 1 & & \\ & & -1 & & & \\ & \iddots & & & & \\ -1 & & & & &} .
\]
We let $P \subset \tilde{G}$ denote the upper-triangular Siegel parabolic subgroup, with Levi subgroup $M \cong \GL_n \times \opn{GL}_1$ and unipotent radical $N$. Let $T \subset \tilde{G}$ denote the standard diagonal torus, which we will describe as elements of the form
\[
\opn{diag}(t_1, \dots, t_n, s t_n^{-1}, \dots, s t_1^{-1}) \in T .
\]
We will almost always denote such an element of the torus by $(t_1, \dots, t_n; s)$. The identification $M \cong \GL_n \times \opn{GL}_1$ takes the element $(t_1, \dots, t_n; s) \in T$ to $\opn{diag}(t_1, \dots, t_n) \times s$; we can therefore identify $T$ with the maximal torus in $M$. We will also use the notation $T_{\GL_n}$ for the standard diagonal torus in $\GL_n$.

Let $T^+ \subset T(\mbb{Q}_p)$ (resp. $T_M^+ \subset T(\mbb{Q}_p)$) denote the submonoid of elements which satisfy $v_p\langle t, \alpha \rangle \geq 0$ for every positive root $\alpha$ in $\tilde{G}$ (resp. every positive root in $M$). Of course, one has $T^+ \subset T_M^+$. Let $W_{\mathscr{G}}$ denote the Weyl group of a split reductive group $\sG$, and let $w_0^{\mathscr{G}}$ denote the longest Weyl element. Let ${^MW}$ denote the set of minimal length representatives for $W_M \backslash W_{\tilde{G}}$. This set has size $2^n$ and the lengths of the elements range from $0$ to $d=n(n+1)/2$ (exhausting this list of possibilities). We note that, for any $w \in {^MW}$, one has $w T^+ w^{-1} \subset T_M^{+}$. 

If we view $W_{\tilde{G}}$ as the subgroup of $S_{2n}$ of permutations $w$ which satisfy $w(i) + w(2n+1-i) = 2n+1$, then we have the following explicit description of the elements of ${^MW}$.

\begin{lemma} \label{Lem:ExplicitMWset}
    Let $\mathcal{B} \subset \{1, \dots, n \}$ be a subset, which we write as $\mathcal{B} = \{ b_1, \dots, b_l \}$ with $b_1 \leq \cdots \leq b_l$. Let $\mathcal{B}^c = \{1, \dots, n \} - \mathcal{B} = \{ b_1', \dots, b_{n-l}' \}$, with $b_1' \leq \cdots \leq b_{n-l}'$. We let $w_{\mathcal{B}} \in S_{2n}$ denote the permutation which satisfies 
    \[
    w_{\mathcal{B}}(c) = \left\{ \begin{array}{cc} n+i & \text{ if } c = b_i \in {\mathcal{B}} \\ i & \text{ if } c = b_i' \in {\mathcal{B}}^c \end{array} \right.
    \]
    for $1 \leq i \leq n$, and $w_{\mathcal{B}}(c) = 2n+1-w_{\mathcal{B}}(2n+1-c)$ for $n+1 \leq c \leq 2n$. By convention, we let $w_{\varnothing}$ denote the identity, and $w_{\{1, \dots, n \}} = w_0^M w_0^{\tilde{G}}$. Then 
    \[
    {^MW} = \{ w_{\mathcal{B}} \in S_{2n} : {\mathcal{B}} \subset \{ 1, \dots, n \} \} 
    \]
    and $l(w_{\mathcal{B}}) = \sum_{b \in {\mathcal{B}}} (n+1-b)$.
\end{lemma}
\begin{proof}
    The Weyl group $W_{\tilde{G}}$ is described as those permutations $w$ in $S_{2n}$ which satisfy $w(i) + w(2n+1-i) = 2n+1$. Furthermore, the set ${^MW}$ is characterised by the condition that $\Phi_M^+ \subset w \Phi^+$ where $\Phi_M^+$ and $\Phi^+$ denote the positive roots in $M$ and $\tilde{G}$ respectively. Equivalently, the set ${^MW}$ consists of the permutations $w \in W_{\tilde{G}}$ which satisfy
    \[
    w^{-1}(1) < \cdots < w^{-1}(n) .
    \]
    This amounts to the description in the statement of the lemma.
\end{proof}

\subsubsection{Algebraic representations}

Let $X^*(T)$ denote the group of algebraic characters of $T$. We will denote such a character by $\lambda = (\lambda_1, \dots, \lambda_n; \lambda_0)$, where $\lambda_i \in \mbb{Z}$ are such that
\[
\lambda(t_1, \dots, t_n; s) = t_1^{\lambda_1} \cdots t_n^{\lambda_n} s^{\lambda_0} .
\]
The weight $\lambda$ is $\tilde{G}$-dominant (denoted $\lambda \in X^*(T)^+$) if $\lambda_1 \geq \cdots \geq \lambda_n \geq 0$, and it is $M$-dominant (denoted $X^*(T)^{ +}_M$) if $\lambda_1 \geq \cdots \geq \lambda_n$. We will often write the group law on characters additively. We let $\rho \in X^*(T)_{\mbb{Q}}$ denote the half-sum of the positive roots of $\tilde{G}$, and for any $w \in W_{\tilde{G}}$ and $\lambda \in X^*(T)$, we let 
\[
    \lambda_w \defeq w \star \lambda \defeq w \cdot (\lambda + \rho) - \rho
\]
denote the shifted Weyl action.

Let $L/\mbb{Q}_p$ be a finite extension, and let $\mathcal{O}$ be its ring of integers. For any $\lambda \in X^*(T)^+$, we let 
\[
V_{\tilde{G}, \lambda} = V_{\tilde{G}, \lambda, \mathcal{O}} = \{ f \colon \tilde{G}_{\mathcal{O}} \to \mbb{A}^{1}_{\mathcal{O}} : f(- \cdot b) = (w_0^{\tilde{G}}\lambda)(b^{-1}) \cdot f(-) \text{ for all } b \in B \}
\]
where $B \subset \tilde{G}$ denotes the upper-triangular Borel subgroup. This is an algebraic representation of $\tilde{G}(\mathcal{O})$ of highest weight $\lambda$, where the action is given by left-translation by the inverse. We equip this representation with an action of $T^+$ by the formula $(t \cdot f)(-) = f(t^{-1} \cdot - \cdot \langle t \rangle)$, where $\langle (t_1, \dots, t_n; s) \rangle = (|t_1|_p^{-1}, \dots, |t_n|_p^{-1}; |s|_p^{-1})$. Here the $p$-adic valuation is normalised so that $| p |_p = p^{-1}$. If $t \in T(\mbb{Z}_p)$, then this action agrees with the usual action given by the left-translation by the inverse. We make similar definitions for the groups $M$ and $\GL_n$ (which carry actions of $T_M^+$   and $T_{\GL_n}^{+}$ respectively), and often drop the group from the notation when it is clear from the context. 

We note that we will sometimes consider $V_{\tilde{G}, \lambda, L} = V_{\tilde{G}, \lambda} \otimes_{\mathcal{O}} L$, but we will \emph{only ever} view this as a representation of $\tilde{G}(\mathcal{O})$ and $T^+$ as above (and not as a representation of $\tilde{G}(L)$ in the usual way). This will mean that, when we consider the cohomology of the locally symmetric space associated with $\tilde{G}$ with coefficients in the local system $V_{\tilde{G},\lambda, L}$, the action of the Hecke operators at $p$ will automatically be optimally normalised.   

\subsubsection{Hecke algebras} \label{SubSub:HeckeAlgebrasPrelims}

Let $S$ be a finite set of primes containing $p$, and fix Haar measures on $\tilde{G}(\mbb{A}_f^S)$, $M(\mbb{A}_f^S)$, $\GL_n(\mbb{A}_f^S)$ such that the volumes of $\tilde{G}(\widehat{\mbb{Z}}^S)$, $M(\widehat{\mbb{Z}}^S)$, $\GL_n(\widehat{\mbb{Z}}^S)$ are $1$. Let 
\begin{align*}
    \tilde{\mbf{T}}^S &\defeq C_{c}^{\infty}(\tilde{G}(\mbb{A}_f^S) /\!/ \tilde{G}(\widehat{\mbb{Z}}^S), \mathcal{O}) \\
    \mbf{T}^S &\defeq C_c^{\infty}(M(\mbb{A}_f^S) /\!/ M(\widehat{\mbb{Z}}^S), \mathcal{O}) \\
    \mbf{T}^S_{\GL_n} &\defeq C_c^{\infty}(\GL_n(\mbb{A}_f^S) /\!/ \GL_n(\widehat{\mbb{Z}}^S), \mathcal{O})
\end{align*}
denote the spherical Hecke algebras over $\mathcal{O}$ for the primes away from $S$ where the convolution product is with respect to the above fixed Haar measures. We have an $\mathcal{O}$-algebra homomorphism (the unnormalised, or twisted, Satake morphism)
\begin{align*}
\mathcal{S} \colon \tilde{\mbf{T}}^S &\to \mbf{T}^S \\
f &\mapsto \left[ m \mapsto \int_{N(\mbb{A}_f^S)} f(mn) dn \right]
\end{align*}
given by restricting to $P$ and integrating over the fibres of $P \to M$ (here, $dn$ is the Haar measure on $N(\mbb{A}_f^S)$ normalised so that $N(\widehat{\mbb{Z}}^S)$ has volume $1$). 

We let $\tilde{\mbf{T}}^{S, \opn{ord}} = \tilde{\mbf{T}}^S \otimes_{\mathcal{O}} \mathcal{O}[T^+]$, $\mbf{T}^{S, \opn{ord}} = \mbf{T}^S \otimes_{\mathcal{O}} \mathcal{O}[T_M^+]$, and $\mbf{T}^{S, \opn{ord}}_{\GL_n} = \mbf{T}^{S}_{\GL_n} \otimes_{\mathcal{O}} \mathcal{O}[T_{\GL_n}^+]$. If $w \in {^M W}$, we will let $\mathcal{S}_w \colon \tilde{\mbf{T}}^{S, \opn{ord}} \to \mbf{T}^{S, \opn{ord}}$ which coincides with $\mathcal{S}$ on the first factor, and is induced by the map $t \mapsto wtw^{-1}$ on the second factor. For any $t \in T^+$, we let $\tilde{U}_t \in \tilde{\mbf{T}}^{S, \opn{ord}}$ denote the element corresponding to the class $[t] \in \mathcal{O}[T^+]$. Similarly, for $t \in T_M^+$ (resp. $t \in T_{\GL_n}^+$), we let $U_t \in \mbf{T}^{S, \opn{ord}}$ (resp. $\mbf{T}_{\GL_n}^{S, \opn{ord}}$) denote the element corresponding to the class $[t] \in \mathcal{O}[T_M^+]$ (resp. $[t] \in \mathcal{O}[T_{\GL_n}^+]$). To simplify notation, for any integer $i=0, 1, \dots, n$, we let $\tilde{U}_i = \tilde{U}_{t_i}$ and $U_i = U_{t_i}$, where 
\[
t_i = (p, \dots, p, 1, \dots, 1; 1) \in T^+
\]
with $i$ lots of $p$ (if we want to view $t \in T^+_{\GL_n}$, we simply forget the last $\opn{GL}_1$-entry).

For any element $u \in \mbb{Z}_p^{\times}$ and $i=1, \dots, n$, we let $\tilde{\langle u \rangle}_i \in \tilde{\mbf{T}}^{S, \opn{ord}, \times}$ and $\langle u \rangle_i \in \mbf{T}^{S, \opn{ord}, \times}$ (or $\langle u \rangle_i \in \mbf{T}_{\GL_n}^{S, \opn{ord}, \times}$) denote the element corresponding to the class
\[
[(1, \dots, 1, u, 1, \dots, 1; 1)] \in \mathcal{O}[T^+]
\]
where $u$ is in the $i$-th place. In the following definition, we let $\opn{Art} \colon \mbb{Q}_p^{\times} \to G_{\mbb{Q}_p}^{\mathrm{ab}}$ denote the Artin reciprocity map, normalised so that $p$ is sent to a geometric Frobenius.

\begin{definition} \label{DefOfUnivGaloisChars} 
Let $\omega \colon G_{\mbb{Q}_p} \to \mbb{Z}_p^{\times}$ denote the $p$-adic cyclotomic character.
    \begin{enumerate}
        \item Let $R$ be any $\tilde{\mbf{T}}^{S, \opn{ord}}$-algebra in which the operators $\tilde{U}_t$ ($t \in T^+$) become invertible. Let $\lambda = (\lambda_1, \dots, \lambda_n; \lambda_0) \in X^*(T)^+$. For $i=1, \dots, 2n+1$, we let $\psi_{\lambda, i} \colon G_{\mbb{Q}_p} \to R^{\times}$ denote the unique continuous character satisfying:
        \[
    \psi_{\lambda, i}(\opn{Art}(u)) = \left\{ \begin{array}{cc} \omega^{n+1-i}(\opn{Art}(u)) u^{(w_0^M\lambda)_{n+1-i}} \tilde{\langle u \rangle}_{n+1-i}^{-1}  & \text{ if } 1 \leq i \leq n \\ 1   & \text{ if } i=n+1 \\ \omega^{n+1-i}(\opn{Art}(u)) u^{-(w_0^M\lambda)_{i-(n+1)}} \tilde{\langle u \rangle}_{i-(n+1)}  & \text{ if } n+2 \leq i \leq 2n+1 \end{array} \right.
    \]
    for $u \in \mbb{Z}_p^{\times}$, and 
    \[
    \psi_{\lambda, i}(\opn{Art}(p)) = \left\{ \begin{array}{cc} \omega^{n+1-i}(\opn{Art}(p)) \tilde{U}_{{n-i}} \tilde{U}_{{n+1-i}}^{-1}  & \text{ if } 1 \leq i \leq n \\ 1  & \text{ if } i=n+1 \\ \omega^{n+1-i}(\opn{Art}(p)) \tilde{U}_{{i-(n+1)}} \tilde{U}^{-1}_{{i-(n+2)}} & \text{ if } n+2 \leq i \leq 2n+1  \end{array} \right. .
    \]
    Here $(w_0^M\lambda)_i = \lambda_{n+1-i}$ for $i = 1, \dots, n$.
    \item Let $R$ be any $\mbf{T}_{\GL_n}^{S, \opn{ord}}$-algebra in which the operators $U_t$ ($t \in T_{\GL_n}^+$) become invertible. Let $\lambda = (\lambda_1, \dots, \lambda_n) \in X^*(T_{\GL_n})^+$. For $i=1, \dots, n$, we let $\chi_{\lambda, i} \colon G_{\mbb{Q}_p} \to R^{\times}$ denote the unique continuous character satisfying:
    \[
    \chi_{\lambda,i}(\opn{Art}(u)) = \omega^{1-i}(\opn{Art}(u)) u^{-(w_0^{\GL_n}\lambda)_i} \langle u \rangle_i
    \]
    for $u \in \mbb{Z}_p^{\times}$, and 
    \[
    \chi_{\lambda,i}(\opn{Art}(p)) = \omega^{1-i}(\opn{Art}(p)) U_{i} U_{{i-1}}^{-1} .
    \]
    Here $(w_0^{\GL_n}\lambda)_i = \lambda_{n+1-i}$ for $i=1, \dots, n$.
    \end{enumerate}
In (2), if $\lambda$ is the trivial character, then we often drop it from the notation and simply write $\chi_i$.    
\end{definition}

Finally, if $C$ is a complex of $\mathcal{O}$-modules (resp.\ $L$-modules) on which 
$\tilde{\mbf{T}}^{S}$, $\tilde{\mbf{T}}^{S, \opn{ord}}$, etc.\  acts, then we write $\tilde{\mbf{T}}^{S}(C)$, $\tilde{\mbf{T}}^{S, \opn{ord}}(C)$, etc. for the image of $\tilde{\mbf{T}}^{S}$, $\tilde{\mbf{T}}^{S, \opn{ord}}$ (resp.\ $\tilde{\mbf{T}}_L^S, \tilde{\mbf{T}}_L^{S,\opn{ord}}$) etc.\  in the endomorphism algebra of $C$. If we view $C$ as an object in the derived category $\mbf{D}(\mathcal{O})$ of $\mathcal{O}$-modules, then we will consider the image of the Hecke algebra in the endomorphism algebra $\opn{End}_{\mbf{D}(\mathcal{O})}(C)$.

\subsubsection{Dual groups and Hecke polynomials}

Recall that the dual group of $\tilde{G} = \opn{GSp}_{2n}$ can (and will) be identified with $\opn{GSpin}_{2n+1}$ (see, e.g., \cite{KretShin} for the definition of the general spin group). We will identify $T$ with the standard maximal torus in $\opn{GSpin}_{2n+1}$. We will consider the Hecke polynomial associated with the standard representation $r_{\opn{GSpin}, \opn{std}} \colon \opn{GSpin}_{2n+1} \to \opn{SO}_{2n+1} \to \opn{GL}_{2n+1}$. 

\begin{definition} \label{HtildeHeckeDef}
    Let $\ell \not\in S$. For $i=0,1, \dots, 2n+1$, let $\tilde{T}_{\ell, i} \in C_c^{\infty}(\tilde{G}(\mbb{Q}_\ell) /\!/ \tilde{G}(\mbb{Z}_\ell), \mathcal{O})$ denote the Hecke operator which corresponds to $\opn{tr}(\wedge^i r_{\opn{GSpin}, \opn{std}})$ under the \emph{normalised} Satake isomorphism.\footnote{Note that $\tilde{T}_{\ell, i}$ does indeed have coefficients in $\mathcal{O}$ and not just $\mathcal{O}[\ell^{\pm 1/2}]$. This is because the half-sum of the positive roots for $\opn{Sp}_{2n}$ is integral (it is $(n, n-1, \dots, 1)$). In particular, $\ell^{n+1-i}\tilde{T}_{\ell, i} = \ell^{n+1-i}\tilde{T}_{\ell, 2n+1-i}$ is the Hecke operator corresponding to $(\ell, \dots, \ell, 1, \dots, 1; 1) \in T(\mbb{Q}_{\ell})$ for $i=1, \dots, n$, and $\tilde{T}_{\ell, 0} = \tilde{T}_{\ell, 2n+1}$ is the identity.} We define:
    \[
    \tilde{H}_{\ell}(X) \defeq 1 - \tilde{T}_{\ell, 1} X + \cdots + (-1)^i \tilde{T}_{\ell, i} X^i + \cdots - \tilde{T}_{\ell, 2n+1} X^{2n+1} 
    \]
    which is a polynomial with coefficients in $\tilde{\mbf{T}}^S$.
\end{definition}

We will also need the standard Hecke polynomials for $\GL_n$. Let $r_{\GL_n, \opn{std}} \colon \GL_n \times \opn{GL}_1 \to \GL_n$ denote the projection to the first factor.

\begin{definition} \label{HXHeckeDef}
    Let $\ell \notin S$. For $i = 0,1, \dots, n$, let $T_{\ell, i} \in C_c^{\infty}(M(\mbb{Q}_{\ell}) /\!/ M(\mbb{Z}_\ell), \mathcal{O}[\ell^{\pm 1/2}])$ denote the Hecke operator which corresponds to $\opn{tr}(\wedge^i r_{\GL_n, \opn{std}})$ under the normalised Satake isomorphism. We define:
    \[
    H_{\ell}(X) \defeq 1 - \ell^{(n-1)/2} T_{\ell, 1} X + \cdots + (-1)^i \ell^{i(n-1)/2} T_{\ell, i} X^i + \cdots + (-1)^n \ell^{n(n-1)/2} T_{\ell, n} X^n
    \]
    which is a polynomial with coefficients in $\mbf{T}^S$. We also let $H_{\ell}(X)$ denote the image of the polynomial under the natural map $\mbf{T}^S[X] \to \mbf{T}^S_{\GL_n}[X]$ (given by projecting to the $\GL_n$ factor in $M$). We also let
    \[
    H^{\vee}_{\ell}(X) \defeq 1 - \ell^{-(n-1)/2} T_{\ell, n}^{-1} T_{\ell, n-1} X + \cdots + (-1)^i \ell^{-i(n-1)/2} T_{\ell, n}^{-1} T_{\ell, n-i} X^i + \cdots + (-1)^n \ell^{-n(n-1)/2} T_{\ell, n}^{-1} X^n
    \]
    which is a polynomial in $\mbf{T}^S[X]$ (note that $T_{\ell, n}$ is invertible in $\mbf{T}^S$).
\end{definition}

We note the following key identity:
\begin{equation} \label{KeyIdentityOfHeckeEqn}
\mathcal{S}(\tilde{H}_{\ell}(X)) = (1-X) \cdot H_{\ell}(\ell \cdot X) \cdot H^{\vee}_{\ell}(\ell^{-1} \cdot X) 
\end{equation}
which can can be proven by decomposing $\wedge^i r_{\opn{GSpin}, \opn{std}}$ as a representation of $\widehat{M} \subset \opn{GSpin}_{2n+1}$ and comparing the normalised Satake isomorphisms for $\tilde{G}$ and $M$ (c.f., the discussion just before \cite[Corollary 5.2.6]{ScholzeTorsion}, noting that our normalisations differ from \emph{loc.cit.} by $\ell^{\pm 1}$).

\subsubsection{Locally symmetric spaces}

For $\mathscr{G} \in \{ \tilde{G}, M, \GL_n \}$ and $K \subset \mathscr{G}(\mbb{A}_f)$ a neat compact open subgroup, we let $X_{\mathscr{G}, K}$ denote the locally symmetric space associated with $\mathscr{G}$ of level $K$. Note that the (real) dimensional of $X_{\mathscr{G}, K}$ is $n(n+1) = 2d$ (resp. $n(n+1)/2 - 1 = d-1$) if $\mathscr{G} = \tilde{G}$ (resp. $\mathscr{G} \in \{ M, \GL_n \}$).

We consider the following compact open subgroups in the $\mbb{Q}_p$-points of these groups. For $c \geq b \geq 0$ with $c \geq 1$, we let 
\[
\tilde{\opn{Iw}}(b,c) \defeq \{ x \in \tilde{G}(\mbb{Z}_p) : x \text{ mod } p^c \in B(\mbb{Z}/p^c \mbb{Z}), x \text{ mod } p^b \in U(\mbb{Z}/p^b \mbb{Z}) \}
\]
recalling $B$ is the upper-triangular Borel in $\tilde{G}$, and where $U \subset B$ denotes its unipotent radical. We also define
\[
\opn{Iw}(b, c) \defeq \{ x \in M(\mbb{Z}_p) : x \text{ mod } p^c \in B_M(\mbb{Z}/p^c \mbb{Z}), x \text{ mod } p^b \in U_M(\mbb{Z}/p^b \mbb{Z}) \}
\]
where $B_M$ denotes the upper-triangular Borel subgroup in $M$ with unipotent radical $U_M$. We also write $\opn{Iw}(b, c)$ for the image of this group under the natural map $M(\mbb{Z}_p) \twoheadrightarrow \GL_n(\mbb{Z}_p)$.

\subsubsection{Borel--Serre compactifications}

For an integer $1 \leq j \leq n$, let $\mbf{j}$ denote a partition $j = j_1 + \cdots + j_a$. Let $Q_{\mbf{j}} \subset \tilde{G}$ denote the standard parabolic subgroup of $\tilde{G}$ with Levi 
\[
M_{\mbf{j}} \cong \opn{GL}_{j_1} \times \cdots \times \opn{GL}_{j_a} \times \opn{GSp}_{2(n-j)},
\]
where our convention is that $\opn{GSp}_0 = \opn{GL}_1$. Let 
    \[
    \ide{X}_{\tilde{G}} = \tilde{G}(\mbb{Q}) \backslash \left(\tilde{G}(\mbb{A}_f) \times X_{\tilde{G}} \right)
    \]
    where $X_{\tilde{G}}$ denotes the Hermitian symmetric domain for the standard Shimura datum associated with $\tilde{G} = \opn{GSp}_{2n}$. Let $\overline{X}_{\tilde{G}}$ denote the Borel--Serre compactification of $X_{\tilde{G}}$, and let $X_{Q_{\mbf{j}}} = e({Q_{\mbf{j}}})$ denote the boundary component corresponding to the parabolic ${Q_{\mbf{j}}}$ (see \cite[\S 3.1.2]{NewtonThorneTorsion}). Set
    \[
    \ide{X}_{Q_{\mbf{j}}} \defeq {Q_{\mbf{j}}}(\mbb{Q}) \backslash \left( {Q_{\mbf{j}}}(\mbb{A}_f) \times X_{Q_{\mbf{j}}} \right), \quad \quad  \opn{Ind}_{{Q_{\mbf{j}}^{\infty}}}^{\tilde{G}^{\infty}} \mathfrak{X}_{Q_{\mbf{j}}} \defeq {Q_{\mbf{j}}}(\mbb{Q}) \backslash \left( \tilde{G}(\mbb{A}_f) \times X_{Q_{\mbf{j}}} \right) .
    \]
    We also let $\mathfrak{X}_{M_{\mbf{j}}} = M_{\mbf{j}}(\mbb{Q}) \backslash (M_{\mbf{j}}(\mbb{A}_f) \times X_{M_{\mbf{j}}})$, where $X_{M_{\mbf{j}}}$ denotes the symmetric space associated with $M_{\mbf{j}}$. If we let $\overline{\ide{X}}_{\tilde{G}}$ denote the Borel--Serre compactification of $\mathfrak{X}_{\tilde{G}}$ then, as explained in \cite[\S 3.1.2]{NewtonThorneTorsion}, one has a $\tilde{G}(\mbb{A}_f)$-equivariant locally closed embedding
    \[
    \opn{Ind}_{Q_{\mbf{j}}^{\infty}}^{\tilde{G}^{\infty}} \ide{X}_{Q_{\mbf{j}}} \hookrightarrow \partial \overline{\ide{X}}_{\tilde{G}}
    \]
    where $\partial \overline{\ide{X}}_{\tilde{G}} = \overline{\ide{X}}_{\tilde{G}} -\ide{X}_{\tilde{G}}$ denotes the boundary. If $Q_{\mbf{j}}$ is a maximal proper parabolic (i.e., $j_1 = j$), we also have an open immersion $\ide{X}_{Q_{\mbf{j}}} \hookrightarrow \partial \overline{\ide{X}}_{\tilde{G}}$.

    Let $\tilde{K} \subset \tilde{G}(\mbb{A}_f)$ be a sufficiently small compact open subgroup, and let $\partial \overline{X}_{\tilde{G}, \tilde{K}}$ denote the boundary of the Borel--Serre compatification of $X_{\tilde{G}, \tilde{K}} = \ide{X}_{\tilde{G}}/\tilde{K}$. Then we have a locally closed embedding
    \[
    X^{Q_{\mbf{j}}}_{\tilde{K}} := \opn{Ind}^{\tilde{G}^{\infty}}_{Q_{\mbf{j}}^{\infty}} \ide{X}_{Q_{\mbf{j}}} / \tilde{K} \hookrightarrow \partial \overline{X}_{\tilde{G}, \tilde{K}} 
    \]
    which is an open immersion if $Q_{\mbf{j}}$ is maximal proper.

\subsubsection{Running set-up}\label{sec:chapter 9 set-up}
We finally fix our running set-up. Recall that $L/\mbb{Q}_p$ is a finite extension with ring of integers $\mathcal{O}$ and unformiser $\varpi$. Let $\ide{m} \subset \mbf{T}^S$ be a maximal ideal with residue field $k = \mathcal{O}/\varpi$, and let $\Theta \colon \mbf{T}^S \twoheadrightarrow k$ denote the corresponding $\mathcal{O}$-algebra morphism with kernel $\ide{m}$.

\begin{assumption} \label{GaloisTypeAssumption}
    We assume that there exists a continuous semisimple Galois representation
    \[
    \overline{\rho}_{\ide{m}} \colon G_{\Q} \to \GL_n(k) 
    \]
    such that for every finite prime $\ell \not\in S$, the Galois representation $\overline{\rho}_{\ide{m}}$ is unramified at $\ell$ and one has
    \[
    \opn{det}(1 - \opn{Frob}_{\ell}^{-1} X | \overline{\rho}_{\ide{m}} ) = \Theta( H_{\ell}(X) ) .
    \]
    We assume that $\ide{m}$ is non-Eisenstein, i.e., $\overline{\rho}_{\ide{m}}$ is absolutely irreducible.
\end{assumption}

Let $\tilde{\ide{m}} = \cS^*(\ide{m}) \subset \tilde{\mbf{T}}^S$ be the pullback of $\ide{m}$ under the unnormalised Satake morphism. We set
\[
\overline{\rho}_{\tilde{\ide{m}}} \defeq \overline{\rho}_{\ide{m}}^{\vee}(1) \oplus \mbf{1} \oplus \overline{\rho}_{\ide{m}}(-1)
\]
where $V(j)$ is shorthand for $V \otimes \omega^j$, and $\mbf{1}$ denotes the trivial representation. By (\ref{KeyIdentityOfHeckeEqn}), we see that $\overline{\rho}_{\tilde{\ide{m}}}$ is a continuous semisimple Galois representation, which is unramified outside $S$ and satisfies $\opn{det}(1 - \opn{Frob}_{\ell}^{-1} X | \overline{\rho}_{\tilde{\ide{m}}}) = (\Theta \circ \mathcal{S})(\tilde{H}_{\ell}(X))$ for $\ell \notin S$.

\subsection{Galois representations}

We now recall the existence of Galois representations associated with characteristic zero and characteristic $p$ cohomology classes following \cite{ArthurBook} and \cite{ScholzeTorsion}. These Galois representations all satisfy local-global compatibility for $\ell \not\in S$, as we make precise. In characteristic zero, we prove local-global compatibility at $\ell=p$ for $\tilde{G}(\A)$ in Theorem \ref{Thm:PiCARoftildeGLG}, and recall it for $\GL_n$ in Theorem \ref{Thm:Char0ExistenceOfGalRACAR}. In characteristic $p$, local-global compatibility for $\GL_n$ at $\ell=p$ (when the Galois representation is absolutely irreducible and decomposed generic) is one of our main results, and we do not prove it until Proposition \ref{Prop:MixedCharLGcompat} (see Remark \ref{Rem:Charpl=pLGcompat}).

\subsubsection{Characteristic zero results}

In this section, we fix an identification $\mbb{C} \cong \Qpb$, but omit this from the notation. Let $\sG$ denote either $\tilde{G}$ or $\GL_n$, let $\sT \subset \sB \subset \sG$ be the diagonal torus and upper-triangular Borel, and let $\sT^+$ be $T^+$ or $T^+_{\GL_n}$ (if $\sG$ is $\tilde{G}$ or $\GL_n$ respectively). let $\Pi$ be a regular algebraic cuspidal automorphic representation of $\sG(\mbb{A})$ of weight $-w_0^{\sG}\lambda$, with $\lambda \in X^*(\sT)^+$. 

\begin{definition} \label{Def:OrdinarityDefinition}
    Let $J_{\sB}(\Pi_p)$ denote the (normalised) Jacquet module of $\Pi_p$. We say $\Pi_p$ is \emph{ordinary} if there exists a one-dimensional subquotient of $J_{\sB}(\Pi_p)$ given by a smooth character $\chi \colon \sT(\mbb{Q}_p) \to \mbb{C}^{\times}$ satisfying
    \[
    v_p(\chi(t)) = v_p\left(\Big(\delta_{\sB}^{1/2} \cdot w_0^{\sG}\lambda\Big)(t)\right)
    \]
    for all $t \in \sT^+$ (via the identification $\mbb{C} \cong \Qpb$ above). 
\end{definition}

\begin{remark}\label{rem:ordinary principal series}
Note that, if $\Pi_p$ is ordinary, it is isomorphic to a subquotient of the normalised induction $\opn{Ind}_{\sB(\Qp)}^{\sG(\Qp)} \chi$. Furthermore, let $\mathscr{I}(b,b)$ be $\tilde{\opn{Iw}}(b,b)$ (if $\sG = \tilde{G}$) or $\opn{Iw}(b,b)$ (if $\sG = \GL_n$); then this ordinarity condition is equivalent to the existence of an integer $b \geq 1$ and a non-zero element $v \in \Pi_p^{\sI(b, b)}$ which satisfies
    \[
    [\sI(b, b) \cdot t_i \cdot \sI(b, b)] \cdot v = \alpha_i v
    \]
    for all $i=0, 1, \dots, n$, where $\alpha_i$ has $p$-adic valuation equal to that of $(w_0^{\sG}\lambda)(t_i)$. Here $[\sI(b, b) \cdot t_i \cdot \sI(b, b)]$ is the Hecke operator associated with the characteristic function of $\sI(b, b) \cdot t_i \cdot \sI(b, b)$. Note in particular that for $\sG = \GL_n$, Definition \ref{Def:OrdinarityDefinition} agrees with \cite[Def.\ 5.3]{Geraghty}.
\end{remark}

\begin{definition}
    With notation as in Remark \ref{rem:ordinary principal series}, we define $(\Pi_p^{\mathscr{I}(b, b)})^{\opn{ord}} \subset \Pi_p^{\mathscr{I}(b, b)}$ to be the maximal subspace stable under $[\sI(b, b) \cdot t_i \cdot \sI(b, b)]$ for all $i=0, 1, \dots, n$, and such that any eigenvalue for $[\sI(b, b) \cdot t_i \cdot \sI(b, b)]$ has valuation equal to the $p$-adic valuation of $(w_0^{\mathscr{G}} \lambda)(t_i)$. We define the \emph{ordinary part} $\Pi_p^{\opn{ord}}$ of $\Pi_p$ to be 
    \[
    \Pi_p^{\mathrm{ord}} \defeq \bigcup_{b \geq 1} (\Pi_p^{\mathscr{I}(b, b)})^{\opn{ord}} .
    \]
    If $\Pi_p$ is ordinary as in Definition \ref{Def:OrdinarityDefinition}, then $\Pi_p^{\opn{ord}}$ is one-dimensional, else it is zero (as $\lambda$ is dominant; for $\GL_n$, this is \cite[Lemma 5.4(2)]{Geraghty}, and for $\tilde{G}$, the same idea works).
\end{definition}

We first consider local-global compatibility at $\ell=p$ for cuspidal automorphic representations of $\tilde{G}(\mbb{A})$.

\begin{theorem} \label{Thm:PiCARoftildeGLG}
    Let $\Pi$ be a regular algebraic cuspidal automorphic representation of $\tilde{G}(\mbb{A})$ of weight $-w_0^{\tilde{G}}\lambda$, with $\lambda = (\lambda_1, \dots, \lambda_n; \lambda_0) \in X^*(T)^+$, and sufficiently regular in the sense that $\lambda_i - \lambda_{i+1} \geq 1$ for $i=1, \dots, n-1$ and $\lambda_n \geq 1$ (so, in particular, $w \cdot \lambda \neq \lambda$ for all $w \in W_{\tilde{G}}$). Let $S$ denote a finite set of primes containing $p$ and all primes where $\Pi$ is ramified. Then there exists a semisimple continuous Galois representation $\rho_{\Pi} \colon G_{\mbb{Q}} \to \opn{GL}_{2n+1}(\Qpb)$ satisfying the property that: for every prime $\ell \not\in S$, $\rho_{\Pi}$ is unramified at $\ell$ and one has
        \[
        \opn{det}(1 - \opn{Frob}_{\ell}^{-1} X | \rho_{\Pi}) = \Theta^S_{\Pi}( \tilde{H}_{\ell}(X) )
        \]
        where $\Theta^S_{\Pi} \colon \tilde{\mbf{T}}_{\Qpb}^S \to \Qpb$ denotes the eigencharacter for the action of the spherical Hecke algebra away from $S$ acting on $(\Pi_{f}^S)^{\tilde{G}(\widehat{\mbb{Z}}^S)}$.

        Moreover, suppose that $\Pi_p$ is ordinary (in the sense of Definition \ref{Def:OrdinarityDefinition}) and let $\Theta_{\Pi, p} \colon \Qpb[T^+] \to \Qpb$ denote the eigencharacter for the action of the normalised $U_p$-Hecke operators acting on $\Pi_p^{\opn{ord}}$. Let $\psi_{\lambda, i} \colon G_{\mbb{Q}_p} \to \Qpb^{\times}$ denote the characters as in Definition \ref{DefOfUnivGaloisChars} associated with $\Theta_{\Pi}^S \otimes \Theta_{\Pi, p}$. Then:
        \[
        \rho_{\Pi}|_{G_{\mbb{Q}_p}} \sim \left( \begin{array}{cccc} \psi_{\lambda,1} & * & * & * \\ & \psi_{\lambda,2} & * & * \\ & & \ddots & \vdots \\ & & & \psi_{\lambda, 2n+1} \end{array} \right) ,
        \]
        that is, we have local-global compatibility at $\ell = p$.
\end{theorem}

We expect this result is well-known to experts in the field, but as we were unable to find a proof in the literature, we include one here.

\begin{proof}
    Without loss of generality, suppose that $\Pi$ is unitary. Let $\Pi_0 \subset \Pi|_{\opn{Sp}_{2n}(\mbb{A})}$ be any (unitary) irreducible constituent of the restriction to an automorphic representation of $\opn{Sp}_{2n}(\mbb{A})$. It is enough to establish the claim for $\Pi_0$ since the standard representation of $\opn{GSpin}_{2n+1} = {^L \opn{GSp}_{2n}}$ factors through $\opn{SO}_{2n+1} = {^L \opn{Sp}_{2n}}$.
    
    By the work of Arthur \cite{ArthurBook} (as explained in \cite[Theorem 5.1.2]{ScholzeTorsion} for example, although note that \emph{loc.cit.} is for scalar weight), $\Pi_0$ appears in the global $A$-packet associated with an elliptic $A$-parameter $\psi$; so there is a partition $2n+1 = l_1 N_1 + \cdots + l_r N_r$ and self-dual cuspidal automorphic representations $\Sigma_i$ of $\opn{GL}_{N_i}(\mbb{A})$, such that:
    \begin{itemize}
        \item For any $\ell \not\in S$, the representations $\Sigma_{i}$ are unramified at $\ell$, and 
        \[
        \eta \varphi_{\Pi_{0, \ell}} = \bigoplus_{i=1}^r \left( \varphi_{\Sigma_{i, \ell}}|\cdot |^{(1-l_i)/2} \oplus \varphi_{\Sigma_{i, \ell}}|\cdot |^{(3-l_i)/2} \oplus \cdots \oplus \varphi_{\Sigma_{i, \ell}}|\cdot |^{(l_i-1)/2} \right)
        \]
        where $\varphi_{\Pi_{0, \ell}}$ (resp. $\varphi_{\Sigma_{i,\ell}}$) denotes the unramified Langlands parameter associated with $\Pi_{0, \ell}$ (resp. $\Sigma_{i, \ell}$), and $\eta \colon {^L \opn{Sp}_{2n}} \to {^L \opn{GL}_{2n+1}}$ is the map of dual groups given by the standard representation of $\opn{SO}_{2n+1}$.
        \item The representation $\sigma_i \defeq \Sigma_{i} \otimes |\!|\cdot |\!|^{(N_i-l_i)/2}$ is regular algebraic of weight $-w_0^{\opn{GL}_{N_i}} \mu_i$, and the union
        \[
        \bigcup_{i=1}^r \{ (w_0^{\opn{GL}_{N_i}} \mu_i)_j + j - k : j=1, \dots, N_i \text{ and } k=1, \dots, l_i \}
        \]
        is equal to 
        \[
        \{ -\lambda_1-n, -\lambda_2 - n+1, \dots, -\lambda_n - 1, 0, \lambda_n + 1, \dots, \lambda_2 + n -1, \lambda_1 + n \}
        \]
        as multisets.\footnote{Note that, by regular algebraic, we mean regular $C$-algebraic (see \cite{BuzzardGee}).} Set $\xi_{i, k}^{(j)} = (w_0^{\opn{GL}_{N_i}} \mu_i)_j + j - k$.
        \item Since $\lambda$ is sufficiently regular, by the preceding bullet point we must have $l_i = 1$ for all $i=1, \dots, r$. In particular, $\psi$ is a generic $A$-parameter. For a finite prime $\ell$, let $\psi_{\ell}$ denote the local $A$-parameter at $\ell$ associated with $\psi$. By \cite[Theorem 1.2]{Caraiani12}, $\Sigma_{i, \ell}$ is tempered\footnote{To see this, choose an imaginary quadratic number field $F/\mbb{Q}$ in which $\ell$ splits and such that $\Sigma_i$ does not have CM by $F$. Then the quadratic base-change $\opn{BC}_F(\Sigma_i)$ of $\Sigma_i$ to an automorphic representation of $\opn{GL}_{N_i}(\mbb{A}_F)$ is cuspidal, regular algebraic, and conjugate self-dual, and satisfies $\opn{BC}_F(\Sigma_i)_{\lambda} \cong \Sigma_{i, \ell}$ for any prime $\lambda | \ell$.} hence $\psi_{\ell}$ is a tempered, generic $A$-parameter (or equivalently, a tempered $L$-parameter).
    \end{itemize}

    Every $\sigma_i$ is a cuspidal essentially self-dual automorphic representation, so by the work of many people (see \cite{ChenevierHarris, BGGTweight} for example), there exists a semisimple continuous representation
    \[
    \rho_{\sigma_i} \colon G_{\mbb{Q}} \to \opn{GL}_{N_i}(\Qpb)
    \]
    such that for every $\ell \not\in S$, one has 
    \[
    L_{\ell}(\rho_{\sigma_i}, s) = L(\varphi_{\sigma_{i, \ell} | \cdot |^{(1-N_i)/2}}, s) .
    \]
    Furthermore, $\rho_{\sigma_i}$ is de Rham at $p$ with Hodge--Tate weights $\{\xi_{i,1}^{(j)} : j=1, \dots, N_i \}$. We then define
    \[
    \rho_{\Pi_0} \defeq \bigoplus_{i=1}^r \left( \rho_{\sigma_i} \oplus \rho_{\sigma_i}(-1) \oplus \cdots \oplus \rho_{\sigma_i}(1-l_i) \right) = \bigoplus_{i=1}^r \rho_{\sigma_i},
    \]
    which satisfies the desired local-global compatibility at $\ell \not\in S$ by the first bullet point above.

    Suppose now that $\Pi_{p}$ is ordinary. By Remark \ref{rem:ordinary principal series}, and as $\lambda$ is sufficiently regular,  Tadi\'{c}'s irreducibility criterion \cite[Thm.\ 7.9]{Tad94} implies that $\Pi_p \cong \opn{Ind}_{B(\Qp)}^{\tilde{G}(\Qp)}\chi$ is the normalised induction of a smooth character $\chi = (\chi_1,...,\chi_n;\chi_0)$ (as in Definition \ref{Def:OrdinarityDefinition}). By Mackey theory, 
    \[
       \Pi_p|_{\opn{Sp}_{2n}(\Q_p)} \cong I_{B'(\Qp)}^{\opn{Sp}_{2n}(\Q_p)}\chi' \cong  \Pi_{0,p},
    \]
    where $B'$ is the standard Borel in $\opn{Sp}_{2n}$, and $\chi' = (\chi_1,\dots,\chi_n)$ is the corresponding character of the torus $T' = T \cap B'$. Here the second isomorphism follows from irreducibility of the middle induction (which is \cite[Thm.\ 7.1]{Tad94}; see also \cite[Thm.\ 2.2.4]{TaibiEigenvarieties}). 
   
    The representation $\Pi_{0, p}$ lies in the local $A$-packet associated with the tempered $L$-parameter $\psi_{p}$, hence $\Pi_{0, p}$ is tempered (see \cite[Theorem 1.5.1(b)]{ArthurBook}). Thus the characters $\chi_i$ are unitary. Since $\Pi_{0, p}$ lies in a unique $A$-packet (see \emph{loc.cit.}), we must have that $\psi_p$ is equivalent to the $A$-parameter
    \[
    \psi_p = \chi_1 \oplus \cdots \oplus \chi_n \oplus \mbf{1} \oplus \chi_n^{-1} \oplus \cdots \oplus \chi_1^{-1},
    \]
    where $\chi_i \colon W_{\mbb{Q}_p} \times \opn{SU}(2) \times \opn{SU}(2) \to \opn{GL}_1(\mbb{C})$ which is trivial on the $\opn{SU}(2)$-factors and coincides with $\chi_i$ on the $W_{\mbb{Q}_p}$-factor. If we let 
    \[
    \pi = \Sigma_1 \boxplus \cdots \boxplus \Sigma_r
    \]
    denote the global base-change of $\Pi_0$ to an automorphic representation of $\opn{GL}_{2n+1}(\mbb{A})$, then we see that $\pi_p \cong I_{B_{\opn{GL}_{2n+1}}(\mbb{Q}_p)}^{\opn{GL}_{2n+1}(\mbb{Q}_p)}(\chi_1 \oplus \cdots \oplus \chi_n \oplus \mbf{1} \oplus \chi_n^{-1} \oplus \cdots \oplus \chi_1^{-1})$. We note that $\pi_p \cong \Sigma_{1, p} \boxplus \cdots \boxplus \Sigma_{r, p}$. As $\Pi_{p}$ is ordinary, one may easily check that $\pi_p$ is also ordinary, and that this can only happen if $\Sigma_{i, p}$ are also ordinary, for all $i=1, \dots, r$.    
    
    Since $\lambda$ is sufficiently regular, so is $\mu_i$ and the Galois representation $\rho_{\sigma_i}$ satisfies local-global compatibility at $\ell = p$ (see \cite{BGGTII}), i.e.,
    \[
    \rho_{\sigma_i}|_{G_{\mbb{Q}_p}} \sim \left( \begin{array}{cccc} \chi_{\mu_i,1} & * & * & * \\ & \chi_{\mu_i,2} & * & * \\ & & \ddots & \vdots \\ & & & \chi_{\mu_i, N_i} \end{array} \right)
    \]
    where $\chi_{\mu_i, j}$ are the characters from Definition \ref{DefOfUnivGaloisChars} associated with the action of the Hecke algebra at $p$ acting on $\sigma_{i, p}$. We note that the collection $\{ \chi_{\mu_i, j} : i=1, \dots, r \text{ and } j=1, \dots, N_i \}$ coincides with the collection $\{ \psi_{\lambda, j} : j=1, \dots, 2n+1 \}$, and by noting that the Hodge--Tate weights of $\chi_{\mu_i, 1}, \dots, \chi_{\mu_i, N_i}$ have to be in increasing order, one can see that the filtration on 
    \[
    \rho_{\Pi_0}|_{G_{\mbb{Q}_p}} = \bigoplus_{i=1}^r \rho_{\sigma_i}|_{G_{\mbb{Q}_p}}
    \]
    can be reordered so that the local-global compatibility property in the theorem holds.
\end{proof}

A key ingredient in the above proof was the known local-global compatibility results for regular algebraic, essentially self-dual, cuspidal automorphic representations of general linear groups. We will also need similar results without the essentially self-dual condition, as found in \cite[\S5.5]{10author}. For this, first recall the following definition from \cite[Def.\ 4.3.1]{10author} (c.f.\ \cite[Def.\ 1.9]{CS17}, although note that the condition in \emph{loc.cit.} is slightly stronger than what we need).

\begin{definition}
    Let $\rho \colon G_{\mbb{Q}} \to \opn{GL}_n(\overline{\mbb{F}}_p)$ be a continuous Galois representation which is unramified outside $S$. We say that $\rho$ is \emph{decomposed generic at a prime $\ell \notin S$} if the eigenvalues of $\opn{Frob}_{\ell}$, denoted $\beta_1, \dots, \beta_n$, satisfy
    \[
    \beta_i \beta_j^{-1} \neq \ell \text{ for all } i \neq j .
    \]
    We say $\rho$ is \emph{decomposed generic} if there exists such a prime $\ell$ (note that if there exists one such prime, then there exist infinitely many by the Chebotarev density theorem).
\end{definition}

\begin{theorem} \label{Thm:Char0ExistenceOfGalRACAR}
    Let $\Pi$ be a regular algebraic cuspidal automorphic representation of $\GL_n(\mbb{A})$ of weight $-w_0^{\GL_n} \lambda$, with $\lambda = (\lambda_1, \dots, \lambda_n) \in X^*(T_{\GL_n})^+$. Let $S$ denote a finite set of primes containing $p$ and all primes where $\Pi$ is ramified. Then there exists a continuous semisimple Galois representation $\rho_{\Pi} \colon G_{\mbb{Q}} \to \GL_n(\Qpb)$ such that: for every $\ell \not\in S$, $\rho_{\Pi}$ is unramified at $\ell$ and one has
    \[
    \opn{det}(1 - \opn{Frob}_{\ell}^{-1} X | \rho_{\Pi}) = \Theta_{\Pi}^S(H_{\ell}(X))
    \]
    where $\Theta^S_{\Pi} \colon \mbf{T}_{\Qpb}^S \to \Qpb$ denotes the eigencharacter for the action of the spherical Hecke algebra away from $S$ acting on $(\Pi_{f}^S)^{\GL_n(\widehat{\mbb{Z}}^S)}$.

    Moreover, suppose that:
    \begin{itemize}
        \item $\Pi_p$ is ordinary, and
        \item the residual representation $\overline{\rho}_{\Pi} \colon G_{\mbb{Q}} \to \opn{GL}_n(\overline{\mbb{F}}_p)$ is decomposed generic and irreducible. 
    \end{itemize}
    Then one has
    \[
    \rho_{\Pi}|_{G_{\mbb{Q}_p}} \sim \left( \begin{array}{cccc} \chi_{\lambda,1} & * & * & * \\ & \chi_{\lambda,2} & * & * \\ & & \ddots & \vdots \\ & & & \chi_{\lambda, n} \end{array} \right)
    \]
    where $\chi_{\lambda, i}$ denotes the characters associated with $\Theta_{\Pi}^S \otimes \Theta_{\Pi, p}$ as in Definition \ref{DefOfUnivGaloisChars}, with $\Theta_{\Pi, p} \colon \Qpb[T^+_{\GL_n}] \to \Qpb$ denoting the eigencharacter for the action of the normalised $U_p$-operators acting on $\Pi_p^{\mathrm{ord}}$. 
\end{theorem}
\begin{proof}
    The existence of the Galois representation follows from \cite{ScholzeTorsion} (or \cite{HLTTrigid}), and the local-global compatibility at $\ell=p$ follows from \cite[Corollary 5.5.2]{10author} (by making a suitable quadratic base-change to an imaginary quadratic field $F/\mbb{Q}$ in which $p$ splits).
\end{proof}

\subsubsection{Characteristic $p$ results}

We now recall the construction of Galois representations associated with torsion classes appearing in \cite{ScholzeTorsion}.  

\begin{theorem}[Scholze] \label{Thm:GaloisRepForTorClass}
    Let $m \geq 1$ be an integer, and let $\ide{n}$ be a maximal ideal with finite residue field in the (prime-to-$S$) Hecke algebra $\tilde{\mbf{T}}^S\left( \opn{H}^*( X_{\tilde{G}, K}, \mathcal{O}/\varpi^m) \right)$ (resp. $\mbf{T}^S\left( \opn{H}^*(X_{\GL_n, K}, \mathcal{O}/\varpi^m) \right)$) for some sufficiently small compact open subgroup $K \subset \tilde{G}(\mbb{A}_f)$ (resp. $K \subset \GL_n(\mbb{A}_f)$). Then there exists a continuous semisimple Galois representation
    \[
    \overline{\rho}_{\ide{n}} \colon G_{\mbb{Q}} \to \opn{GL}_{2n+1}(\overline{\mbb{F}}_p) \quad \text{ (resp. } \overline{\rho}_{\ide{n}} \colon G_{\mbb{Q}} \to \opn{GL}_{n}(\overline{\mbb{F}}_p) \text{ ) }
    \]
    such that for all $\ell \not\in S$, we have $\overline{\rho}_{\ide{n}}$ is unramified at $\ell$ and
    \begin{align*}
    \opn{det}(1 - \opn{Frob}_{\ell}^{-1} X | \overline{\rho}_{\ide{n}}) &= \Theta_{\ide{n}}(\tilde{H}_{\ell}(X)) \\
    \text{ (resp. } \opn{det}(1 - \opn{Frob}_{\ell}^{-1} X | \overline{\rho}_{\ide{n}}) &= \Theta_{\ide{n}}(H_{\ell}(X)) \text{ )}.
    \end{align*}
    Here $\Theta_{\ide{n}} \colon \tilde{\mbf{T}}^S \to \overline{\mbb{F}}_p$ (resp. $\Theta_{\ide{n}} \colon \mbf{T}^S \to \overline{\mbb{F}}_p$) denotes the homomorphism with kernel $\ide{n}$.
\end{theorem}
\begin{proof}
    For the standard representation for general symplectic groups, the existence is explained in the proof of \cite[Corollary 5.2.6]{ScholzeTorsion} (and relies on Arthur's endscopic classification for symplectic groups \cite{ArthurBook}). The existence for general linear groups again follows from \cite{ScholzeTorsion} (which, in fact, uses the existence for general symplectic groups as input). In both cases, Scholze constructs a continuous determinant with the desired properties, valued in a quotient of the Hecke algebra; further quotienting by $\ide{n}$, one gets determinants valued in $\overline{\F}_p$, and then concludes by \cite[Thm.\ 2.12]{ChenevierPadic}.
\end{proof}

\subsection{Vanishing results and boundary cohomology}

We now collect together key cohomological results; firstly, a vanishing result of Yang and Zhu in characteristic $p$, and secondly, an analysis of boundary cohomology that will allow the transfer of classes from $\tilde{G}$ to $M$ via a degree-shifting argument. Throughout, $\m$ and $\tilde\m$ are the maximal ideals from \S\ref{sec:chapter 9 set-up}; in particular, $\m$ satisfies Assumption \ref{GaloisTypeAssumption}.

\subsubsection{Vanishing results} \label{SubSub:VanishingResults} 

We will use the following vanishing result.

\begin{theorem}[{Yang--Zhu \cite{YangZhu}}] \label{Thm:YZVanishing}
    Suppose that $p > 2n$. Let $K \subset \tilde{G}(\mbb{A}_f)$ be a neat compact open subgroup of the form $K = K^S K_{S \backslash \{p\}} K_p$ with $K^S = \tilde{G}(\widehat{\mbb{Z}}^S)$, $K_{S \backslash \{p\}} \subset \tilde{G}(\mbb{Q}_{S \backslash \{p\}})$, and $K_p \subset \tilde{G}(\mbb{Q}_p)$. Let $\lambda \in X^*(T)^+$. Suppose that $\overline{\rho}_{\tilde{\ide{m}}}$ is decomposed generic at an odd prime $\ell \not\in S$. Then for all $m \geq 1$,
    \[
    \opn{H}^i\left( X_{\tilde{G}, K}, V_{\lambda}/\varpi^m \right)_{\tilde{\ide{m}}} = 0
    \]
    if $i < n(n+1)/2$, and
    \[
    \opn{H}_c^i\left( X_{\tilde{G}, K}, V_{\lambda}/\varpi^m \right)_{\tilde{\ide{m}}} = 0
    \]
    if $i > n(n+1)/2$. 
\end{theorem}
\begin{proof}
    By deepening $K_p$ and using the Hochschild--Serre spectral sequence, it suffices to prove the theorem when $\lambda$ is trivial. But now we note that $\overline{\rho}_{\tilde{\ide{m}}}$ being decomposed generic at $\ell$ implies that the $L$-parameter $W_{\mbb{Q}_{\ell}} \to {^LT}(\overline{\mbb{F}}_p)$ corresponding to $\tilde{\ide{m}}$ is generic in the sense of \cite[Def.\ 1.1]{YangZhu} (note that $\opn{GSpin}_{2n+1}$ is a central extension of $\opn{SO}_{2n+1}$). Applying \cite[Thm.\ 1.5]{YangZhu} yields the result for $m=1$, noting that $2n$ is the Coxeter number of $\opn{Sp}_{2n}$. To conclude for general $m$, we argue inductively using the long exact sequence associated with $0 \to V_{\lambda}/\varpi^{m-1} \to V_{\lambda}/\varpi^m \to V_{\lambda}/\varpi \to 0$.
\end{proof}

\begin{remark}
    We note that in the special case of $\opn{GSp}_4$, the above vanishing results were originally obtained by Hamann--Lee \cite{HamannLee} under some additional hypotheses. We are primarily interested in the group $\opn{GSp}_6$ in this article, which is not covered by their work. 
\end{remark}

\subsubsection{Boundary cohomology}

As a consequence of the existence of Galois representations above, we obtain certain vanishing results of the boundary cohomology of the Borel--Serre compactification of $\tilde{G}$. Let $\tilde{K} \subset \tilde{G}(\mbb{A}_f)$ be a sufficiently small compact open subgroup hyperspecial outside $S$.

\begin{proposition} \label{Prop:OnlyOneBoundarySurvives}
    Recall the ideal $\tilde{\ide{m}}$ from \S\ref{sec:chapter 9 set-up}.
    \begin{enumerate}
        \item If  $Q_{\mbf{j}} \neq P$, then $R\Gamma_c(X^{Q_{\mbf{j}}}_{\tilde{K}}, \mathcal{O})_{\tilde{\ide{m}}} = R\Gamma(X^{Q_{\mbf{j}}}_{\tilde{K}}, \mathcal{O})_{\tilde{\ide{m}}} = 0$.
        \item One has a natural $\tilde{\mbf{T}}^S$-equivariant quasi-isomorphism
        \[
        R\Gamma( X_{\tilde{K}}^P, \mathcal{O})_{\tilde{\ide{m}}} \xrightarrow{\sim} R\Gamma(\partial \overline{X}_{\tilde{G}, \tilde{K}}, \mathcal{O})_{\tilde{\ide{m}}} . 
        \]
    \end{enumerate}
\end{proposition}
\begin{proof}
    The proof of this proposition is very similar to that of \cite[Theorem 2.4.2]{10author}. By Poincar\'e duality and the fact that $R\Gamma(X^{Q_{\mbf{j}}}_{\tilde{K}}, \mathcal{O})$ is a perfect complex, it suffices to show
    \[
    R\Gamma(X^{Q_{\mbf{j}}}_{\tilde{K}}, k)_{\tilde{\ide{m}}} = 0,
    \]
     i.e.\ that $\tilde{\ide{m}}$ doesn't descend to a maximal ideal of $\tilde{\mbf{T}}^S(H^*(X^{Q_{\mbf{j}}}_{\tilde{K}}, k))$. Suppose for contradiction that it does. We may freely increase the level $\tilde{K}$ at primes in $S$ such that we have a decomposition
    \[
    \tilde{K}_{Q_{\mbf{j}}} = \tilde{K} \cap Q_{\mbf{j}}(\mbb{A}_f) = (\tilde{K} \cap N_{\mbf{j}}(\mbb{A}_f)) \cdot \tilde{K}_{M_{\mbf{j}}},
    \]
    where $\tilde{K}_{M_{\mbf{j}}}$ is the projection of $\tilde{K}_{Q_{\mbf{j}}}$ to the Levi, and $N_{\mbf{j}} \subset Q_{\mbf{j}}$ is the unipotent radical. 
    
    In this proof, we will denote the Hecke algebra associated with $M_{\mbf{j}}$ by $\mbf{T}_{\mbf{j}}^S$. We note that we have an unnormalised Satake morphism 
\[
\mathcal{S}_{M_{\mbf{j}}} \colon \tilde{\mbf{T}}^S \to \mbf{T}_{\mbf{j}}^S
\]
defined in the exact same way as in \S \ref{SubSub:HeckeAlgebrasPrelims} (restricting to $Q_{\mbf{j}}$ and integrating along the fibres of the map $Q_{\mbf{j}} \to M_{\mbf{j}}$). We also have the following identity:
\begin{align*}
\mathcal{S}_{M_{\mbf{j}}}(\tilde{H}_{\ell}(X)) = \tilde{H}^{(2(n-j))}_{\ell}(X) \cdot H^{(j_1)}_{\ell}(\ell^{n+1-j_1} X) \cdot H^{(j_2)}_{\ell}(\ell^{n+1-j_1-j_2} X) \cdots H^{(j_a)}_{\ell}(\ell^{n+1-j} X) \cdot \\
\cdot H^{(j_a), \vee}_{\ell}(\ell^{j-1-n} X) \cdots H^{(j_2), \vee}_{\ell}(\ell^{j_1+j_2 -1 -n} X) \cdot H^{(j_1), \vee}_{\ell}(\ell^{j_1-1-n} X)
\end{align*}
analogous to (\ref{KeyIdentityOfHeckeEqn}), where $\tilde{H}^{(2(n-j))}_{\ell}(X)$, $H_{\ell}^{(j_i)}(X)$ and $H_{\ell}^{(j_i), \vee}(X)$ denote the versions of the Hecke polynomials in Definition \ref{HtildeHeckeDef} and Definition \ref{HXHeckeDef} respectively for the groups $\opn{GSp}_{2(n-j)}$ and $\opn{GL}_{j_i}$.
    
    By a similar argument as in \cite[Theorem 4.2]{NewtonThorneTorsion} (and possibly increasing the level $\tilde{K}$ further at primes in $S$ combined with \cite[Lemma 4.3]{NewtonThorneTorsion}), as $R\Gamma(X^{Q_{\mbf{j}}}_{\tilde{K}}, k)_{\tilde{\ide{m}}} \neq 0$ one has $\tilde{\ide{m}} = \mathcal{S}^*_{M_{\mbf{j}}}(\ide{n})$ for some ideal $\ide{n} \subset \mbf{T}_{\mbf{j}}^S(\opn{H}^*(X_{M_{\mbf{j}}, \tilde{K}_{M_{\mbf{j}}}}, k))$. By deepening the level further, we may assume that $\tilde{K}_{M_{\mbf{j}}}$ can be written as a product
    \[
    K_1 \times \cdots \times K_a \times K_0
    \]
    where $K_i \subset \opn{GL}_{j_i}(\mbb{A}_f)$ ($i=1, \dots, a$) and $K_0 \subset \opn{GSp}_{2(n-j)}(\mbb{A}_f)$. By the K\"unneth formula, the ideal $\ide{n}$ can be written as the product
    \[
    \ide{n}_1 \otimes \ide{n}_2 \otimes \cdots \otimes \ide{n}_a \otimes \ide{n}_0
    \]
    for $\ide{n}_i \subset \mbf{T}^S(\opn{H}^*(X_{\opn{GL}_{j_i}, K_i}, k))$ ($i=1, \dots, a$) and $\ide{n}_0 \subset \tilde{\mbf{T}}^S(\opn{H}^*(X_{\opn{GSp}_{2(n-j)}, K_0}, k))$.

    By Theorem \ref{Thm:GaloisRepForTorClass}, we have
    \[
    \overline{\rho}_{\tilde{\ide{m}}} \cong \overline{\rho}_{\ide{n}_1}(j_1 - n-1) \oplus \cdots \oplus \overline{\rho}_{\ide{n}_a}(j-n-1) \oplus \overline{\rho}_{\ide{n}_0} \oplus \overline{\rho}_{\ide{n}_a}^{\vee}(n+1-j) \oplus \cdots \oplus \overline{\rho}_{\ide{n}_1}(n+1-j_1)
    \]
    On the other hand, we know that $\overline{\rho}_{\tilde{\ide{m}}} \cong \overline{\rho}_{\ide{m}}(-1) \oplus \mbf{1} \oplus \overline{\rho}_{\ide{m}}^{\vee}(1)$, and $\overline{\rho}_{\ide{m}}$ is absolutely irreducible (since $\ide{m}$ is non-Eisenstein). The only way we can have both of these presentations of $\overline{\rho}_{\tilde{\ide{m}}}$ is if $a=1$ and $j=n$, which is a contradiction to our assumption that $Q_{\mbf{j}} \neq P$. This completes the proof of part (1). For part (2), we follow the same argument as in \cite[Theorem 2.4.2]{10author}, namely the isomorphism arises from the excision exact triangle.
\end{proof}

\subsubsection{} \label{Subsub:IndAppearsAsDirectSummand}

We now analyse the cohomology $R\Gamma( X_{\tilde{K}}^P, \mathcal{O})_{\tilde{\ide{m}}}$. Let $\tilde{K} = \tilde{K}^p \tilde{K}_p \subset \tilde{G}(\mbb{A}_f)$ be a sufficiently small compact open subgroup and suppose that 
\[
\tilde{K} \cap P(\mbb{A}_f) = (\tilde{K} \cap N(\mbb{A}_f)) \cdot K
\]
for some sufficiently small compact open subgroup $K = K^p K_p \subset M(\mbb{A}_f)$. Suppose that $\tilde{K}$ is hyperspecial outside $S$, and that $\tilde{K}^p_S$ consists of those matrices in $\tilde{G}(\mbb{Z}_{S-\{p\}})$ which are congruent to the identity modulo $N_S$, for some integer $N_S$ divisible only by primes in $S-\{p\}$. We let $R\Gamma(X^P_{\tilde{K}^p}, \mathcal{O}/\varpi^m)^{\opn{sm}}$ denote the (derived) $\tilde{G}(\mbb{Q}_p)$-smooth vectors in the cohomology of $X^P_{\tilde{K}^p} = \opn{Ind}_{P^{\infty}}^{G^{\infty}} \ide{X}_P / \tilde{K}^p$ (as in \cite[\S5.2.1]{10author}).

\begin{proposition} \label{Prop:IndAppearsAsDirectSummand}
    Let $m \geq 1$. Then we have a $\tilde{\mbf{T}}^S$-equivariant quasi-isomorphism
    \[
    R\Gamma(X^P_{\tilde{K}^p}, \mathcal{O}/\varpi^m)^{\opn{sm}} \cong \bigoplus_{P(\mbb{Q}) \backslash \tilde{G}(\mbb{Z}_{S-\{p\}})/\tilde{K}^p_S} \opn{Ind}_{P(\mbb{Q}_p)}^{\tilde{G}(\mbb{Q}_p)} R\Gamma(X_{M, K^p}, \mathcal{O}/\varpi^m )^{\opn{sm}} 
    \]
    in the derived category of smooth representations of $\tilde{G}(\mbb{Q}_p)$. Here $\tilde{\mbf{T}}^S$ acts on the right-hand side through the Satake morphism $\mathcal{S} \colon \tilde{\mbf{T}}^S \to \mbf{T}^S$. 
\end{proposition}
\begin{proof}
    This follows from the exact same argument as in the proof of \cite[Theorem 5.4.1]{10author}, noting that $\tilde{K}^p_S$ is a normal subgroup of $\tilde{G}(\mbb{Z}_{S-\{p\}})$ (there is no place in the argument, nor in the citations to \cite{NewtonThorneTorsion}, which is specific to the authors' choice of reductive groups). 
\end{proof}

\subsection{The degree shifting argument}

Throughout this section, fix $m \geq 1$ and let $\mathcal{C}$ be a bounded from below complex in the derived category of smooth $\mathcal{O}/\varpi^m[M(\mbb{Q}_p)]$-modules. We will often view this as a representation of $P(\mbb{Q}_p)$ via inflation. We recall (see \cite[\S 5.2.1]{10author}) that the derived ordinary part of $\mathcal{C}$ is 
\[
\mathcal{C}^{\opn{ord}} = \opn{ord} R\Gamma(U_M(\mbb{Z}_p), \mathcal{C}) =  R\Gamma(U_M(\mbb{Z}_p), \mathcal{C}) \otimes^{L}_{\mathcal{O}/\varpi^m[T_M^+]} \mathcal{O}/\varpi^m[T(\mbb{Q}_p)] 
\]
where $U_M$ is the unipotent of the upper-triangular Borel in $M$. Here the action of $t \in T_M^+$ on $R\Gamma(U_M(\mbb{Z}_p), \mathcal{C})$ is through the Hecke operator
\[
[ U_M(\mbb{Z}_p) \cdot t \cdot U_M(\mbb{Z}_p)] \cdot v = \sum_{u \in U_M(\mbb{Z}_p)/t U_M(\mbb{Z}_p) t^{-1}} u t \cdot v . 
\]
We have a similar definition for smooth representations of $\tilde{G}(\mbb{Q}_p)$, by first passing to derived $U(\mbb{Z}_p)$ invariants and then localising along $T^+ \hookrightarrow T(\mbb{Q}_p)$. The goal of this section is to compute the ordinary part of the smooth induction $\opn{Ind}_{P(\mbb{Q}_p)}^{\tilde{G}(\mbb{Q}_p)} (\mathcal{C})$ following \cite[\S 5.3]{10author}. 

We recall the following functors from \emph{loc.cit.}. Let $w \in {^M W}$ and set $S_w = P(\mbb{Q}_p) \cdot w \cdot B(\mbb{Q}_p) = P(\mbb{Q}_p) \cdot w \cdot U(\mbb{Q}_p)$ and $S_w^{\circ} = P(\mbb{Q}_p) \cdot w \cdot U(\mbb{Z}_p)$. Set
\[
G_{\geq i} = \bigsqcup_{\substack{w \in {^M W} \\ l(w) \geq i}} S_w,
\]
which is an open subset of $\tilde{G}(\mbb{Q}_p)$.

\begin{definition}
    Suppose $\mathcal{C}$ is a smooth representation of $P(\mbb{Q}_p)$. We define
    \begin{align*}
        I_{\geq i}(\mathcal{C}) &= \left\{ f \colon G_{\geq i} \to \mathcal{C} : \begin{array}{c} f \text{ loc. constant, compact support mod } P(\mbb{Q}_p) \\ f(p \cdot -) = p \cdot f(-) \; \forall p \in P(\mbb{Q}_p)  \end{array} \right\} \\
        I_w(\mathcal{C}) &= \left\{ f \colon S_w \to \mathcal{C} : \begin{array}{c} f \text{ loc. constant, compact support mod } P(\mbb{Q}_p) \\ f(p \cdot -) = p \cdot f(-) \; \forall p \in P(\mbb{Q}_p)  \end{array} \right\} \\
        I^{\circ}_w(\mathcal{C}) &= \{ f \in I_w(\mathcal{C}) : \opn{supp}(f) \subset S^{\circ}_w \} .
    \end{align*}
    The first and second are representations of $B(\mbb{Q}_p)$ via right-translation. The third is a representation of $B(\mbb{Q}_p)^+ = U(\mbb{Z}_p) \cdot T^+ \cdot U(\mbb{Z}_p)$ via right-translation. Of course, $I_{\geq 0}$ is the usual smooth induction from $P(\mbb{Q}_p)$ to $\tilde{G}(\mbb{Q}_p)$, viewed as a representation of $B(\mbb{Q}_p)$. By the same argument as in the proof of \cite[Proposition 5.3.1]{10author} (see, also, \cite[\S 2]{Hauseux}), all of these functors are exact and for any smooth representation $\mathcal{C}$ of $P(\mbb{Q}_p)$, we have functorial short exact sequences 
    \[
0 \to I_{\geq i+1}(\mathcal{C}) \to I_{\geq i}(\mathcal{C}) \to \oplus_{\substack{w \in {^M W} \\ l(w) = i}} I_w(\mathcal{C}) \to 0 .
\]
Furthermore, by the same proof as in \cite[Lemma 5.3.4]{10author} (see, also, \cite[Lemme 3.3.1]{Hauseux}), we have $I^{\circ}_w(\mathcal{C})^{\opn{ord}} \cong I_w(\mathcal{C})^{\opn{ord}}$ for any bounded below complex $\mathcal{C}$ in the derived category of smooth representations of $P(\mbb{Q}_p)$.
\end{definition}

Set $U_w = P(\mbb{Q}_p) \cap w U(\mbb{Z}_p) w^{-1}$, which contains $U_M(\mbb{Z}_p)$, and recall that $N$ denotes the unipotent radical of $P$. For a character $\chi \colon T(\mbb{Q}_p) \to \mathcal{O}^{\times}$, let $\mathcal{O}(\chi)$ denote the rank one $\mathcal{O}$-module equipped with an action of $T(\mbb{Q}_p)$ through the character $\chi$.

\begin{proposition} \label{Prop:AbstractDegShifting}
Let $w \in {^MW}$ and let $\chi_w \colon T(\mbb{Q}_p) \to \mathcal{O}^{\times}$ denote the character given by
\[
\chi_w(t) = \frac{\opn{det}( \opn{Ad}(wtw^{-1}) | \opn{Lie}(N(\mbb{Q}_p) \cap w U(\mbb{Q}_p) w^{-1}) )^{-1} }{|\opn{det}( \opn{Ad}(wtw^{-1}) | \opn{Lie}(N(\mbb{Q}_p) \cap w U(\mbb{Q}_p) w^{-1}) )|_p} .
\]
Then:
    \begin{enumerate}
        \item One has a $T^+$-equivariant quasi-isomorphism $R\Gamma(U(\mbb{Z}_p), I^{\circ}_w(\mathcal{C})) \cong R\Gamma(U_w, \mathcal{C})$, where $T^+$ acts on the right-hand side through the Hecke operator $[U_w \cdot (wtw^{-1}) \cdot U_w]$.
        \item Suppose that the action of $P(\mbb{Q}_p)$ on $\mathcal{C}$ factors through $M(\mbb{Q}_p)$. Then
        \[
        I_w(\mathcal{C})^{\opn{ord}} \cong R\Gamma(U_w, \mathcal{C})^{\opn{ord}} \cong \mathcal{O}(\chi_w) \otimes_{\mathcal{O}} \tau_w^{-1} \mathcal{C}^{\opn{ord}} [l(w) - n(n+1)/2]
        \]
        where $\tau_w^{-1}$ means the action of $T(\mbb{Q}_p)$ is twisted through the map $t \mapsto wtw^{-1}$.
    \end{enumerate}
\end{proposition}
\begin{proof}
    For part (1), we follow \cite[Lemma 5.3.5]{10author}, namely we just need to show that the natural map 
    \[
    \opn{H}^0(U(\mbb{Z}_p), I^{\circ}_w(\mathcal{C})) \to \opn{H}^0(U_w, \mathcal{C})
    \]
    given by $f \mapsto f(w)$ is equivariant for the Hecke action. But this amounts to checking 
    \[
    \sum_{u \in U(\mbb{Z}_p)/t U(\mbb{Z}_p) t^{-1}} f(w \cdot u t) = \sum_{v \in U_w / (wtw^{-1}) U_w (w t w^{-1})^{-1}} v (wtw^{-1}) \cdot f(w)
    \]
    which follows from exactly the same proof as in \emph{loc.cit.}.

    For part (2), we follow \cite[Lemma 5.3.7]{10author} (note that, in our notation, the roles of $U$ and $N$ are reversed compared to \emph{loc.cit.}). Let $U_{w, N} = U_w \cap N(\mbb{Q}_p)$. The key point is to then compute $A \defeq R\Gamma(U_{w, N}, \mathcal{O}/\varpi^m)$ which is equipped with the natural action of $U_M(\mbb{Z}_p)$ (since $U_{w}/U_{w, N} = U_M(\mbb{Z}_p)$) and $t \in T^+$ via the Hecke action of $wtw^{-1}$. More precisely, we want to compute
    \[
    \alpha A = \alpha R\Gamma(U_{w, N}, \mathcal{O}/\varpi^m) \defeq R\Gamma(U_{w, N}, \mathcal{O}/\varpi^m) \otimes_{\mathcal{O}[T^+]} \mathcal{O}[T_M^+]
    \]
    where the localisation is via the map $T^+ \to T_M^+$ given by $t \mapsto wtw^{-1}$. Note that $R\Gamma(U_{w, N}, \mathcal{O}/\varpi^m)$ already carries an action of $T^{+}$ via the (non-twisted) Hecke action of $t \in T^+$ (since $t U_{w, N} t^{-1} \subset U_{w, N}$), and $\alpha A$ is then the localisation along the non-twisted map $T^+ \subset T_M^+$.

    Set $d = n(n+1)/2$ and $z_p = \opn{diag}(p, \dots, p, 1, \dots, 1) \in T^+$, where there are $n$ lots of $p$. Then this acts on $\opn{H}^i(U_{w, N}, \mathcal{O}/\varpi^{m})$ by multiplication by $p^{d-l(w)-i}$ (since $U_{w, N}$ has $\mbb{Z}_p$-rank $d - l(w)$), but applying $\alpha$ inverts this action. Hence $\alpha A$ is concentrated in degree $d-l(w)$, and \cite[Proposition 3.1.8]{Hauseux} tells us that $\alpha \opn{H}^{d-l(w)}(U_{w, N}, \mathcal{O}/\varpi^m)$ is free of rank one over $\mathcal{O}/\varpi^m$ with the action of $T_M^+$ given by the character 
    \[
    t \mapsto \frac{\opn{det}(\opn{Ad}(t) | \opn{Lie}U_{w, N})^{-1}}{|\opn{det}(\opn{Ad}(t) | \opn{Lie}U_{w, N})|_p} .
    \]
    The rest of the proposition follows as in \cite[Lemma 5.3.7]{10author}.
\end{proof}

\subsection{Transfer map on Hecke algebras}

We now specialise to the ordinary setting. We explain how to apply the degree shifting results in Proposition \ref{Prop:AbstractDegShifting} to construct a map of ordinary Hecke algebras (acting faithfully on certain  cohomology groups) from $\tilde{G}$ to $M$, and subsequently to $\GL_n$. 

Let $p > 2n$. Let $m \geq 1$ and $c \geq b \geq 0$ with $c \geq 1$, and let $\tilde{K}$ and $K$ be as in \S \ref{Subsub:IndAppearsAsDirectSummand} with the additional condition that $\tilde{K}_p = \tilde{\opn{Iw}}(b, c)$ and $K_p = \opn{Iw}(b, c)$. Set $d = n(n+1)/2$ for brevity. Let $\tilde{\lambda}$ be a dominant weight for $\tilde{G}$.  Recall $\m$ and $\tilde\m$ from \S\ref{sec:chapter 9 set-up}; we assume that $\overline{\rho}_{\tilde{\ide{m}}}$ (also defined in \S\ref{sec:chapter 9 set-up}) is decomposed generic. From the excision exact sequence,  we obtain a Hecke-equivariant map 
\begin{equation} \label{Eqn:FirstSurj}
\opn{H}^{d}( X_{\tilde{G}, \tilde{K}}, V_{\tilde{\lambda}}/\varpi^m )^{\opn{ord}}_{\tilde{\ide{m}}} \twoheadrightarrow \opn{H}^{d}( \partial \overline{X}_{\tilde{G}, \tilde{K}}, V_{\tilde{\lambda}}/\varpi^m )^{\opn{ord}}_{\tilde{\ide{m}}} 
\end{equation}
which is surjective by Yang--Zhu's vanishing result $\opn{H}^{d+1}_c(X_{\tilde{G}, \tilde{K}}, V_{\tilde{\lambda}}/\varpi^m )_{\tilde{\ide{m}}} = 0$ (Theorem \ref{Thm:YZVanishing}). Here, the ordinary part is as usual the direct summand on which the $\tilde{U}_t$-operators act invertibly. 

We will construct the transfer by applying degree-shifting to $\mathcal{C} = R\Gamma(X_{M, K^p}, \mathcal{O}/\varpi^m)^{\opn{sm}}_{\ide{m}}$. The next two lemmas isolate the weight contribution and relate the cohomologies of $\partial\overline{X}_{\tilde{G},\tilde{K}}$ and $X_{M,K}$. Let $T(b) \subset T(\mbb{Z}_p)$ denote the subgroup of elements which are congruent to the identity modulo $p^b$, and recall that $w_0^{\tilde{G}}$ (resp. $w_0^M$) denotes the longest Weyl element of $\tilde{G}$ (resp. $M$). 

\begin{lemma}\label{lem:T(b) cohomology}
Let $\tilde{\lambda}$ and $\lambda$ be $\tilde{G}$-dominant and $M$-dominant weights respectively, which we view as continuous characters $T(\mbb{Q}_p) \to \mathcal{O}^{\times}$ by precomposing with the natural map $T(\mbb{Q}_p) \twoheadrightarrow T(\mbb{Z}_p)$. We have a $\tilde{\mbf{T}}^{S,\opn{ord}}$-equivariant (resp.\ $\mbf{T}^{S, \opn{ord}}$-equivariant) quasi-isomorphism
\begin{align*}
R\Gamma( \partial \overline{X}_{\tilde{G}, \tilde{K}}, V_{\tilde{\lambda}}/\varpi^m )^{\opn{ord}} &\cong R\Gamma(T(b), \mathcal{O}(w_0^{\tilde{G}} \tilde{\lambda}) \otimes R\Gamma( \partial \overline{X}_{\tilde{G}, \tilde{K}^p}, \mathcal{O}/\varpi^m )^{\opn{ord}} ),\\
\text{(resp. }R\Gamma( X_{M, K}, V_{\lambda}/\varpi^m )^{\opn{ord}} &\cong R\Gamma(T(b), \mathcal{O}(w_0^{M} \lambda) \otimes R\Gamma(X_{M, K^p}, \mathcal{O}/\varpi^m )^{\opn{ord}})).
\end{align*}
\end{lemma}

\begin{proof}
The second quasi-isomorphism is \cite[Props.\ 5.2.15, 5.2.17]{10author}. The first is proved via the same arguments.
\end{proof}

\begin{lemma} \label{Lem:SurjectionToI>=0}
    One has a surjective morphism
    \[
\opn{H}^d\left( R\Gamma(T(b), \mathcal{O}(w_0^{\tilde{G}} \tilde{\lambda}) \otimes R\Gamma( \partial \overline{X}_{\tilde{G}, \tilde{K}^p}, \mathcal{O}/\varpi^m )^{\opn{ord}}_{\tilde{\ide{m}}} ) \right)  \twoheadrightarrow \opn{H}^d\left( R\Gamma(T(b), \mathcal{O}(w_0^{\tilde{G}} \tilde{\lambda}) \otimes I_{\geq 0}(\mathcal{C})^{\opn{ord}}) \right) .
\]
This is $\tilde{\mbf{T}}^{S, \opn{ord}}$-equivariant, where the action of $\tilde{\mbf{T}}^{S}$ on the target is via $\cS \colon \tilde{\mbf{T}}^{S} \to \mbf{T}^{S}$ and the action of $T^+$ is induced from its action on $\mathcal{O}(w_0^{\tilde{G}} \tilde{\lambda}) \otimes I_{\geq 0}(\mathcal{C})^{\opn{ord}}$.
\end{lemma}
\begin{proof}
    We follow the proof of \cite[Theorem 5.4.1]{10author}. Let $R\Gamma_{\tilde{K}^p, \opn{sm}}$ denote the derived functor of taking $\tilde{K}^p$-invariants and $\tilde{G}(\mbb{Q}_p)$-smooth vectors on representations of $\tilde{G}^{\infty} = \tilde{G}(\mbb{A}_f)$.
    
    We first claim that one has a Hecke-equivariant quasi-isomorphism: 
    \begin{align*}
    (R\Gamma_{\tilde{K}^p, \opn{sm}}R\Gamma(\opn{Ind}_{P^{\infty}}^{\tilde{G}^{\infty}}\ide{X}_P, \mathcal{O}/\varpi^m))_{\tilde{\ide{m}}} &\cong (R\Gamma_{\tilde{K}^p, \opn{sm}}R\Gamma(\partial \overline{\ide{X}}_{\tilde{G}}, \mathcal{O}/\varpi^m))_{\tilde{\ide{m}}} \\
    &= \tilde{R\Gamma}(\partial \overline{X}_{\tilde{G},\tilde{K}^p}, \mathcal{O}/\varpi^m )^{\opn{sm}}_{\tilde{\ide{m}}} .
    \end{align*}
    Indeed, this follows from the same argument as in \emph{loc.cit.} using Proposition \ref{Prop:OnlyOneBoundarySurvives}. The lemma now follows from Proposition \ref{Prop:IndAppearsAsDirectSummand}.
\end{proof}

Now, following the same argument as in \cite[Lemma 5.3.3]{10author}, for any $i \geq 0$ we have Hecke-equivariant short exact sequences
\begin{align}
0 \to \opn{H}^d&\left( R\Gamma(T(b), \mathcal{O}(w_0^{\tilde{G}} \tilde{\lambda}) \otimes I_{\geq i+1}(\mathcal{C})^{\opn{ord}}) \right) \to \opn{H}^d\left( R\Gamma(T(b), \mathcal{O}(w_0^{\tilde{G}} \tilde{\lambda}) \otimes I_{\geq i}(\mathcal{C})^{\opn{ord}}) \right) \nonumber \\ &\to \bigoplus_{l(w) = i} \opn{H}^d\left( R\Gamma(T(b), \mathcal{O}(w_0^{\tilde{G}} \tilde{\lambda}) \otimes I_{w}(\mathcal{C})^{\opn{ord}}) \right) \to 0.  \label{Eqn:I>=SES}
\end{align}

Arguing inductively, and combining with Lemmas \ref{lem:T(b) cohomology} and \ref{Lem:SurjectionToI>=0}, for each $w \in {^M}W$ we obtain a $\tilde{\mbf{T}}^{S,\opn{ord}}$-equivariant subquotient
\begin{equation}\label{eq:transfer 1}
\begin{tikzcd}
{\h^d( \partial \overline{X}_{\tilde{G}, \tilde{K}}, V_{\tilde{\lambda}}/\varpi^m )^{\opn{ord}}_{\tilde{\ide{m}}}} \arrow[r, two heads, dotted, hook] & {\opn{H}^d\left( R\Gamma(T(b), \mathcal{O}(w_0^{\tilde{G}} \tilde{\lambda}) \otimes I_{w}(\mathcal{C})^{\opn{ord}}) \right).}
\end{tikzcd}
\end{equation}
Here the notation ``$\begin{tikzcd} \! \arrow[r, two heads, dotted, hook] & \! \end{tikzcd}$'' means the target is a subquotient of the source. Now we apply degree-shifting to the image. Let $x = w_0^M w w_0^{\tilde{G}} \in {^MW}$ and recall $\lambda_x = x \star \tilde{\lambda} = x \cdot (\tilde{\lambda} + \rho) - \rho$.  Note that $l(w) = d - l(x) = n(n+1)/2 - l(x)$.

\begin{lemma}\label{lem:applying degree shifting}
We have 
\[
\opn{H}^d\left( R\Gamma(T(b), \mathcal{O}(w_0^{\tilde{G}} \tilde{\lambda}) \otimes I_{w}(\mathcal{C})^{\opn{ord}}) \right) \cong \tau_w^{-1} \opn{H}^{l(w)}( X_{M, K}, V_{\lambda_{x}}/\varpi^m )^{\opn{ord}}_{\ide{m}}.
\]
\end{lemma}
\begin{proof}
Applying Proposition \ref{Prop:AbstractDegShifting} to $\cC$, we have 
\begin{align*}
\opn{H}^d\left( R\Gamma(T(b), \mathcal{O}(w_0^{\tilde{G}} \tilde{\lambda}) \otimes I_{w}(\mathcal{C})^{\opn{ord}}) \right) &\cong \opn{H}^{l(w)}\left( R\Gamma(T(b), \mathcal{O}(w_0^{\tilde{G}} \tilde{\lambda}+\chi_w) \otimes \tau_w^{-1}\mathcal{C}^{\opn{ord}} ) \right)\\
&\cong \tau_w^{-1} \opn{H}^{l(w)}\left( R\Gamma(T(b), \mathcal{O}(w(w_0^{\tilde{G}} \tilde{\lambda}+\chi_w)) \otimes \mathcal{C}^{\opn{ord}} ) \right).
\end{align*}
Note $\chi_w$ factors through the natural map $T(\mbb{Q}_p) \twoheadrightarrow T(\mbb{Z}_p)$ and coincides with $(w^{-1}w_0^M w_0^{\tilde{G}}) \cdot \rho - \rho$ when restricted to $T(\mbb{Z}_p)$. We claim that $\lambda_x|_{T(\mbb{Z}_p)} = w_0^Mw(w_0^{\tilde{G}} \tilde{\lambda}+\chi_w)|_{T(\mbb{Z}_p)}$; then the result follows from  Lemma \ref{lem:T(b) cohomology} (as $w_0^M = (w_0^M)^{-1}$). But, after restricting to $T(\mbb{Z}_p)$, we have
\[
w_0^Mw(w_0^{\tilde{G}}\tilde{\lambda} + \chi_w) = w_0^Mww_0^{\tilde{G}}\tilde{\lambda} + w_0^{\tilde{G}}\rho - w_0^Mw\rho = x\tilde{\lambda} - \rho + x\rho = \lambda_x,
\]
as required. In the last step we have used the property $w_0^{\tilde{G}}\rho = - \rho$.
\end{proof}

Recall that $\tau_w^{-1}$ does not affect the underlying space; it means only that the $T(\Qp)$-action is twisted through the map $t \mapsto wtw^{-1}$. In particular, combining \eqref{Eqn:FirstSurj}, \eqref{eq:transfer 1} and Lemma \ref{lem:applying degree shifting}, we obtain a subquotient
\begin{equation}\label{eq:transfer 2}
\begin{tikzcd}
{\h^d( X_{\tilde{G}, \tilde{K}}, V_{\tilde{\lambda}}/\varpi^m )^{\opn{ord}}_{\tilde{\ide{m}}}} \arrow[r, two heads, dotted, hook] & {\opn{H}^{d-l(x)}( X_{M, K}, V_{\lambda_{x}}/\varpi^m )^{\opn{ord}}_{\ide{m}},}
\end{tikzcd}
\end{equation}
which is $\tilde{\mbf{T}}^{S,\opn{ord}}$-equivariant if $\tilde{\mbf{T}}^{S,\opn{ord}}$ acts on the source in the usual way, and  on the target via $\mathcal{S}_w \colon \tilde{\mbf{T}}^{S, \opn{ord}} \to \mbf{T}^{S, \opn{ord}}$.  Here recall $\mathcal{S}_w$ is the ring homomorphism coinciding with $\mathcal{S}$ away from $p$, and which at $p$ is induced from the map
\begin{align*}
\mathcal{O}[T^+] &\to \mathcal{O}[T_M^+] \\
[t] &\mapsto [w t w^{-1}] .
\end{align*}
The map on Hecke algebras induced from \eqref{eq:transfer 2} is the desired transfer map modulo $\varpi^m$. To obtain the map in characteristic zero, we pass to the inverse limit over $m$. Let
\begin{align*}
    \tilde{\mbf{T}}^{S, \opn{ord}}(b, c, \tilde{\lambda}) &\defeq \tilde{\mbf{T}}^{S, \opn{ord}}\left( \opn{H}^{d}(X_{\tilde{G}, \tilde{K}}, V_{\tilde{\lambda}})^{\opn{ord}}_{\tilde{\ide{m}}} \right) = \tilde{\mbf{T}}^{S, \opn{ord}}\left( \varprojlim_m \opn{H}^{d}(X_{\tilde{G}, \tilde{K}}, V_{\tilde{\lambda}}/\varpi^m)^{\opn{ord}}_{\tilde{\ide{m}}} \right) \\
    \mbf{T}^{S, \opn{ord}}(b, c, \lambda_x) &\defeq \mbf{T}^{S, \opn{ord}}\left( \opn{H}^{d-l(x)}(X_{M, K}, V_{\lambda_x})^{\opn{ord}}_{\ide{m}} \right) = \mbf{T}^{S, \opn{ord}}\left( \varprojlim_m \opn{H}^{d-l(x)}(X_{M, K}, V_{\lambda_x}/\varpi^m)^{\opn{ord}}_{\ide{m}} \right) .
\end{align*}
Combining everything in this section, we obtain the following:

\begin{proposition} \label{Prop:SwHomOnHecke}
    Let $p > 2n$. One has an $\mathcal{O}$-algebra homomorphism
    \[
    \mathcal{S}_w \colon \tilde{\mbf{T}}^{S, \opn{ord}}(b, c, \tilde{\lambda}) \to \mbf{T}^{S, \opn{ord}}(b, c, \lambda_x)
    \]
    induced by $\mathcal{S}_w$. 
\end{proposition}
\begin{proof}
 Note that the morphisms in (\ref{Eqn:FirstSurj}) and Lemma \ref{Lem:SurjectionToI>=0} remain surjective after passing to the limit over $m \geq 1$, because the modules are finite. Similarly, the sequence in (\ref{Eqn:I>=SES}) remains exact after passing to the limit over $m \geq 1$, for the same reason. Thus \eqref{eq:transfer 2} remains a subquotient in the limit, implying the existence of the desired morphism of Hecke algebras.
\end{proof}

We also note the following:

\begin{proposition} \label{Prop:DescFiltOnOrd}
    With notation as above, the ordinary part $\opn{H}^d(X_{\tilde{K}}^P, V_{\tilde{\lambda}})^{\opn{ord}}$ carries a Hecke-equivariant descending filtration 
    \[
    \opn{Fil}^i \opn{H}^d(X_{\tilde{K}}^P, V_{\tilde{\lambda}})^{\opn{ord}}
    \]
 with the properties that:
    \begin{itemize}
        \item $\opn{Fil}^{0}\opn{H}^d(X_{\tilde{K}}^P, V_{\tilde{\lambda}})^{\opn{ord}} = \opn{H}^d(X_{\tilde{K}}^P, V_{\tilde{\lambda}})^{\opn{ord}}$ and $\opn{Fil}^{d+1}\opn{H}^d(X_{\tilde{K}}^P, V_{\tilde{\lambda}})^{\opn{ord}} = 0$;
        \item The graded pieces for $i=0, \dots, d$ satisfy
        \[
        \opn{Fil}^i/\opn{Fil}^{i+1} = \bigoplus_{g \in P(\mbb{Q})\backslash \tilde{G}(\mbb{Z}_{S-\{p\}}) / \tilde{K}^p_S} \bigoplus_{\substack{w \in {^M W} \\ l(w) = i}} \opn{H}^{l(w)}(X_{M, K}, V_{\lambda_x})^{\opn{ord}}
        \]
        where $x = w_0^Mww_0^{\tilde{G}}$ and $\tilde{\mbf{T}}^{S, \opn{ord}}$ acts on the summands for $w$ through $\mathcal{S}_w$. 
    \end{itemize}
\end{proposition}
\begin{proof}
First we work over $\mathcal{O}/\varpi^m$. Let $\mathcal{C} = R\Gamma(X_{M,K^p}, \cO/\varpi^m)^{\mathrm{sm}}$ (i.e., we no longer localise at $\ide{m}$). We have Hecke-equivariant quasi-isomorphisms
\begin{align*}
R\Gamma(X^P_{\tilde{K}}, V_{\tilde{\lambda}}/\varpi^m )^{\opn{ord}} &\cong R\Gamma(T(b), \mathcal{O}(w_0^{\tilde{G}} \tilde{\lambda}) \otimes R\Gamma( X^P_{\tilde{K}^p}, \mathcal{O}/\varpi^m )^{\opn{ord}})\\
&\cong \bigoplus_{g\in P(\Q)\backslash \tilde{G}(\Z_{S-\{p\}})/\tilde{K}^p_S} R\Gamma\Big(T(b), \mathcal{O}(w_0^{\tilde{G}} \tilde{\lambda}) \otimes I_{\geq 0}(\cC))^{\opn{ord}}\Big);
\end{align*}
the first is directly analogous to Lemma \ref{lem:T(b) cohomology}, whilst the second is Proposition \ref{Prop:IndAppearsAsDirectSummand}. 
This inherits the desired filtration (modulo $\varpi^m$) via \eqref{Eqn:I>=SES}, using Lemma \ref{lem:applying degree shifting} to describe the graded pieces.

To get the claimed filtration in characteristic zero, we pass to the limit over $m$, noting that all of the relevant inverse systems satisfy the Mittag-Leffler property (as they are inverse systems of finite $\mathcal{O}/\varpi^m$-modules). 
\end{proof}

\subsection{Local-global compatibility}

We now discuss the local-global compatibility at $\ell = p$.

\subsubsection{Characteristic zero local-global compatibility}

Recall that $p > 2n$ and let $\ide{m} \subset \mbf{T}^S$ and $\tilde{\ide{m}} \subset \tilde{\mbf{T}}^S$ be as in \S\ref{sec:chapter 9 set-up}. Recall $\ide{m}$ is assumed to be decomposed generic and non-Eisenstein, and we assume that $\tilde{\ide{m}}$ is decomposed generic. Let $c \geq b \geq 0$ with $c \geq 1$, and let $\tilde{K}$ and $K$ be as in \S \ref{Subsub:IndAppearsAsDirectSummand} with the additional condition that $\tilde{K}_p = \tilde{\opn{Iw}}(b, c)$ and $K_p = \opn{Iw}(b, c)$.

\begin{proposition} \label{Prop:Char0LGforHecke}
    Let $\tilde{\lambda} = (\tilde{\lambda}_1, \dots, \tilde{\lambda}_n; \tilde{\lambda}_0) \in X^*(T)^+$ be a sufficiently regular weight, in the sense that $\tilde{\lambda}_i - \tilde{\lambda}_{i+1} \geq n$ for all $i=1, \dots, n-1$, and $\tilde{\lambda}_n \geq n$. Then:
    \begin{enumerate}
        \item $\opn{H}^d (X_{\tilde{G}, \tilde{K}}, V_{\tilde{\lambda}, \mathcal{O}} )_{\tilde{\ide{m}}}^{\opn{ord}}$ is $\varpi$-torsion free, and $\opn{H}^d (X_{\tilde{G}, \tilde{K}}, V_{\tilde{\lambda}, \mathcal{O}} )_{\tilde{\ide{m}}}^{\opn{ord}}[1/p]$ is a semisimple module for the Hecke algebra $\tilde{\mbf{T}}^{S, \opn{ord}}$.
        \item The Hecke algebra
        \[
        \tilde{\mbf{T}}^{S, \opn{ord}}\big(\opn{H}^d (X_{\tilde{G}, \tilde{K}}, V_{\tilde{\lambda}, \mathcal{O}} )_{\tilde{\ide{m}}}^{\opn{ord}}[1/p]\big)
        \]
        is a finite product of characteristic zero fields. Each maximal ideal $\ide{n}$ corresponds to a Hecke eigensystem appearing in $\opn{H}^d (X_{\tilde{G}, \tilde{K}}, V_{\tilde{\lambda}, L} )$ which is ordinary at $p$. Furthermore, there exists a continuous semisimple Galois representation $\rho_{\ide{n}} \colon G_{\mbb{Q}} \to \opn{GL}_{2n+1}(\Qpb)$, such that for all $\ell \not\in S$, $\rho_{\ide{n}}$ is unramified at $\ell$ and 
        \[
        \opn{det}(1 - \opn{Frob}_{\ell}^{-1} X | \rho_{\ide{n}} ) 
        \]
        equals the image of $\tilde{H}_{\ell}(X)$ in $\tilde{\mbf{T}}^{S, \opn{ord}}(\opn{H}^d (X_{\tilde{G}, \tilde{K}}, V_{\tilde{\lambda}, \mathcal{O}} )_{\tilde{\ide{m}}}^{\opn{ord}}[1/p])_{\ide{n}}[X]$. Moreover, the residual representation $\overline{\rho}_{\ide{n}}$ is isomorphic to $\overline{\rho}_{\tilde{\ide{m}}}$.
        \item One has local-global compatibility at $\ell = p$, namely
        \[
        \rho_{\ide{n}}|_{G_{\mbb{Q}_p}} \sim \left( \begin{array}{cccc} \psi_{\tilde{\lambda},1} & * & * & * \\ & \psi_{\tilde{\lambda},2} & * & * \\ & & \ddots & \vdots \\ & & & \psi_{\tilde{\lambda}, 2n+1} \end{array} \right)
        \]
        where $\psi_{\tilde{\lambda}, i} \colon G_{\mbb{Q}_p} \to \Qpb^{\times}$ is the character as in Definition \ref{DefOfUnivGaloisChars}, associated with the algebra homomorphism $\tilde{\mbf{T}}^{S, \opn{ord}}_L \to \Qpb$ with kernel $\ide{n}$. 
    \end{enumerate}
\end{proposition}
\begin{proof} 
Recall that we are assuming $ n \geq 2$. The torsion subgroup $\opn{H}^d(X_{\tilde{G}, \tilde{K}}, V_{\tilde{\lambda}, \mathcal{O}})_{\tilde{\ide{m}}}[\varpi]$ is a quotient of $\opn{H}^{d-1}(X_{\tilde{G}, \tilde{K}}, V_{\tilde{\lambda}}/\varpi)_{\tilde{\ide{m}}}$, which is zero by Theorem \ref{Thm:YZVanishing}; thus the cohomology is $\varpi$-torsion-free.

Fix an isomorphism $\mbb{C} \cong \Qpb$. Via this isomorphism, we can consider the singular cohomology
\[
\opn{H}^d(X_{\tilde{G}, \tilde{K}}, V_{\tilde{\lambda}, \mbb{C}}) \cong \opn{H}^d(X_{\tilde{G}, \tilde{K}}, V_{\tilde{\lambda}, \Qpb}).
\]
This cohomology group can be computed in terms of automorphic representations following \cite[\S 2.2]{FSdecomposition}. In particular, it has a Hecke-equivariant direct summand given by the cuspidal cohomology $\opn{H}^d_{\opn{cusp}}(X_{\tilde{G}, \tilde{K}}, V_{\tilde{\lambda}, \Qpb})$.  As $\tilde{\lambda}$ is sufficiently regular, by \cite[Corollary 2.3]{Schwermer94} we have $\opn{H}^d_{\opn{cusp}}(X_{\tilde{G}, \tilde{K}}, V_{\tilde{\lambda}, \Qpb}) = \opn{H}^d_{!}(X_{\tilde{G}, \tilde{K}}, V_{\tilde{\lambda}, \Qpb})$; explicitly, $\h^d_{\opn{cusp}}$ is the image of $\hc{d}$ in the excision exact sequence
\begin{equation}\label{eq:excision}
\hc{d}(X_{\tilde{G}, \tilde{K}}, V_{\tilde{\lambda}, \Qpb}) \to \h^d(X_{\tilde{G}, \tilde{K}}, V_{\tilde{\lambda}, \Qpb}) \to \h^d(\partial\overline{X}_{\tilde{G}, \tilde{K}}, V_{\tilde{\lambda}, \Qpb}) \to \hc{d+1}(X_{\tilde{G}, \tilde{K}}, V_{\tilde{\lambda}, \Qpb}).
\end{equation}
Now take ordinary parts and let $\ide{n}$ be a maximal ideal as in the statement. Base-changing to $\overline{\Q}_p$, we obtain a maximal ideal $\ide{b} \subset \tilde{\mbf{T}}^{S}_{\Qpb}$ lying above $\ide{n}$ and $\tilde{\ide{m}}$. Theorem \ref{Thm:YZVanishing} implies $\opn{H}^{d+1}_{c}(X_{\tilde{G}, \tilde{K}}, V_{\tilde{\lambda}, \mathcal{O}})^{\opn{ord}}_{\tilde\m} = 0,$  so the same is true after inverting $p$ and localising at $\ide{b}$; so \eqref{eq:excision} yields a short exact sequence $0 \to (\h^d_{\opn{cusp}})^{\opn{ord}}_{\ide{b}} \to (\h^d)^{\opn{ord}}_{\ide{b}} \to (\h^d_{\partial})^{\opn{ord}}_{\ide{b}} \to 0$. This splits by \cite[(7)]{FSdecomposition}, giving a Hecke-equivariant isomorphism
\begin{equation}\label{eq:cusp + boundary tildeG}
\opn{H}^d(X_{\tilde{G}, \tilde{K}}, V_{\tilde{\lambda}, \Qpb})^{\opn{ord}}_{\ide{b}} \cong \opn{H}^d_{\opn{cusp}}(X_{\tilde{G}, \tilde{K}}, V_{\tilde{\lambda}, \Qpb})^{\opn{ord}}_{\ide{b}} \oplus \opn{H}^d(\partial \overline{X}_{\tilde{G}, \tilde{K}}, V_{\tilde{\lambda}, \Qpb})^{\opn{ord}}_{\ide{b}}.
\end{equation}
 If $\ide{b}$ contributes to $\opn{H}^d(X_{\tilde{G}, \tilde{K}}, V_{\tilde{\lambda}, \Qpb})^{\opn{ord}}$, it thus contributes to at least one of cuspidal or boundary cohomology. Suppose first it contributes to the cuspidal part; then by the usual description of cuspidal cohomology, $\ide{b}$ is the prime-to-$S$ Hecke eigensystem associated with a regular algebraic cuspidal automorphic representation $\sigma$ of $\tilde{G}(\A)$ of weight $-w_0^{\tilde{G}}\tilde\lambda$ which is ordinary at $p$, and we may take $\rho_{\ide{n}} = \rho_\sigma$ via Theorem \ref{Thm:PiCARoftildeGLG}.

To prove the rest of (1) and (2), we use a similar result for boundary cohomology; namely, we'll show that if $\ide{b}$ contributes to boundary cohomology, then it arises as the pullback under $\cS$ of the Hecke eigensystem of a (uniquely determined) regular algebraic cuspidal automorphic representation $\sigma$ of $M(\A)$. 

Tensoring Proposition \ref{Prop:OnlyOneBoundarySurvives} by $\overline{\Q}_p$ and localising at $\ide{b}$ yields
\[
\opn{H}^d(\partial \overline{X}_{\tilde{G}, \tilde{K}}, V_{\tilde{\lambda}, \Qpb})_{\ide{b}}^{\opn{ord}} \cong \opn{H}^d(X^P_{\tilde{K}}, V_{\tilde{\lambda}, \Qpb})_{\ide{b}}^{\opn{ord}} .
\]
This latter space has a filtration induced from that in Proposition \ref{Prop:DescFiltOnOrd}, whose graded pieces are direct sums over ${^M}W$. Suppose that $\opn{H}^d(X^P_{\tilde{K}}, V_{\tilde{\lambda}, \Qpb})_{\ide{b}}^{\opn{ord}} \neq 0$; then some graded piece is non-zero, and hence there exists some $w \in {^M W}$ and an ideal $\ide{a} \subset \mbf{T}^S_{\Qpb}$ such that 
\[
    \mathcal{S}^*(\ide{a}) = \ide{b} \qquad \text{and}\qquad \opn{H}^{l(w)}(X_{M, K}, V_{\lambda_x, \Qpb})_{\ide{a}}^{\opn{ord}} \neq 0
\]
(recalling $x = w_0^M w w_0^{\tilde{G}}$). We now show, over two claims, that there is at most one such $w$, and that $\ide{a}$ is associated with a cuspidal automorphic representation of $M(\A)$.

\begin{claim}\label{claim:cuspidal}
Suppose $\opn{H}^d(X^P_{\tilde{K}}, V_{\tilde{\lambda}, \Qpb})_{\ide{b}}^{\opn{ord}} \neq 0$, and let $\ide{a} \subset \mbf{T}^S_{\Qpb}$  be as above. Then $\ide{a}$ is the prime-to-$S$ Hecke eigensystem associated with a regular algebraic cuspidal automorphic representation $\sigma$ of $M(\mbb{A})$ of weight $-w_0^M \lambda_x$ which is ordinary at $p$.
\end{claim}

\noindent\emph{Proof of claim:} The statements of Theorem \ref{Thm:GaloisRepForTorClass} also hold in characteristic zero (by the same arguments as in the proof of \cite[Corollary 5.4.2]{ScholzeTorsion}). This implies there exist semisimple Galois representations $\rho_{\ide{b}} \colon G_{\mbb{Q}} \to \opn{GL}_{2n+1}(\Qpb)$ and $\rho_{\ide{a}} \colon G_{\mbb{Q}} \to \opn{GL}_{n}(\Qpb)$ associated with $\ide{b}$ and $\ide{a}$ respectively satisfying local-global compatibility for $\ell \not\in S$. As $\mathcal{S}^*(\ide{a}) = \ide{b}$, by \eqref{KeyIdentityOfHeckeEqn} these compatibilities imply 
\[
\rho_{\ide{b}} \cong \rho_{\ide{a}}(-1) \oplus \mbf{1} \oplus \rho_{\ide{a}}^{\vee}(1).
\]
Since $\overline{\rho}_{\ide{b}} \cong \overline{\rho}_{\tilde{\ide{m}}} \cong \overline{\rho}_{\ide{m}}(-1) \oplus \mbf{1} \oplus \overline{\rho}_{\ide{m}}^{\vee}(1)$, we see that $\overline{\rho}_{\ide{a}}$ is isomorphic to either $\overline{\rho}_{\ide{m}}$ or $\overline{\rho}_{\ide{m}}^{\vee}(2)$; thus $\rho_{\ide{a}}$ (and its residual representation) is absolutely irreducible. This implies that
\[
\opn{H}^{l(w)}(X_{M, K}, V_{\lambda_x, \Qpb})^{\opn{ord}}_{\ide{a}} \cong \opn{H}_{!}^{l(w)}(X_{M, K}, V_{\lambda_x, \Qpb})^{\opn{ord}}_{\ide{a}} \cong \opn{H}_{\opn{cusp}}^{l(w)}(X_{M, K}, V_{\lambda_x, \Qpb})^{\opn{ord}}_{\ide{a}}.
\]
Indeed, $\ide{a}$ cannot contribute to the non-cuspidal part of the cohomology of $X_{M, K}$ via the decomposition in \cite{FSdecomposition}. If it did, then as explained in the proof of \cite[Theorem 2.4.10]{10author}, $\ide{a}$ would correspond to a Hecke eigensystem appearing in an induction of a cuspidal automorphic representation of a proper Levi subgroup of $M$, and its associated Galois representation $\rho_{\ide{a}}$ would therefore be reducible, a contradiction. Thus $\ide{a}$ contributes only to interior cohomology, giving the first isomorphism. The second isomorphism follows from \cite[Corollary 2.3]{Schwermer94} and the fact that $\lambda_x$ is regular (because $\tilde{\lambda}$ is sufficiently regular).

Claim \ref{claim:cuspidal} now follows from the usual description of cuspidal cohomology.

\medskip

We note that $\overline{\rho}_{\ide{a}} = \overline{\rho}_{\sigma}$ is non-Eisenstein and decomposed generic (because it is isomorphic to either $\overline{\rho}_{\ide{m}}$ or $\overline{\rho}_{\ide{m}}^{\vee}(2)$), hence $\rho_{\sigma}$ satisfies local-global compatibility at $\ell = p$ (Theorem \ref{Thm:Char0ExistenceOfGalRACAR}).

\begin{claim}\label{claim:at most one w}
There is at most one $w \in {^M}W$ such that $\h^{l(w)}(X_{M, K}, V_{\lambda_x, \Qpb})_{\ide{b}}^{\opn{ord}} \neq 0$.
\end{claim}

\noindent\emph{Proof of claim:} Suppose we have two distinct $w$ and $w'$ (with associated $x,x'$) with this non-vanishing property. Then as above, there exist $\ide{a}, \ide{a}'\subset \mbf{T}^S_{\overline{\Q}_p}$ such that $\ide{b} = \cS^*(\ide{a}) = \cS^*(\ide{a}')$ and
\[
\opn{H}^{l(w)}(X_{M, K}, V_{\lambda_x, \Qpb})_{\ide{a}}^{\opn{ord}} \neq 0 \text{ and } \opn{H}^{l(w')}(X_{M, K}, V_{\lambda_{x'}, \Qpb})_{\ide{a}'}^{\opn{ord}} \neq 0 .
\]
By Claim \ref{claim:cuspidal}, $\ide{a}$ and $\ide{a}'$ correspond to cuspidal automorphic representations $\sigma$ and $\sigma'$ of $M(\A)$. By strong multiplicity one for cuspidal automorphic representations of $\GL_n$, we either have $\sigma' \cong \sigma$ or $\sigma' \cong \sigma^{\vee}(2)$ (the dual of $\sigma$ twisted by $|\!|-|\!|^{2}$). The former cannot happen because $\lambda_x \neq \lambda_{x'}$; so the latter holds. This implies that in $X^*(T_{\GL_n})$ we have
\begin{align*}
\lambda_{x'} = -w_0^{\GL_n}\lambda_x + (-2, \dots, -2) = -ww_0^{\tilde{G}}\cdot(\tilde\lambda + \rho) + \big(w_0^{\GL_n}\cdot \rho\big) + (-2, \dots, -2).
\end{align*} 
Here the actions are computed in $X^*(T)$, but equality only holds after we project under $X^*(T) \to X^*(T_{\GL_n})$ (omitting the last coordinate). In $X^*(T_{\GL_n})$ we have $-w_0^{\tilde{G}}\cdot\tilde\lambda = \tilde\lambda$, so
\begin{align}\label{eq:x' to w}
\lambda_{x'} &= w\cdot(\tilde\lambda+\rho) + \big(w_0^{\GL_n}\cdot \rho\big) + (-2, \dots, -2)\\
&= \Big(w\cdot(\tilde\lambda+\rho) - \rho\Big) + \Big(\rho + \big(w_0^{\GL_n}\cdot \rho\big) + (-2, \dots, -2)\Big)= \lambda_w + (n-1, \dots, n-1),\notag
\end{align}
as $\rho = (n,\dots,1)$. In particular, since $n \geq 2$, we see $\lambda_{x'} \neq \lambda_w$, hence $x' \neq w$.

Write $\tilde{\lambda} = (\tilde{\lambda}_1, \dots, \tilde{\lambda}_n ; \tilde\lambda_0)$ with $\tilde{\lambda}_1 > \tilde{\lambda}_2 > \cdots > \tilde{\lambda}_n \geq n$ sufficiently regular (recalling $\tilde\lambda_i - \tilde\lambda_{i+1} \geq n$). Suppose that $w = w_{\cB}$ and $x' = w_{\cC}$ for $\cB,\cC \subset \{1,\dots,n\}$, with notation as in Lemma \ref{Lem:ExplicitMWset}. Explicitly, we have
\[
(\lambda_{w_{\cB}})_i = \left\{ \begin{array}{cc} (\tilde{\lambda}+\rho)_{b_i'} - \rho_i & \text{ if } 1 \leq i \leq \# \cB^c \\ -(\tilde{\lambda}+\rho)_{b_{n+1-i}} - \rho_i & \text{ if } \# \cB^c + 1 \leq i \leq n  \end{array} \right.,
\]
and similarly for $\lambda_{w_{\cC}}$. As $w \neq x'$ we have $\cB \neq \cC$, so $\cB^c \neq \cC^c$. We consider three cases:
\begin{itemize}
    \item If $\# \cC^c > \# \cB^c$, let $i = \# \cB^c + 1$. Then 
    \[
        (\tilde{\lambda}+\rho)_{c_i'} = -(\tilde{\lambda}+\rho)_{b_{n+1-i}} + (n-1) \iff \tilde{\lambda}_{c_i'} + \tilde{\lambda}_{b_{n+1-i}} = c_i' + b_{n+1-i} - n - 3,
    \]
    which can't happen as $\tilde\lambda$ is sufficiently regular; the left-hand side is $\geq 2n$, whilst the right-hand side is $\leq n -3$.
    
    \item If $\# \cB^c > \# \cC^c$, let $i = \# \cC^c + 1$. Then
    \[
        -(\tilde{\lambda}+\rho)_{c_{n+1-i}} = (\tilde{\lambda}+\rho)_{b_{i}'} + (n-1),
    \]
    impossible as the left-hand side is $<0$ whereas the right-hand side is $>0$. 
    \item If $\# \cB^c = \# \cC^c$, let $1 \leq i \leq \# \cB^c$ such that $b_i' \neq c_i'$. Then  $(\tilde{\lambda}+\rho)_{c_i'} = (\tilde{\lambda}+\rho)_{b'_{i}} + (n-1) \iff$
    \begin{equation}\label{eq:third case}
     \tilde{\lambda}_{c_i'} - \tilde{\lambda}_{b'_{i}} = c_i' - b_i' + n -1.
    \end{equation}
    As $\tilde\lambda$ is sufficiently regular, we have $|\tilde\lambda_r - \tilde\lambda_s| \geq |r-s|n$ for all $r,s \in \{1,\dots,n\}$. The equality \eqref{eq:third case} implies that
    \[
    |c_i' - b_i'|n \leq |\tilde{\lambda}_{c_i'} - \tilde{\lambda}_{b_i'}| =  |c_i' - b_i' + n-1| \leq |c_i' - b_i'| + n-1 ,
    \]
    hence $|c_i' - b_i'| = 1$ and $|\tilde\lambda_{c_i'} - \tilde\lambda_{b_i'}| = n$. This is incompatible with \eqref{eq:third case}, noting that by regularity $\tilde\lambda_r-\tilde\lambda_s$ and $r-s$ have different signs if $r\neq s$.
\end{itemize} 
Thus $\lambda_{x'} \neq \lambda_w + (n-1,...,n-1)$, whence $\sigma' \not\cong \sigma^\vee(2)$. By contradiction, Claim \ref{claim:at most one w} holds.

\medskip

We now complete the proof of Proposition \ref{Prop:Char0LGforHecke}. Combining \eqref{eq:cusp + boundary tildeG}, the filtration from Proposition \ref{Prop:DescFiltOnOrd}, and both claims, we have
\begin{equation}\label{eq:everything cuspidal}
\opn{H}^d(X_{\tilde{G}, \tilde{K}}, V_{\tilde{\lambda}, \Qpb})^{\opn{ord}}_{\ide{b}} = \opn{H}_{\opn{cusp}}^d(X_{\tilde{G}, \tilde{K}}, V_{\tilde{\lambda}, \Qpb})^{\opn{ord}}_{\ide{b}} \oplus \bigoplus_{g \in P(\mbb{Q})\backslash \tilde{G}(\mbb{Z}_{S-\{p\}}) / \tilde{K}^p_S} \opn{H}^{l(w)}_{\opn{cusp}}(X_{M, K}, V_{\lambda_x, \Qpb})^{\opn{ord}}_{\ide{a}}
\end{equation}
for some $w \in {^MW}$ and $\ide{a}$ as above (where the second summand is zero if no such $w$ exists). Thus $\ide{b}$ arises from a regular algebraic cuspidal automorphic representation $\sigma$ of $\tilde{G}(\mbb{A})$ (if it contributes to the first summand) or $M(\mbb{A})$ (the second).  As the ordinary part of $\sigma$ is at most 1-dimensional, we see that $\tilde{\mbf{T}}^{S, \opn{ord}}$ acts semisimply on $\opn{H}^d(X_{\tilde{G}, \tilde{K}}, V_{\tilde{\lambda}, \Qpb})^{\opn{ord}}_{\ide{b}}$, hence on $\opn{H}^d (X_{\tilde{G}, \tilde{K}}, V_{\tilde{\lambda}, \mathcal{O}} )_{\tilde{\ide{m}}}^{\opn{ord}}[1/p]$, completing the proof of (1). 

Finally, let $\ide{n}$ be as in the statement of the proposition. Base-changing to $\overline{\Q}_p$ and using the above, we see $\ide{n}$ arises from some $p$-ordinary regular algebraic cuspidal automorphic representation $\sigma$ of $\tilde{G}(\A)$ or $M(\A)$ via the above. The first two parts of (2) follow. If $\ide{n}$ contributes to the cuspidal cohomology, as above we let $\rho_{\ide{n}} = \rho_\sigma$ and conclude (2) and (3) by Theorem \ref{Thm:PiCARoftildeGLG}. If $\ide{n}$ contributes to the boundary cohomology, then we set
\[
\rho_{\ide{n}} = \rho_{\sigma}(-1) \oplus \mbf{1} \oplus \rho_{\sigma}^{\vee}(1) .
\]
By construction this satisfies the properties desired in (2). It remains to deduce (3) in this case. As explained above, the Galois representation $\rho_{\sigma}$ satisfies local-global compatibility at $\ell = p$ by Theorem \ref{Thm:Char0ExistenceOfGalRACAR}. Since $w \in {^M W}$, we may re-order the filtration on $\rho_{\ide{n}}|_{G_{\mbb{Q}_p}}$ so that all the desired properties are satisfied. More precisely the ordered set of parameters 
\[
\{ \psi_{\tilde{\lambda},n+2}, \dots, \psi_{\tilde{\lambda}, 2n+1}, \psi_{\tilde{\lambda}, 1}, \dots, \psi_{\tilde{\lambda}, n} \}
\]
for $\rho_{\ide{n}}$ (excluding the trivial parameter) is equal to the ordered set
\[
\{ \zeta_{w(1)}, \zeta_{w(2)}, \dots, \zeta_{w(n)}, \zeta_{w(n+1)}, \dots, \zeta_{w(2n)} \} 
\]
where 
\[
\{ \zeta_{1}, \dots, \zeta_{2n} \} = \{ \chi_{\lambda_x, 1} \otimes \omega^{-1}, \dots, \chi_{\lambda_x, n} \otimes \omega^{-1}, \chi_{\lambda_x, n}^{-1} \otimes \omega, \dots, \chi_{\lambda_x,1}^{-1} \otimes \omega \} 
\]
as ordered sets. Here $\chi_{\lambda_x, i}$ are the characters associated with $\sigma$ as in Definition \ref{DefOfUnivGaloisChars}.
\end{proof}

We note two important consequences of this proposition. After inverting $p$, we obtain a Galois representation attached to $\tilde{\m}$ valued in the Hecke algebra satisfying local-global compatibility at $\ell\not\in S$ and at $p$. Integrally, whilst we don't have a full representation, we can construct a Galois determinant.

\begin{corollary} \label{Cor:Char0HeckeContDeterminant}
Consider notation and assumptions as in Proposition \ref{Prop:Char0LGforHecke}. 
\begin{itemize}
\item[(1)] After inverting $p$, there exists a continuous semisimple Galois representation
\[
    \sigma_{\tilde{\m}} : G_{\Q} \to \GL_{2n+1}\Big(\tilde{\mbf{T}}^{S,\opn{ord}}\Big(\h^d(X_{\tilde{G},\tilde{K}}, V_{\tilde{\lambda},\cO})^{\opn{ord}}_{\tilde{\m}}\big[\tfrac{1}{p}\big]\Big)\Big)
\]
such that
    \begin{itemize}
        \item[(a)] for all $\ell \notin S$, the (reverse) characteristic polynomial $\opn{det}(1-\opn{Frob}_{\ell}^{-1}X|\sigma_{\tilde{\m}})$ is equal to the image of $\tilde{H}_{\ell}(X)$ in $\tilde{\mbf{T}}^{S,\opn{ord}}(\h^d(X_{\tilde{G},\tilde{K}}, V_{\tilde{\lambda},\cO})^{\opn{ord}}_{\tilde{\m}}[\tfrac{1}{p}])[X]$;
        \item[(b)] one has
        \[
        \sigma_{\tilde{\m}}|_{G_{\mbb{Q}_p}} \sim \left( \begin{array}{cccc} \psi_{\tilde{\lambda},1} & * & * & * \\ & \psi_{\tilde{\lambda}, 2} & * & * \\ & & \ddots & \vdots \\ & & & \psi_{\tilde{\lambda},2n+1} \end{array} \right),
        \]
         where $\psi_{\tilde{\lambda}, i} \colon G_{\mbb{Q}_p} \to \tilde{\mbf{T}}^{S,\opn{ord}}\Big(\h^d(X_{\tilde{G},\tilde{K}}, V_{\tilde{\lambda},\cO})^{\opn{ord}}_{\tilde{\m}}\big[\tfrac{1}{p}\big]\Big)^\times$ is as in Definition \ref{DefOfUnivGaloisChars}.
    \end{itemize}

\item[(2)] Integrally, there exists a continuous $(2n+1)$-dimensional determinant
    \[
    \tilde{D} \colon G_{\mbb{Q}} \to \tilde{\mbf{T}}^{S, \opn{ord}}( \opn{H}^d(X_{\tilde{G}, \tilde{K}}, V_{\tilde{\lambda}, \mathcal{O}})_{\tilde{\ide{m}}}^{\opn{ord}})
    \]
    such that $\tilde{D} = \opn{det}(\sigma_{\tilde{\m}})$.
    \end{itemize}
\end{corollary}
\begin{proof}
For (1), we take $\sigma = \oplus_{\ide{n}} \rho_{\ide{n}}$, where $\ide{n}$ ranges over the maximal ideals as Proposition \ref{Prop:Char0LGforHecke}.

For (2), by Proposition \ref{Prop:Char0LGforHecke}(1/2), we have an embedding
    \[
    \tilde{\mbf{T}}^{S, \opn{ord}}( \opn{H}^d(X_{\tilde{G}, \tilde{K}}, V_{\tilde{\lambda}, \mathcal{O}})_{\tilde{\ide{m}}}^{\opn{ord}}) \hookrightarrow \prod_{\ide{n}} \tilde{\mbf{T}}^{S, \opn{ord}}( \opn{H}^d(X_{\tilde{G}, \tilde{K}}, V_{\tilde{\lambda}, \mathcal{O}})_{\tilde{\ide{m}}}^{\opn{ord}}[1/p])_{\ide{n}} . 
    \]
    The result then follows from \cite[Example 2.32]{ChenevierPadic}.
\end{proof}

\subsubsection{Some important lemmas}

We pause to give two important technical lemmas.

For a weight $\lambda = (\lambda_1, \dots, \lambda_n) \in X^*(T_{\GL_n})$ and integer $a \in \mbb{Z}$, let $\lambda(a) \in X^*(T)$ denote the weight
\[
\lambda(a) = (\lambda(a)_1, \dots, \lambda(a)_n; 0) = (\lambda_1 + a, \dots, \lambda_n + a; 0) .
\]

\begin{lemma}[{c.f., \cite[Lemma 5.4.8]{10author}}] \label{Lem:choosinglambdawa}
    Let $m \geq 1$ be an integer. Then there exists a $\GL_n$-dominant weight $\lambda \in X^*(T_{\GL_n})$ such that:
    \begin{enumerate}
        \item $\mathcal{O}(\lambda)/\varpi^m \cong \mathcal{O}/\varpi^m$;
        \item For all $0 \leq j \leq d$, there exists $a_j \in (p-1)\mbb{Z}$ and $w_j \in {^MW}$ with $l(w_j) = d-j$, such that
        \[
        \tilde{\lambda}_j \defeq x_j^{-1} \star \lambda(a_j) = x_j^{-1} \cdot (\lambda(a_j) + \rho) - \rho
        \]
        is $\tilde{G}$-dominant. Here, $x_j = w_0^M w_j w_0^{\tilde{G}}$ and $\rho$ denotes the half-sum of the positive roots of $\tilde{G}$.
        \item The weight $\tilde{\lambda}_j$ in (2) is sufficiently regular as in Proposition \ref{Prop:Char0LGforHecke}.
    \end{enumerate}
\end{lemma}
\begin{proof}
    Recall the explicit description of ${^MW}$ from Lemma \ref{Lem:ExplicitMWset}, and note that $w_0^M w_{\mathcal{B}} w_0^{\tilde{G}} = w_{{\mathcal{B}}^c}$. Let $1 \leq j \leq d$, and let
    \[
    {\mathcal{B}}_j = \{ 1, 2, \dots, y-1, y , z \}
    \]
    where $0 \leq y < z \leq n+1$ are integers such that $j = (y+1)(n+1) - (y+1)y/2 - z$ (if $y=0$, then ${\mathcal{B}}_j = \{ z \}$, and if $z=n+1$, then ${\mathcal{B}}_j = \{1, \dots, y \}$). We set ${\mathcal{B}}_0 = \varnothing$. Let $a \in \mbb{Z}$ be an integer, and set $x_j = w_{{\mathcal{B}}_j}$ and $w_j = w_{{\mathcal{B}}_{j}^c}$. 

    Let $C \gg n$ be a sufficiently large integer (depending only on $n$) such that the absolute value of any entry of $w^{-1} \cdot \rho - \rho$ (written as $(v_1, \dots, v_n; 0)$, say) for any $w \in {^MW}$ is $< C$. Let $M > 8(C+n)$ be a non-negative integer divisible by $8n(p-1)\# (\mathcal{O}/\varpi^m)^{\times}$. We claim that the weight 
    \[
    \lambda = ((2^n - 2)M, (2^n - 4)M, (2^n - 8)M, \dots, 0 ) 
    \]
    works (i.e., $\lambda = (\lambda_1, \dots, \lambda_n)$ with $\lambda_i = (2^n - 2^i)M$). Clearly part (1) of the lemma is satisfied for this weight, so we will check properties (2) and (3).

    We first calculate $x_j^{-1} \cdot \lambda(a)$ using the description in Lemma \ref{Lem:ExplicitMWset}. We have:
    \begin{align*} 
    x_j^{-1} \cdot \lambda(a) &= (\lambda(a)_{x_j(1)}, \dots, \lambda(a)_{x_j(n)}; *) \\
     &= (-\lambda(a)_{n}, \dots, -\lambda(a)_{n+1-y}, \lambda(a)_{1}, \dots, \lambda(a)_{z-y-1}, -\lambda(a)_{n-y}, \lambda(a)_{z-y}, \dots, \lambda(a)_{n-y-1}; *) .
    \end{align*} 
    If $j=0$ (resp. $j=d$), then this is $\tilde{G}$-dominant if only if $a \geq 0$ (resp. $a \leq -(2^n - 2)M$). Clearly, in both cases, we can find $a_j \in (p-1)\mbb{Z}$ such that $x_j^{-1} \star \lambda(a_j)$ is dominant and sufficiently regular (because $2(M-C) \geq n$). We therefore assume that $j \neq 0,d$.
    
    The weight $x_j^{-1} \cdot \lambda(a)$ is $\tilde{G}$-dominant if and only if the following four inequalities are satisfied:\footnote{Here, in any of these inequalities, if the argument on the left-hand side (resp. right-hand side) is undefined, we replace it with $+\infty$ (resp. $0$).}
    \begin{align*}
        -(2^n - 2^{n+1-y})M - a &\geq (2^n - 2)M + a, \\
        (2^n - 2^{z-y-1})M + a &\geq -(2^n - 2^{n-y})M - a, \\
        -(2^n - 2^{g-y})M - a &\geq (2^n - 2^{z-y})M + a, \\
        (2^n - 2^{g-y-1})M + a &\geq 0 .
    \end{align*}
    If $y = 0$ (so $z \leq n$ by our above assumption), then we see that this is implied if 
    \[
    -(2^n - 2^{z-1}) \tfrac{M}{2} \leq a \leq -(2^n-2^{z})\tfrac{M}{2} .
    \]
    Note that this range has length $\geq M/2$. In particular, we can take $a$ to be an integer in this range congruent to $M/8$ modulo $M/2$, and $\tilde{\lambda}_j$ is dominant and sufficiently regular (because one must have $ -(2^n - 2^{z-1}) \tfrac{M}{2} + C + g \leq a \leq -(2^n-2^{z})\tfrac{M}{2}- C-n$ by the condition $M > 8(C+n)$). 

    Finally, we assume that $1 \leq y \leq n-1$ and $z > y+1$. Then the four inequalities are implied by 
    \begin{align*}
        -(2^{n+1}-2^{n-y}-2^{z-y-1})\tfrac{M}{2} \leq a \leq -(2^{n+1}-2^{n-y}-2^{z-y})\tfrac{M}{2} & \quad \text{ if } z \neq n+1 \\
        -(2^{n+1} - 2^{n+1-y})\tfrac{M}{2} \leq a \leq -(2^{n+1} - 2^{n+1-y}-2)\tfrac{M}{2} & \quad \text{ if } z = n+1
    \end{align*}
    and these ranges again have size $\geq M/2$. We can then take $a_j$ in this range congruent to $M/8$ modulo $M/2$, and for the same reasons above, $\tilde{\lambda}_j$ will be dominant and sufficiently regular.
\end{proof}

\begin{lemma} \label{Lem:CharacterAvoidance}
    Let $p$ be an odd prime. Let $S_0$ be a finite set of primes containing $p$ and let $\ell \not\in S_0$. Then, for any finite set $T \subset \overline{\mbb{F}}^{\times}_p$, there exists a finite set $S$ of primes containing $S_0$ but not containing $\ell$, and a finite-order character $\psi \colon G_{\mbb{Q}} \to \overline{\mbb{F}}_p^{\times}$ such that
    \begin{itemize}
        \item $\psi$ is unramified outside $S \backslash S_0$, and
        \item $\psi(\opn{Frob}_{\ell}), \psi(\opn{Frob}_p) \not\in T$.
    \end{itemize}
\end{lemma}
\begin{proof}
    Let $M \geq 1$ be the largest integer such that $T$ contains an $M$-th root of unity in $\overline{\mbb{F}}_p^{\times}$ (note that $M$ always exists). Let $q$ be an odd prime with the following properties:
    \begin{itemize}
        \item $q$ is coprime to $\ell p$,
        \item $q \not\equiv 1$ modulo $p$, and
        \item $\ell^{M} < q$ and $p^{M} < q$.
    \end{itemize}
    Such a prime exists by the Chebotarev density theorem applied to $\mbb{Q}(\mu_p)$. Let $a$ denote the order of $\ell \in \left(\mbb{Z}/q\mbb{Z} \right)^{\times}$ and $b$ the order of $p \in \left(\mbb{Z}/q\mbb{Z}\right)^{\times}$. Note that $a, b$ are both prime to $p$ by the second bullet point, and $a,b > M$ by the third bullet point. 

    Let $S = S_0 \cup \{ q \}$ and let $\chi$ be a primitive Dirichlet character modulo $q$ which maps a fixed choice of primitive root modulo $q$ to a fixed choice of $(q-1)$-th root of unity $\zeta_{q-1} \in \overline{\mbb{Q}}^{\times}$. Then $\chi(\ell)$ is a primitive $a$-th root of unity and $\chi(p)$ is a primitive $b$-th root of unity. Let $\psi = \bar{\chi}$ denote the reduction of $\chi$ modulo $p$ (viewed as a Galois character in the usual way). Then as $a, b$ are prime to $p$, both $\psi(\opn{Frob}_\ell)$ and $\psi(\opn{Frob}_p)$ are roots of unity in $\overline{\mbb{F}}_p^{\times}$ of order $> M$, so not contained in $T$.
\end{proof}

\subsubsection{Mixed characteristic local-global compatibility}

We now discuss the mixed-characteristic local-global compatibility. Let $S$ be a finite set of primes containing $p$, and let $\ide{m}_0 \subset \mbf{T}_{\opn{GL}_n}^S$ denote a maximal ideal with residue field $k = \mathcal{O}/\varpi$. We suppose that there exists a continuous semisimple Galois representation
\[
\overline{\rho}_{\ide{m}_0} \colon G_{\mbb{Q}} \to \opn{GL}_n(k)
\]
associated with $\ide{m}_0$ as in Assumption \ref{GaloisTypeAssumption}. If $\psi \colon G_{\mbb{Q}} \to \mathcal{O}^{\times}$ is a continuous finite-order character which is unramified outside $S$, then we let $\ide{m}_0(\psi) \subset \mbf{T}_{\opn{GL}_n}^S$ denote the twist of $\ide{m}_0$ by $\psi$. Explicitly, if $\Theta \colon \mbf{T}^S_{\opn{GL}_n} \to k$ is the $\mathcal{O}$-algebra homomorphism with kernel $\ide{m}_0$, then $\ide{m}_0(\psi)$ is the kernel of 
\[
[\opn{GL}_n(\widehat{\mbb{Z}}^S) \cdot m \cdot \opn{GL}_n(\widehat{\mbb{Z}}^S)] \mapsto \overline{\psi}(\opn{det}m) \cdot \Theta([\opn{GL}_n(\widehat{\mbb{Z}}^S) \cdot m \cdot \opn{GL}_n(\widehat{\mbb{Z}}^S)]), \quad \quad m \in \opn{GL}_n(\mbb{A}_f^S) ,
\]
where $\overline{\psi}$ denotes the mod $\varpi$ reduction of $\psi$. Note that one has a Galois representation associated with $\ide{m}_0(\psi)$ as in Assumption \ref{GaloisTypeAssumption}, namely $\overline{\rho}_{\ide{m}_0(\psi)} = \overline{\rho}_{\ide{m}_0} \otimes \overline{\psi}$.

\begin{proposition}[{c.f.\ \cite[Proposition 5.4.18]{10author}}] \label{Prop:MixedCharLGcompat}
    Let $p > 2n$. Suppose that $\overline{\rho}_{\ide{m}_0}$ is absolutely irreducible and decomposed generic. Let $c \geq b \geq 0$ with $c \geq 1$, and let $K \subset \GL_n(\mbb{A}_f)$ denote a neat compact open subgroup which is hyperspecial outside $S$ and equal to $\opn{Iw}(b, c)$ at $p$. Let $m \geq 1$ and let $\lambda$, $\{a_j \}$ and $\{w_j\}$ be as in Lemma \ref{Lem:choosinglambdawa}. Let $1 \leq j \leq d$, and set
    \[
    \mbf{T}_{\GL_n}^{S, \opn{ord}}(b, c, \lambda; j) = \mbf{T}_{\GL_n}^{S, \opn{ord}}\left( \opn{H}^{d-j}(X_{\GL_n, K}, V_{\lambda, \mathcal{O}})^{\opn{ord}}_{\ide{m}_0} \right) .
    \] 
    Then there exists a nilpotent ideal $J \subset \mbf{T}_{\GL_n}^{S, \opn{ord}}(b, c, \lambda; j)$ (with nilpotence degree only depending on $n$), and a continuous semisimple representation 
    \[
    \varrho_{j} \colon G_{\mbb{Q}} \to \GL_n(\mbf{T}_{\GL_n}^{S, \opn{ord}}(b, c, \lambda; j)/J)
    \]
    such that:
    \begin{enumerate}
        \item For every $\ell \notin S$, $\varrho_{j}$ is unramified at $\ell$ and the characteristic polynomial
        \[
        \opn{det}(1 - \opn{Frob}_{\ell}^{-1} X | \varrho_j)
        \]
        equals the image of the Hecke polynomial $H_{\ell}(X)$ in $(\mbf{T}_{\GL_n}^{S, \opn{ord}}(b, c, \lambda; j)/J)[X]$.
        \item For every $x \in G_{\mbb{Q}_p}$, the characteristic polynomial of $\varrho_j(x)$ equals $\prod_{i=1}^n (X - \chi_{\lambda, i}(x))$. Here we are viewing the characters $\chi_{\lambda, i}$ from Definition \ref{DefOfUnivGaloisChars} as being valued in $(\mbf{T}_{\GL_n}^{S, \opn{ord}}(b, c, \lambda; j)/J)^{\times}$.
        \item For each $x_1, \dots, x_n \in G_{\mbb{Q}_p}$, we have
        \[
        (\varrho_j(x_1) - \chi_{\lambda, 1}(x_1)) \cdots (\varrho_j(x_n) - \chi_{\lambda, n}(x_n)) = 0 .
        \]
    \end{enumerate}
\end{proposition}
\begin{proof}
    We first note that a Galois representation $\varrho_{j}$ satisfying (1) already exists from the work of Scholze \cite{ScholzeTorsion}. We therefore want to verify (2) and (3) for this representation. Since this only concerns the behaviour at $p$, we are free to enlarge $S$, increase $\mathcal{O}$ and deepen the level $K$ at primes in $S - \{p\}$ as we wish. Here we are using the fact that semisimple Galois representations over the Henselian semi-local $\mathcal{O}$-algebra $\mbf{T}_{\GL_n}^{S, \opn{ord}}(b, c, \lambda; j)$, reducing to $\overline{\rho}_{\ide{m}_0}$ modulo each of its maximal ideals, are determined (up to isomorphism) by their associated determinant (see \cite[Theorem 2.22]{ChenevierPadic}). 

    The first step is to set things up so that our running assumptions in the previous sections are satisfied. Suppose that $\ell \not\in S$ is an odd prime where $\overline{\rho}_{\ide{m}_0}$ is decomposed generic. After possibly enlarging $\mathcal{O}$ and $S$, we can find a finite-order character $\psi \colon G_{\mbb{Q}} \to \mathcal{O}^{\times}$, unramified outside $S - \{p\}$, such that 
    \begin{equation} \label{Eqn:TwistedDecomGen}
    (\overline{\rho}_{\ide{m}_0} \otimes \overline{\psi} \otimes \omega^{-1}) \oplus \mbf{1} \oplus (\overline{\rho}_{\ide{m}_0}^{\vee} \otimes \overline{\psi}^{-1} \otimes \omega )
    \end{equation}
    is decomposed generic at $\ell$ (because $\overline{\rho}_{\ide{m}_0}$ is). Moreover, we may assume that $(\overline{\rho}_{\ide{m}_0} \otimes \overline{\psi} \otimes \omega^{-1})|_{G_{\mbb{Q}_p}}$, $\mbf{1}|_{G_{\mbb{Q}_p}}$, and $(\overline{\rho}_{\ide{m}_0}^{\vee} \otimes \overline{\psi}^{-1} \otimes \omega )|_{G_{\mbb{Q}_p}}$ have no Jordan--H\"older factors in common with each other. Indeed, the existence of $\psi$ follows from Lemma \ref{Lem:CharacterAvoidance} for the set 
    \[
    T = \{ \beta_i^{-1}, \beta_i^{-1}\ell^2, \pm \beta_i^{-1} \ell^{1/2}, \pm \beta_i^{-1} \ell^{3/2}, \pm \alpha_i^{-1} \}
    \]
    where $\beta_1, \dots, \beta_n$ are the eigenvalues of $\opn{Frob}_{\ell}$ on $\overline{\rho}_{\ide{m}_0}$, and $\alpha_1, \dots, \alpha_n$ are the eigenvalues of a fixed choice of lift of Frobenius $\opn{Frob}_p \in G_{\mbb{Q}_p}$ at $p$ on $\overline{\rho}_{\ide{m}_0}$. Note that $\ide{m}_0(\psi)$ satisfies the same assumptions as $\ide{m}_0$ after possibly deepening the level $K$ at primes in $S-\{p\}$. We will prove the result for this twist of $\ide{m}_0$, then untwist to get the result we actually want. 

    After enlarging $\mathcal{O}$ and increasing $S$, we may assume that there exists a character $\chi \colon G_{\mbb{Q}} \to \mathcal{O}^{\times}$ which is unramified outside $S - \{ p\}$, such that the representations $[(\ref{Eqn:TwistedDecomGen}) \otimes \bar{\chi}]|_{G_{\mbb{Q}_p}}$ and $[(\ref{Eqn:TwistedDecomGen}) \otimes \bar{\chi}^{-1}]|_{G_{\mbb{Q}_p}}$ have no common Jordan--H\"older factors, and 
    \[
    (\overline{\rho}_{\ide{m}_0} \otimes \overline{\psi}\bar{\chi} \otimes \omega^{-1}) \oplus \mbf{1} \oplus (\overline{\rho}_{\ide{m}_0}^{\vee} \otimes \overline{\psi}^{-1}\bar{\chi}^{-1} \otimes \omega )
    \]
    is decomposed generic at $\ell$. We may also assume that $\bar{\chi}|_{G_{\mbb{Q}_p}}$ is non-trivial and doesn't appear in $(\overline{\rho}_{\ide{m}_0} \otimes \overline{\psi} \otimes \omega^{-1})|_{G_{\mbb{Q}_p}}$. Indeed, the existence of $\chi$ follows from Lemma \ref{Lem:CharacterAvoidance} for the set
    \begin{align*}
    T = \{ \pm 1, \pm \alpha_i^{\pm 1/2} \bar{\psi}(\opn{Frob}_p)^{\pm 1/2},& \pm \alpha_i^{\pm 1} \bar{\psi}(\opn{Frob}_p)^{\pm 1}, \beta_i^{-1}\bar{\psi}(\opn{Frob}_{\ell})^{-1}, \beta_i^{-1}\bar{\psi}(\opn{Frob}_{\ell})^{-1}\ell^2 , \\ & \pm \beta_i^{-1}\bar{\psi}(\opn{Frob}_{\ell})^{-1} \ell^{1/2}, \pm \beta_i^{-1}\bar{\psi}(\opn{Frob}_{\ell})^{-1} \ell^{3/2} \}
    \end{align*}
    where $\{ \alpha_i \}$ and $\{ \beta_i \}$ are as above. We can deepen $K$ to include the ramification of $\chi$.

    \begin{notation*}
        Set $\ide{M}_0 \subset \mbf{T}^S_{\opn{GL}_n}$ to be either $\ide{m}_0(\psi)$ or $\ide{m}_0(\psi\chi)$. Here, $S$ is a finite set independent of the choice, which contains $p$ and all primes where $\psi$ and $\chi$ ramify. We let $K_M \subset M(\mbb{A}_f) = \GL_n(\mbb{A}_f) \times \opn{GL}_1(\mbb{A}_f)$ be a sufficiently small compact open subgroup, which decomposes as $K_M = K \times C$ for some compact open subgroup $C \subset \opn{GL}_1(\mbb{A}_f)$. After possibly shrinking $K$ at primes in $S - \{p\}$ and choosing $C$ appropriately, we may assume that there exists a sufficiently small compact open subgroup $\tilde{K} \subset \tilde{G}(\mbb{A}_f)$ such that 
    \[
    \tilde{K} \cap P(\mbb{A}_f) = (\tilde{K} \cap N(\mbb{A}_f)) K_M
    \]
    and $\tilde{K}_p = \tilde{\opn{Iw}}(b, c)$. We can also assume that $\tilde{K}$ is hyperspecial outside $S$ and $\tilde{K}^p_S$ consists of matrices in $\tilde{G}(\mbb{Z}_{S-\{p\}})$ which are congruent to the identity modulo $N_S$, for some integer $N_S$ divisible only by primes in $S$. In other words $\tilde{K}$ and $K$ satisfy the conditions in \S \ref{Subsub:IndAppearsAsDirectSummand}. Finally, $\overline{\rho}_{\ide{M}_0}$ is absolutely irreducible and decomposed generic, and $\overline{\rho}_{\ide{M}_0}(-1) \oplus \mbf{1} \oplus \overline{\rho}^{\vee}_{\ide{M}_0}(1)$ is decomposed generic.
    \end{notation*}
    
    Let $\ide{m} \subset \mbf{T}^S$ denote the preimage of $\ide{M}_0$ under the natural map $\mbf{T}^S \to \mbf{T}^S_{\opn{GL}_n}$ induced from the projection $M = \opn{GL}_n \times \opn{GL}_1 \to \opn{GL}_n$. Set $\tilde{\ide{m}} = \mathcal{S}^*(\ide{m})$ and note that $\overline{\rho}_{\tilde{\ide{m}}}$ is isomorphic to $\overline{\rho}_{\ide{M}_0}(-1) \oplus \mbf{1} \oplus \overline{\rho}^{\vee}_{\ide{M}_0}(1)$.

    Note that the (real) dimension of the locally symmetric space associated with $M$ is $d-1 = n(n+1)/2-1$. The K\"{u}nneth formula provides us with an isomorphism
    \[
    \opn{H}^{d-j}(X_{M, K_M}, V_{\lambda(a_j), \mathcal{O}}) \cong \opn{H}^{d-j}(X_{\GL_n, K}, V_{\lambda(a_j), \mathcal{O}}) \otimes_{\mathcal{O}} \opn{H}^0(X_{\opn{GL}_1, C}, \mathcal{O}) .
    \]
    Note $\opn{H}^0(X_{\opn{GL}_1, C}, \mathcal{O})$ is a finite free $\mathcal{O}$-module with rank equal to the number of connected components of $X_{\opn{GL}_1, C}$, and contains the trivial representation $\mathcal{O}$ of the Hecke algebra for $\opn{GL}_1$. This implies that we have a natural map
    \[
    \opn{pr}_j \colon \mbf{T}^{S, \opn{ord}}(\opn{H}^{d-j}(X_{M, K_M}, V_{\lambda(a_j), \mathcal{O}})^{\opn{ord}}_{\ide{m}}) \twoheadrightarrow \mbf{T}^{S, \opn{ord}}_{\GL_n}(\opn{H}^{d-j}(X_{\GL_n, K}, V_{\lambda(a_j), \mathcal{O}})^{\opn{ord}}_{\ide{M}_0})
    \]
    induced from mapping a Hecke operator for $M$ to one for $\GL_n$ (via the projection $M \twoheadrightarrow \GL_n$).

    Let $\chi_{\opn{cycl}}(x) = |\!|x|\!| x_p$ for $x \in \mbb{A}_f^{\times}$ denote the adelic $p$-adic cyclotomic character. By precomposing with the natural map $\GL_n \xrightarrow{\opn{det}^{-a_j}} \opn{GL}_1$, this character gives rise to a class in $\opn{H}^0(X_{\GL_n, K}, V_{\mu(a_j), \mathcal{O}})$, where $\mu(a_j) = (a_j, \dots, a_j)$. By considering the cup product with this class, we obtain an isomorphism
    \[
    \opn{H}^{d-j}(X_{\GL_n, K}, V_{\lambda(a_j), \mathcal{O}}) \cong \opn{H}^{d-j}(X_{\GL_n, K}, V_{\lambda, \mathcal{O}})\{a_j\}
    \]
    of $\mbf{T}^{S,\opn{ord}}_{\GL_n}$-modules, where the notation $\{ a_j \}$ means the $\mbf{T}^{S,\opn{ord}}_{\GL_n}$-action is twisted as follows:
    \begin{itemize}
    \item for $m \in \GL_n(\mbb{A}_f^S)$, the symbol ``$[K^S m K^S]$'' in $\mbf{T}^{S,\opn{ord}}_{\GL_n}$ acts as $\chi_{\opn{cycl}}(\opn{det}(m))^{-a_j}[K^S m K^S]$,
    \item for $u \in \mbb{Z}_p^{\times}$, the symbol ``$\langle u \rangle_i$'' acts via $u^{a_j}\langle u \rangle_i$,
    \item for $t \in T_{\GL_n}^+$, the operator $U_t$ acts as normal.
\end{itemize}
In particular, this isomorphism is $U_t$-equivariant (recalling that $U_t$ is automatically optimally normalised). This implies that the ordinary parts match up, and since $(p-1)$ divides $a_j$, one also obtains an isomorphism after localising both sides at $\ide{M}_0$. We therefore obtain an isomorphism
    \[
    \nu_j \colon \mbf{T}^{S, \opn{ord}}_{\GL_n}(\opn{H}^{d-j}(X_{\GL_n, K}, V_{\lambda(a_j), \mathcal{O}})^{\opn{ord}}_{\ide{M}_0}) \xrightarrow{\sim} \mbf{T}^{S, \opn{ord}}_{\GL_n}(\opn{H}^{d-j}(X_{\GL_n, K}, V_{\lambda, \mathcal{O}})^{\opn{ord}}_{\ide{M}_0})
    \]
    which sends $[K^S m K^S]$ to $\chi_{\opn{cycl}}(\opn{det}(m))^{-a_j}[K^S m K^S]$, $U_t$ to itself, and $\langle u \rangle_i$ to $u^{a_j}\langle u \rangle_i$.   

    \begin{notation*}
        The ideal $\ide{M}_0 \subset \mbf{T}_{\GL_n}^S$ satisfies the hypotheses of the proposition. Let
        \[
        \varrho_j \colon G_{\mbb{Q}} \to \GL_n\left( \mbf{T}^{S, \opn{ord}}_{\GL_n}(\opn{H}^{d-j}(X_{\GL_n, K}, V_{\lambda, \mathcal{O}})^{\opn{ord}}_{\ide{M}_0}) / J_j \right)
        \]
        denote the semisimple Galois representation attached to $\ide{M}_0$ in \cite{ScholzeTorsion}, which satisfies property (1) from the statement.  Set $\varrho'_j = (\nu^{-1}_j \circ \varrho_j) \otimes \omega^{-a_{j}}$, and $J_j' = \nu^{-1}_j(J_j)$; so the representation $\varrho'_j$ is valued in $\GL_n(\mbf{T}^{S, \opn{ord}}_{\GL_n}(\opn{H}^{d-j}(X_{\GL_n, K}, V_{\lambda(a_{j}), \mathcal{O}})^{\opn{ord}}_{\ide{M}_0})/J_{j}')$.  Let $\chi_{\lambda(a_{j}), i}' = \chi_{\lambda(a_{j}), i} \cdot \omega^{-a_{j}}$ valued in $\mbf{T}^{S, \opn{ord}}_{\GL_n}(\opn{H}^{d-j}(X_{\GL_n, K}, V_{\lambda(a_{j}), \mathcal{O}})^{\opn{ord}}_{\ide{M}_0})^{\times}$.
    \end{notation*}

  The twist by $\omega^{-a_j}$ ensures $\varrho_j'$ satisfies the following local-global compatibility at $\ell \not\in S$: the (reverse) characteristic polynomial $\det(1-\opn{Frob}_\ell^{-1}X|\rho_j')$ equals the image of the Hecke polynomial $H_\ell(X)$ in $(\mbf{T}_{\GL_n}^{S,\opn{ord}}(\h^{d-j}(X_{\GL_n,K},V_{\lambda(a_j),\cO})^{\opn{ord}}_{\ide{M}_0})/J_j')[X]$.

Recall we defined
\[
\tilde{\mbf{T}}^{S, \opn{ord}}(b, c, \tilde{\lambda}_{j}) = \tilde{\mbf{T}}^{S, \opn{ord}}\left( \opn{H}^{d}(X_{\tilde{G}, \tilde{K}}, V_{\tilde{\lambda}_j,\cO})^{\opn{ord}}_{\tilde{\ide{m}}} \right).
\]
As $d-l(x_j) = l(w_j) = d-j$ (by Lemma \ref{Lem:choosinglambdawa}), Proposition \ref{Prop:SwHomOnHecke} (and its proof) yields an $\mathcal{O}$-algebra morphism
    \[
    \mathcal{S}_{w_{j}} \colon \tilde{\mbf{T}}^{S, \opn{ord}}(b, c, \tilde{\lambda}_{j}) \to \mbf{T}^{S, \opn{ord}}(\opn{H}^{d-j}(X_{M, K_M}, V_{\lambda(a_{j}), \mathcal{O}})^{\opn{ord}}_{\ide{m}})
    \]
    induced from $\mathcal{S}_{w_{j}}$. For $i=1, \dots, 2n+1$, let $\psi_{\tilde{\lambda}_j, i} \colon G_{\mbb{Q}_p} \to \tilde{\mbf{T}}^{S, \opn{ord}}(b, c, \tilde{\lambda}_j)^{\times}$ denote the character from Definition \ref{DefOfUnivGaloisChars}.

    Let 
    \[
    \mathcal{T}_{w_j} \colon \tilde{\mbf{T}}^{S, \opn{ord}}(b, c, \tilde{\lambda}_j) \to \mbf{T}^{S, \opn{ord}}_{\GL_n}(\opn{H}^{d-j}(X_{\GL_n, K}, V_{\lambda(a_j), \mathcal{O}})^{\opn{ord}}_{\ide{M}_0})
    \]
    denote the composition $\opn{pr}_j \circ \; \mathcal{S}_{w_j}$. One has for $1 \leq i \leq n$:
    \begin{equation}\label{eq:image of tw}
    \mathcal{T}_{w_j} \circ \psi_{\tilde{\lambda}_j, n+1-i}^{-1} = \mathcal{T}_{w_j} \circ \psi_{\tilde{\lambda}_j, n+1+i} = \left\{ \begin{array}{cc} \chi_{\lambda(a_j), w_j(i)}' \otimes \omega^{-1} & \text{ if } 1 \leq w_j(i) \leq n \\ (\chi_{\lambda(a_j), 2n+1-w_j(i)}')^{-1} \otimes \omega & \text{ if } n+1 \leq w_j(i) \leq 2n \end{array} \right.
    \end{equation}
    and $\mathcal{T}_{w_j} \circ \psi_{\tilde{\lambda}_j, n+1}$ is the trivial character. Furthermore $\mathcal{T}_{w_j} \circ \tilde{H}_{\ell}(X)$ modulo $J_j'$ is the (reverse) characteristic polynomial of $\opn{Frob}_{\ell}^{-1}$ on $\varrho_j'(-1) \oplus \mbf{1} \oplus (\varrho'_j)^{\vee}(1)$.

    Our set-up allows us to apply Corollary \ref{Cor:Char0HeckeContDeterminant} to $\tilde{\ide{m}}$, giving a semisimple Galois representation  
    \[  
    \sigma_{\tilde{\m}} \colon G_{\mbb{Q}} \to \opn{GL}_{2n+1}(\tilde{\mbf{T}}^{S, \opn{ord}}(b, c, \tilde{\lambda}_j)[1/p])
    \]
    satisfying local-global compatibility at $\ell\not\in S$ and at $p$, and a continuous determinant 
    \[
    \tilde{D} \colon G_{\mbb{Q}} \to \tilde{\mbf{T}}^{S, \opn{ord}}(b, c, \tilde{\lambda}_j)
    \]
    such that $\tilde{D} = \det(\sigma_{\tilde\m})$.  After inverting $p$, and combining part (1b) of the corollary with \eqref{eq:image of tw}, we see $\cT_{w_j} \circ \sigma_{\tilde{\m}}|_{G_{\Qp}}$ is upper-triangular with diagonal entries $(\chi'_{\lambda(a_j),\bullet})^{\pm 1}\otimes\omega^{\mp 1}$ or the trivial character. In particular, if $D \defeq \mathcal{T}_{w_j} \circ \tilde{D}$, then $D|_{G_{\Qp}}$ is the continuous determinant of the semisimplification
    \begin{equation}\label{eq:semisimplification 1}
    (\cT_{w_j}\circ \sigma_{\tilde{\m}}|_{G_{\Qp}})^{\opn{ss}} = \mbf{1} \oplus \bigoplus_{i=1}^n \chi'_{\lambda(a_j), i}\omega^{-1} \oplus (\chi'_{\lambda(a_j), i})^{-1}\omega,
    \end{equation}
    a representation $G_{\Q} \to \GL_{2n+1}(\mbf{T}^{S, \opn{ord}}_{\GL_n}(\opn{H}^{d-j}(X_{\GL_n, K}, V_{\lambda(a_j), \mathcal{O}})^{\opn{ord}}_{\ide{M}_0}))$. Note this description of the semisimplification follows from (1b) of the corollary, and we thus deduce it takes integral values, even though $\cT_{w_j}\circ \sigma_{\tilde{\m}}|_{G_{\Qp}}$ itself is not necessarily integral.
    
     By considering local-global compatibility at $\ell \not\in S$, we see $D$ modulo $J_j'$ is the continuous determinant associated with $\varrho_j'(-1) \oplus \mbf{1} \oplus (\varrho'_j)^{\vee}(1)$. To ease notation, set 
    \begin{align*}
    \tilde{A}_0 &= \tilde{\mbf{T}}^{S, \opn{ord}}(b, c, \tilde{\lambda}_j)  = \tilde{\mbf{T}}^{S, \opn{ord}}(\opn{H}^{d}(X_{\tilde{G}, \tilde{K}}, V_{\tilde{\lambda}_j,\cO})^{\opn{ord}}_{\tilde{\ide{m}}}), \\
    A_0 &= \mbf{T}^{S, \opn{ord}}_{\GL_n}(\opn{H}^{d-j}(X_{\GL_n, K}, V_{\lambda(a_j), \mathcal{O}})^{\opn{ord}}_{\ide{M}_0})
    \end{align*}
    which are both semi-local finite $\mathcal{O}$-algebras. Let $A$ be a local direct factor of $A_0$ and let $\tilde{A}$ be the corresponding local direct factor of $\tilde{A}_0$. It suffices to work over these local algebras. Let $J$ denote the ideal $J'_j$ in the local factor $A$; then $D$ modulo $J$ is the continuous determinant associated with $\varrho_j'(-1) \oplus \mbf{1} \oplus (\varrho'_j)^{\vee}(1)$. Let $\ide{n}_A$ be the maximal ideal of $A$; then $D \newmod{\ide{n}_A} : G_{\Q} \to \overline{\F}_p$ is the continuous determinant of $\overline{\varrho}_j'(-1) \oplus \mbf{1} \oplus (\overline{\varrho}'_j)^{\vee}(1) = \overline{\rho}_{\tilde{\m}}$. Its restriction to $G_{\Qp}$ is also the continuous determinant associated to \eqref{eq:semisimplification 1} modulo $\ide{n}_A$. As semisimple Galois representations over $\overline{\F}_p$ are uniquely determined by their continuous determinant by \cite[Thm.\ 2.12]{ChenevierPadic}, we deduce 
    \begin{equation} \label{Eqn:StdRepModMaxDecomp}
    \left( \overline{\varrho}_j'|_{G_{\mbb{Q}_p}}(-1) \oplus \mbf{1} \oplus (\overline{\varrho}'_j|_{G_{\mbb{Q}_p}})^{\vee}(1) \right)^{\opn{ss}} \cong \mbf{1} \oplus \bigoplus_{i=1}^n \overline{\chi}'_{\lambda(a_j), i}\omega^{-1} \oplus (\overline{\chi}'_{\lambda(a_j), i})^{-1}\omega 
    \end{equation}
    where the overline denotes reduction modulo $\ide{n}_A$. 

    From now on, any of the the above notation refers to the objects associated with $\ide{M}_0 = \ide{m}_0(\psi)$ only. Rewriting (\ref{Eqn:StdRepModMaxDecomp}) for the case $\ide{M}_0 = \ide{m}_0(\psi\chi)$, we obtain the decomposition:
    \[
    \left( (\overline{\varrho}_j' \otimes \overline{\chi})|_{G_{\mbb{Q}_p}}(-1) \oplus \mbf{1} \oplus ((\overline{\varrho}'_j \otimes \overline{\chi})|_{G_{\mbb{Q}_p}})^{\vee}(1) \right)^{\opn{ss}} \cong \mbf{1} \oplus \bigoplus_{i=1}^n \overline{\chi}'_{\lambda(a_j), i}\overline{\chi} \omega^{-1} \oplus (\overline{\chi}'_{\lambda(a_j), i})^{-1} \overline{\chi}^{-1} \omega .
    \]
    Our conditions on $\chi$ then imply that
    \begin{equation} \label{Eqn:DesiredSS}
    (\overline{\varrho}'_j|_{G_{\mbb{Q}_p}})^{\opn{ss}} \cong \bigoplus_{i=1}^n \overline{\chi}'_{\lambda(a_j), i} .
    \end{equation}

    We now argue as in the proof of \cite[Proposition 4.4.6, Proposition 5.4.18]{10author}. More precisely, let $D_{A/J} = D \otimes_A A/J$; then we can identify
    \[
    (A/J)[G_{\mbb{Q}}]/\opn{ker}(D_{A/J}) = M_n(A/J) \times M_1(A/J) \times M_n(A/J)
    \]
    where the projection to the first factor corresponds to $\varrho_j'(-1)$, the second to $\mbf{1}$, and the third to $(\varrho_j')^{\vee}(1)$.

    On the other hand, the natural map $\tilde{A}[G_{\mbb{Q}}] \to (A/J)[G_{\mbb{Q}}]/\opn{ker}(D_{A/J})$ factors through $\tilde{A}[G_{\mbb{Q}}]/\opn{ker}(\tilde{D})$, and we have an algebra embedding 
    \[
    \tilde{A}[G_{\mbb{Q}}]/\opn{ker}(\tilde{D}) \subset M_{2n+1}(\tilde{A} \otimes_{\mathcal{O}} \Qpb)
    \]
    corresponding to $\sigma_{\tilde{\m}}$. From the local-global compatibility for $\sigma_{\tilde{\m}}|_{G_{\mbb{Q}_p}}$, we see that for elements $Y, Y_1, \dots, Y_{2n+1}$ in $(A/J)[G_{\mbb{Q}_p}]$, we have
    \begin{equation} \label{Eq:FirstIdentityInLG}
    \opn{det}(X - \varrho_j'(-1)(Y)) \cdot (X-\mbf{1}(Y)) \cdot \opn{det}(X - (\varrho_j')^{\vee}(1)(Y)) = \prod_{i=1}^{2n+1} (X - \psi_{\tilde{\lambda}_j, i}(Y))
    \end{equation}
    in $(A/J)[X]$ (noting that, under $\cT_{w_j}$, we may view $\psi_{\tilde{\lambda}_j, i}$ as valued in $A/J$), and
    \begin{equation} \label{Eq:SecondIdinLG}
    \bigg(\prod_{i=1}^{2n+1}\Big(\varrho'_j(-1)(Y_i) - \psi_{\tilde{\lambda}_j, i}(Y_i)\Big), \prod_{i=1}^{2n+1} \Big(\mbf{1}(Y) - \psi_{\tilde{\lambda}_j, i}(Y_i)\Big), \prod_{i=1}^{2n+1}\Big((\varrho'_j)^{\vee}(1)(Y_i) - \psi_{\tilde{\lambda}_j, i}(Y_i)\Big) \bigg) = (0, 0, 0)
    \end{equation}
    in $M_n(A/J) \times M_1(A/J) \times M_n(A/J)$. By choosing suitable idempotents as in the proof of \cite[Proposition 5.4.18]{10author}, we can deduce the desired properties for $\varrho'_j$. More precisely, by \eqref{Eqn:DesiredSS} and our running assumptions, the representations $\varrho_j'(-1)|_{G_{\Qp}}, \mathbf{1}$ and $(\varrho_j')^\vee(1)|_{G_{\Qp}}$ have no Jordan--H\"older factors in common, so we can find an idempotent $e \in (A/J)[G_{\mbb{Q}_p}]$ such that $\varrho'_j(-1)(e) = 1$ and the image of $e$ under the trivial representation and $(\varrho_j')^{\vee}(1)$ is zero. We can also assume that $\psi_{\tilde{\lambda}_j, i}(e) = 1$ if $\overline{\psi}_{\tilde{\lambda}_j, i}$ appears in $\overline{\varrho}'_j(-1)$, and is zero otherwise.

    Let $x, x_1, \dots, x_n \in G_{\mbb{Q}_p}$. Taking $Y = x e$ in the identity (\ref{Eq:FirstIdentityInLG}), and using the description \eqref{eq:image of tw} of $\cT_{w_j}\circ \psi_{\tilde{\lambda}_j,i}$, we see that
    \[
    X^{n+1} \opn{det}(X - \varrho'_j(-1)(x)) = X^{n+1} \prod_{i=1}^n (X - \chi'_{\lambda(a_j), i}\omega^{-1}(x)) .
    \]
    Similarly, let $Y_i = x_k e$ if $\mathcal{T}_{w_j} \circ \psi_{\tilde{\lambda}_j, i} = \chi'_{\lambda(a_j), k} \otimes \omega^{-1}$, and $Y_i = e$ otherwise. Then plugging this into (\ref{Eq:SecondIdinLG}), we find that
    \[
    \bigg( \prod_{k=1}^n \Big(\varrho'_j(-1)(x_k) - \chi'_{\lambda(a_j),k}\omega^{-1}(x_k)\Big), 0, 0\bigg) = (0, 0, 0) .
    \]
    Note that exactly these $n$ terms survive by \eqref{Eqn:DesiredSS}.
    
    This proves the proposition for $\varrho'_j(-1)$ and hence $\varrho'_j$. Noting that $\nu_j \circ \chi_{\lambda(a_j), i} = \chi_{\lambda, i}$, we now untwist to obtain the proposition for $\varrho_j$, i.e., for the maximal ideal $\ide{m}_0(\psi)$. 
    
    Finally, we note that because $\psi$ is unramified outside $S-\{p\}$, we have an isomorphism
    \[
    \mbf{T}_{\GL_n}^{S, \opn{ord}}\left( \opn{H}^{d-j}(X_{\GL_n, K}, V_{\lambda, \mathcal{O}})^{\opn{ord}}_{\ide{m}_0} \right) \cong \mbf{T}_{\GL_n}^{S, \opn{ord}}\left( \opn{H}^{d-j}(X_{\GL_n, K}, V_{\lambda, \mathcal{O}})^{\opn{ord}}_{\ide{m}_0(\psi)} \right)
    \]
    which identifies the Galois representation $\varrho_j$ (associated with $\ide{m}_0$) with the one associated with $\ide{m}_0(\psi)$ twisted by the character $\psi^{-1}$. This implies the required local-global compatibility at $\ell = p$ for $\ide{m}_0$.
\end{proof}

\begin{remark} \label{Rem:Charpl=pLGcompat}
    With notation as above, suppose that $\ide{m}_0$ is in the support of $\opn{H}^i(X_{\opn{GL}_n, K}, k)^{\opn{ord}}$ for some appropriately chosen $i$ and $b, c$, and let $\ide{m}_0' \subset \mbf{T}_{\opn{GL}_n}^{S, \opn{ord}}(\opn{H}^i(X_{\opn{GL}_n, K}, k)^{\opn{ord}})$ be any maximal ideal extending $\ide{m}_0$. After possibly increasing $k$, assume that the residue field of $\ide{m}_0'$ is $k$. Then, as part of the proof of Proposition \ref{Prop:MixedCharLGcompat}, we have also established local-global compatibility at $\ell = p$ for the Galois representation $\overline{\rho}_{\ide{m}_0}$ (assuming it is absolutely irreducible and decomposed generic). In particular, if the characters $\chi_i \colon G_{\mbb{Q}_p} \to k^{\times}$ (from Definition \ref{DefOfUnivGaloisChars} associated with the morphism $\Theta' \colon \mbf{T}^{S, \opn{ord}}_{\opn{GL}_n} \to \mbf{T}^{S, \opn{ord}}_{\opn{GL}_n}/\ide{m}_0' = k$) are distinct, then combining the above with \cite[Lemma 6.2.11]{10author} shows
    \[
    \overline{\rho}_{\ide{m}_0}|_{G_{\mbb{Q}_p}} \sim \left( \begin{array}{cccc} \chi_1 & * & * & * \\ & \chi_2 & * & * \\ & & \ddots & \vdots \\ & & & \chi_n \end{array} \right).
    \]

\end{remark}

\subsection{The main results}

We now discuss the main result of this section.

\begin{theorem} \label{Thm:LevelbcmLGcompat}
    Let $p > 2n$. Let $c \geq b \geq 0$ with $c \geq 1$, and let $K \subset \GL_n(\mbb{A}_f)$ denote a neat compact open subgroup which is hyperspecial outside $S$ and equal to $\opn{Iw}(b, c)$ at $p$. Let $m \geq 1$ and let 
    \[
    \mbf{T}_{\GL_n}^{S, \opn{ord}}(b, c; m) \defeq \mbf{T}_{\GL_n}^{S, \opn{ord}}( R\Gamma(X_{\GL_n, K}, \mathcal{O}/\varpi^m )^{\opn{ord}}) \subset \opn{End}_{\mbf{D}(\mathcal{O})}( R\Gamma(X_{\GL_n, K}, \mathcal{O}/\varpi^m )^{\opn{ord}}) .
    \]
    Let $\ide{m}_0 \subset \mbf{T}_{\GL_n}^{S, \opn{ord}}(b, c; m)$ denote a maximal ideal with residue field $k = \mathcal{O}/\varpi$, and suppose that the associated Galois representation $\overline{\rho}_{\ide{m}_0}$ (see Theorem \ref{Thm:GaloisRepForTorClass}) is absolutely irreducible and decomposed generic.

    Then there exists a nilpotent ideal $J \subset \mbf{T}_{\GL_n}^{S, \opn{ord}}(b, c; m)_{\ide{m}_0}$ with nilpotence degree only depending on $n$, and a continuous semisimple representation 
    \[
    \rho_{b, c, m} \colon G_{\mbb{Q}} \to \GL_n( \mbf{T}_{\GL_n}^{S, \opn{ord}}(b, c; m)_{\ide{m}_0}/J)
    \]
    such that:
    \begin{enumerate}
        \item For every $\ell \not\in S$, one has
        \[
        \opn{det}(1 - \opn{Frob}_{\ell}^{-1} X | \rho_{b, c, m}) = \vartheta(H_{\ell}(X))
        \]
        where $\vartheta \colon \mbf{T}^{S, \opn{ord}}_{\GL_n} \to \mbf{T}_{\GL_n}^{S, \opn{ord}}(b, c; m)_{\ide{m}_0}/J$ is the natural map.
        \item For every $x \in G_{\mbb{Q}_p}$, the characteristic polynomial of $\rho_{b, c, m}(x)$ equals $\prod_{i=1}^n(X -\chi_i(x))$, where $\chi_i$ are the characters from Definition \ref{DefOfUnivGaloisChars} associated with $\vartheta$.
        \item For each $x_1, \dots, x_n \in G_{\mbb{Q}_p}$, we have
        \[
        (\rho_{b, c, m}(x_1) - \chi_1(x_1)) \cdots (\rho_{b, c, m}(x_n) - \chi_n(x_n)) = 0 .
        \]
    \end{enumerate}
\end{theorem}
\begin{proof}
To simplify notation, in this proof only we write $X \defeq X_{\GL_n,K}$ and $\mbf{T} \defeq \mbf{T}_{\GL_n}^{S,\opn{ord}}$. By  \cite[Lem.\ 2.5(2)]{KT17}, the kernel $J_1$ of the map
    \begin{align*}
   \theta_1 \colon    \mbf{T}_{\GL_n}^{S, \opn{ord}}(b, c; m) = \mbf{T}\Big( R\Gamma(X, \mathcal{O}/\varpi^m )^{\opn{ord}}\Big) &\longrightarrow \oplus_{q=0}^{d-1}\mbf{T}\Big( \opn{H}^q(X, \mathcal{O}/\varpi^m )^{\opn{ord}}\Big)
    \end{align*}
 is nilpotent with nilpotence degree at most $d$ (hence only depending on $n$). 
 
 We first construct a representation over the target, working separately with each summand. Fix then $1 \leq j \leq d$ (so $q = d-j$), and let $\lambda$ be as in Lemma \ref{Lem:choosinglambdawa}. We have a short exact sequence
    \[
    0 \to \opn{H}^{d-j}(X, V_{\lambda, \mathcal{O}})^{\opn{ord}}/\varpi^m \to \opn{H}^{d-j}(X, \mathcal{O}/\varpi^m )^{\opn{ord}} \to \opn{H}^{d-j+1}(X, V_{\lambda, \mathcal{O}})^{\opn{ord}}[\varpi^m] \to 0
    \]
    as $\mathcal{O}(\lambda)/\varpi^m \cong \mathcal{O}/\varpi^m$ and the ordinary part mod $\varpi^m$ only depends on the weight mod $\varpi^m$. Let 
    \begin{align*}
        \mbf{T}(d-j,\lambda;m) &\defeq \mbf{T}(\opn{H}^{d-j}(X, V_{\lambda, \mathcal{O}})^{\opn{ord}}/\varpi^m),\\
        \mbf{T}(d-j;m) & \defeq \mbf{T}(\opn{H}^{d-j}(X, \mathcal{O}/\varpi^m )^{\opn{ord}}),\\
        \mbf{T}'(d-j+1,\lambda;m) &\defeq \mbf{T}(\opn{H}^{d-j+1}(X, V_{\lambda, \mathcal{O}})^{\opn{ord}}[\varpi^m])
    \end{align*}
   be the respective Hecke algebras. From the short exact sequence, we obtain a surjective map
\[
    \theta_2^j \colon \mbf{T}(d-j;m) \longrightarrow \mbf{T}(d-j,\lambda;m) \times \mbf{T}'(d-j+1,\lambda;m).
\]
Let $J_{2,j} = \ker(\theta_2^j)$; a simple computation shows $T^2 = 0$ for all $T \in J_{2,j}$ (so $J_{2,j}$ is nilpotent). 

   Combining Proposition \ref{Prop:MixedCharLGcompat} with the natural maps of Hecke algebras, we have representations
    \[
        \varrho_j' : G_{\Q} \to \GL_n(\mbf{T}(d-j,\lambda;m)_{\m_0}/J_{j}'), \qquad \varrho''_j : G_{\Q} \to \GL_n(\mbf{T}'(d-j+1,\lambda;m)_{\m_0}/J_{j}'')
    \]
    satisfying the required conditions, where $J_{j}'$ and $J_{j}''$ are nilpotent. Let $J_{3,j} = (\theta_2^j)^{-1}(J_j' \times J_j'')$, which is nilpotent with nilpotence degree depending only on $n$ (noting $J_{2,j} \subset J_{3,j}$). Combining $\rho_j'$ and $\rho_j''$ thus gives a representation
    \[
        \varrho_j : G_{\Q} \to \GL_n(\mbf{T}(d-j;m)_{\m_0}/J_{3,j})
    \]
    satisfying the required compatibilities. Let $\varrho = \oplus_{j=1}^d \varrho_j$.
    
   Let $J \subset  \mbf{T}_{\GL_n}^{S, \opn{ord}}(b, c; m)_{\m_0}$ be the kernel of the composition of $\theta_1$ (localised at $\m_0$) with projection modulo $J_{3,1} \times \cdots \times J_{3,d}$. Note $J$, which contains $J_1$, is nilpotent of degree depending only on $n$, and $\theta_1$ induces an injective map
    \[
        \theta_1' \colon  \mbf{T}_{\GL_n}^{S, \opn{ord}}(b, c; m)_{\m_0}/J \hookrightarrow \oplus_{j=1}^{d}\mbf{T}\big(d-j;m\big)_{\m_0}/J_{3,j}.
    \]
    Now let $D_{\varrho}$ be the continuous determinant attached to $\varrho$, which satisfies the desired local-global compatibility (inherited from Proposition \ref{Prop:MixedCharLGcompat}). By \cite[Example 2.32]{ChenevierPadic}, $D$ takes values in $\opn{Im}(\theta_1')$, and pulling back under $\theta_1'$ we obtain a continuous (Cayley--Hamilton) determinant $D \colon G_{\Q} \to  \mbf{T}_{\GL_n}^{S, \opn{ord}}(b, c; m)_{\m_0}/J$ satisfying properties (1)--(3). We now obtain $\rho_{b,c,m}$ as the lift of $D$ under \cite[Theorem 2.22(i)]{ChenevierPadic}, using the fact that $\overline{\rho}_{\ide{m}_0}$ is assumed to be absolutely irreducible.
\end{proof}

\begin{remark}
The Galois representation $\rho_{b, c, m}$ satisfying (1) was previously constructed in the work of Scholze \cite{ScholzeTorsion} and Newton--Thorne \cite{NewtonThorneTorsion} (see also \cite{CGHJMRS}). 
\end{remark}

We note the following corollary, which is what we will use in practice. Note that, if we let $K(b, c) = K$ to emphasise the dependence on $b, c$, then since $R\Gamma( T(b), R\Gamma(X_{\GL_n, K(b', c')}, \mathcal{O}/\varpi^m)^{\opn{ord}}) \cong R\Gamma(X_{\GL_n, K(b, c)}, \mathcal{O}/\varpi^m)^{\opn{ord}}$ for $b' \geq b$ and $c' \geq c$, we obtain a natural $\mbf{T}^{S, \opn{ord}}_{\GL_n}$-algebra morphism $\mbf{T}^{S, \opn{ord}}_{\GL_n}(b', c'; m) \to \mbf{T}^{S, \opn{ord}}_{\GL_n}(b, c; m)$. Similarly, we have natural maps $\mbf{T}^{S, \opn{ord}}_{\GL_n}(b, c; m') \to \mbf{T}^{S, \opn{ord}}_{\GL_n}(b, c; m)$ for $m' \geq m$. Here it is important that we are working with Hecke algebras in the derived categories, otherwise these maps do not exist in general. 

Let $\mathcal{T} = \varprojlim_{b, c, m} \mbf{T}^{S, \opn{ord}}_{\GL_n}(b, c; m)$, which we equip with the inverse limit topology. The ring $\mathcal{T}$ is naturally a finite $\Lambda$-algebra, where $\Lambda = \mathcal{O}[\![T_{\opn{GL_n}}(\mbb{Z}_p)]\!]$ (equipped with the usual profinite topology), and the topology on $\mathcal{T}$ coincides with the canonical topology induced from that on $\Lambda$. Note that any maximal ideal $\ide{m}_0 \subset \mathcal{T}$ with finite residue field descends to a maximal ideal of some $\mbf{T}^{S, \opn{ord}}_{\GL_n}(b, c; m)$.

\begin{corollary} \label{Cor:MainTriangulation}
    Let $\ide{m}_0 \subset \mathcal{T}$ be a maximal ideal with residue field $k = \mathcal{O}/\varpi$, and suppose that the associated Galois representation $\overline{\rho}_{\ide{m}_0} \colon G_{\mbb{Q}} \to \opn{GL}_n(k)$ is absolutely irreducible and decomposed generic. Then there exists a nilpotent ideal $J \subset \mathcal{T}_{\ide{m}_0}$ (with nilpotence degree only depending on $n$) and a continuous semisimple representation
    \[
    \rho \colon G_{\mbb{Q}} \to \GL_n(\mathcal{T}_{\ide{m}_0}/J)
    \]
    such that:
    \begin{enumerate}
        \item For every $\ell \notin S$, the representation $\rho$ is unramified at $\ell$ and $\opn{det}(1 - \opn{Frob}_{\ell}^{-1} X | \rho)$ equals the image of $H_{\ell}(X)$ in $(\mathcal{T}_{\ide{m}_0}/J)[X]$.
        \item For every $x \in G_{\mbb{Q}_p}$, one has
        \[
        \opn{det}(X - \rho(x)) = \prod_{i=1}^n (X - \chi_i(x))
        \]
        where $\chi_i$ denote the characters of Definition \ref{DefOfUnivGaloisChars} valued in $(\mathcal{T}_{\ide{m}_0}/J)^{\times}$.
        \item For each $x_1, \dots, x_n \in G_{\mbb{Q}_p}$, one has
        \[
        (\rho(x_1) - \chi_1(x_1)) \cdots (\rho(x_n) - \chi_n(x_n)) = 0 .
        \]
    \end{enumerate}
    Moreover, let $L[\epsilon]$ denote the ring of dual numbers and suppose that $f \colon \mathcal{T}_{\ide{m}_0}/J \to L[\epsilon]$ is a $\mbf{T}^{S,\opn{ord}}_{\GL_n}$-algebra morphism. Set $\rho_{L[\epsilon]} = f \circ \rho$. Suppose that the characters $f \circ \chi_i$ are pairwise distinct modulo $\epsilon$. Then
    \[
    \rho_{L[\epsilon]}|_{G_{\mbb{Q}_p}} \sim \left( \begin{array}{cccc} f \circ \chi_1 & * & * & * \\ & f \circ \chi_2 & * & * \\ & & \ddots & \vdots \\ & & & f \circ \chi_n \end{array} \right) .
    \]
\end{corollary}
\begin{proof}
    There exists a nilpotent ideal $J$ and continuous determinant $D$ valued in $\mathcal{T}_{\ide{m}_0}/J$ with properties (1)--(3) by \cite[Example 2.32]{ChenevierPadic} and Theorem \ref{Thm:LevelbcmLGcompat} (crucially using that the nilpotence degree of the ideal is independent of $b, c, m$). Note that $\mathcal{T}_{\ide{m}_0}/J$ is a local Henselian ring and this determinant reduces to $\opn{det}(\overline{\rho}_{\pi})$ modulo the maximal ideal of $\mathcal{T}_{\ide{m}_0}/J$. The determinant $D$ is Cayley--Hamilton (as it is built from Cayley--Hamilton determinants). We now obtain $\rho$ by \cite[Theorem 2.22(i)]{ChenevierPadic} and the fact that $\overline{\rho}_{\ide{m}_0}$ is assumed to be absolutely irreducible.

    For the last part, we follow \cite[Lemma 6.2.11]{10author}. Namely, let $V$ denote the finite free $L[\epsilon]$-module (of rank $n$) on which $\rho_{L[\epsilon]}$ acts. Define, iteratively, $W_i \subset V/W_{i-1}$ to be the largest submodule for which $\rho_{L[\epsilon]}$ acts through the character $f \circ \chi_i$ (with $W_0 = 0$). Let $V_i \subset V$ denote the preimage of $W_i$ under the map $V \to V/V_{i-1}$. Note that by property (3), we have $V_n = V$. We need to show that $V_i$ is free of rank $i$.

    Consider $V/V_{n-1}$. By property (3), we already know that $G_{\mbb{Q}_p}$ acts through $f \circ \chi_n$ on $V/V_{n-1}$. Furthermore, by \cite[Lemma 6.2.11]{10author} (applied to the field $L = L[\epsilon]/\epsilon$), we  must have
    \[
    \opn{dim}_L(V/V_{n-1})/\epsilon = 1.
    \] 
   By Nakayama's lemma, we see that $V/V_{n-1}$ is cyclic over $L[\epsilon]$. We claim it is free of rank one; to see this, it suffices to show it is $\epsilon$-torsion free.

    Suppose there exists $v \in V$ which is non-zero in $V/V_{n-1}$ and satisfies $\epsilon v \in V_{n-1}$. By assumption, we can choose $x_i \in G_{\mbb{Q}_p}$ such that $f\chi_n(x_i) - f\chi_i(x_i) \in L[\epsilon]$ is invertible for all $i=1, \dots, n-1$. Then
    \[
    0 = \prod_{i=1}^{n-1}(\rho_{L[\epsilon]}(x_i) - f\chi_i(x_i)) \epsilon v = \epsilon \prod_{i=1}^{n-1}(f\chi_n(x_i) - f\chi_i(x_i)) v + \epsilon w
    \]
    for some $w \in V_{n-1}$ (because $\rho_{L[\epsilon]}(x_i) - f\chi_n(x_i)$ kills $V/V_{n-1}$). Since $V$ is $\epsilon$-torsion free, we see that
    \[
    v = -\prod_{i=1}^{n-1}(f\chi_n(x_i) - f\chi_i(x_i))^{-1} w \; \in \; V_{n-1},
    \]
    which is a contradiction. Hence $V/V_{n-1}$ is $\epsilon$-torsion free, and hence free of rank one over $L[\epsilon]$. This implies $V_{n-1}$ is free of rank $n-1$ over $L[\epsilon]$, and we conclude by an inductive argument.
\end{proof}
 
\subsection{Proof of Theorem \ref{TriangulationInFamilies}} \label{SubSub:ProofOfTriangThm}

We now specialise to $n = 3$ and use notation from \S \ref{SubSec:FamiliesDefect1} freely. In particular $\Sigma \subset \Omega \subset \cW$ are affinoids in weight space. We recall that $\pi$ is a regular algebraic, cuspidal automorphic representation of $\opn{GL}_3(\mbb{A})$ which has trivial cohomological weight and is Steinberg at $p$. In particular, $\pi_p$ is ordinary. Recall that Theorem \ref{TriangulationInFamilies} asks for a triangulation of the Galois representation in an infinitesimal neighbourhood of $\pi$ in the eigenvariety constructed with overconvergent cohomology; whilst Corollary \ref{Cor:MainTriangulation} gives such a triangulation for the eigenvariety constructed using completed cohomology. To prove Theorem \ref{TriangulationInFamilies}, we essentially need to compare these eigenvarieties (around ordinary points); we do this below. A general comparison around finite-slope points will appear in forthcoming work of Johansson, McDonald and Tarrach.

Let $\ide{m}_0 \subset \mbf{T}^{S, \opn{ord}}_{\opn{GL}_3}$ denote the mod $\varpi$ Hecke eigensystem associated with $\pi$, and note that $\ide{m}_0$ is in the support of $\mbf{T}^{S, \opn{ord}}_{\opn{GL}_3}(0, 1; m)$ for some $m \geq 1$. We are assuming that $\overline{\rho}_{\ide{m}_0} = \overline{\rho}_{\pi}$ is absolutely irreducible and decomposed generic. Recall the characters $\underline{\alpha}^S, \underline{\alpha}_p$ from \S\ref{SubSec:FamiliesDefect1}, describing families of eigenvalues appearing in overconvergent cohomology. We obtain an $\mathcal{O}$-algebra homomorphism 
\[
\mbf{T}^{S, \opn{ord}}_{\opn{GL}_3} \xrightarrow{\underline{\alpha}^S \otimes \underline{\delta}_p} \mathcal{O}(\Sigma) \to L[\epsilon].
\] 
Here the second map is induced from $T_1 \mapsto v_1 \epsilon$ and $T_2 \mapsto v_2 \epsilon$, and $\underline{\delta}_p \colon T^+ \to \mathcal{O}(\Sigma)^{\times}$ denotes the character 
\[
\underline{\delta}_p(t) = \underline{\alpha}_p(t) \kappa_{\Sigma}([t])
\]
with $\kappa_{\Sigma}$ equal to the universal character of $\Sigma$ and $[t] \in T(\mbb{Z}_p)$ denoting the image of $t$ under the natural projection map $T(\mbb{Q}_p) \twoheadrightarrow T(\mbb{Z}_p)$. To prove Theorem \ref{TriangulationInFamilies}, it suffices to show that this map factors through the natural map $\mbf{T}^{S, \opn{ord}}_{\opn{GL}_3} \to \mathcal{T}_{\ide{m}_0}^{\opn{red}}$ (where $\mathcal{T}_{\ide{m}_0}^{\opn{red}}$ is the maximal reduced quotient of $\mathcal{T}_{\ide{m}_0}$). We can then apply Corollary \ref{Cor:MainTriangulation}. Since $\mathcal{O}(\Sigma)$ is reduced, it suffices to show that $\underline{\alpha}^S \otimes \underline{\delta}_p$ factors through $\mbf{T}^{S, \opn{ord}}_{\opn{GL}_3} \to \mathcal{T}_{\ide{m}_0} = (\varprojlim_{b, c, m}\mbf{T}^{S, \opn{ord}}_{\opn{GL}_3}(b, c; m))_{\ide{m}_0}$.

At the expense of shrinking $\Omega$, we may assume that $\Omega \subset \Omega^{\circ} \subset \Omega_0$, where $\Omega_0$ is a (sufficiently small) open affinoid neighbourhood of the trivial weight in $\mathcal{W}$ and $\Omega^{\circ}$ is a wide open disc, i.e., $\Omega^{\circ}$ is the adic generic fibre of $\opn{Spf}\mathcal{O}[\![T_1/p^{\rho}, T_2/p^{\rho}]\!]$ for some $\rho \in \mbb{Q}_{>0}$. 

Let $A_{\Omega^{\circ}}^{\circ,r\opn{-an}}$ be the space of functions $f: \opn{Iw} \to \mathcal{O}^+(\Omega^{\circ})$ that are $r$-analytic in the variables corresponding to $\opn{Iw}$, and satisfy $f(b \cdot -) = \kappa_{\Omega^{\circ}}(b) \cdot f(-)$ for all $b \in \overline{B}(\mbb{Q}_p) \cap \opn{Iw}$. By restriction, we can identify $A_{\Omega^\circ}^{\circ,r\opn{-an}}$ with the space of $r$-analytic functions $N(\Zp) \cong \Zp^3 \to \cO^+(\Omega^\circ)$. Let $D_{\Omega^{\circ}}^{\circ,r\opn{-an}}$ denote the $\mathcal{O}^+(\Omega^{\circ})$-linear dual of $A_{\Omega^{\circ}}^{\circ,r\opn{-an}}$. The module $D_{\Omega^{\circ}}^{\circ,r\opn{-an}}$ carries actions of $(\opn{Iw}, T^+)$ via the same formulae as in \S \ref{SubSec:FamiliesDefect1}.

The representation $D_{\Omega^{\circ}}^{\circ,r\opn{-an}}$ is profinite, and has a $(\opn{Iw}, T^+)$-stable filtration $\opn{Fil}^i D_{\Omega^{\circ}}^{\circ,r\opn{-an}}$ satisfying (see \cite[\S 2]{Han17} and \cite[\S 5.1]{JoNewExt}): 
\begin{itemize}
    \item $\opn{Fil}^0 D_{\Omega^{\circ}}^{\circ,r\opn{-an}} = D_{\Omega^{\circ}}^{\circ,r\opn{-an}}$ and $\bigcap_i \opn{Fil}^i D_{\Omega^{\circ}}^{\circ,r\opn{-an}} = \{0\}$;
    \item $D_{\Omega^{\circ}}^{\circ,r\opn{-an}} = \varprojlim_i D_{\Omega^{\circ}}^{\circ,r\opn{-an}}/\opn{Fil}^i$;
    \item Each $D_{\Omega^{\circ}}^{\circ,r\opn{-an}}/\opn{Fil}^i$ is finite, and the action of $\opn{Iw}_{i+r} \defeq \opn{ker}(\opn{Iw} \to \opn{GL}_3(\mbb{Z}/p^{i+r} \mbb{Z}))$ on this module is trivial.
\end{itemize}
Let $F_{\Omega^{\circ}} \colon B(\mbb{Z}_p) \to T(\mbb{Z}_p) \xrightarrow{\kappa_{\Omega^{\circ}}} \cO^+(\Omega^{\circ})$ denote the natural character, which (for sufficiently large $r$) we view as an element of $A_{\Omega^{\circ}}^{\circ,r\opn{-an}}$ in the natural way. Note that $F_{\Omega^\circ}$ corresponds to the constant function $1$ on $N(\Zp)$. We have an $\mathcal{O}^+(\Omega^{\circ})$-linear map
\[
\Lambda \colon D_{\Omega^{\circ}}^{\circ,r\opn{-an}} \to \mathcal{O}^+(\Omega^{\circ})(\kappa_{\Omega^{\circ}}^{-1})
\]
given by evaluating at $F_{\Omega^{\circ}}$. This map is equivariant for the action of $B(\mbb{Z}_p)$ as well as the action of $T^+$ (where $T^+$ acts via the projection $T^+ \twoheadrightarrow T(\mbb{Z}_p)$ on the right-hand side). Let $\Lambda_m \colon D_{\Omega^{\circ}}^{\circ,r\opn{-an}}/\opn{Fil}^m \to (\mathcal{O}^+(\Omega^{\circ})/\ide{a}^m)(\kappa_{\Omega^{\circ}}^{-1})$ denote the induced map, where $\ide{a}$ denotes the maximal ideal of $\mathcal{O}^+(\Omega^{\circ})$. We have the following lemma:

\begin{lemma}
\begin{itemize}
\item[(i)]    Let $m \geq 1$ and $c \geq b \geq m+r$. The map $\Lambda_m$ induces a $\mbf{T}_{\opn{GL}_3}^{S, \opn{ord}}$-equivariant quasi-isomorphism
    \begin{align}
    R\Gamma(X_{\opn{GL}_3, K^p \opn{Iw}}, &D_{\Omega^{\circ}}^{\circ,r\opn{-an}}/\opn{Fil}^m)^{\opn{ord}} \cong R\Gamma(X_{\opn{GL}_3, K^p \opn{Iw}(0,c)}, D_{\Omega^{\circ}}^{\circ,r\opn{-an}}/\opn{Fil}^m)^{\opn{ord}} \notag\\ 
    &\xrightarrow{\sim} R\Gamma(X_{\opn{GL}_3, K^p \opn{Iw}(0, c)}, (\mathcal{O}^+(\Omega^{\circ})/\ide{a}^m)(\kappa_{\Omega^{\circ}}^{-1}))^{\opn{ord}} \label{eq:lambda_m}\\ 
    &\cong R\Gamma(T(\mbb{Z}_p)/T(b), R\Gamma(X_{\opn{GL}_3, K^p \opn{Iw}(b, c)}, \mathcal{O}/\varpi^m)^{\opn{ord}} \otimes_{\mathcal{O}/\varpi^m} (\mathcal{O}^+(\Omega^{\circ})/\ide{a}^m)(\kappa_{\Omega^{\circ}}^{-1})) .\notag
    \end{align}
   \item[(ii)] The map 
    \[
    \mbf{T}_{\opn{GL}_3}^{S, \opn{ord}} \to \mbf{T}_{\opn{GL}_3}^{\opn{S}, \opn{ord}}(R\Gamma(X_{\opn{GL}_3, K^p \opn{Iw}}, D_{\Omega^{\circ}}^{\circ,r\opn{-an}}/\opn{Fil}^m)^{\opn{ord}})
    \]
    which is the natural one on $\mbf{T}_{\GL_3}^S$ and given by $t \mapsto U_t \cdot \kappa_{\Omega^{\circ}}([t])$ for $t \in T^+$, factors through 
    \[
    \mbf{T}^{S, \opn{ord}}_{\opn{GL}_3}(b, c; m) =  \mbf{T}_{\opn{GL}_3}^{S, \opn{ord}}(R\Gamma(X_{\opn{GL}_3, K^p \opn{Iw}(b, c)}, \mathcal{O}/\varpi^m)^{\opn{ord}})
    \]
    compatibly with the change of $b, c, m$ maps.
\end{itemize}
\end{lemma}
\begin{proof}
Part (ii) is a direct consequence of the Hecke-equivariance in (i). To see (i), first we explain why the claimed map exists. The first quasi-isomorphism is standard, using the fact that the ordinary part at level $\opn{Iw}(b, c)$ is independent of $c$ . The morphism $\Lambda_m$ is equivariant for the action of $(\opn{Iw}(0, c), T^+)$, so the next map on cohomology complexes is well-defined. The third isomorphism is proved by first applying \cite[Lemma 5.2.9]{10author} to pass from the cohomology of $X_{\GL_3,K^p\opn{Iw}(0,c)}$ to that of $T(\Zp)$; then exploiting $R\Gamma(T(\Zp),-) = R\Gamma(T(\Zp)/T(b), R\Gamma(T(b),-))$; then applying \cite[Lemma 5.2.9]{10author} again to replace $R\Gamma(T(b),-)$ with the cohomology of $X_{\GL_3,K^p\opn{Iw}(b,c)}$, pulling out $(\cO^+(\Omega^\circ)/\mathfrak{a}^m)(\kappa_{\Omega^\circ}^{-1})$ (upon which $\opn{Iw}(b,c)$ acts trivially). 
    
    It remains to prove that the second map \eqref{eq:lambda_m} is a quasi-isomorphism. We use the following:

    \begin{claim}
    Let  $t=\opn{diag}(p^2, p, 1) \in T^+$. We have $t^{r+1}\cdot \opn{ker}(\Lambda_m) \subset p\opn{ker}(\Lambda_m)$.
    \end{claim}

    Given the claim, the $U_p$-operator acts on the cohomology with coefficients in $\opn{ker}\Lambda_m$ with strictly positive slope, and hence its ordinary part vanishes. The map \eqref{eq:lambda_m} is thus a quasi-isomorphism from the long exact sequence attached to $0 \to \ker(\Lambda_m) \to D_{\Omega^\circ}^{\circ,r\opn{-an}}/\opn{Fil}^m \to \cO^+(\Omega^\circ)/\mathfrak{a}^m \to 0$.

    \medskip

    \emph{Proof of claim:} If $f \in A^{\circ,r\opn{-an}}_{\Omega^\circ}$, then it is standard that $t^{-r} \cdot f \in A^{\circ, 0\opn{-an}}_{\Omega^\circ}$ is analytic (since $t^rN(\Zp)t^{-r} \subset N(p^r\Zp)$, upon which $f$ is analytic).  In particular, we can identify $t^{-r} \cdot f$ uniquely with an element of $\cO^+(\Omega^\circ)\langle X_1,X_2,X_3 \rangle$ under the identification $N(\Zp) \cong \Zp^3$ (with $X_1$ corresponding to the $(1,2)$ matrix entry, $X_2$ to the $(1,3)$ entry, and $X_3$ to the $(2,3)$ entry). Write $t^{-r} \cdot f = a_0 + g$, where $g(X_1,X_2,X_3) \defeq \sum_{\mathbf{i} \in \N^3\backslash \{(0,0,0)\}} a_{\mathbf{i}} X_1^{i_1}X_2^{i_2}X_3^{i_3}$. Note that 
    \[
        (t^{-r-1}\cdot f)(X_1,X_2,X_3) = a_0 + g(pX_1,p^2X_2,pX_3).
    \]
    In particular, as $g$ has no constant term, we can write $t^{-r-1}\cdot f = a_0 + pg'$, for $g' \in \cO^+(\Omega^\circ)\langle X_1,X_2,X_3\rangle$ with $g'(0,0,0) = 0$. Note the association $f \mapsto g'$ is $\cO^+(\Omega^\circ)$-linear.

    Now let $[\mu] \in \ker(\Lambda_m)$, represented by some $\mu$ in the dual of $r$-analytic functions on $N(\Zp$); then $\mu(a_0) = a_0 \mu(F_{\Omega^\circ}) \equiv 0 \newmod{\mathfrak{a}^m}$ for any constant function $a_0 \in \cO^+(\Omega^\circ)$ (recalling $F_{\Omega^\circ}$ corresponds to the constant function $1$). Define $\lambda \in D^{\circ,r\opn{-an}}_{\Omega^\circ}$ by $\lambda(f) = \mu(g')$, for notation as above. If $f = F_{\Omega^\circ}$, then $g' = 0$; so $\lambda(F_{\Omega^\circ}) = \mu(0) = 0$, and hence $[\lambda] \in \ker(\Lambda_m)$. We find
    \[
        (t^{r+1}\cdot\mu)(f) = \mu(t^{-r-1}\cdot f) = \mu(a_0 + pg') \equiv p\lambda(f) \newmod{\mathfrak{a}^m}.
    \]
    We see $t^{-r-1}\cdot [\mu] = [t^{-r-1}\cdot \mu] = p[\lambda] \in p \ker(\Lambda_m)$, as required.
\end{proof}

Let $\ide{n}_0 \subset \mbf{T}^{S, \opn{ord}}_{\opn{GL}_3, \mathcal{O}^+(\Omega^{\circ})}$ denote the maximal ideal given by the preimage of $\ide{m}_0$ under the map $\mbf{T}^{S, \opn{ord}}_{\opn{GL}_3, \mathcal{O}^+(\Omega^{\circ})} \to \mbf{T}^{S, \opn{ord}}_{\opn{GL}_3, \mathcal{O}}$ given by reduction modulo $(T_1/p^{\rho}, T_2/\rho^{\rho})$. Since $\ide{m}_0$ is non-Eisenstein, we also have 
\[
\mbf{T}_{\opn{GL}_3}^{\opn{S}, \opn{ord}}(R\Gamma(X_{\opn{GL}_3, K^p \opn{Iw}}, D_{\Omega^{\circ}}^{\circ,r\opn{-an}}/\opn{Fil}^m)_{\ide{n}_0}^{\opn{ord}}) = \mbf{T}_{\opn{GL}_3}^{\opn{S}, \opn{ord}}(R\Gamma_c(X_{\opn{GL}_3, K^p \opn{Iw}}, D_{\Omega^{\circ}}^{\circ,r\opn{-an}}/\opn{Fil}^m)_{\ide{n}_0}^{\opn{ord}}) .
\]
Passing to the limit, and using the fact that $R\Gamma_c(X_{\opn{GL}_3, K^p \opn{Iw}}, D_{\Omega^{\circ}}^{\circ,r\opn{-an}})_{\ide{n}_0}^{\opn{ord}}$ is a perfect complex in the derived category of $\mathcal{O}^+(\Omega^{\circ})$-modules, we see that the map
\begin{align*}
\mbf{T}_{\opn{GL}_3}^{S, \opn{ord}} &\to \varprojlim_m \mbf{T}_{\opn{GL}_3}^{S, \opn{ord}}(R\Gamma_c(X_{\opn{GL}_3, K^p \opn{Iw}}, D_{\Omega^{\circ}}^{\circ,r\opn{-an}}/\opn{Fil}^m)_{\ide{n}_0}^{\opn{ord}}) \\ &= \mbf{T}_{\opn{GL}_3}^{S, \opn{ord}}(R\Gamma_c(X_{\opn{GL}_3, K^p \opn{Iw}}, D_{\Omega^{\circ}}^{\circ,r\opn{-an}})_{\ide{n}_0}^{\opn{ord}}) \\ &\to \mbf{T}_{\opn{GL}_3}^{S, \opn{ord}}(\opn{H}^*_c(X_{\opn{GL}_3, K^p \opn{Iw}}, D_{\Omega^{\circ}}^{\circ,r\opn{-an}})_{\ide{n}_0}^{\opn{ord}})
\end{align*}
satisfying $t \mapsto U_t \cdot \kappa_{\Omega^{\circ}}([t])$ for $t \in T^+$, factors through $\mathcal{T}_{\ide{m}_0}$ (note that $\kappa_{\Omega^{\circ}}$ is trivial modulo $(T_1/p^{\rho}, T_2/p^{\rho})$). 

To conclude the proof of Theorem \ref{TriangulationInFamilies}, we now note that
\[
\opn{H}_c^*(X_{\opn{GL}_3, K^p \opn{Iw}}, D_{\Omega^{\circ}}^{\circ,r\opn{-an}})^{\opn{ord}} \otimes_{\mathcal{O}^+(\Omega^{\circ})} \mathcal{O}(\Omega) \cong \opn{H}_c^*(X_{\opn{GL}_3, K^p \opn{Iw}}, D_{\Omega}^{r\opn{-an}})^{\opn{ord}}
\]
Hecke-equivariantly, and we have a natural map 
\begin{equation} \label{Eqn:OrdEtoIrredComp}
\mbf{T}^{S, \opn{ord}}_{\opn{GL}_3}(\opn{H}_c^*(X_{\opn{GL}_3, K^p \opn{Iw}}, D_{\Omega}^{r\opn{-an}})^{\opn{ord}}) \to \mathcal{O}(\Sigma)
\end{equation}
given by identifying $\Sigma$ as an irreducible component of $\mathscr{E} \times_{\mathcal{W}} \Omega$ passing through the point $x_{\pi}$. Here $\mathscr{E} \to \mathcal{W}$ is the $\opn{GL}(3)$-eigenvariety in \cite{Han17} constructed from (compactly-supported) overconvergent cohomology, and $x_{\pi} \in \mathscr{E}$ is the point corresponding to the unique $p$-stabilisation in $\pi$. By design, the map \eqref{Eqn:OrdEtoIrredComp} factors $\underline{\alpha}^S \otimes \underline{\alpha}_p$, and the induced map $\mbf{T}^{S, \opn{ord}}_{\opn{GL}_3}(\opn{H}_c^*(X_{\opn{GL}_3, K^p \opn{Iw}}, D_{\Omega^{\circ}}^{\circ,r\opn{-an}})^{\opn{ord}}) \to \mathcal{O}(\Sigma)$ factors through the localisation at $\ide{n}_0$ (because the specialisation of $\underline{\alpha}^S \otimes \underline{\alpha}_p$ at the trivial weight is the Hecke eigensystem associated with $\pi$).

\begin{remark}
  We have specialised to $\GL_{3}$ for our intended application; but following this method, one can also obtain results about triangulations in families for ordinary points on the $\GL_n$-eigenvariety, for general $n \geq 2$.
\end{remark}

\newcommand{\etalchar}[1]{$^{#1}$}
\renewcommand{\MR}[1]{}
\providecommand{\bysame}{\leavevmode\hbox to3em{\hrulefill}\thinspace}
\providecommand{\MR}{\relax\ifhmode\unskip\space\fi MR }
\providecommand{\MRhref}[2]{%
  \href{http://www.ams.org/mathscinet-getitem?mr=#1}{#2}
}
\providecommand{\href}[2]{#2}

\Addresses


\begin{thebibliography}{BLGGT14b}

\bibitem[ACC{\etalchar{+}}23]{10author}
Patrick~B. Allen, Frank Calegari, Ana Caraiani, Toby Gee, David Helm, Bao~V. Le~Hung, James Newton, Peter Scholze, Richard Taylor, and Jack~A. Thorne, \emph{Potential automorphy over {CM} fields}, Ann. of Math. (2) \textbf{197} (2023), no.~3, 897--1113. \MR{4564261}

\bibitem[APS08]{AshPollackStevens}
Avner Ash, David Pollack, and Glenn Stevens, \emph{Rigidity of {$p$}-adic cohomology classes of congruence subgroups of {$\mathrm{GL}(n,\mathbb{Z})$}}, Proc. Lond. Math. Soc. (3) \textbf{96} (2008), no.~2, 367--388. \MR{2396124}

\bibitem[Art13]{ArthurBook}
James Arthur, \emph{The endoscopic classification of representations}, American Mathematical Society Colloquium Publications, vol.~61, American Mathematical Society, Providence, RI, 2013, Orthogonal and symplectic groups. \MR{3135650}

\bibitem[ATI{\etalchar{+}}24]{LocalIntertwiningRelations}
Hiraku {Atobe}, Wee {Teck Gan}, Atsushi {Ichino}, Tasho {Kaletha}, Alberto {M{\'\i}nguez}, and Sug~Woo {Shin}, \emph{{Local Intertwining Relations and Co-tempered $A$-packets of Classical Groups}}, arXiv e-prints (2024), arXiv:2410.13504.

\bibitem[BDJ22]{BDJ17}
Daniel Barrera, Mladen Dimitrov, and Andrei Jorza, \emph{{$p$}-adic {$L$}-functions of {H}ilbert cusp forms and the trivial zero conjecture}, J. Eur. Math. Soc. (JEMS) \textbf{24} (2022), no.~10, 3439--3503. \MR{4432904}

\bibitem[Bei86]{Bei86}
A.~A. Beilinson, \emph{Higher regulators of modular curves}, Applications of algebraic {$K$}-theory to algebraic geometry and number theory, {P}art {I}, {II} ({B}oulder, {C}olo., 1983), Contemp. Math., vol.~55, Amer. Math. Soc., Providence, RI, 1986, pp.~1--34. \MR{862627}

\bibitem[Ben11]{Benois11}
Denis Benois, \emph{A generalization of {G}reenberg's {$\mathscr{L}$}-invariant}, Amer. J. Math. \textbf{133} (2011), no.~6, 1573--1632. \MR{2863371}

\bibitem[BG14]{BuzzardGee}
Kevin Buzzard and Toby Gee, \emph{The conjectural connections between automorphic representations and {G}alois representations}, Automorphic forms and {G}alois representations. {V}ol. 1, London Math. Soc. Lecture Note Ser., vol. 414, Cambridge Univ. Press, Cambridge, 2014, pp.~135--187. \MR{3444225}

\bibitem[BLGGT14a]{BGGTII}
Thomas Barnet-Lamb, Toby Gee, David Geraghty, and Richard Taylor, \emph{Local-global compatibility for {$l=p$}, {II}}, Ann. Sci. \'Ec. Norm. Sup\'er. (4) \textbf{47} (2014), no.~1, 165--179. \MR{3205603}

\bibitem[BLGGT14b]{BGGTweight}
\bysame, \emph{Potential automorphy and change of weight}, Ann. of Math. (2) \textbf{179} (2014), no.~2, 501--609. \MR{3152941}

\bibitem[BW19]{BW17}
Daniel {Barrera Salazar} and Chris Williams, \emph{Exceptional zeros and {$\mathcal{L}$}-invariants of {B}ianchi modular forms}, Trans. Amer. Math. Soc. \textbf{372} (2019), no.~1, 1--34. \MR{3968760}

\bibitem[BW21]{BW20}
\bysame, \emph{Parabolic eigenvarieties via overconvergent cohomology}, Math. Z. \textbf{299} (2021), no.~1-2, 961--995. \MR{4311626}

\bibitem[Car12]{Caraiani12}
Ana Caraiani, \emph{\href{https://doi.org/10.1215/00127094-1723706}{{L}ocal-global compatibility and the action of monodromy on nearby cycles}}, Duke Mathematical Journal \textbf{161} (2012), no.~12, 2311 -- 2413.

\bibitem[CGH{\etalchar{+}}20]{CGHJMRS}
Ana Caraiani, Daniel~R. Gulotta, Chi-Yun Hsu, Christian Johansson, Lucia Mocz, Emanuel Reinecke, and Sheng-Chi Shih, \emph{Shimura varieties at level {$\Gamma _1(p^\infty)$} and {G}alois representations}, Compos. Math. \textbf{156} (2020), no.~6, 1152--1230. \MR{4108870}

\bibitem[CH13]{ChenevierHarris}
Ga{\"e}tan Chenevier and Michael Harris, \emph{\href{https://dx.doi.org/10.4310/CJM.2013.v1.n1.a2}{Construction of automorphic {G}alois representations, {II}}}, Cambridge Journal of Mathematics \textbf{1} (2013), no.~1, 53--73.

\bibitem[Che14]{ChenevierPadic}
Ga\"etan Chenevier, \emph{The {$p$}-adic analytic space of pseudocharacters of a profinite group and pseudorepresentations over arbitrary rings}, Automorphic forms and {G}alois representations. {V}ol. 1, London Math. Soc. Lecture Note Ser., vol. 414, Cambridge Univ. Press, Cambridge, 2014, pp.~221--285. \MR{3444227}

\bibitem[Clo90]{Clo90}
Laurent Clozel, \emph{Motifs et formes automorphes: applications du principe de fonctorialit\'{e}}, Automorphic forms, {S}himura varieties, and {$L$}-functions, {V}ol. {I} ({A}nn {A}rbor, {MI}, 1988), Perspect. Math., vol.~10, Academic Press, Boston, MA, 1990, pp.~77--159. \MR{1044819}

\bibitem[CM09]{CM09}
Frank Calegari and Barry Mazur, \emph{Nearly ordinary {G}alois deformations over arbitrary number fields}, J. Inst. Math. Jussieu \textbf{8} (2009), no.~1, 99--177. \MR{2461903}

\bibitem[Coa89]{coates89}
John Coates, \emph{On {$p$}-adic {$L$}-functions attached to motives over {${\bf Q}$}. {II}}, Bol. Soc. Brasil. Mat. (N.S.) \textbf{20} (1989), no.~1, 101--112. \MR{1129081}

\bibitem[CPR89]{coatesperrinriou89}
John Coates and Bernadette Perrin-Riou, \emph{On {$p$}-adic {$L$}-functions attached to motives over {${\bf Q}$}}, Algebraic number theory, Adv. Stud. Pure Math., vol.~17, Academic Press, Boston, MA, 1989, pp.~23--54. \MR{1097608}

\bibitem[CS17]{CS17}
Ana Caraiani and Peter Scholze, \emph{\href{http://www.jstor.org/stable/26395700}{On the generic part of the cohomology of compact unitary {S}himura varieties}}, Annals of Mathematics \textbf{186} (2017), no.~3, 649--766.

\bibitem[Din19]{DingLinvariants}
Yiwen Ding, \emph{Simple {$\mathcal{L}$}-invariants for {$\mathrm{GL}_n$}}, Trans. Amer. Math. Soc. \textbf{372} (2019), no.~11, 7993--8042. \MR{4029688}

\bibitem[Eme17]{Eme17}
Matthew Emerton, \emph{Locally analytic vectors in representations of locally {$p$}-adic analytic groups}, Mem. Amer. Math. Soc. \textbf{248} (2017), no.~1175, iv+158. \MR{3685952}

\bibitem[FS98]{FSdecomposition}
Jens Franke and Joachim Schwermer, \emph{A decomposition of spaces of automorphic forms, and the {E}isenstein cohomology of arithmetic groups}, Math. Ann. \textbf{311} (1998), no.~4, 765--790. \MR{1637980}

\bibitem[Geh21]{AutomorphicLinvariants}
Lennart Gehrmann, \emph{Automorphic {${\mathscr{L}}$}-invariants for reductive groups}, J. Reine Angew. Math. \textbf{779} (2021), 57--103. \MR{4319061}

\bibitem[Gel75]{Gel75}
Stephen~S. Gelbart, \emph{Automorphic forms on ad\`ele groups}, Annals of Mathematics Studies, vol. No. 83, Princeton University Press, Princeton, NJ; University of Tokyo Press, Tokyo, 1975. \MR{379375}

\bibitem[Ger19]{Geraghty}
David Geraghty, \emph{Modularity lifting theorems for ordinary {G}alois representations}, Math. Ann. \textbf{373} (2019), no.~3-4, 1341--1427. \MR{3953131}

\bibitem[GJ78]{GelbartJacquet}
Stephen Gelbart and Herv\'e Jacquet, \emph{A relation between automorphic representations of {${\rm GL}(2)$}\ and {${\rm GL}(3)$}}, Ann. Sci. \'Ecole Norm. Sup. (4) \textbf{11} (1978), no.~4, 471--542. \MR{533066}

\bibitem[GR22]{GehrmannRosso}
Lennart Gehrmann and Giovanni Rosso, \emph{Big principal series, {$p$}-adic families and {$ \mathcal L$}-invariants}, Compos. Math. \textbf{158} (2022), no.~2, 409--436. \MR{4413750}

\bibitem[Gre94]{GreenbergTrivialZeros}
Ralph Greenberg, \emph{Trivial zeros of {$p$}-adic {$L$}-functions}, {$p$}-adic monodromy and the {B}irch and {S}winnerton-{D}yer conjecture ({B}oston, {MA}, 1991), Contemp. Math., vol. 165, Amer. Math. Soc., Providence, RI, 1994, pp.~149--174. \MR{1279608}

\bibitem[GS93]{GS93}
Ralph Greenberg and Glenn Stevens, \emph{$p$-adic {$L$}-functions and $p$-adic periods of modular forms}, Invent. Math. \textbf{111} (1993), 407 -- 447.

\bibitem[Han17]{Han17}
David Hansen, \emph{Universal eigenvarieties, trianguline {G}alois representations and $p$-adic {L}anglands functoriality}, J. Reine. Angew. Math. \textbf{730} (2017), 1--64.

\bibitem[Hau16]{Hauseux}
Julien Hauseux, \emph{Extensions entre s\'eries principales {$p$}-adiques et modulo {$p$} de {$G(F)$}}, J. Inst. Math. Jussieu \textbf{15} (2016), no.~2, 225--270. \MR{3480966}

\bibitem[{Hev}23]{HevesiOrdinaryParts}
Bence {Hevesi}, \emph{{Ordinary parts and local-global compatibility at $\ell=p$}}, arXiv e-prints (2023), arXiv:2311.13514.

\bibitem[{Hey}21]{SmoothParts}
Claudius {Heyer}, \emph{{{O}n the {S}mooth {P}art {F}unctor}}, arXiv e-prints (2021), arXiv:2108.05262.

\bibitem[Hid04]{HidaLinvariant}
Haruzo Hida, \emph{Greenberg's {$\mathscr{L}$}-invariants of adjoint square {G}alois representations}, Int. Math. Res. Not. (2004), no.~59, 3177--3189. \MR{2097038}

\bibitem[HL23]{HamannLee}
Linus {Hamann} and Si~Ying {Lee}, \emph{{Torsion Vanishing for Some Shimura Varieties}}, arXiv e-prints (2023), arXiv:2309.08705.

\bibitem[HLTT16]{HLTTrigid}
Michael Harris, Kai-Wen Lan, Richard Taylor, and Jack Thorne, \emph{On the rigid cohomology of certain {S}himura varieties}, Res. Math. Sci. \textbf{3} (2016), Paper No. 37, 308. \MR{3565594}

\bibitem[HT17]{HansenThorne}
David Hansen and Jack~A. Thorne, \emph{On the {$\mathrm{GL}_n$}-eigenvariety and a conjecture of {V}enkatesh}, Selecta Math. (N.S.) \textbf{23} (2017), no.~2, 1205--1234. \MR{3624909}

\bibitem[JN19a]{JoNewExt}
Christian Johansson and James Newton, \emph{Extended eigenvarieties for overconvergent cohomology}, Algebra Number Theory \textbf{13} (2019), no.~1, 93--158. \MR{3917916}

\bibitem[JN19b]{JoNew}
\bysame, \emph{Irreducible components of extended eigenvarieties and interpolating {L}anglands functoriality}, Math. Res. Lett. \textbf{26} (2019), no.~1, 159--201. \MR{3963980}

\bibitem[JPSS83]{JPSS-RankinSelberg}
Herv\'{e} Jacquet, Ilja~Iosifovitch Piatetski-Shapiro, and Joseph Shalika, \emph{Rankin--{S}elberg convolutions}, Amer. J. Math. \textbf{105} (1983), no.~2, 367--464. \MR{701565}

\bibitem[KS12]{KS12}
Jan Kohlhaase and Benjamin Schraen, \emph{Homological vanishing theorems for locally analytic representations}, Math. Ann. \textbf{353} (2012), no.~1, 219--258. \MR{2910788}

\bibitem[KS23]{KretShin}
Arno Kret and Sug~Woo Shin, \emph{Galois representations for general symplectic groups}, J. Eur. Math. Soc. (JEMS) \textbf{25} (2023), no.~1, 75--152. \MR{4556781}

\bibitem[KT17]{KT17}
Chandrashekhar~B. Khare and Jack~A. Thorne, \emph{Potential automorphy and the {L}eopoldt conjecture}, Amer. J. Math. \textbf{139} (2017), no.~5, 1205--1273. \MR{3702498}

\bibitem[LR20]{LR20}
Zheng Liu and Giovanni Rosso, \emph{Non-cuspidal {H}ida theory for {S}iegel modular forms and trivial zeros of {$p$}-adic {$L$}-functions}, Math. Ann. \textbf{378} (2020), no.~1-2, 153--231. \MR{4150915}

\bibitem[LS25]{LiuSunRelativeCC}
Dongwen {Liu} and Binyong {Sun}, \emph{{Relative completed cohomologies and modular symbols}}, arXiv e-prints (2025), arXiv:1709.05762.

\bibitem[LW12]{loeffler-weinstein}
David Loeffler and Jared Weinstein, \emph{On the computation of local components of a newform}, Math. Comp. \textbf{81} (2012), no.~278, 1179--1200, \url{https://doi.org/10.1090/S0025-5718-2011-02530-5}. \MR{2869056}

\bibitem[LW20]{LW18}
David Loeffler and Chris Williams, \emph{$p$-adic {A}sai {$L$}-functions of {B}ianchi modular forms}, Algebra Number Theory \textbf{14} (2020), no.~7, 1669--1710. \MR{4150247}

\bibitem[LW21]{LW21}
David {Loeffler} and Chris {Williams}, \emph{{P-adic L-functions for GL(3)}}, arXiv e-prints (2021), arXiv:2111.04535.

\bibitem[{McD}25]{McDonald-Trianguline}
Vaughan {McDonald}, \emph{{Eigenvarieties over CM fields and trianguline representations}}, arXiv e-prints (2025), arXiv:2504.18319.

\bibitem[Mok12]{MokLinvariant}
Chung~Pang Mok, \emph{{$\mathscr{L}$}-invariant of the adjoint {G}alois representation of modular forms of finite slope}, J. Lond. Math. Soc. (2) \textbf{86} (2012), no.~2, 626--640. \MR{2980928}

\bibitem[MTT86]{MTT86}
Barry Mazur, John Tate, and Jeremy Teitelbaum, \emph{On $p$-adic analogues of the {B}irch and {S}winnerton-{D}yer conjecture}, Invent. Math. \textbf{84} (1986), 1 -- 48.

\bibitem[NT16]{NewtonThorneTorsion}
James Newton and Jack~A. Thorne, \emph{Torsion {G}alois representations over {CM} fields and {H}ecke algebras in the derived category}, Forum Math. Sigma \textbf{4} (2016), Paper No. e21, 88. \MR{3528275}

\bibitem[NT21]{NewtonThorneSymII}
\bysame, \emph{Symmetric power functoriality for holomorphic modular forms, {II}}, Publ. Math. Inst. Hautes \'Etudes Sci. \textbf{134} (2021), 117--152. \MR{4349241}

\bibitem[Pan22]{Pan22}
Lue Pan, \emph{On locally analytic vectors of the completed cohomology of modular curves}, Forum Math. Pi \textbf{10} (2022), Paper No. e7, 82. \MR{4390302}

\bibitem[Ram14]{RamakrishnanSelfDual}
Dinakar Ramakrishnan, \emph{An exercise concerning the selfdual cusp forms on {$\mathrm{GL}(3)$}}, Indian J. Pure Appl. Math. \textbf{45} (2014), no.~5, 777--785. \MR{3286086}

\bibitem[Ree90]{Reeder}
Mark Reeder, \emph{Modular symbols and the {S}teinberg representation}, Cohomology of arithmetic groups and automorphic forms ({L}uminy-{M}arseille, 1989), Lecture Notes in Math., vol. 1447, Springer, Berlin, 1990, pp.~287--302. \MR{1082970}

\bibitem[RJW25]{ENT-notes}
Joaqu\'in Rodrigues~Jacinto and Chris Williams, \emph{An introduction to p-adic {L}-functions}, Essent. Number Theory \textbf{4} (2025), no.~1, 101--216. \MR{4888596}

\bibitem[Ros15]{RossoSym2II}
Giovanni Rosso, \emph{Derivative of symmetric square {$p$}-adic {$L$}-functions via pull-back formula}, Arithmetic and geometry, London Math. Soc. Lecture Note Ser., vol. 420, Cambridge Univ. Press, Cambridge, 2015, pp.~373--400. \MR{3467131}

\bibitem[Ros16]{RossoSym2}
\bysame, \emph{A formula for the derivative of the {$p$}-adic {$L$}-function of the symmetric square of a finite slope modular form}, Amer. J. Math. \textbf{138} (2016), no.~3, 821--878. \MR{3506387}

\bibitem[Sch94]{Schwermer94}
Joachim Schwermer, \emph{Eisenstein series and cohomology of arithmetic groups: the generic case}, Invent. Math. \textbf{116} (1994), no.~1-3, 481--511. \MR{1253202}

\bibitem[Sch15]{ScholzeTorsion}
Peter Scholze, \emph{\href{https://doi.org/10.4007/annals.2015.182.3.3}{On torsion in the cohomology of locally symmetric varieties}}, Ann. of Math. (2) \textbf{182} (2015), no.~3, 945--1066. \MR{3418533}

\bibitem[Ski16]{SkinnerSplitMult}
Christopher Skinner, \emph{Multiplicative reduction and the cyclotomic main conjecture for {${\rm GL}_2$}}, Pacific J. Math. \textbf{283} (2016), no.~1, 171--200. \MR{3513846}

\bibitem[Spi14]{Spi14}
Michael Spie{\ss}, \emph{On special zeros of $p$-adic {$L$}-functions of {H}ilbert modular forms}, Invent. Math. \textbf{196 (1)} (2014), 69 -- 138.

\bibitem[ST02]{ST02}
P.~Schneider and J.~Teitelbaum, \emph{Banach space representations and {I}wasawa theory}, Israel J. Math. \textbf{127} (2002), 359--380. \MR{1900706}

\bibitem[ST03]{ST03}
Peter Schneider and Jeremy Teitelbaum, \emph{Algebras of {$p$}-adic distributions and admissible representations}, Invent. Math. \textbf{153} (2003), no.~1, 145--196. \MR{1990669}

\bibitem[SU14]{SU14}
Christopher Skinner and Eric Urban, \emph{The {I}wasawa {M}ain {C}onjectures for {GL}$(2)$}, Invent. Math. \textbf{195 (1)} (2014), 1--277.

\bibitem[Tad94]{Tad94}
Marko Tadi\'c, \emph{Representations of {$p$}-adic symplectic groups}, Compositio Math. \textbf{90} (1994), no.~2, 123--181. \MR{1266251}

\bibitem[Ta{\"{i}}16]{TaibiEigenvarieties}
Olivier Ta{\"{i}}bi, \emph{Eigenvarieties for classical groups and complex conjugations in {G}alois representations}, Math. Res. Lett. \textbf{23} (2016), no.~4, 1167--1220. \MR{3554506}

\bibitem[VW21]{VW19}
Guhan Venkat and Chris Williams, \emph{Stark-{H}eegner cycles attached to {B}ianchi modular forms}, J. Lond. Math. Soc. (2) \textbf{104} (2021), no.~1, 394--422. \MR{4313251}

\bibitem[Wed08]{WedLLC}
Torsten Wedhorn, \emph{The local {L}anglands correspondence for {${\rm GL}(n)$} over {$p$}-adic fields}, School on {A}utomorphic {F}orms on {${\rm GL}(n)$}, ICTP Lect. Notes, vol.~21, Abdus Salam Int. Cent. Theoret. Phys., Trieste, 2008, pp.~237--320. \MR{2508771}

\bibitem[YZ25]{YangZhu}
Xiangqian {Yang} and Xinwen {Zhu}, \emph{{On the generic part of the cohomology of Shimura varieties of abelian type}}, arXiv e-prints (2025), arXiv:2505.04329.

\end{thebibliography}
\end{document}